\documentclass[11pt,reqno]{amsart}
\usepackage[dvipsnames,svgnames,table]{xcolor}
\usepackage{amssymb,amsfonts,amsmath,amscd,mathrsfs}
\usepackage{mathtools,verbatim,enumitem,url}
\usepackage{fancybox,framed}
\usepackage[all]{xy}
\usepackage{tikz-cd}
\usepackage{hyperref}

\allowdisplaybreaks[4]

\newcommand{\cred}{\color{Red}}
\newcommand{\cgreen}{\color{Green}}
\newcommand{\cblue}{\color{Blue}}
\newcommand{\corange}{\color{Orange}}
\newcommand{\cteal}{\color{Teal}}
\newcommand{\colive}{\color{Olive}}
\newcommand{\cviolet}{\color{Violet}}

\newtheorem{theorem}{Theorem}[section]
\newtheorem{lemma}[theorem]{Lemma}     
\newtheorem{corollary}[theorem]{Corollary}
\newtheorem{proposition}[theorem]{Proposition}
\newtheorem*{theorem*}{Main Theorem}
\newtheorem*{theoremb}{Theorem}
\newtheorem*{theoremu}{Theorem~\ref{FirstStep}}
\newtheorem*{theoremdos}{Theorem~\ref{SecondStep}}
\newtheorem*{theoremtres}{Theorem~\ref{ThirdStep}}
\newtheorem*{theoremquatre}{Theorem~\ref{FourthStep}}
\newtheorem*{theoremcinc}{Theorem~\ref{FifthStep}}
\numberwithin{equation}{section}

\theoremstyle{definition}
\newtheorem{definition}[theorem]{Definition}
\newtheorem{remark}[theorem]{Remark}
\newtheorem{summary}[theorem]{Summary}
\newtheorem{example}[theorem]{Example}
\newtheorem{notation}[theorem]{Notation}

\newcommand{\mbk}{\mbox{$\Bbbk$}}

\newcommand{\mbn}{\mbox{$\mathbb{N}$}}

\newcommand{\mbq}{\mbox{$\mathbb{Q}$}}

\newcommand{\mbz}{\mbox{$\mathbb{Z}$}}

\newcommand{\mcc}{\mbox{$\mathcal{C}$}}

\newcommand{\mcb}{\mbox{$\mathcal{B}$}}

\newcommand{\mfm}{\mbox{$\mathfrak{m}$}}
\newcommand{\mfn}{\mbox{$\mathfrak{n}$}}
\newcommand{\mfp}{\mbox{$\mathfrak{p}$}}

\newcommand{\msb}{\mbox{$\mathscr{B}$}}
\newcommand{\msc}{\mbox{$\mathscr{C}$}}
\newcommand{\msd}{\mbox{$\mathscr{D}$}}

\newcommand{\la}{\mbox{$\mathscr{L}_a$}}

\newcommand{\lser}{\mbox{$[\mkern-1.5mu [ $}}
\newcommand{\rser}{\mbox{$]\mkern-1.5mu ] $}}
\newcommand{\mon}{\mbox{${\rm m}$}}
\newcommand{\monx}{\mbox{$\underline{\rm x}$}}

\newcommand{\ulf}{\operatorname{ULF}}
\newcommand{\rem}{\varepsilon}
\newcommand{\nums}{\mbox{$\mathcal{S}$}}
\newcommand{\fac}{\mbox{${\rm F}$}}
\newcommand{\frob}{\mbox{${\rm Frob}$}}
\newcommand{\setl}{\operatorname{L}}
\newcommand{\good}{\mbox{${\rm G}$}}
\newcommand{\ord}{\mbox{${\rm ord}$}}

\newcommand{\sord}{\mbox{${\rm ord}_{\sigma}$}}
\newcommand{\lform}{\mbox{${\rm LF}$}}

\newcommand{\supp}{\mbox{${\rm Supp}$}}

\newcommand{\im}{\mbox{${\rm Im}$}}
\newcommand{\height}{\mbox{${\rm ht}$}}
\newcommand{\grade}{\mbox{${\rm grade}$}}
\newcommand{\depth}{\mbox{${\rm depth}$}}
\newcommand{\length}{\mbox{${\rm length}$}}
\newcommand{\rank}{\mbox{${\rm rank}$}}
\newcommand{\card}{\mbox{${\rm card}$}}

\begin{document}
\title[Explicit minimal generating sets of prime ideals]
{Explicit minimal generating sets of a family of prime
ideals with unbounded minimal number of generators 
in a three-dimensional power series ring}

\author[L. Gonz\'alez]{Laura Gonz\'alez} 
\author[F. Planas-Vilanova]{Francesc Planas-Vilanova}

\address{Departament de Matem\`atiques, Universitat Polit\`ecnica de
Catalunya. Diagonal 647, ETSEIB, E-08028 Barcelona.} 
\email{laura.gonzalez.hernandez@upc.edu}
\email{francesc.planas@upc.edu} 
\thanks{Both authors are partially supported by Grant PID2023-146936NB-I00 financed by the Spanish State Agency MCIN/AEI/10.13039/501100011033/ FEDER, UE and AGAUR project 2021 SGR 00603}


\subjclass[2020]{13E15, 13H05, 13J05, 13P10, 14H50}

\keywords{Power series ring, prime ideal, minimal generating set, monomial curve}


\begin{abstract}
We display a new family of prime ideals with unbounded minimal number of generators
in a three-dimensional power series ring over a field of characteristic zero.
These primes are obtained as the kernel of a quasi-monomial algebra homomorphism. 
Up to constant coefficients, determined by some specific linear systems with binomial entries, 
we describe their minimal generating polynomial sets. The advantage of our family with respect to some previous work is, on the one hand, the explicit description of the generating sets and, 
on the other hand, the simplicity of the exponents of the aforementioned quasi-monomial homomorphism. 
We also provide a code in Python which states and solves the linear systems that 
lead to a complete description of the minimal generating sets 
with a ``Gr\"obner-free'' approach. 
\end{abstract}

\maketitle

\tableofcontents

\section{Introduction}\label{sec-intro}

Let $R=\mbk\lser x,y,z\rser$ be the power series ring in the variables $x,y,z$ over a field $\mbk$ of characteristic zero. In 1974, T.T. Moh \cite{moh1} inspired by 
the work of S.S. Abhyankar on Macaulay's primes \cite{abhyankar} 
and by the research of J. Herzog \cite{herzog} (see also \cite[p.~137]{kunz})
on three-dimensional affine monomial curves, considered the following family of prime ideals in $R$: let $n\geq 1$ be an odd integer, $m=(n+1)/2$ and $\lambda$ an integer greater than $n(n+1)m$, with $\gcd(\lambda,m)=1$; let $\rho_n:\mbk\lser x,y,z\rser\to\mbk\lser t\rser$ be the $\mbk$-algebra ``quasi-monomial'' homomorphism defined by 
$\rho_n(x)=t^{nm}(1+t^{\lambda})$, $\rho_n(y)=t^{(n+1)m}$ and 
$\rho_n(z)=t^{(n+2)m}$ and let $Q_n=\ker(\rho_n)$. Then he showed that the minimal number of generators of $Q_n$ is $\mu(Q_n)\geq n+1$. In his final remarks in \cite{moh1}, Moh writes that,  
when looking for a three-dimensional irreducible 
affine curve which is not set theoretically a complete intersection, if there is any, it would seem reasonable to start with examples of prime ideals which are ``highly non-complete intersection''. 

Shortly after, in \cite{moh2}, Moh himself shows that, in fact, $\mu(Q_n)=n+1$. 
Observe that, for each $n$, $Q_n$ is a perfect prime ideal of height $2$, hence, by the Hilbert-Burch Theorem (see, e.g., \cite[\S~1.4]{bh}), $Q_n$ is the ideal of the $n\times n$ minors of a $n\times (n+1)$ matrix with entries in $R$. The structure of this matrix for the leading forms of the generators of $Q_n$ is given in \cite{moh2}. 
In addition, in both papers, \cite{moh1}, \cite{moh2}, Moh states that as a consequence of the existence of $Q_n$ in $R$, one can deduce the same result in the polynomial ring $A=\mbk[x,y,z]$. 

J. Sally in her book \cite[p. 58 and ff.]{sally} gives an excellent overview on the work of Moh and
suggests that the limitation to characteristic zero can be avoided a posteriori. 
Moh in \cite{moh2} quotes the work of J. Sally as an ``interesting argument to drop the restriction on the characteristic of the ground field $\mbk$''. However, in \cite{gp1}, 
the authors of the present paper show that the minimal number of generators of $Q_n$ does depend on the characteristic of the ground field and, in fact, may decrease.

In 1980, in an unpublished work \cite{maurer}, J. Maurer affirms that Moh finds inductively the power series which generate $Q_n$, but that it is not clear whether they are or not polynomials. 
This is a relevant point, for instance, in order to deduce the existence of such primes in the polynomial ring $A=\mbk[x,y,z]$. Maurer slightly modifies the prime ideals of Moh by considering the $\mbk$-algebra homomorphism
$\rho_n:\mbk\lser x,y,z\rser\to\mbk\lser t\rser$ defined by $\rho_n(x)=t^{a_n}$, 
$\rho_n(y)=t^{b_n}$ and $\rho_n(z)=t^{c_n}(1+t)$, where $(a_n,b_n,c_n)=\frac{n+1}{2}(n,n+1,n+2)$ if $n$ is odd, and $(a_n,b_n,c_n)=\frac{n}{2}(n+1,n+2,n+3)$ if $n$ is even. Then he states that
$\ker(\rho_n)$ is generated by the $n+1$ polynomials obtained as the $n\times n$ minors of an explicit, up to constants to be found, $n\times (n+1)$ matrix with entries in $A=\mbk[x,y,z]$. We are convinced that the extraordinary work of Maurer, as well as those of Moh, will give birth to new ideas and results. 
We would like to thank Anna Oneto for providing us with a copy of the work of Maurer. 

Perhaps in this direction, recently, Mehta, Saha and Sengupta \cite{mss} give an iterative method for finding a set of polynomials that minimally generates the prime ideals $Q_n$ of Moh, though, the authors say, ``it is not very clear to us whether they are the same as those conceived in Moh's work or whether there is any underlying determinantal structure''. 

As said by Moh, obtaining families of prime ideals in $R=\mbk\lser x,y,z\rser$ or $A=\mbk[x,y,z]$ which need an arbitrarily large number of generators may lead to help to understand the open problem of whether or not a 
three-dimensional affine curve is set-theoretically a complete intersection. In the same vein, it may help to face the open problem stated by Sally in \cite[Remark~p. 53]{sally} on whether a uniform bound
on the number of generators of prime ideals implies that the ring has dimension at most two. Or to undertake the apparently more simple conjecture stated by Shimoda on whether a noetherian local ring such that all its prime ideals different from the maximal ideal are complete intersection, has dimension at most two (see, e.g., \cite{gop1}). 

Following the suggestion of Moh and some of his ideas, 
we consider here the following family of prime ideals $P_n$. 
Let $n$ be an integer, $n\geq 3$, $a=(n-1)n/2$ and $q=2\lfloor(n+1)/2\rfloor-1$.
Let $\rho_n:\mbk\lser x,y,z\rser\to \mbk\lser t\rser$ be the
$\mbk$-algebra ``quasi-monomial'' homomorphism defined by $\rho_n(x)=t^a(1+t^q)$,
$\rho_n(y)=t^{a+1}$, $\rho_n(z)=t^{a+2}$. We prove that the prime ideal 
$P_n=\ker(\rho_n)$ is minimally generated by an explicitly given, up to constants, set of $n$ polynomials
$f_1,\ldots,f_n$ (see the next section for the definition of the $f_i$ and the statement of the main result).
Moreover, the constant coefficients of each $f_i$ are just found by solving at most four determined consistent linear systems of size at most $n$. We provide the reader with a code in Python to help in finding $f_1,\ldots,f_n$ (see Remark~\ref{code}). 
We emphasize that there is no need to use a Gr\"obner basis approach to calculate the polynomials $f_1,\ldots,f_n$, which clearly shortens the time to obtain
the minimal generating sets of each $P_n$, especially for higher integers $n$. 

The advantage of our primes with respect to those given by Moh 
resides in that we can explicitly find their minimal generating sets. With respect to Maurer's examples we explicit the linear systems that have to be solved to find the coefficients of the generators.
Another advantage with respect to both the primes of Moh and Maurer is in the exponents of the homomorphism $\rho_n$. Namely, the exponents $nm$, $(n+1)m$ and $(n+2)m$ in Moh's examples and the exponents $(a_n,b_n,c_n)$ in Maurer's examples are not coprime, whereas our exponents $(a,a+1,a+2)$ are coprime. In particular, the set 
$\{a,a+1,a+2\}$ generates a numerical semigroup, whose structure will be crucial to our purposes. 
On the other hand, having coprime
exponents, we believe, may be crucial in extending these results to a noetherian, non necessarily regular, local ring (see, e.g., the argument in \cite[page~4]{planas}). 
Here we thank the software system {\sc Singular} \cite{dgps} for serving as an excellent source of inspiration. In particular, we would like to highlight the merit of the previous work of Moh and Maurer, when no such software was available.  
Since $P_n$ is minimally generated by $n$ polynomials $f_1,\ldots,f_n$,
for each $n\geq 3$, we can deduce the existence of 
prime ideals in $\mbk[x,y,z]$ and 
$\mbk[x,y,z]_{(x,y,z)}$ minimally generated by at least $n$ elements
(see Corollary~\ref{cor-inA}). As said before, and in our opinion, this aspect seems 
insufficiently justified in Moh's papers. 
With respect to the question of Sally, we deduce that a regular local ring containing a field of 
characteristic zero has, for each integer $n\geq 3$, 
a prime ideal minimally generated by $n$ elements (see Corollary~\ref{cor-Sally}). 

Some questions should be addressed in the future. 
The first question would be to find the syzygy Hilbert-Burch matrix to the polynomials 
$f_1,\ldots,f_n$. The second problem is related to 
the characteristic of the ground field: 
we hope that our family contains subfamilies with unbounded minimal 
number of generators in characteristic other than zero. 
Some work has been done in this direction in characteristic $3$. An affirmative answer
would lead to the solution of the question of Sally in equicharacteristic regular local rings. Thirdly, we would like to prove (or disprove) that our primes are set-theoretically a complete intersection. 

The proof of our main result is based, roughly speaking, in three pillars. 
The first pillar is a stratification of the elements of the numerical semigroup $\nums=\langle a,a+1,a+2\rangle$ which follows from a previous work in \cite{ggp}. The second pillar
is \cite[Theorem~3.2]{gp1}, an extension of Moh's \cite[Theorem~4.3]{moh1}, 
in which subsets of linearly independent elements 
of some vector spaces $V_r$, $r\in\nums$, related to the leading forms of the elements of $P_n$, 
determine parts of a minimal generating set. Once the candidates $f_1,\ldots,f_n$
of a minimal generating set of $P_n$ are found so that $(f_1,\ldots,f_n)\subseteq P_n$, 
the third pillar is the calculation of the length of the ideals $(f_1,\ldots,f_n,z)$ and $(P_n,z)$ 
to deduce the equality $(f_1,\ldots,f_n)=P_n$.

An overview of the whole paper is given in the next section. 
\section{Statement of the main result and overview of the paper}\label{sec-main}

In this section we state the main result while, at the same time, 
we introduce the indispensable notations to follow the paper.
Subsequently, we give a general overview of the whole article. 

\begin{notation}\label{not-fac}
Throughout, $n$ is an integer, $n\geq 3$, and $a=(n-1)n/2$, so that $a$ is
an integer, $a\geq 3$.  Let $q=2\lfloor(n+1)/2\rfloor-1$, where $\lfloor \lambda\rfloor$ stands for
the floor function of $\lambda$, the 
largest integer less than or equal to $\lambda\in\mbq$. 
Let $s_0=a(n-1)+2$. This integer, we will see, will play an important role.

Let $\mbk$ be a field
of characteristic zero. Let $\mbk\lser x,y,z\rser$ and $\mbk\lser
t\rser$ be the power series rings in three variables $x,y,z$ and in
one variable $t$, respectively, with coefficients in $\mbk$. 

Let $\nums=\langle a,a+1,a+2\rangle$ be the numerical
semigroup generated by $a,a+1,a+2$. Its Frobenius number, i.e.,
the largest integer that does not belong to $\nums$, is $\frob(\nums)=\lfloor a/2\rfloor a-1$.
For $r\in\mbn$, the set of natural numbers, including zero, let
\begin{eqnarray*}
\fac(r,\nums)=\{\alpha=(\alpha_1,\alpha_2,\alpha_3)\in\mbn^3\mid\alpha_1 a+\alpha_2(a+1)+\alpha_3(a+2)=r\}
\end{eqnarray*}
be the set of factorizations of $r$. If $r\not\in\nums$, then $\fac(r,\nums)=\emptyset$. 
Let $\mon^{\alpha}=x^{\alpha_1}y^{\alpha_1}z^{\alpha_3}$ denote the monomial in $\monx=x,y,z$ with 
multi-index exponent $\alpha=(\alpha_1,\alpha_2,\alpha_3)\in\mbn^3$. The degree of $\mon^{\alpha}$
is defined as $\deg(\mon^\alpha)=|\alpha|=\alpha_1+\alpha_2+\alpha_3$. 
Any power series of $\mbk\lser\monx\rser$ can be written as $f=\sum_{\alpha\in\mathbb{N}^3}\lambda_{\alpha}\mon^{\alpha}$, with
$\lambda_{\alpha}\in\mbk$. The support of $f$, $f\neq 0$, is $\supp(f)=\{\alpha\in\mbn^3\mid\lambda_{\alpha}\neq 0\}$.  The order of $f$, $f\neq 0$, is $\ord(f)=\min\{|\alpha|\in\mbn\mid \alpha\in\supp(f)\}$. 
The leading form of $f$, $f\neq 0$, is $\lform(f)=
\sum_{|\alpha|={\rm ord}(f)}\lambda_{\alpha}\mon^{\alpha}$; 
$f$ is said to be homogeneous if $f=\lform(f)$. 

Let $\sigma:\mbk\lser\monx\rser\rightarrow\mbk\lser\monx\rser$ be the $\mbk$-algebra homomorphism defined by 
$\sigma(x)=x^{a}$, $\sigma(y)=y^{a+1}$, $\sigma(z)=z^{a+2}$. Given 
$f=\sum_{\alpha\in\mathbb{N}^d}\lambda_{\alpha}\mon^\alpha
\in \mbk\lser\monx\rser$, $f\neq 0$, the $\sigma$-order of $f$ is defined as 
$\sord(f)=\ord(\sigma(f))$. Note that $\sord(f)\in\nums$. 
The $\sigma$-leading form of $f$ is
$f^{\sigma}=\sigma^{-1}(\lform(\sigma(f)))=
\sum_{\alpha\in{\rm F}(r,\mathcal{S})}\lambda_{\alpha}\mon^{\alpha}$, where $r=\sord(f)$; $f$ is
said to be $\sigma$-homogeneous if $\sigma(f)$ is homogeneous, that is, if
$f=f^{\sigma}$. Let $f^{\tau}=f-f^{\sigma}$, which is called the tail of $f$. 
\end{notation}

\begin{notation}\label{notation-binomials}
Let $i,j,k,l\in\mbn$. Let $[k,l]=\{k,k+1,\ldots,l\}$ stand for the set of consecutive natural
numbers from $k$ to $l$. When $k>l$, we understand $[k,l]=\emptyset$. 
Let $B=(b_{i,j})_{i,j\geq 0}$ be the infinite matrix of
binomial coefficients, where $b_{i,j}=\binom{i}{j}$.
If $j>i$, set $b_{i,j}=0$. Rows and columns in $B$ are counted starting from zero.
If $I=\{i_1,\ldots,i_k\}$ and $J=\{j_1,\ldots,j_l\}$
are two subsets of integers, with $0\leq
i_1<\ldots<i_k$ and $0\leq j_1<\ldots<j_l$, 
let $B^I_J$ denote the $k\times l$ submatrix of the
binomial matrix $B$ defined by the (intersection of the) rows $i_1,\ldots,i_k$ 
and the columns $j_1,\ldots,j_l$ of $B$. Let $P^J_I=\left(B^I_J\right)^{\top}\cdot\Theta_{k}$, 
where $\Theta_k$ is the $k\times k$ exchange matrix, i.e., 
the anti-diagonal matrix with 1's on the diagonal going 
from the lower left corner to the upper right corner. Concretely, 
\begin{eqnarray*}
B^{I}_{J}=\left(\begin{array}{ccc}
b_{i_{1},j_1}&\ldots&b_{i_{1},j_{l}}\\\vdots&&\vdots\\
b_{i_{k},j_1}&\ldots&b_{i_{k},j_{l}}
\end{array}\right)\;\mbox{ and }\;
P_{I}^{J}=\left(\begin{array}{ccc}
  b_{i_{k},j_1}&\ldots&b_{i_{1},j_1}
  \\ \vdots&&\vdots\\b_{i_{k},j_{l}}&\ldots&
  b_{i_{1},j_{l}}
\end{array}\right).
\end{eqnarray*}
Loosely speaking, $P^J_I$ is obtained by applying a ``90 degree clockwise
rotation'' to $B^{I}_{J}$. The last row of $B^{I}_{J}$ becomes the first column of $P^J_I$, 
and so on. If $k=l$, let $b^I_J=\det(B^I_J)$ 
be the binomial determinant defined by rows $I$ and columns $J$ (see, e.g., \cite{gv} or \cite{gp2}).
\end{notation}

\begin{definition}\label{def-gk}
For each $n\geq 3$, let $g_1,\ldots,g_n$ and $f_1,\ldots,f_n$ be the polynomials in $\mbk[x,y,z]$ defined as follows.

\vspace*{0.2cm}

\noindent \underline{Suppose first that $n$ is odd}.

\vspace*{0.2cm}

\noindent $\bullet$ If $n=3$, let 
\begin{eqnarray}\label{equality-gfn=3}
\begin{array}{lcl}
g_1=xz-y^2 &\mbox{ and }&f_1=g_1-xy^2+yz^2;\\
g_2=x^3-yz&\mbox{ and }&f_2=g_2-3xyz+z^3-x^2yz+xz^3;\\
g_3=x^2y-z^2&\mbox{ and }&f_3=g_3-y^2z-xz^2. 
\end{array}
\end{eqnarray}

\vspace*{0.2cm}

\noindent $\bullet$ If $n=5$, let 
\begin{eqnarray}\label{equality-gfn=5}
\begin{array}{lcl}
g_1=x^3z-x^2y^2&\mbox{ and }&f_1=g_1-yz^3-2x^3y^2+6xyz^3-y^3z^2+2y^5z;\\
g_2=x^2yz-xy^3&\mbox{ and }&f_2=g_2-z^4-x^2y^3+xz^4+y^2z^3;\\
g_3=x^2z^2-2xy^2z +y^4&\mbox{ and }&f_3=g_3-xy^4+yz^4;\\
g_4=xyz^2-y^3z&\mbox{ and }&f_4=g_4-x^5+10xy^3z-5y^5 +x^2y^3z +9xy^5 -6y^2z^4;\\
g_5=xz^3-y^2z^2&\mbox{ and }&f_5=g_5-x^4y+6xy^2z^2-2y^4z+xy^4z+3y^6.\\
\end{array}
\end{eqnarray}

\vspace*{0.2cm}

\noindent \underline{Suppose that $n$ is odd and $n\geq 7$}.

\vspace*{0.2cm}

\noindent $\bullet$ If $1\leq k\leq n-4$, with $k$ odd, let
\begin{eqnarray}\label{equality-gknoddkodd}
g_k=\sum_{j=0}^{\frac{k+1}{2}}(-1)^j
b^{\left[n-k-2,\frac{2n-k-3}{2}\right]\setminus\{\frac{2n-k-3-2j}{2}\}}
_{\{0\}\cup\left[ \frac{n-k}{2},\frac{n-3}{2}\right]}\cdot 
x^{\frac{2n-k-3-2j}{2}}y^{2j}z^{\frac{k+1-2j}{2}} 
\end{eqnarray}
and, if $\lambda_{l,j}$ are the unique solutions of the determined consistent linear systems in 
Lemma~\ref{lemma-h-i}, let
\begin{multline}\label{poly-fknoddkodd}
f_k=g_k+
\sum_{j=0}^{\frac{n-k-4}{2}}\lambda_{1,j}x^{\frac{n-k-4-2j}{2}}y^{1+2j}z^{\frac{n+k-2j}{2}}+
\sum_{j=1}^{\frac{k+1}{2}}\lambda_{2,j}x^{\frac{2n-k-1-2j}{2}}y^{2j}z^{\frac{k+1-2j}{2}}+\\
\sum_{j=0}^{\frac{n-k-2}{2}}\lambda_{3,j}x^{\frac{n-k-2-2j}{2}}y^{1+2j}z^{\frac{n+k-2j}{2}}
+\sum_{j=2}^{\frac{n-k}{2}}
\lambda_{4,j}x^{\frac{n-k-2j}{2}}y^{1+2j}z^{\frac{n+k+2j}{2}}.
\end{multline}

\vspace*{0.2cm}

\noindent $\bullet$ If $2\leq k\leq n-3$, with $k$ even, let
\begin{eqnarray}\label{equality-gknoddkeven}
g_k=\sum_{j=0}^{\frac{k}{2}}(-1)^j
b^{\left[n-k-2,\frac{2n-k-4}{2}\right]\setminus\{\frac{2n-k-4-2j}{2}\}}
_{\{0\}\cup\left[\frac{n-k+1}{2},\frac{n-3}{2}\right]}\cdot 
x^{\frac{2n-k-4-2j}{2}}y^{1+2j}z^{\frac{k-2j}{2}}
\end{eqnarray}
and, if $\lambda_{l,j}$ are the unique solutions of the determined consistent linear systems in 
Lemma~\ref{lemma-h-i+1}, let
\begin{multline}\label{poly-fknoddkeven}
f_k=g_k+
\sum_{j=0}^{\frac{n-k-3}{2}}\lambda_{1,j}x^{\frac{n-k-3-2j}{2}}y^{2j}z^{\frac{n+k+1-2j}{2}}+
\sum_{j=1}^{\frac{k}{2}}
\lambda_{2,j}x^{\frac{2n-k-2-2j}{2}}y^{1+2j}z^{\frac{k-2j}{2}}+\\
\sum_{j=0}^{\frac{n-k-1}{2}}\lambda_{3,j}x^{\frac{n-k-1-2j}{2}}y^{2j}z^{\frac{n+k+1-2j}{2}}+
\sum_{j=3}^{\frac{n-k+1}{2}}\lambda_{4,j}x^{\frac{n-k+1-2j}{2}}y^{2j}z^{\frac{n+k+1+2j}{2}},
\end{multline}
where the last summation is taken void if $k=n-3$. 

\vspace*{0.2cm}

\noindent $\bullet$ If $k=n-2$, let
\begin{eqnarray}\label{equality-gk00}
g_{n-2}=\sum_{j=0}^{\frac{n-1}{2}}(-1)^jb_{\frac{n-1}{2},\frac{n-1-2j}{2}}\cdot 
x^{\frac{n-1-2j}{2}}y^{2j}z^{\frac{n-1-2j}{2}} 
\end{eqnarray}
and 
\begin{eqnarray}\label{poly-fk00}
f_{n-2}=g_{n-2}-xy^{n-1}+yz^{n-1}.
\end{eqnarray}

\vspace*{0.2cm}

\noindent $\bullet$ If $k=n-1$, let
\begin{eqnarray}\label{equality-gk21}
g_{n-1}=\sum_{j=0}^{\frac{n-3}{2}}(-1)^jb_{\frac{n-3}{2},\frac{n-3-2j}{2}}\cdot 
x^{\frac{n-3-2j}{2}}y^{1+2j}z^{\frac{n-1-2j}{2}}
\end{eqnarray}
and, if $\lambda_{l,j}$ are the unique solutions of the determined consistent linear systems in 
Lemma~\ref{lemma-h-21}, let
\begin{multline}\label{poly-fk21}
f_{n-1}=g_{n-1}-x^n+
\sum_{j=1}^{\frac{n-1}{2}}\lambda_{2,j}x^{\frac{n-1-2j}{2}}y^{1+2j}z^{\frac{n-1-2j}{2}}+\\
\sum_{j=1}^{\frac{n-1}{2}}\lambda_{3,j}x^{\frac{n+1-2j}{2}}y^{1+2j}z^{\frac{n-1-2j}{2}}+
\lambda_{4,1}y^2z^{n-1}.
\end{multline}

\vspace*{0.2cm}

\noindent $\bullet$ If $k=n$, let
\begin{eqnarray}\label{equality-gk22}
g_n=\sum_{j=0}^{\frac{n-3}{2}}(-1)^jb_{\frac{n-3}{2},\frac{n-3-2j}{2}}\cdot 
x^{\frac{n-3-2j}{2}}y^{2j}z^{\frac{n+1-2j}{2}}
\end{eqnarray}
and, if $\lambda_{l,j}$ are the unique solutions of the determined consistent linear systems in 
Lemma~\ref{lemma-h-22}, let
\begin{eqnarray}\label{poly-fk22}
f_{n}=g_{n}-x^{n-1}y+
\sum_{j=1}^{\frac{n-1}{2}}\lambda_{2,j}x^{\frac{n-1-2j}{2}}y^{2j}z^{\frac{n+1-2j}{2}}+
\sum_{j=2}^{\frac{n+1}{2}}\lambda_{3,j}x^{\frac{n+1-2j}{2}}y^{2j}z^{\frac{n+1-2j}{2}}.
\end{eqnarray}

\vspace*{0.2cm}

\noindent \underline{Suppose that $n$ is even}. 

\vspace*{0.2cm}

\noindent $\bullet$ If $n=4$, let 
\begin{eqnarray}\label{equality-gfn=4}
\begin{array}{lcl}
g_1=xy^2-x^2z&\mbox{ and }&f_1=g_1+yz^2+x^2y^2-2xyz^2+z^4;\\
g_2=xyz-y^3&\mbox{ and }&f_2=g_2-z^3;\\
g_3=xz^2-y^2z&\mbox{ and }&f_3=g_3-x^3y+x^2z^2+xy^2z+y^4;\\
g_4=x^4-z^3&\mbox{ and }&f_4=g_4-4x^2yz+2y^2z^2-xy^2z^2+yz^4.\\
\end{array}
\end{eqnarray}

\vspace*{0.2cm}

\noindent \underline{Suppose that $n$ is even and $n\geq 6$}.

\vspace*{0.2cm}

\noindent $\bullet$ If $1\leq k\leq n-5$, with $k$ odd, let
\begin{eqnarray}\label{equality-gknevenkodd}
g_k=\sum_{j=0}^{\frac{k+1}{2}}(-1)^j
b^{\left[n-k-2,\frac{2n-k-3}{2}\right]\setminus\{\frac{2n-k-3-2j}{2}\}}
_{\{0\}\cup\left[\frac{n-k+1}{2},\frac{n-2}{2}\right]}\cdot 
x^{\frac{2n-k-3-2j}{2}}y^{2j}z^{\frac{k+1-2j}{2}}
\end{eqnarray}
and, if $\lambda_{l,j}$ are the unique solutions of the determined consistent linear systems in 
Lemma~\ref{lemma-h-i}, let
\begin{multline}\label{poly-fknevenkodd}
f_k=g_k+\sum_{j=0}^{\frac{n-k-3}{2}}
\lambda_{1,j}x^{\frac{n-k-3-2j}{2}}y^{1+2j}z^{\frac{n+k-1-2j}{2}}+
\sum_{j=1}^{\frac{k+1}{2}}\lambda_{2,j}x^{\frac{2n-k-1-2j}{2}}y^{2j}z^{\frac{k+1-2j}{2}}+\\
\sum_{j=0}^{\frac{n-k-1}{2}}\lambda_{3,j}x^{\frac{n-k-1-2j}{2}}y^{1+2j}z^{\frac{n+k-1-2j}{2}}+
\sum_{j=3}^{\frac{n-k+1}{2}}\lambda_{4,j}x^{\frac{n-k+1-2j}{2}}y^{1+2j}z^{\frac{n+k-1-2j}{2}}.
\end{multline}

\vspace*{0.2cm}

\noindent $\bullet$ If $2\leq k\leq n-4$, with $k$ even, let 
\begin{eqnarray}\label{equality-gknevenkeven}
g_k=\sum_{j=0}^{\frac{k}{2}}(-1)^j
b^{\left[n-k-2,\frac{2n-k-4}{2}\right]\setminus\{\frac{2n-k-4-2j}{2}\}}
_{\{0\}\cup\left[\frac{n-k+2}{2},\frac{n-2}{2}\right]}\cdot 
x^{\frac{2n-k-4-2j}{2}}y^{1+2j}z^{\frac{k-2j}{2}}
\end{eqnarray}
and, if $\lambda_{l,j}$ are the unique solutions of the determined consistent linear systems in 
Lemma~\ref{lemma-h-i+1}, let
\begin{multline}\label{poly-fknevenkeven}
f_k=g_k+\sum_{j=0}^{\frac{n-k-2}{2}}
\lambda_{1,j}x^{\frac{n-k-2-2j}{2}}y^{2j}z^{\frac{n+k-2j}{2}}+
\sum_{j=1}^{\frac{k}{2}}\lambda_{2,j}x^{\frac{2n-k-2-2j}{2}}y^{1+2j}z^{\frac{k-2j}{2}}+\\
\sum_{j=0}^{\frac{n-k}{2}}\lambda_{3,j}x^{\frac{n-k-2j}{2}}y^{2j}z^{\frac{n+k-2j}{2}}+
\sum_{j=4}^{\frac{n-k+2}{2}}\lambda_{4,j}x^{\frac{n-k+2-2j}{2}}y^{2j}z^{\frac{n+k-2j}{2}},
\end{multline}
where the last summation is taken void if $k=n-4$. 

\vspace*{0.2cm}

\noindent $\bullet$ If $k=n-3$, let
\begin{eqnarray}\label{equality-gk1-1}
g_{n-3}=\sum_{j=0}^{\frac{n-2}{2}}(-1)^jb^{\left[1,\frac{n}{2}\right]
\setminus\{\frac{n-2j}{2}\}}
_{\{0\}\cup\left[2,\frac{n-2}{2}\right]}\cdot 
x^{\frac{n-2j}{2}}y^{2j}z^{\frac{n-2-2j}{2}}
\end{eqnarray}
and
\begin{eqnarray}\label{poly-fk1-1}
f_{n-3}=g_{n-3}-\left(\sum_{j=0}^{\frac{n-2}{2}}(-1)^j
b_{\{0\}\sqcup[2,\frac{n-2}{2}]}^{[1,\frac{n}{2}]\setminus\{\frac{n-2j}{2}\}}
b_{\frac{n-2j}{2},1}\right)yz^{n-2}-x^2y^{n-2}+xyz^{n-2}+y^3z^{n-3}.
\end{eqnarray}

\vspace*{0.2cm}

\noindent $\bullet$ If $k=n-2$, let
\begin{eqnarray}\label{equality-gk10a}
g_{n-2}=xy^{n-3}z-y^{n-1}
\end{eqnarray}
and
\begin{eqnarray}\label{poly-fk10a}
f_{n-2}=xy^{n-3}z-y^{n-1}-z^{n-1};
\end{eqnarray}

\vspace*{0.2cm}

\noindent $\bullet$ If $k=n-1$, let
\begin{eqnarray}\label{equality-gk10b}
g_{n-1}=\sum_{j=0}^{\frac{n-2}{2}}(-1)^{j}b_{\frac{n-2}{2},\frac{n-2-2j}{2}}
\cdot x^{\frac{n-2-2j}{2}}y^{1+2j}z^{\frac{n-2-2j}{2}}
\end{eqnarray}
and, if $\lambda_{l,j}$ are the unique solutions of the determined consistent linear systems in 
Lemma~\ref{lemma-h-10}, let
\begin{multline}\label{poly-fk10b}
f_{n-1}=g_{n-1}-x^n+
\sum_{j=0}^{\frac{n-2}{2}}\lambda_{3,j}x^{\frac{n-2j}{2}}y^{1+2j}z^{\frac{n-2-2j}{2}}+\\
\sum_{j=2}^{\frac{n-2}{2}}\lambda_{4,j}x^{\frac{n+2-2j}{2}}y^{1+2j}z^{\frac{n-2-2j}{2}}+
\lambda_{5,1}xy^2z^{n-2}+\lambda_{5,2}y^4z^{n-3}.
\end{multline}

\vspace*{0.2cm}

\noindent $\bullet$ If $k=n$, let
\begin{eqnarray}\label{equality-gk11}
g_n=\sum_{j=0}^{\frac{n-2}{2}}(-1)^{j}b_{\frac{n-2}{2},\frac{n-2-2j}{2}}
\cdot x^{\frac{n-2-2j}{2}}y^{2j}z^{\frac{n-2j}{2}}
\end{eqnarray}
and, if $\lambda_{l,j}$ are the unique solutions of the determined consistent linear systems in 
Lemma~\ref{lemma-h-11}, let
\begin{eqnarray}\label{poly-fk11}
f_n=g_n-x^{n-1}y+
\sum_{j=1}^{\frac{n}{2}}\lambda_{2,j}x^{\frac{n-2j}{2}}y^{2j}z^{\frac{n-2j}{2}}+
\sum_{j=3}^{\frac{n}{2}}\lambda_{3,j}x^{\frac{n+2-2j}{2}}y^{2j}z^{\frac{n-2j}{2}}+
\lambda_{4,1}y^3z^{n-2}.
\end{eqnarray}
\end{definition}

The purpose of the present work is to prove the following result.

\begin{theorem*}
Let $n$ be an integer, $n\geq 3$. Let $a=(n-1)n/2$ and $q=2\lfloor(n+1)/2\rfloor-1$.
Let $\rho:\mbk\lser x,y,z\rser\to \mbk\lser t\rser$ be the
$\mbk$-algebra homomorphism with $\rho(x)=t^a(1+t^q)$,
$\rho(y)=t^{a+1}$, $\rho(z)=t^{a+2}$. Then the prime ideal 
$\ker(\rho)$ is minimally generated by the $n$ polynomials
$f_1,\ldots,f_n$. Their $\sigma$-leading forms satisfy
$f_1^{\sigma}=g_1,\ldots,f_n^{\sigma}=g_n$.
The coefficients of each tail $f_i^{\tau}$ can be found by solving at most four determined consistent linear systems of equations of size at most $n$.
\end{theorem*}

Now, we sketch the proof. To do this, we require more notations. 

\begin{notation}\label{not-wv}
(See \cite{moh1} and \cite[Remark~2.3 and Lemma~2.4]{gp1}.)
Let $r\in\mbn$. Set 
\begin{eqnarray*}
W_r=\{g\in\mbk\lser\monx\rser\mid g\neq 0,\; g=g^\sigma ,\; \sord(g)=r\}\cup \{0\}.
\end{eqnarray*}
Clearly, $W_r$ is the $\mbk$-vector space spanned by the 
monomials $\mon^{\alpha}$, with $\alpha\in\fac(r,\nums)$. Set $\delta_r=\dim W_r$. 
If $r\in\nums$, then $W_r\neq 0$ and $\delta_r\geq 1$. If $r=0$, then $W_0=\langle 1\rangle$.  
If $r\not\in\nums$, we just define $W_r=\{0\}$. 

From now on, fix $r\in\nums$, $r\neq 0$. Set 
\begin{eqnarray*}
V_r=\{g\in W_r\mid g\neq 0,\mbox{ there exists }f\in \ker(\rho), f\neq 0, \mbox{ such that
}g=f^{\sigma}\}\cup\{0\}.
\end{eqnarray*}
Then $V_r$ is a $\mbk$-vector subspace of $W_r$. Moreover, if $g\in W_r$, $g\neq 0$, then
$g\in V_r$ if, and only if, $g\in\ker(\rho)$ or there exists $h\in\mbk\lser\monx\rser$,
$h\neq 0$, such that $\sord(h)>r$ and such that $f=g+h\in\ker(\rho)$. 
In such a case, $g=f^{\sigma}$ and $h=f^{\tau}$.
\end{notation}

The proof of the main theorem is based on \cite[Theorem~3.2]{gp1}, an extended version of
\cite[Theorem~4.3]{moh1} (see also line 3 of the proof of the main result in \cite[p. 310]{moh2}).
Below, and for the sake of completeness and easy reference, 
we write an ad-hoc version, just called \emph{Extending to Basis Method}, 
sufficient to our purposes.
Given a subset $\msb$ of $\mbk\lser\monx\rser$, $\msb\neq\emptyset,\{0\}$, let
$\msb^{\sigma}=\{f^{\sigma}\mid f\in\msb, f\neq 0\}$ denote the set of its
$\sigma$-leading forms. If $\msb=\emptyset$ or $\msb=\{0\}$, then we just define
$\msb^{\sigma}=\emptyset$. 

\begin{theoremb}\label{moh43} \text{\cite[Theorem~3.2]{gp1}} {Extending to Basis Method.} 
Let $s=\min\{\sord(f)\mid f\in \ker(\rho), f\neq 0\}$, $s\geq 1$.
Let $\msc$ be a minimal generating set of $\ker(\rho)$. 
Let $\msd_0,\ldots,\msd_{a-1}\subset \ker(\rho)$ be such that,
for each $i=0,\ldots,a-1$, $\msd_i^{\sigma}$ is a (possibly empty)
linearly independent subset of $V_{s+i}$. 
Then $\msd_0\cup\ldots\cup\msd_{a-1}$ can be extended
to a minimal generating set of $\ker(\rho)$ with elements of 
$\langle \msc\rangle_{\Bbbk}$.
\end{theoremb}

So we start by determining the integer $s=\min\{\sord(f)\mid f\in \ker(\rho), f\neq 0\}$. 
Before, we need to recall some results on the 
numerical semigroup $\nums=\langle a,a+1,a+2\rangle$.

\begin{summary}\label{summary-semigroup}
{\rm (See \cite{ggp} and, more specifically,
\cite[Notations~2.9, Remark~2.10 and Theorems~3.1 and 4.1]{ggp})}
Let $\la=\lfloor a/2\rfloor (a+2)$, if $a$ is even, or $\la=(\lfloor
a/2\rfloor +2)a$, if $a$ is odd.  Clearly, $\la\in\nums$. 
Note that $\frob(\nums)<\la$. In fact, 
\begin{eqnarray*}
\frob(\nums)=\left\lfloor\frac{a}{2}\right\rfloor a-1<\left\lfloor
\frac{a}{2}\right\rfloor (a+2) \leq\left(\left\lfloor
\frac{a}{2}\right\rfloor +1\right)a<\left(\left\lfloor
\frac{a}{2}\right\rfloor +2\right)a.
\end{eqnarray*}
Let $\ell_r=\lfloor r/a\rfloor$ and $\rem_r=r-a\ell_r$
be the quotient and the remainder, respectively, of the 
Euclidean division in $\mbn$ of $r$ by $a$ and $\phi_r=(\phi_{r,1},\phi_{r,2},\phi_{r,3})=
\left(\ell_r-\left\lfloor \frac{\rem_r+1}{2}\right\rfloor ,
\rem_r-2\left\lfloor \frac{\rem_r}{2}\right\rfloor ,\left\lfloor
\frac{\rem_r}{2} \right\rfloor \right)$. Note that $\phi_{r,2}+2\phi_{r,3}=\rem_r$ and 
$|\phi_r|=\phi_{r,1}+\phi_{r,2}+\phi_{r,3}=\ell_r$.
It is known that, if $r\in\nums$, then $\phi_r\in\mbn^3$, $\rem_r\leq
2\ell_r$, $\phi_r\in\fac(r,\nums)$ and $\ell_r\in\setl(r,\nums)=\{|\alpha|\mid \alpha\in\fac(r,\nums)\}$ 
(see \cite[Theorem~3.1]{ggp}.)
One denotes $\kappa_r=\min(\phi_{r,1},\phi_{r,3})$,
$\iota_r=\phi_{r,2}+|\phi_{r,1}-\phi_{r,3}|$ and
$c_{r}=\phi_{r,3}-\phi_{r,1}$. Then, $\kappa_r$, $\iota_r\in\mbn$ and $c_r\in\mbz$. Moreover,
$c_r=\rem_r-\ell_r$ and $\ell_r-\iota_r=2\kappa_r$. In particular, 
$\ell_r$ and $\iota_r$ have the same parity. 

Set $\omega=(-1,2-1)\in\mbz^3$. Suppose that $r<\la$, so that one can apply
\cite[Notations~2.9, Remark~2.10 and Theorems~3.1 and 4.1]{ggp} (see also Proposition~\ref{theo31ggp}). 
Then,
\begin{eqnarray*}
\mcb_r=\{\mon^{\phi_r+j\omega}\mid 0\leq j\leq\kappa_r\}=
\{x^{\phi_{r,1}-j}y^{\phi_{r,2}+2j}z^{\phi_{r,3}-j}\mid 0\leq j\leq\kappa_r\},
\end{eqnarray*}
is an (ordered) $\mbk$-basis of $W_r$. In particular, $\delta_r=\kappa_r+1$.
Given $i\in\mbn$, let $\Gamma_i$ be defined as
$\Gamma_0=\{0\}$, $\Gamma_1=\{-1,0,1\}$ and $\Gamma_{i}:=\{-i,-i+1,i-1,i\}$, for $i\geq 2$.
One shows that $c_r\in\Gamma_{\iota_r}$.
\end{summary}

\begin{notation}\label{notation-IH}
We define the two sets of indices $I_r=[\phi_{r,1}-\kappa_r,\phi_{r,1}]$ and 
$H_r=[0,\phi_{r,1}]$. Consider $\mcc_r=\{t^{r+kq}\mid k\in H_r\}$,
as an ordered subset of $\mbk\lser t\rser$, and let $T_r$ be the $\mbk$-vector space spanned
by the $\mbk$-linearly independent elements of $\mcc_r$. Then 
$\rho(W_r)\subseteq T_r$ (see Lemma~\ref{lemma2-rho}). 
Let $\rho_r:W_r\to T_r$ be the $\mbk$-linear map restriction 
of $\rho$ over $W_r$ and let $\rho_r^L:W_r\to T_r^L$ be the composition 
of $\rho_r$ with the projection linear map
$T_r\to T_r^L$, where $T_r^{L}$ is the  $\mbk$-vector space spanned by
$\{t^{r+kq}\mid k\in L\}$, where $L\subseteq H_r$. 

For $k\in [0,\phi_{r,1}]$, we say that
$k$ is $r$-{\em good} if $r+kq\neq s+jq$, for all $s\in\nums$, $s>r$,
and for all $j\in [0,\phi_{s,1}]$. Let
$\good_r=\{k\in [0,\phi_{r,1}]\mid k\mbox{ is }r\mbox{-good}\}$ denote
the set of $r$-good elements. Note that $\{0\}\subseteq\good_r\subseteq
[0,\phi_{r,1}]$ and $1\leq\card(\good_r)\leq\phi_{r,1}+1$ (see Definition~\ref{goods}).
For ease of reference, we call ``\emph{$r$-good kernel}'' to
the kernel of the linear map $\rho_r^{{\rm G}_r}:W_r\to T_r^{{\rm G}_r}$. 
\end{notation}

The first step of the whole proof, shown in Section~\ref{sec-goods}, is 
``\emph{the inclusion on the good kernels}'':

\begin{theoremu}
Let $r\in\nums$, $r<\la$. Then $V_r\subseteq\ker(\rho_r^{{\rm G}_r})$.
\end{theoremu}

Recall that $s_0=a(n-1)+2$. 
Then $s_0\in\nums$ and $s_0<\la$ (see Section~\ref{sec-unders0}).
The second step is to ensure ``\emph{when the good kernels are zero}'':

\begin{theoremdos}
Let $r\in \nums$, $r<s_0$. Then $V_r=\ker(\rho_r^{{\rm G}_r})=0$. In particular, 
\begin{eqnarray*}
\min\{\sord(f)\mid f\in \ker(\rho),\; f\neq 0\}\geq s_0.
\end{eqnarray*}
\end{theoremdos}

Thus, in order to apply the \emph{Extending to Basis Method}, 
one must work with elements $r\in\nums$, $r\geq s_0$. We consider
the interval $[s_0,s_0+n)$, where 
$[s_0,s_0+n)\subset \nums\cap [0,\la)$ (Remark~\ref{remn5} and Lemma~\ref{lemma-pre}). 
Write the elements of $[s_0,s_0+n)$ as $r_k=s_0+k-1$, for $1\leq k\leq n$.
The next stage is to determine the $\mbk$-linear subspaces 
$\ker\left(\rho_{r_k}^{{\rm G}_{r_k}}\right)$ and, subsequently, to prove that
they coincide with $V_{r_k}$.
Our proof heavily relies on a stratification of the elements of $\nums$. Concretely, 
it is known that any $r\in\nums$, $r<\la$, can be written as
$r=(a+1)(2\delta_r-2+\iota_r)+c_r$, where $\ell_r=2\delta_r-2+\iota_r$ 
(see \cite[Theorem~3.4]{ggp}). According to the value of the pair $(\iota_r,c_r)$, we prove
in Section~\ref{sec-kernels} the third step, ``\emph{bases for the good kernels}'':

\begin{theoremtres}
Let $n\geq 6$. Let $r_k=s_0+k-1\in [s_0,s_0+n)$, where $1\leq k\leq n$. 

\noindent Suppose that $n$ is odd. 
\begin{itemize}
\item Then $g_k$ is a basis of the one-dimensional vector space 
$\ker\left(\rho_{r_k}^{{\rm G}_{r_k}}\right)$, for each $1\leq k\leq n$. 
\end{itemize}
Suppose that $n$ is even.
\begin{itemize}
\item  If $1\leq k\leq n-3$, then $g_k$ is a basis of the one-dimensional vector space 
$\ker\left(\rho_{r_k}^{{\rm G}_{r_k}}\right)$. 
\item $g_{n-2},g_{n-1}$ is a basis of the two-dimensional vector space 
$\ker\left(\rho_{r_{n-2}}^{{\rm G}_{r_{n-2}}}\right)$.
\item $g_{n}$ is a basis of the one-dimensional vector space 
$\ker\left(\rho_{r_{n-1}}^{{\rm G}_{r_{n-1}}}\right)$. 
\end{itemize}
\end{theoremtres}

In Section~\ref{sec-Vr} we find the coefficients of the polynomials
$f_1,\ldots,f_n$ defined in Definition~\ref{def-gk} so that $f_i\in\ker(\rho)$ and $f_i^{\sigma}=g_i$.
In particular, $V_r=\ker(\rho_r)$, for all $r\in [s_0,s_0+n)$ and 
\begin{eqnarray*}
s_0=\min\{\sord(f)\mid f\in \ker(\rho), f\neq 0\}.
\end{eqnarray*}
As a consequence, and using the \emph{Extending to Basis Method}, we get the fourth step, 
``\emph{a part of a minimal generating set of the kernel of $\rho$}'':

\begin{theoremquatre}
The polynomials $f_1,\ldots,f_n$ are part of a minimal generating set of $\ker(\rho)$.
\end{theoremquatre}

In Section~\ref{sec-min}, we consider $I=(f_1,\ldots,f_n)R$ the ideal generated by the polynomials $f_1,\ldots,f_n$. 
We prove that the ideals $(I,z)$ and $(\ker(\rho),z)$ have equal (finite) length. 
By the Nakayama Lemma we obtain the last and fifth step, ``\emph{a minimal generating set of 
the kernel of $\rho$}'':

\begin{theoremcinc}
The polynomials $f_1,\ldots,f_n$ define a minimal generating set of $\ker(\rho)$.
\end{theoremcinc}

As a corollary, we deduce, for each $n\geq 3$, the existence of prime ideals in 
$\mbk[x,y,z]$ and $\mbk[x,y,z]_{(x,y,z)}$ minimally
generated by $n$ elements (see Corollary~\ref{cor-inA}).
We also deduce that the Sally question has an affirmative answer for regular local rings containing a field of characteristic zero (see Corollary~\ref{cor-Sally}). 

We finish the paper with the case $n=6$, as an example, and by giving a code to calculate the polynomials $f_1,\ldots,f_n$. 

\section{Definition and inclusion on the good kernels}\label{sec-goods}

In this section we prove that $V_r$ is included in
the kernel of $\rho_r^{{\rm G}_r}:W_r\to T_r^{{\rm G}_r}$, 
where $\good_r$ is the set of $r$-good integers
(see Definition~\ref{goods}). This is the analogue of \cite[Proposition~4.1]{moh1}. 

\begin{lemma}\label{lemma1-rho}
Let $r\in\nums$, $r\neq 0$, and let $\alpha\in\fac(r,\nums)$ and 
$\mon^\alpha=x^{\alpha_1}y^{\alpha_2}z^{\alpha_3}$. The
following hold.
\begin{itemize}
\item[$(1)$]
  $\rho(\mon^\alpha)=t^r(1+t^q)^{\alpha_1}=
  b_{\alpha_1,0}t^r+b_{\alpha_1,1}t^{r+q}+\cdots
  +b_{\alpha_1,\alpha_1}t^{r+\alpha_1q}$, where
  $b_{i,j}=\binom{i}{j}$.
\item[$(2)$]
$\supp(\rho(\mon^\alpha))=\{r+kq\mid k\in[0,\alpha_1]\}$.  In particular,
$\ord(\rho(\mon^{\alpha}))=r$.
\item[$(3)$] If $\dim W_r=1$, then $V_r=0$.
\item[$(4)$] Let $g=\sum_{\alpha\in{\rm
    F}(r,\mathcal{S})}\lambda_{\alpha}\mon^{\alpha}\in W_r$, $g\neq
  0$. If $g\in V_r$, then $\sum_{\alpha\in{\rm
      F}(r,\mathcal{S})}\lambda_{\alpha}=0$.
\item[$(5)$] If $W_r=\langle \mon^{\alpha},\mon^{\beta}\rangle$, with
  $\alpha,\beta\in\fac(r,\nums)$, $\alpha\neq\beta$, then $V_r=0$ or
  $V_r=\langle \mon^{\alpha}-\mon^{\beta}\rangle$.
\end{itemize}
\end{lemma}
\begin{proof}
Item $(1)$ follows from the definition of $\rho$ and the hypothesis
$\alpha\in\fac(r,S)$. Item $(2)$ follows directly from $(1)$ and the
fact that $b_{\alpha_1,j}\neq 0$, for all $j=0,\ldots,\alpha_1$.

Suppose that $W_r=\langle \mon^\alpha\rangle$, where $\alpha\in\fac(r,S)$.
Suppose that there exists $g\in V_r$, $g\neq 0$.  By definition, $g\in
W_r$, so $g=\lambda \mon^{\alpha}$, for some $\lambda\neq 0$. By Item
$(2)$, $\ord(\rho(g))=r$, thus, in particular, $g\not\in
\ker(\rho)$. Since $g\in V_r$ and $g\not\in\ker(\rho)$, by definition,
there exists $h\in\mbk\lser x,y,z\rser$, $h\neq 0$, such that
$\sord(h)=s>r$ and $\rho(g)+\rho(h)=0$.  Since $\rho(g)$ is not zero, $\rho(h)$
cannot be zero. Applying Item $(1)$ to each monomial of $h$, one
deduces that $\ord(\rho(h))\geq s>r$, which is a contradiction with
$\rho(g)+\rho(h)=0$. Therefore, $V_r$ must be zero, which proves
$(3)$.

Now we turn to the proof of $(4)$. By Item $(1)$,
\begin{eqnarray*}
\rho(g)=\sum_{\alpha\in{\rm F}(r,\mathcal{S})}\lambda_{\alpha}t^r(1+t^q)^{\alpha_1}=
\sum_{\alpha\in{\rm F}(r,\mathcal{S})}\lambda_{\alpha}t^r+p(t),
\end{eqnarray*}
where $p(t)$ is a polynomial in $t$ whose terms have all degree bigger than $r$.
If $g\in\ker(\rho)$, then clearly $\sum_{\alpha\in{\rm F}(r,\mathcal{S})}\lambda_{\alpha}=0$. 
If $g\not\in\ker(\rho)$,
since $g\in V_r$, then, by definition, there exists $h\in\mbk\lser
x,y,z\rser$, $h\neq 0$, $\sord(h)>r$, such that
$\rho(g)+\rho(h)=0$. Since $g\not\in\ker(\rho)$, $\rho(h)\neq 0$.
Applying Item $(1)$ to each monomial of $h$, one deduces that
$\ord(\rho(h))\geq s>r$, which forces 
$\sum_{\alpha\in{\rm F}(r,\mathcal{S})}\lambda_{\alpha}=0$.

Suppose that $W_r=\langle \mon^{\alpha},\mon^{\beta}\rangle$, where
$\alpha,\beta\in\fac(r,\nums)$, $\alpha\neq\beta$.  If $V_r\neq 0$,
then there exists $g\in V_r$, $g\neq 0$, $g=\lambda \mon^{\alpha}+ \mu
\mon^{\beta}$, $\lambda,\mu\in\mbk$, $(\lambda,\mu)\neq (0,0)$. By
$(4)$, $\lambda+\mu=0$, thus $g=\lambda(\mon^{\alpha}-\mon^{\beta})$ and
$V_r=\langle \mon^{\alpha}-\mon^{\beta}\rangle$.
\end{proof}

From now on, until the end of the section, we suppose that $r\in\nums$, $r\neq 0$, $r<\la$. 
Recall that $\delta_r-1=\kappa_r=\min(\phi_{r,1},\phi_{r,3})\leq\phi_{r,1}$, so $\delta_r\leq\phi_{r,1}+1$;
that $T_r$ denotes the $\mbk$-vector space spanned by the linearly independent monomials 
$\mcc_r=\{t^{r+kq}\mid k\in [0,\phi_{r,1}]\}$ and that 
$I_r=[\phi_{r,1}-\kappa_r,\phi_{r,1}]$ and $H_r=[0,\phi_{r,1}]$
(see Summary~\ref{summary-semigroup} and Notation~\ref{notation-IH}). Recall too the notation
of $P^J_I$ in Notation~\ref{notation-binomials}.

\begin{lemma}\label{lemma2-rho}
Let $r\in\nums$ $r\neq 0$, $r<\la$. 
Then $\rho(W_r)\subseteq T_r$. Moreover, the matrix $M_{\mathcal{B}_r,\mathcal{C}_r}(\rho_r)$
of the restriction linear map 
$\rho_r:W_r\to T_r$ in the bases $\mcb_r=\{\mon^{\phi_r+j\omega}\mid 0\leq j\leq\kappa_r\}$ and 
$\mcc_r=\{t^{r+kq}\mid 0\leq k\leq \phi_{r,1}\}$ is the $(\phi_{r,1}+1)\times\delta_r$ matrix with
binomial entries
\begin{eqnarray*}
P^{H_r}_{I_r}=\left(\begin{array}{ccc}
  b_{\phi_{r,1},0}&\ldots&b_{\phi_{r,1}-\kappa_r,0}
  \\ \vdots&&\vdots\\b_{\phi_{r,1},\phi_{r,1}}&\ldots&
  b_{\phi_{r,1}-\kappa_r,\phi_{r,1}}
\end{array}\right).
\end{eqnarray*}
\end{lemma}
\begin{proof}
By Lemma~\ref{lemma1-rho}, for each $j=0,\ldots,\kappa_r$, we have
\begin{eqnarray*}
\rho(\mon^{\phi_r+j\omega})=t^r(1+t^q)^{(\phi_{r,1}-j)}=
b_{\phi_{r,1}-j,0}t^r+b_{\phi_{r,1}-j,1}t^{r+q}+\cdots
+b_{\phi_{r,1}-j,\phi_{r,1}-j}t^{r+(\phi_{r,1}-j)q},
\end{eqnarray*} 
which clearly implies $\rho(W_r)\subseteq T_r$. 
Let $g\in W_r$ and let
$\Lambda^{\top}:=(\lambda_0,\ldots,\lambda_{\kappa_r})\in\mbk^{\delta}$, $\delta_r=\kappa_r+1$,
be the components of $g$ in the basis $\mcb_r$, so
\begin{eqnarray*}
g=\lambda_0\mon^{\phi_r}+\ldots+\lambda_{j-1}\mon^{\phi_r+(j-1)\omega}+\ldots+
\lambda_{\kappa_r}\mon^{\phi_r+\kappa_r\omega}.
\end{eqnarray*}
By Lemma~\ref{lemma1-rho},
\begin{eqnarray*}
\rho(g)&=&\lambda_{0}t^r(1+t^q)^{\phi_{r,1}}+\cdots+
\lambda_{j-1}t^r(1+t^q)^{\phi_{r,1}-(j-1)}+\cdots+
\lambda_{\kappa_r}t^r(1+t^q)^{\phi_{r,1}-\kappa_r}\nonumber \\&=&
\lambda_{0}\left( b_{\phi_{r,1},0}t^r+b_{\phi_{r,1},1}t^{r+q}+\cdots
+b_{\phi_{r,1},\phi_{r,1}}t^{r+\phi_{r,1}q} \right)+\cdots+\nonumber \\&&
\lambda_{j-1}\left(
b_{\phi_{r,1}-(j-1),0}t^r+b_{\phi_{r,1}-(j-1),1}t^{r+q}+\cdots
+b_{\phi_{r,1}-(j-1),\phi_{r,1}}t^{r+\phi_{r,1}q} \right)+\cdots+\nonumber \\&&
\lambda_{\kappa_r}\left( b_{\phi_{r,1}-\kappa_r,0}t^r+b_{\phi_{r,1}-\kappa_r,1}t^{r+q}+\cdots +
b_{\phi_{r,1}-\kappa_r,\phi_{r,1}}t^{r+\phi_{r,1}q}\right)\nonumber \\&=&
\sum_{k=0}^{\phi_{r,1}}\left(b_{\phi_{r,1},k}\lambda_{0}+\cdots+
b_{\phi_{r,1}-(j-1),k}\lambda_{j-1}+\cdots+
b_{\phi_{r,1}-\kappa_r,k}\lambda_{\kappa_r}\right) t^{r+kq}=\sum_{k=0}^{\phi_{r,1}}\mu_kt^{r+kq},
\end{eqnarray*}
where $\mu_k=b_{\phi_{r,1},k}\lambda_{0}+\cdots+
b_{\phi_{r,1}-(j-1),k}\lambda_{j-1}+\cdots+
b_{\phi_{r,1}-\kappa_r,k}\lambda_{\kappa_r}$.
Let $M^{\top}=(\mu_0,\ldots,\mu_{\phi_{r,1}})\in\mbk^{\phi_{r,1}}$ be the
components of $\rho(g)$ in $\mcc_r$. Since $M=P^{H_r}_{I_r}\Lambda$, it follows that 
$M_{\mathcal{B}_r,\mathcal{C}_r}(\rho_r)=P^{H_r}_{I_r}$.
\end{proof}

\begin{example}\label{Example-PB}
Suppose that $n=9$ and $r=295$. Then $\phi_{r,1}=4$, $\delta_r=4$, 
$I_r=[1,4]$, $H_r=[0,4]$ and
\begin{eqnarray*}
(B^{I_r}_{H_r})^{\top}\Theta_{\delta_r}=
\left(\begin{array}{ccccc}
1&1&0&0&0\\1&2&1&0&0 \\1&3&3&1&0\\1&4&6&4&1
\end{array}\right)^{\top}\cdot
\left(\begin{array}{cccc}
0&0&0&1\\0&0&1&0\\0&1&0&0\\1&0&0&0
\end{array}\right)=
\left(\begin{array}{ccccc}
1&1&1&1\\4&3&2&1\\6&3&1&0\\4&1&0&0\\1&0&0&0
\end{array}\right)=P^{H_r}_{I_r}.
\end{eqnarray*}  
Let  $I=\{1,3\}\subset I_r=[1,4]$ and $J=\{0,1,3\}\subset H_r=[0,4]$, then
\begin{eqnarray*}
(B^{I}_{J})^{\top}\Theta_{2}=
\left(\begin{array}{ccc}
1&1&0\\1&3&1\end{array}\right)^{\top}\cdot
\left(\begin{array}{cc}
0&1\\1&0\end{array}\right)=
\left(\begin{array}{cc}
1&1\\3&1\\1&0\end{array}\right)=P^{J}_{I}.
\end{eqnarray*}
\end{example}

\begin{notation}
Let $\mcc_r^L=\{t^{r+k_1q},\ldots,t^{r+k_{\eta}q}\}$ be the (ordered) subset of $\mcc_r$ defined by the set of indices $L=\{k_1,\ldots,k_{\eta}\}\subseteq H_r=[0,\phi_{r,1}]$, $0\leq k_1<\ldots<k_{\eta}\leq \phi_{r,1}$, $1\leq \eta\leq\phi_{r,1}+1$.
Let $T_r^L=\langle \mcc_r^L\rangle\subseteq T_r$ be the $\mbk$-vector space spanned by $\mcc_r^L$. 
Let $\rho_r^L:W_r\to T_r^L$ be the $\mbk$-linear map composition of $\rho_r:W_r\to T_r$ with the natural projection $T_r\to T_r^L$.
\end{notation}

\begin{lemma}\label{lemma3-rho}
Let $r\in\nums$, $r\neq 0$, $r<\la$. 
\begin{itemize}
\item[$(1)$] The matrix of $\rho_r^L:W_r\to T_r^L$ in the bases $\mcb_r$ and 
$\mcc_r^L$ is $M_{\mathcal{B}_r,\mathcal{C}_r^L}(\rho_r^L)=P_{I_r}^L$.
\item[$(2)$] $\dim\ker(\rho_r^L)=\delta_r-\min(\card(L),\delta_r)$.
\item[$(3)$]  If $\card(L)=\delta_r-1$, then $\dim\ker(\rho_r^L)=1$ and a basis is given by
the polynomial 
\begin{eqnarray*}
\sum_{j=0}^{\kappa_r}(-1)^jb_L^{I_r\setminus\{\phi_{r,1}-j\}}\mon^{\phi_r+j\omega}.
\end{eqnarray*}
\item[$(4)$] The map $\rho_r:W_r\to T_r$ is a monomorphism. In particular, if $g\in W_r$, $g\neq 0$, 
then $\rho(g)\neq 0$ and $\supp(\rho(g))\subseteq \{r+kq\mid k\in[0,\phi_{r,1}]\}$.
\item[$(5)$] Let $g\in W_r$, $g\neq 0$. Then, $g\in V_r$ if and only if
  there exists $h\in\mbk\lser x,y,z\rser$, $h\neq 0$, $\sord(h)>r$,
  such that $\rho(g)+\rho(h)=0$. In particular, $\rho(h)\neq 0$.
\item[$(6)$] If $h\in\mbk\lser x,y,z\rser$, with $\rho(h)\neq 0$, then
  $\supp(\rho(h))\subseteq\bigcup_{\{s\in\mathcal{S}, s\geq\sigma{\rm
    ord}(h)\}} \{s+jq\mid j\in [0,\phi_{s,1}]\}$.
\item[$(7)$] If $L_1,L_2$ are two subsets of $H_r$, then
$\ker\left(\rho_r^{L_1\cup L_2}\right)=
\ker\left(\rho_r^{L_1}\right)\cap \ker\left(\rho_r^{L_2}\right)$.
\item[$(8)$] If $H_r=L_1\sqcup L_2$ is a partition of $H_r$, then 
$\rho_r(g)=\rho_r^{L_1}(g)+\rho_r^{L_2}(g)$,  for all $g\in W_r$, where the sum is considered in $T_r$.
\end{itemize}
\end{lemma}
\begin{proof}
The matrix of $\rho_r^L$ in the bases $\mcb_r$ and $\mcc_r^L$ is clearly obtained from $P^{H_r}_{I_r}$ by omitting
the rows not belonging to $L$. This proves $(1)$. 

One has $I_r=[\phi_{r,1}-\kappa_r,\phi_{r,1}]\subseteq [0,\phi_{r,1}]=H_r$, where $\delta_r-1=\kappa_r\leq\phi_{r,1}$. 
By \cite[Corollary~2.6]{gp2}, $\rank(B^{I_r}_{J})=\delta_r$, 
for every $J=\{j_1,\ldots,j_{\delta_r}\}\subseteq H_r$, $0\leq j_1<\ldots<j_{\delta_r}\leq\phi_{r,1}$.
Therefore, $\rank(P_{I_r}^J)=\rank(B^{I_r}_J)=\delta_r$, so every $\delta_r\times\delta_r$ 
submatrix of $P^{H_r}_{I_r}$ has rank $\delta_r$. 
Since $L=\{k_1,\ldots,k_{\eta}\}\subseteq H_r$, it follows that
$\rank(P_{I_r}^L)=\min(\eta,\delta_r)$. Thus,
\begin{eqnarray*}
\dim\ker(\rho_r^L)=\delta_r-\rank(P_{I_r}^L)=\delta_r-\min(\eta,\delta_r)=
\delta_r-\min(\card(L),\delta_r), 
\end{eqnarray*}
which shows $(2)$. Suppose that $\card(L)=\delta_r-1$, where, recall $\delta_r-1=\kappa_r$. Then,
$\dim\ker(\rho_r^L)=1$. Let $g=\sum_{j=0}^{\kappa_r}\lambda_{j}\mon^{\phi_r+j\omega}\in W_r$, $g\neq 0$, 
and write $\Lambda^{\top}=(\lambda_1,\ldots,\lambda_{\kappa_r})$. Then, $g\in\ker(\rho_r^L)$ if, and only if,
\begin{eqnarray*}
0=P_{I_r}^L\Lambda=\left(B^{I_r}_L\right)^{\top}\cdot\Theta_{\delta_r}\cdot
\Lambda=(B^{I_r}_L)^{\top}\cdot
\left(\begin{array}{c}\lambda_{\delta_r}\\\vdots\\\lambda_1\end{array}\right).
\end{eqnarray*}
Let momentarily $M=(B^{I_r}_L)^{\top}$, which is a $(\delta_r-1)\times\delta_r$ matrix. 
Given a matrix $A$, let $A^{i}$ and $A_j$ stand for the submatrices of $A$ after having removed row $i$ and column $j$, respectively. Note that $\left(A^{\top}\right)_j=(A^j)^{\top}$. 
Thus, $M_j=\left((B^{I_r}_L)^{\top}\right)_j=\left((B^{I_r}_L)^j\right)^{\top}=
\left(B^{I_r\setminus\{\phi_{r,1}-\kappa_r-(j-1)\}}_L\right)^{\top}$ and
$\det(M_j)=\det(B^{I_r\setminus\{\phi_{r,1}-\kappa_r-(j-1)\}}_L)$.
By Cramer's rule, $(\lambda_{\delta_r},\ldots,\lambda_1)$ is a multiple of the vector
\begin{multline*}
((-1)\det(M_1),\ldots,(-1)^{\delta_r}\det(M_{\delta_r}))=
((-1)\det(B^{I_r\setminus\{\phi_{r,1}-(\delta_r-1)\}}_L),\ldots,
(-1)^{\delta_r}\det(B^{I_r\setminus\{\phi_{r,1}\}}_L)).
\end{multline*}
Therefore, $g$ is a multiple of the polynomial 
$\sum_{j=0}^{\kappa_r}(-1)^jb_L^{I_r\setminus\{\phi_{r,1}-j\}}\mon^{\phi_r+j\omega}$, where 
$b^I_J=\det(B^I_J)$. This shows $(3)$. 

If $L=H_r$, then $\rho_r^L=\rho_r$. Since
$\card(L)=\phi_{r,1}+1\geq \delta_r$, by Item $(2)$, $\rho_r:W_r\to T_r$ is a monomorphism. 
Therefore, if $g\in W_r$, $g\neq 0$, then $\rho(g)\neq 0$ and $\rho(g)\in T_r$, so
$\supp(\rho(g))\subseteq \{r+kq\mid k\in[0,\phi_{r,1}]\}$, 
which shows $(4)$. 

Let $g\in W_r$, $g\neq 0$. By \cite[Lemma~2.4]{gp1}, $g\in V_r$ if, and only if, $g\in\ker(\rho)$ 
or there exists $h\in\mbk\lser\monx\rser$,
$h\neq 0$, such that $\sord(h)>r$ and such that $f=g+h\in\ker(\rho)$. 
By $(4)$, $\rho(g)\neq 0$. Thus, there exists $h\in\mbk\lser\monx\rser$,
$h\neq 0$, such that $\sord(h)>r$ and such $\rho(g)+\rho(h)=0$, where $\rho(h)\neq 0$. This proves $(5)$.

Given $h=\sum_{\alpha\in\mathbb{N}^d}\lambda_{\alpha}\mon^\alpha\in
\mbk\lser\monx\rser$ and $s\in\nums$, let 
$h_{(s)}=\sum_{\alpha\in{\rm F}(s,\mathcal{S})}\lambda_{\alpha}\mon^{\alpha}\in W_s$ 
be the summation of all the monomial terms of $h$ of $\sigma$-order $s$. Since $\fac(s,\nums)$
is finite, then $h_{(s)}$ is a polynomial and 
$\sum_{s\in\mathcal{S}}h_{(s)}$ coincides with $h$. 
By Item $(4)$, if $h_{(s)}\neq 0$, then 
$\rho(h_{(s)})\neq 0$ and $\supp(\rho(h_{(s)}))\subseteq \{s,s+q,\ldots,s+\phi_{s,1}q\}$,  
which shows $(6)$.

Item $(7)$ follows from the equivalence $P^{L_1\cup L_2}_{I_r}\Lambda =0\Leftrightarrow 
P^{L_1}_{I_r}\Lambda=0$ and $P^{L_2}_{I_r}\Lambda=0$. 

Let $H_r=L_1\sqcup L_2$ a partition. Though
$\rho_r^{L_1}:W_r\to T_r^{L_1}$ and $\rho_r^{L_2}:W_r\to T_r^{L_2}$
are two linear maps that can not be formally added, if $g\in W_r$, we
can write $\rho_r(g)=\rho_r^{L_1}(g)+\rho_r^{L_2}(g)$, understanding
that $T_r^{L_1},T_r^{L_2}\subseteq T_r$. 
\end{proof}

\begin{definition}\label{goods}
Let $r\in\nums$, $r\neq 0$, $r<\la$. For $k\in [0,\phi_{r,1}]$, we say that
$k$ is $r$-{\em good} if $r+kq\neq s+jq$, for all $s\in\nums$, $s>r$,
and for all $j\in [0,\phi_{s,1}]$. Let
$\good_r=\{k\in [0,\phi_{r,1}]\mid k\mbox{ is }r\mbox{-good}\}$ denote
the set of $r$-good integers.
Note that $\{0\}\subseteq\good_r\subseteq [0,\phi_{r,1}]$ and $1\leq\card(\good_r)\leq\phi_{r,1}+1$.
\end{definition}

\begin{lemma}\label{lemma4-rho}
Let $r\in\nums$, $r\neq 0$, $r<\la$. The following hold.
\begin{itemize}
\item[$(1)$] If $g\in V_r$, $g\neq 0$, and $k\in\good_r$, then
  $r+kq\not\in\supp(\rho(g))$ and $g\in\ker(\rho_r^{\{k\}})$.
\item[$(2)$] If $g\in V_r$, $g\neq 0$, then $\supp(\rho(g))\subseteq
  \{r+kq\mid k\in [0,\phi_{r,1}]\; ,\; k\not\in\good_r\}$.
\end{itemize}
\end{lemma}
\begin{proof}
Let $g\in W_r$, $g\neq 0$. Write $\rho(g)=\sum_{k=0}^{\phi_{r,1}}\mu_kt^{r+kq}$. 
By Lemma~\ref{lemma3-rho}, $(5)$, 
if $g\in V_r$, there exists $h\in\mbk\lser x,y,z\rser $, $h\neq 0$, with
$\sord(h)>r$, such that $\rho(g)+\rho(h)=0$. By Lemma~\ref{lemma3-rho}, $(6)$,
$\supp(\rho(h))\subseteq\bigcup_{s\in \nums,s\geq \sigma{\rm
    ord}(h)}\{s+jq\mid j\in [0,\phi_{s,1}]\}$.  By definition, if
$k\in\good_r$, then $r+kq\neq s+jq$, for all $s\in\nums$, $s>r$, and
for all $j\in[0,\phi_{s,1}]$.  Thus,
$r+kq\not\in\supp(\rho(h))$. Since $\rho(g)+\rho(h)=0$, necessarily,
$r+kq\not\in\supp(\rho(g))$. Hence $\mu_k=0$ and
$\rho_r^{\{k\}}(g)=\mu_kt^{r+kq}=0$. This proves $(1)$.

By Lemma~\ref{lemma3-rho}, $(4)$, if $g\in W_r$, $g\neq 0$, then $\supp(\rho(g))\subseteq
\{r+kq\mid k\in [0,\phi_{r,1}]\}$. By Item $(1)$, if $g\in V_r$,
$g\neq 0$, and $k\in\good_r$, then
$r+kq\not\in\supp(\rho(g))$. Therefore, we conclude that, if $g\in
V_r$, $g\neq 0$, then $\supp(\rho(g))\subseteq \{r+kq\mid k\in
[0,\phi_{r,1}]\; , \; k\not\in\good_r\}$, which proves $(2)$.
\end{proof}

Now we can state and prove the main result of the section, the first step, 
the inclusion of $V_r$ on the good kernels.

\begin{theorem}\label{FirstStep}
Let $r\in\nums$, $r\neq 0$, $r<\la$. If
$L\subseteq\good_r$, then $V_r\subseteq\ker(\rho_r^L)$.
\end{theorem}
\begin{proof}
Let $L=\{k_1,\ldots,k_{\eta}\}\subseteq\good_r$, with $0\leq
k_1<\ldots<k_{\eta}\leq \phi_{r,1}$. Let $g\in V_r$, $g\neq 0$.  By
Lemma~\ref{lemma4-rho}, $g\in \ker(\rho_r^{\{k_i\}})$, for every $i\in
\{1,\ldots,\eta\}$. So $g\in
\cap_{i=1}^{\eta}\ker(\rho_r^{\{k_i\}})=\ker(\rho_r^L)$.
\end{proof}

Observe that, for $L=\{0\}\subseteq\good_r$, we recover Item $(4)$ of
Lemma~\ref{lemma1-rho}. 

\begin{corollary}\label{corollary-vr0}
Let $r\in\nums$, $r\neq 0$, $r<\la$. Then,
\begin{eqnarray*}
\dim V_r\leq \dim\ker(\rho_r^{{\rm G}_r})=\delta_r-\min(\card(\good_r),\delta_r).
\end{eqnarray*}
Furthermore, if $\delta_r=1$, or $\card(\good_r)\geq\delta_r$, or
$\phi_{r,1}=0$, then $V_r=\ker(\rho_r^{{\rm G}_r})=0$.
\end{corollary}
\begin{proof}
By Theorem~\ref{FirstStep} and Lemma~\ref{lemma3-rho}, $V_r\subseteq
\ker(\rho_r^{\rm G})$ and $\dim\ker(\rho_r^{{\rm G}_r})=
\delta_r-\min(\card(\good_r),\delta_r)$. Since $0\in\good_r$, if
$\delta_r=1$ or if $\card(\good_r)\geq\delta_r$, then
$\delta_r-\min(\card(\good_r),\delta_r)=0$. If $\phi_{r,1}=0$, since
$0\leq\delta_r-1\leq\phi_{r,1}$, then $\delta_r=1$.
\end{proof}

\begin{example}\label{example-vrn3}
Suppose that $n=3$, so $a=(n-1)n/2=3$, $q=2\lfloor(n+1)/2\rfloor-1=3$, $s_0=a(n-1)+2=8$,
$\la=(\lfloor a/2\rfloor +2)a=9$,
$\nums=\langle 3,4,5\rangle$ 
and $\rho(x)=t^{3}(1+t^3)$, $\rho(y)=t^{4}$, $\rho(z)=t^{5}$. 
Since $\fac(3,\nums)=\{(1,0,0)\}$, $\fac(4,\nums)=\{(0,1,0)\}$, 
$\fac(5,\nums)=\{(0,0,1)\}$, $\fac(6,\nums)=\{(2,0,0)\}$ and $\fac(7,\nums)=\{(1,1,0)\}$, 
it follows that $W_3=\langle x\rangle$, 
$W_4=\langle y\rangle$,
$W_5=\langle z\rangle$,
$W_6=\langle x^2\rangle$ and $W_7=\langle xy\rangle$. Thus, for $r=3,4,5,6,7$,
$\delta_r=\dim W_r=1$ and, by Corollary~\ref{corollary-vr0}, $V_r=\ker(\rho_r^{{\rm G}_r})=0$.
\end{example}

We finish this section with three lemmas that will be very useful in the sequel.

\begin{lemma}\label{lemma-good}
Let $r\in\nums$, $r\neq 0$, $r<\la$. Suppose that $\delta_r\geq 2$. Let $\eta\in\mbn$,
$1\leq\eta\leq\phi_{r,1}$.
\begin{itemize}
\item[$(1)$] If $\{r+kq\mid k\in [1,\eta]\}\cap\nums=\emptyset$,
  then $[0,\eta]\subseteq\good_r$.
\item[$(2)$] If $\{r+kq\mid k\in
  [1,\delta_{r}-1]\}\cap\nums=\emptyset$, then $V_r=\ker(\rho_r^{{\rm G}_r})=0$.
\end{itemize}
\end{lemma}
\begin{proof}
Note that $1\leq \delta_r-1=\kappa_r=\phi_{r,1}$, so $\phi_{r,1}\geq 1$ and 
$[1,\phi_{r,1}]$ is not empty.
Suppose that there exist $k\in[1,\eta]$ such that $k$ is not $r$-good
(recall that $0\in\good_r$). Thus, there exist $s\in\nums$, $s>r$, and
$j\in [0,\phi_{s,1}]$, such that $r+kq=s+jq$. Let $k'=k-j$, so
$r+k'q=s\in\nums$. Since $r<s$, $j\geq 0$ and $k\leq\eta$, then
$0<k'=k-j\leq k\leq\eta$. So $r+k'q\in\{r+lq\mid l\in
[1,\eta]\}\cap\nums\neq\emptyset$.  This proves $(1)$. Suppose that
$\{r+kq\mid k\in[1,\delta_{r}-1]\}\cap\nums=\emptyset$.  Since
$1\leq\delta_r-1\leq\phi_{r,1}$, then, by $(1)$,
$[0,\delta_r-1]\subseteq\good_r$. Thus, $\card(\good_r)\geq\delta_r$. By Corollary~\ref{corollary-vr0}, $V_r=\ker(\rho_r^{\rm G})=0$.
\end{proof}

\begin{lemma}\label{lemma-diffcase}
Let $r\in\nums$, $r\neq 0$, $r<\la$. Suppose that $\delta_r\geq 2$, that $u=r+q\in\nums$ and that
$2+\phi_{u,1}\leq\phi_{r,1}$. Let $\eta\in [2+\phi_{u,1},\phi_{r,1}]$.
If $\{r+kq\mid k\in [2,\eta]\}\cap\nums=\emptyset$, then $\{0\}\cup
[2+\phi_{u,1},\eta]\subseteq\good_r$.
\end{lemma}
\begin{proof}
Recall that $0\in\good_r$. If there were $k\in [2+\phi_{u,1},\eta]$
such that $k$ is not $r$-good, then there would be $s\in\nums$, $s>r$,
and $j\in [0,\phi_{s,1}]$, such that $r+kq=s+jq$. Hence
$r+k'q=s\in\nums$, for $k':=k-j$, where $1\leq k'=k-j\leq
k\leq\eta$. Since $r+q\in\nums$ and $r+kq\not\in\nums$, for $2\leq
k\leq\eta$, then, necessarily, $k'=1$ and $u=r+q=r+k'q=s$. In
particular, $\phi_{u,1}=\phi_{s,1}$ and $k\geq
2+\phi_{u,1}=2+\phi_{s,1}$. Since $j\leq\phi_{s,1}$, then
$1=k'=k-j\geq 2+\phi_{s,1}-\phi_{s,1}=2$, a contradiction.
\end{proof}

\begin{lemma}\label{lg}
Let $r\in\nums$, $r\neq 0$, $r<\la$. If $u=r+lq\in\nums$, for
some $l\geq 1$, then $[l,l+\phi_{u,1}]\cap\good_r=\emptyset$. 
\end{lemma}
\begin{proof}
This follows from Definition~\ref{goods}. Indeed, since
$\good_r\subseteq [0,\phi_{r,1}]$, if $l>\phi_{r,1}$, then
$l\not\in\good_r$ and $[l,l+\phi_{u,1}]\cap\good_r=\emptyset$.  If
$l\in [1,\phi_{r,1}]$, and since $u=r+lq\in\nums$, then, by definition,
$l\not\in\good_r$. Take $k\in [l,l+\phi_{u,1}]$ and write $k=l+j$,
with $j\in [0,\phi_{u,1}]$. Let us see $k\not\in\good_r$. But this is
clear, since $r+kq=r+lq+jq=u+jq$, where $u=r+lq\in\nums$, $u>r$ and
$j\in [0,\phi_{u,1}]$.
\end{proof}

\section{Ensuring when the good kernels are zero}\label{sec-unders0}

In this section we prove that if $r<s_0=a(n-1)+2$, then $V_r=\ker(\rho_r^{{\rm G}_r})=0$.
Note that $s_0=(n-2)a+(a+2)\in\nums$ (see Notation~\ref{not-fac}).

\begin{example}\label{vrn3}
Suppose that $n=3$, so $a=(n-1)n/2=3$ and $s_0=a(n-1)+2=8$. We have seen in Example~\ref{example-vrn3}
that $V_r=\ker(\rho_r^{{\rm G}_r})=0$, for all $r<s_0=8$. 
\end{example}

Throughout, we intensively use Summary~\ref{summary-semigroup} and a 
version of \cite[Theorem~3.1]{ggp}. We include it here for the sake of easy reference.

\begin{proposition}\label{theo31ggp}
Let $n\geq 3$, $a=(n-1)n/2$ and $\nums=\langle a,a+1,a+2\rangle$. Let $r\in\mbn$ and
let $r=a\ell_r+\rem_r$ be the Euclidean division in $\mbn$ of $r$ by $a$, 
where $\ell_r=\lfloor r/a\rfloor$ is the quotient and 
$\rem_r=r-a\ell_r$ is the remainder. Let $r=au+v$, with $u,v\in\mbn$. The following hold.
\begin{itemize}
\item[$(1)$] $r\in\nums$ if, and only if, $\rem_r\leq 2\ell_r$.
\item[$(2)$] If $0\leq v\leq 2u$, then $r\in\nums$.
\item[$(3)$] If $0\leq v\leq 2u$ and $r<\la$, then $r\in\nums$ and $(u,v)=(\ell_r,\rem_r)$.
\item[$(4)$] If $2u<v<a$, then $(u,v)=(\ell_r,\rem_r)$ and $r\not\in\nums$.
\end{itemize}
\end{proposition}
\begin{proof}
Item $(1)$ is the equivalence $(i)\Leftrightarrow (ii)$ in \cite[Theorem~3.1]{ggp}. 
Item $(2)$ follows from the equivalence $(i)\Leftrightarrow (iv)$ in \cite[Theorem~3.1]{ggp}. 
Suppose that $0\leq v\leq 2u$ and $r<\la$. Then, by \cite[Theorem~2.2]{ggp}, $r\in\ulf(\nums)$, i.e., all the factorizations of $r$ have the same length. 
By \cite[Theorem~3.1, (1)]{ggp} one deduces that $(u,v)=(\ell_r,\rem_r)$, which shows Item $(3)$.
Finally, if $r=au+v$ with $2v\leq v<a$, then $v<a$ and $v=\rem_r$ and $u=\ell_r$. Thus, 
$2\ell_r<\rem_r$, and by Item $(1)$, one deduces that $r\not\in\nums$.
\end{proof}

\begin{remark}\label{rem-s0}
If $n\geq 4$ and $r<s_0$, then $r<\la$ and one can apply \cite[Theorems~3.1, 3.4]{ggp}.
\end{remark}
\begin{proof}
If $n=4$, then $s_0=a(n-1)+2=20<\lfloor a/2\rfloor (a+2)=\la=24$. Suppose that 
$n\geq 5$. By Summary~\ref{summary-semigroup}, $\frob(\nums)<\la$. Let us see that
$s_0\leq\frob(\nums)$. Indeed, if $n=2m+1$, with
$m\geq 2$, then $\lfloor a/2\rfloor =\lfloor  (2m+1)m/2\rfloor \geq 2m+1=n>n-1$.  If $n=2m+2$,
with $m\geq 2$, then $\lfloor  a/2\rfloor=(2m+1)(m+1)\geq 2(m+1)=n>n-1$, as
well.  Therefore, $\lfloor  a/2\rfloor >n-1$, so $\lfloor  a/2\rfloor -(n-1)\geq 1$. Thus,
$a(\lfloor a/2\rfloor -(n-1))\geq a\geq 3=2+1$ and $\frob(\nums)=a\lfloor a/2\rfloor -1\geq s_0$.
\end{proof}

\begin{lemma}\label{lemma-n-1}
Let $n\geq 4$. Let $r\in \nums$, $r\neq 0$, $r<s_0$.  Then $\ell_r\leq n-1$. 
If in addition $\ell_r=n-1$, then $\delta_r=1$ and $V_r=\ker(\rho_r^{{\rm G}_r})=0$.
\end{lemma}
\begin{proof}
By definition, $r=a\ell_r+\rem_r$, where $\ell_r=\lfloor r/a\rfloor$, and, by 
Proposition~\ref{theo31ggp}, $0\leq \rem_r\leq 2\ell_r$. If $\ell_r>n-1$, then $r=a\ell_r+\rem_r\geq
an+\rem_r=a(n-1)+(a+\rem_r)\geq a(n-1)+3>s_0$, a contradiction. Suppose that
$\ell_r=n-1$.  By \cite[Theorem~3.4]{ggp}, 
$n-1=\ell_r=2(\delta_r-1)+\iota_r$, where $\delta_r\geq 1$ and $\iota_r\geq
0$. Thus, $0\leq \iota_r\leq n-1$.  Moreover, $a(n-1)+\rem_r=a\ell_r+\rem_r=r<s_0$, which forces
$\rem_r\in\{0,1\}$. Then, 
$c_r=\rem_r-\ell_r=\rem_r-n+1\in\{-n+1,-n+2\}$. Since $n\geq 4$, then $c_r\leq
-2<0$.  Recall that $c_r\in\Gamma_{\iota_r}$, where $\Gamma_0=\{0\}$,
$\Gamma_1=\{-1,0,1\}$ and, for $\iota_r\geq 2$,
$\Gamma_{\iota_r}=\{-\iota_r,-\iota_r+1,\iota_r-1,\iota_r\}$. Thus $\iota_r\geq 2$ and
$c_r\in\Gamma_{\iota_r}\bigcap\mbz_{<0}=\{-\iota_r,-\iota_r+1\}$. Suppose that
$c_r=-\iota_r$. Then $1\geq \rem_r=\ell_r+c_r=n-1-\iota_r$ and $\iota_r\geq n-2$. Thus,
$n-2\leq \iota_r\leq n-1$. Since $\ell_r=n-1$ and $\iota_r$ and $\ell_r$ have the same parity, then $\iota_r=n-1$.
In particular,
$n-1=\ell_r=2(\delta_r-1)+\iota_r=2(\delta_r-1)+(n-1)$ and $\delta_r=1$. Suppose that $c_r=-\iota_r+1$, where $\iota_r\geq 2$.  Then $1\geq \rem_r=\ell_r+c_r=n-1-\iota_r+1=n-\iota_r$, so $\iota_r\geq n-1$ and, 
since $\iota_r\leq \ell_r=n-1$, it follows that $\iota_r=n-1$. Therefore
$n-1=\ell_r=2(\delta_r-1)+\iota_r=2(\delta_r-1)+(n-1)$ and $\delta_r=1$ again.  Since $\delta_r=1$, 
then, by Corollary~\ref{corollary-vr0}, $V_r=\ker(\rho_r^{{\rm G}_r})=0$.
\end{proof}

\begin{lemma}\label{lemma-remainder}
Let $n\geq 4$. Let $r\in\nums$, $r\neq 0$, $r<\la$. Suppose that $\delta_r\geq 2$.
If $2\ell_r<\ell_r+c+kq<a$, for all $k\in [1,\delta_r-1]$, then $V_r=\ker(\rho_r^{{\rm G}_r})=0$.
\end{lemma}
\begin{proof}
Write $r=a\ell_r+\rem_r=a\ell_r+\ell_r+c_r$, where $c_r=\rem_r-\ell_r$. 
So, $r+kq=a\ell_r+(\ell_r+c_r+kq)$. If $2\ell_r<\ell_r+c+kq<a$, for all $k\in [1,\delta_r-1]$, 
then, by Proposition~\ref{theo31ggp}, $r+kq\not\in\nums$, for all $k\in [1,\delta_r-1]$. Therefore, 
$\{r+kq\mid k\in[1,\delta_r-1]\}\cap\nums=\emptyset$. Thus, by Lemma~\ref{lemma-good}, 
$V_r=\ker(\rho_r^{\rm G})=0$.
\end{proof}

\begin{remark}\label{aboutq}
Recall that $q=2\lfloor (n+1)/2\rfloor-1$. In particular, $(q+1)/2=\lfloor (n+1)/2\rfloor$. 
Note that, if $n$ is odd, then $q=n$ and,  if $n$ is even, then $q=n-1$. In particular, $\lfloor n/2\rfloor q=a$, $q$ is always odd and $n-1\leq q\leq n$.
\end{remark}

Now we can state and prove the main result of the section, the second step, 
which ensures that the good kernels are zero for small enough $r\in\nums$. 

\begin{theorem}\label{SecondStep}
Let $n\geq 3$. Let $r\in \nums$, $r\neq 0$, $r<s_0=a(n-1)+2$. 
Then $V_r=\ker(\rho_r^{{\rm G}_r})=0$. In particular, 
\begin{eqnarray*}
\min\{\sord(f)\mid f\in \ker(\rho), f\neq 0\}\geq s_0.
\end{eqnarray*}
\end{theorem}
\begin{proof}
By the Example~\ref{vrn3}, we can suppose that $n\geq 4$.

Using \cite[Theorem~3.4]{ggp}, write $r=(a+1)(2\delta_r-2+\iota_r)+c_r$, with $\delta_r\geq
1$, $\ell_r=2\delta_r-2+\iota_r$, $0\leq \iota_r\leq \ell_r$, and
$c_r\in\Gamma_{\iota_r}$, where $\Gamma_0=\{0\}$, $\Gamma_1=\{-1,0,1\}$ and $\Gamma_{i}:=\{-i,-i+1,i-1,i\}$, for $i\geq 2$.
By Corollary~\ref{corollary-vr0} and
Lemma~\ref{lemma-n-1}, we can suppose that $\delta_r\geq 2$ and
$\ell_r\leq n-2$. The proof is divided into the following four cases:
$\iota_r=0$, $\iota_r=1$, $\iota_r\geq 2$ and $c_r$ positive, and, finally, $\iota_r\geq 2$
and $c_r$ negative. In the first three cases one shows that $2\ell_r<\ell_r+c_r+kq<a$, for all $1\leq k\leq \delta_r-1$,
and applying Lemma~\ref{lemma-remainder}, one deduces $V_r=\ker(\rho_r^{{\rm G}_r})=0$. The proof of the fourth case
is slightly different though the idea is the same. To make less burdensome the proof we simply write $\ell$, $d$, $i$ and $c$ instead of $\ell_r$, $\delta_r$, $\iota_r$ and $c_r$, if no need of disambiguation.  

\vspace*{0,3cm}

\noindent \underline{{\sc Case 1}: $r=(a+1)(2d-2)$, with $d\geq 2$, $\ell=2d-2$, $i=0$ and $c=0$.}

\vspace*{0,2cm}

\noindent Observe that $\ell=2d-2$ is even. So, if $\ell=n-2$,
then $n$ is even too. Therefore, if $n$ is odd, then $\ell\leq n-3$.
Suppose first that $n$ is odd, so $q=2\lfloor(n+1)/2\rfloor-1=n$. For
$k\leq d-1$, one has
\begin{eqnarray*}
\ell+c+kq\leq n-3+(d-1)n=-3+dn=-3+\frac{\ell+2}{2}n\leq -3+\frac{n-3+2}{2}n=-3+a<a.
\end{eqnarray*}
Since $\ell\leq n-3$, $c=0$, $k\geq 1$ and $q=n$, then
$\ell+c+kq\geq \ell+n\geq\ell+\ell+3>2\ell$.

\noindent Suppose that $n$ is even. Then $q=2\lfloor(n+1)/2\rfloor-1=n-1$ and $\ell\leq n-2$.
For $k\leq d-1$, then
\begin{eqnarray*}
&&\ell+c+kq\leq n-2+(d-1)(n-1)=-1+d(n-1)=
  \\&&-1+\frac{\ell+2}{2}(n-1)\leq
  -1+\frac{(n-2)+2}{2}(n-1)=-1+a<a.
\end{eqnarray*}
Since $\ell\leq n-2$, $c=0$, $k\geq 1$ and $q=n-1$, then
$\ell+c+kq\geq \ell+n-1\geq\ell+\ell+1>2\ell$. 
Therefore, $2\ell<\ell+c+kq<a$, for all $1\leq k\leq d-1$. By Lemma~\ref{lemma-remainder}, 
$V_r=\ker(\rho_r^{{\rm G}_r})=0$.

\vspace*{0,3cm}

\noindent \underline{{\sc Case 2}: $r=(a+1)(2d-1)+c$, with
$d\geq 2$, $\ell=2d-2+i=2d-1$, $i=1$,
$c\in\Gamma_1=\{-1,0,1\}$.}

\vspace*{0,2cm}

\noindent Using $k\leq d-1$, $\ell\leq n-2$, $c\leq 1$, $q\leq n$
and $d=(\ell+1)/2\leq (n-1)/2$, then
\begin{eqnarray*}
  \ell+c+kq\leq (n-2)+1+(d-1)n\leq
  n-1+\left(\frac{n-1}{2}-1\right)n=n-1+a-n<a.
\end{eqnarray*}
Suppose that $n$ is odd, so $q=n$. Since $\ell\leq n-2$, $c\geq
-1$, $k\geq 1$ and $q=n$, then
\begin{eqnarray*}
\ell+c+kq\geq \ell-1+n\geq \ell-1+\ell+2>2\ell.
\end{eqnarray*}
Suppose that $n$ is even. Then $q=n-1$ and $n-2$ is even.  Since
$\ell=2d-1$ is odd, then $\ell\neq n-2$ and $\ell\leq n-3$.
Since $\ell\leq n-3$, $c\geq -1$, $k\geq 1$ and $q=n-1$, then
\begin{eqnarray*}
\ell+c+kq\geq \ell-1+n-1\geq \ell-1+\ell+2>2\ell.
\end{eqnarray*}  
Thus, $2\ell<\ell+c+kq<a$, for all $1\leq k\leq d-1$. By Lemma~\ref{lemma-remainder}, 
$V_r=\ker(\rho_r^{{\rm G}_r})=0$.

\vspace*{0,3cm}

\noindent \underline{{\sc Case 3}: $r=(a+1)(2d-2+i)+c$, with $d\geq 2$, $\ell=2d-2+i$, $2\leq i\leq \ell$,
$c\in\{i-1,i\}\subset\Gamma_i$.}

\vspace*{0,2cm}

\noindent Suppose that $\ell=n-2$ and that $n$ is odd. In
particular, $q=n$, $\ell=2d-2+i$ is odd and $i$ must be odd, with
$3\leq i\leq \ell=n-2$. Using $\ell=n-2$, $c\leq i\leq n-2$,
$k\leq d-1$, $q=n$, $d=(\ell+2-i)/2=(n-i)/2$ and $i\geq 3$, then
\begin{eqnarray*}
&&\ell+c+kq\leq n-2+n-2+(d-1)n=2n-4+\frac{n-2-i}{2}n=\\&&
  2n-4+a-\frac{1+i}{2}n\leq 2n-4+a-2n<a.
\end{eqnarray*}
Moreover, since $\ell=n-2$, $c\geq i-1\geq 2$, $k\geq 1$ and
$q=n$, then
\begin{eqnarray*}
\ell+c+kq\geq \ell+2+n=\ell+2+\ell+2>2\ell.
\end{eqnarray*} 

\vspace*{0,2cm}

\noindent Suppose that $\ell=n-2$ and that $n$ is even. In
particular, $q=n-1$, $\ell=2d-2+i$ is even and $i$ must be even,
with $2\leq i\leq \ell=n-2$. Using $\ell=n-2$, $c\leq i\leq
n-2$, $k\leq d-1$, $q=n-1$, $d=(\ell+2-i)/2=(n-i)/2$ and $i\geq 2$,
then
\begin{eqnarray*}
&&\ell+c+kq\leq n-2+n-2+(d-1)(n-1)=2n-4+\frac{n-2-i}{2}(n-1)=\\&&
  2n-4+a-\frac{2+i}{2}(n-1)\leq 2n-4+a-2(n-1)<a.
\end{eqnarray*}
Moreover, since $\ell=n-2$, $c\geq i-1\geq 1$, $k\geq 1$ and
$q=n-1$, then
\begin{eqnarray*}
\ell+c+kq\geq \ell+1+n-1=\ell+\ell+2>2\ell.
\end{eqnarray*}

\vspace*{0,2cm}

\noindent Now suppose that $\ell\leq n-3$.  Using $\ell\leq n-3$,
$c\leq i\leq \ell\leq n-3$, $k\leq d-1$, $q\leq n$,
$d=(\ell+2-i)/2\leq (n-1-i)/2$ and $i\geq 2$, then
\begin{eqnarray*}
&&\ell+c+kq\leq n-3+n-3+(d-1)n\leq 2n-6+\frac{n-3-i}{2}n=\\&&
  2n-6+a-\frac{2+i}{2}n\leq 2n-6+a-2n<a.
\end{eqnarray*}
Furthermore, since $\ell\leq n-3$, $c\geq i-1\geq 1$, $k\geq 1$ and
$q\geq n-1$, then
\begin{eqnarray*}
\ell+c+kq\geq \ell+1+n-1\geq \ell+\ell+3>2\ell.
\end{eqnarray*}  
So, we have shown that $2\ell<\ell+c+kq<a$, for all $1\leq k\leq d-1$. By Lemma~\ref{lemma-remainder}, 
$V_r=\ker(\rho_r^{{\rm G}_r})=0$.

\vspace*{0,3cm}

\noindent \underline{{\sc Case 4}: $r=(a+1)(2d-2+i)+c$, with $d\geq 2$, $\ell=2d-2+i$, $2\leq i\leq \ell$,
$c\in\{-i,-i+1\}\subset\Gamma_i$.}

\vspace*{0,2cm}

\noindent \underline{\sc Claim}. Set
$\eta=\min(\phi_{r,1},\lfloor n/2\rfloor-1)$. Then, $\eta\geq \delta_r-1\geq 1$ and
$\{r+kq\mid k\in [2,\eta]\}\cap\nums=\emptyset$.
\begin{proof}[Proof of the Claim] 
Recall that $d-1=\kappa_r=\min(\phi_{r,1},\phi_{r,3})\leq\phi_{r,1}$. On
the other hand, since $\ell=2d-2+i$, $\ell\leq n-2$ and $i\geq 2$,
then $2d=\ell+2-i\leq n-i\leq n-2$ and $d\leq \lfloor n/2\rfloor-1$. Since
$d\geq 2$, then $1\leq d-1\leq \min(\phi_{r,1},\lfloor n/2\rfloor -1)=\eta$.

Let us prove that $r+kq\not\in\nums$, for all $2\leq k\leq\eta$. Since
$r+kq=a\ell+(\ell+c+kq)$, by Proposition~\ref{theo31ggp}, it
is enough to show that $2\ell<\ell+c+kq<a$. First, suppose that $n$ is odd. 
Then $\lfloor n/2\rfloor=(n-1)/2$ and
$k\leq\eta\leq ((n-1)/2)-1$. Since $\ell\leq n-2$, $c\leq -i+1$,
$k\leq ((n-1)/2)-1$, $q=n$ and $i\geq 2$, then
\begin{eqnarray*}
\ell+c+kq\leq n-2-i+1+\left(\frac{n-1}{2}-1\right)n\leq n-3+a-n<a.
\end{eqnarray*}
Suppose that $n$ is even. Then $\lfloor n/2\rfloor=n/2$ and $k\leq\eta\leq (n/2)-1$.
Since $\ell\leq n-2$, $c\leq -i+1$, $k\leq (n/2)-1$, $q=n-1$ and
$i\geq 2$, then
\begin{eqnarray*}
\ell+c+kq\leq n-2-i+1+\left(\frac{n}{2}-1\right)(n-1)\leq
n-3+a-(n-1)<a.
\end{eqnarray*}
Moreover, for any $n\geq 4$, since $d\geq 2$, $\ell=2d-2+i$, $c\geq -i$,
$k\geq 2$, $q\geq n-1$ and $\ell\leq n-2$, then
\begin{eqnarray*}
\ell+c+kq\geq 2d-2+i-i+2(n-1)\geq 2+2(n-1)=2n> 2\ell,
\end{eqnarray*}
which finishes de proof of the {\sc Claim}.
\end{proof}

\vspace*{0,2cm}

\noindent \underline{{\sc Subcase 4.1}: $\ell<\lfloor(n+1)/2\rfloor$}. 

\vspace*{0,2cm}

\noindent We begin by proving that $r+q\not\in\nums$. Since
$r+q=a\ell+(\ell+c+q)$, by Prposition~\ref{theo31ggp}, it is
enough to prove that $2\ell<\ell+c+q<a$.  Using that $c\leq -i+1$, $i\geq 2$, $q\leq n$, and $3<n$, then
\begin{eqnarray*}
\ell+c+q\leq\left\lfloor\frac{n+1}{2}\right\rfloor-1-i+1+n\leq \left\lfloor\frac{n+1}{2}\right\rfloor-2+n\leq 
\frac{3n-3}{2}=\frac{n-1}{2}3<\frac{n-1}{2}n=a.
\end{eqnarray*}
On the other hand, since $d\geq 2$, $\ell=2d-2+i\geq 2+i$, $c\geq
-i$ and $q\geq n-1$, then
\begin{eqnarray*}
\ell+c+q\geq 2+i-i+n-1=n+1=2\left(\frac{n+1}{2}\right)\geq
2\left(\left\lfloor\frac{n+1}{2}\right\rfloor\right)>2\ell.
\end{eqnarray*}
Therefore, $r+q\not\in\nums$. 
The {\sc Claim}, together with $r+q\not\in\nums$, implies $\{r+kq\mid k\in [1,\eta]\}\cap\nums=\emptyset$, where
$1\leq\eta\leq\phi_{r,1}$.  By Lemma~\ref{lemma-good},
$[0,\eta]\subseteq\good_r$. Since 
$\eta\geq\delta_r-1$, it follows that
$\card(\good_r)\geq\eta+1\geq\delta_r$. By Corollary~\ref{corollary-vr0},
$V_r=\ker(\rho_r^{{\rm G}_r})=0$.

\vspace*{0,2cm}

\noindent \underline{{\sc Subcase 4.2}: $\ell\geq \lfloor(n+1)/2\rfloor$ and $c>\ell-q$}. 

\vspace*{0,2cm}

\noindent Let us see that $r+q\not\in\nums$. Since
$r+q=a\ell+(\ell+c+q)$, by Proposition~\ref{theo31ggp}, it is
enough to prove that $2\ell<\ell+c+q<a$. Using that $\ell\leq n-2$, $c\leq -i+1$, $i\geq 2$, $q\leq n$ and $4\leq n$, then
\begin{eqnarray*}
\ell+c+q\leq n-2-i+1+n\leq 2n-3=\frac{4(n-1)}{2}-1\leq
\frac{n-1}{2}n-1=a-1<a.
\end{eqnarray*}
Moreover, since $c\geq \ell-q+1$, then
\begin{eqnarray*}
\ell+c+q\geq \ell+\ell-q+1+q\geq 2\ell+1>2\ell.
\end{eqnarray*}
The {\sc Claim}, together with $r+q\not\in\nums$, implies 
$\{r+kq\mid k\in [1,\eta]\}\cap\nums=\emptyset$, where
$1\leq\eta\leq\phi_{r,1}$.  By Lemma~\ref{lemma-good},
$[0,\eta]\subseteq\good_r$. Since 
$\eta\geq\delta_r-1$, $\card(\good_r)\geq\delta_r$ and  $V_r=\ker(\rho_r^{{\rm G}_r})=0$.

\vspace*{0,2cm}

\noindent \underline{{\sc Subcase 4.3}: $\ell\geq \lfloor (n+1)/2\rfloor$ and $c\leq \ell-q$.}

\vspace*{0,2cm}

\noindent We begin by proving that $u=r+q\in\nums$. 
Note that $u=r+q=a\ell+(\ell+c+q)$. Since $\ell\leq n-2$, $c\leq\ell-q$ and $n\geq 4$, 
then $\ell+c+q\leq \ell+\ell-q+q=2\ell$, where $2\ell\leq 2(n-2)<(n-1)n/2=a$. 
Since $\ell=2d-2+i\geq 2+i$, $c\geq -i$ and $q\geq n-1$, then
$\ell+c+q\geq 2+i-i+n-1\geq n+1\geq 0$. By Proposition~\ref{theo31ggp}, $r+q\in\nums$. 
Moreover, since $0\leq (\ell+c+q)<a$, then
$\ell+c+q$ is the remainder of the Euclidean division of $u$ by $a$ in $\mbn$, so
$\rem_u=\ell+c+q$ and hence $\ell_u=\ell$. 

Let us see that $\eta-\phi_{u,1}\geq\delta_r$. Suppose that
$c=-i$. Since $\ell=2d-2+i$, then $\ell+c=2d-2$ is
even. Moreover, since $q$ is always odd, then $q+1$ is even. Therefore
$\ell+c+q+1$ is even and $\lfloor(\ell+c+q+1)/2\rfloor=(\ell-i+q+1)/2$.
Suppose that $c=-i+1$. Since $\ell=2d-2+i$, then $\ell+c=2d-1$ is
odd. Since $q+1$ is even, then $\ell+c+q+1=\ell-i+1+q+1$ is odd and 
$\lfloor(\ell+c+q+1)/2\rfloor=(\ell-i+q+1)/2$.
again. Therefore,
\begin{eqnarray*}
\phi_{u,1}=\ell_u-\left\lfloor\frac{\rem_u+1}{2}\right\rfloor=
\ell-\left\lfloor\frac{\ell+c+q+1}{2}\right\rfloor=\ell-
\frac{\ell-i+q+1}{2}=\frac{\ell+i}{2}-\frac{q+1}{2}.
\end{eqnarray*}
Note that, in this subcase, $\eta=\min(\phi_{r,1},\lfloor n/2\rfloor-1)=\lfloor n/2\rfloor-1$.
Indeed, since $c\leq\ell-q$, $q$ is odd and $(q-1)/2=\lfloor (n-1)/2\rfloor$, then
\begin{eqnarray*}
\phi_{r,1}=\ell_r-\left\lfloor\frac{\rem_r+1}{2}\right\rfloor=\ell-
\left\lfloor\frac{\ell+c+1}{2}\right\rfloor\geq \ell-
\left\lfloor\frac{2\ell-q+1}{2}\right\rfloor=\left\lfloor\frac{q-1}{2}\right\rfloor=
\left\lfloor\frac{n-1}{2}\right\rfloor\geq\left\lfloor\frac{n}{2}\right\rfloor-1.
\end{eqnarray*}
Suppose that $n$ is odd, so $q=n$ and $\eta=\frac{n-1}{2}-1$. Since $\ell=2d-2+i$ and $\ell\leq n-2$, then
\begin{eqnarray*}
\eta-\phi_{u,1}=\frac{n-1}{2}-1-\frac{\ell+i}{2}+\frac{n+1}{2}=n-1
-\frac{\ell+i}{2}\geq \ell+1-(d-1+i)=d.
\end{eqnarray*}
Suppose that $n$ is even, so $q=n-1$ and $\eta=\frac{n}{2}-1$. Since $\ell=2d-2+i$ and $\ell\leq n-2$, 
\begin{eqnarray*}
\eta-\phi_{u,1}=\frac{n}{2}-1-\frac{\ell+i}{2}+\frac{n}{2}=n-1
-\frac{\ell+i}{2}\geq \ell+1-(d-1+i)=d.
\end{eqnarray*}
This proves $\eta-\phi_{u,1}\geq\delta_r\geq 2$, so $2+\phi_{u,1}\leq\eta\leq\phi_{r,1}$. By the {\sc Claim}, 
$\{r+kq\mid k\in [2,\eta]\}\cap\nums=\emptyset$. Since
all the hypotheses of Lemma~\ref{lemma-diffcase} are satisfied, then $\{0\}\cup [2+\phi_{u,1},\eta]\subseteq\good_r$.
Thus, $\card(\good_r)\geq 1+\eta-(2+\phi_{u,1}-1)=\eta-\phi_{u,1}\geq\delta_r$. By
Corollary~\ref{corollary-vr0}, $V_r=\ker(\rho_r^{\rm G})=0$.
\end{proof}

We give an instance that all the cases and subcases in the proof of 
Theorem~\ref{SecondStep} can occur.
From now on, we distinguish the pairs $(\iota_r,c_r)$ by using the same colors as in \cite{ggp}.

\begin{example}\label{ex-canoccur}
Let $n=5$, so $a=10$ and $s_0=a(n-1)+2=46$. Clearly, $44=10\cdot 4\in\nums$ and $44<s_0$. Then,
$r=44$ is an example of {\sc Case 1}, because $44=11\cdot 4=(a+1)(2\delta_r-2-\iota_r)+c_r$, with 
$\ell_r=4$, $\delta_r=2$, $\iota_r=0$ and $c_r=0$.

Let $n=6$, so $a=15$ and $s_0=77$. Clearly, $47=2\cdot 15+17\in\nums$ and $47<s_0$. 
Then, $r=47$ is an example of {\sc Case 2}, because $47=16\cdot 3-1=(a+1)(2\delta_r-2+\iota_r)+c_r$, with $\ell_r=3$, $\delta_r=2$, and $(\iota_r,c_r)={\cred (1,-1)}$. 

Let $n=6$, $a=15$, $s_0=77$. Clearly, $48=15+16+17\in\nums$, $49=15+2\cdot 17\in\nums$, 
and $48,49<s_0$. Then, $r=48=16\cdot 3$ and $r=49=16\cdot 3+1$ are examples of {\sc Case 2}, with 
$\ell_r=3$, $\delta_r=2$ and$(\iota_r,c_r)={\cgreen (1,0)}$ and $(\iota_r,c_r)={\cblue (1,1)}$, respectively. 

Let $n=7$, so $a=21$ and $s_0=128$. Clearly, $112=21+22+3\cdot 23\in\nums$, 
$113=21+4\cdot 23\in\nums$, and $112,113<s_0$. Then, $r=112=22\cdot 5+2$ and $r=113=22\cdot 5+3$ 
are examples of {\sc Case 3}, with $\ell_r=5$, 
$\delta_r=2$ and $(\iota_r,c_r)={\colive (3,2)}$ and 
$(\iota_r,c_r)={\cviolet (3,3)}$, respectively.  

Let $n=9$, so $a=36$ and $s_0=290$. Clearly, $146=3\cdot 36+38\in\nums$, 
$147=2\cdot 36+37+38\in\nums$, and $146,147<s_0$. Then, $r=146=37\cdot 4-2$ and 
$r=147=37\cdot 4-1$ are examples of {\sc Subcase 4.1}, with 
$\ell_r=4<5=\lfloor (n+1)/2\rfloor$, $\delta_r=2$, and $(\iota_r,c_r)={\corange (2,-2)}$ and $(\iota_r,c_r)={\cteal (2,-1)}$, respectively. 

Let $n=9$, $a=36$, $s_0=290$. Clearly, $220=4\cdot 36+2\cdot 38\in\nums$, 
$221=3\cdot 36+37+2\cdot38\in\nums$, and $220,221<s_0$. Then, $r=220=37\cdot 6-2$ and $r=221=37\cdot 6-1$ are examples of {\sc Subcase 4.2}, with $\ell_r=6\geq 5=\lfloor (n+1)/2\rfloor$, $\delta_r=3$ and $(\iota_r,c_r)={\corange (2,-2)}$ and $(\iota_r,c_r)={\cteal (2,-1)}$, respectively, where $c_r>\ell_r-q=6-9=-3$.

Finally, let $n=9$, $a=36$, $s_0=290$. Clearly, $218=5\cdot 36+38\in\nums$, $219=4\cdot 36+37+38\in\nums$, and $218,219<s_0$. Then, $r=218=37\cdot 6-4$ and $r=219=37\cdot 6-3$ are examples of {\sc Subcase 4.3}, with $\ell_r=6\geq 5=\lfloor (n+1)/2\rfloor$, $\delta_r=2$ and $(\iota_r,c_r)={\corange (4,-4)}$ and $(\iota_r,c_r)={\cteal (4,-3)}$, respectively, where $c_r\leq\ell_r-q=6-9=-3$.
\end{example}

\begin{remark}
The idea of the proof of Theorem~\ref{SecondStep} is to find a large enough subset $L$ 
of $\good_r$, so as to ensure that $V_r=\ker(\rho_r^L)=0$. This is what we do
in Cases 1, 2 and 3, where we show $[0,\delta_r-1]\subseteq\good_r$ and then apply
Lemma~\ref{lemma-good}. In Case 4, we prove $\{r+kq\mid k\in[2,\eta]\}\cap\nums=\emptyset$, where 
$\eta=\min(\phi_{r,1},\lfloor n/2-1\rfloor-1)\geq \delta_r-1$. In Subcases 4.1 and 4.1, we show $1\in\good_r$, so
$[0,\eta]\subseteq\good_r$, and since $\eta\geq\delta_r-1$, we deduce $V_r=\ker(\rho_r^{{\rm G}_r})=0$.
The most complicated case is Subcase 4.3, the only one in which $1\not\in\good_r$. The reason of the election of $\eta$
can be understood in the following way. On the one hand, $\eta$ must be smaller than or equal to $\phi_{r,1}$, since
$\good_r\subset [0,\phi_{r,1}]$. On the other hand, since $\lfloor n/2\rfloor q=a$, then
$r+\lfloor n/2\rfloor q=r+a\in\nums$ and $\lfloor n/2\rfloor\not\in\good_r$. 
Taking $\eta=\min(\phi_{r,1},\lfloor n/2\rfloor-1)$, we are able to show that
$\{0\}\cup [2+\phi_{u,1},\eta]\subseteq\good_r$, which is enough to our purposes.
\end{remark}

\section{Bases for the good kernels}\label{sec-kernels}

In this section we give a basis of $\ker(\rho_r^{{\rm G}_r})$, 
for $r\in [s_0,s_0+n)$, where $s_0=a(n-1)+2$ and $n\geq 6$. 

\begin{remark}\label{remn5}
Note that, if $n\geq 5$, then
$n(n-5)\geq 0$, so $n(n-1)/4\geq n$ and $\lfloor a/2\rfloor\geq n\geq
1$. Thus, $a(\lfloor a/2\rfloor-n)+2(\lfloor a/2\rfloor-1)\geq 0$ and
$\left\lfloor a/2\right\rfloor(a+2)\geq an+2=s_0+a>s_0+n$. Therefore, 
$s_0+n<\left\lfloor a/2\right\rfloor(a+2)\leq \la$.

Though all the results in this section hold for $n=5$, to facilitate the reading and to avoid 
distinguishing the limit case $n=5$, we suppose that $n\geq 6$. We also suppose $r<s_0+n$. In particular, $r<\la$, and
one can apply \cite[Notations~2.9 and Theorems~3.1 and 3.4]{ggp} (see also Proposition~\ref{theo31ggp}). 
Therefore, given $r\in\nums$, $r\neq 0$, $r<s_0+n$, we can write
\begin{eqnarray*}
r=a\ell_r+\rem_r=(a+1)(2\delta_r-2+\iota_r)+c_r, \mbox{ with }
\end{eqnarray*}
$\ell_r=\lfloor r/a\rfloor$, $0\leq \rem_r<a$, $\rem_r\leq 2\ell_r$, $\delta_r\geq 1$,
$\ell_r=2\delta_r-2+\iota_r$, $0\leq \iota_r\leq\ell_r$, $c_r\in\Gamma_{\iota_r}$ and $c_{r}=\rem_r-\ell_r$.
\end{remark}

\begin{lemma}\label{lemma-pre}
Let $n\geq 6$. 
If $r\in [s_0,s_0+n)$, then $\ell_r=n-1$. Moreover, $[s_0,s_0+n)\subset\nums$.
\end{lemma}
\begin{proof}
Let $r\in [s_0,s_0+n)$. If $\ell_r\geq n$, then $an+\rem_r\leq a\ell_r+\rem_r=r<s_0+n=a(n-1)+2+n$ and $a+\rem_r<n+2$. 
Since $n\geq 6$, then $(n-1)/2\geq 2$. Thus, $2n\leq (n-1)n/2=a\leq a+\rem_r<n+2$ and $n<2$, a contradiction. 
If $\ell\leq n-2$, then $a(n-1)+2=s_0\leq r=a\ell_r+\rem_r\leq a(n-2)+\rem_r$ and $a+2\leq \rem_r<a$, a contradiction. Thus,
$\ell_r=n-1$ and $r=a(n-1)+\rem_r$. If $r\leq s_0+n-1=a(n-1)+n+1$, 
then $\rem_r\leq n+1\leq 2(n-1)=2\ell_r$.  
By Proposition~\ref{theo31ggp},~$(1)$, $r\in\nums$.
\end{proof}

As before, we distinguish the pairs $(\iota_r,c_r)$ by using the same colors as in \cite{ggp}.

\begin{lemma}\label{lemma-general}
Let $n\geq 6$. Let $r\in [s_0,s_0+n)$. Write $r=(a+1)(2\delta_r-2+\iota_r)+c_r$. The following hold.
\begin{itemize}  
\item[$(1)$] $\iota_r$ is odd if and only if $n$ is even, i.e., 
$\iota_r\equiv n-1\!\pmod{2}$.
\item[$(2)$] If $c_r\geq 0$, then $\iota_r\leq 3$ and the pair
$(\iota_r,c_r)$ is in $\{(0,0),{\cgreen(1,0)},{\cblue(1,1)},{\colive (2,1)},{\cviolet (2,2)},{\colive(3,2)}\}$.
\item[$(3)$] If $c_r<0$, then $\iota_r\leq n-3$ and the pair $(\iota_r,c_r)$ is in 
\begin{eqnarray*}
\{{\cred (1,-1)}\}\cup \{{\corange (i,-i)},{\cteal (i,-i+1)}\mid 2\leq i\leq n-3,\; i\equiv n-1
\hspace{-0.3cm}\pmod{2}\}.
\end{eqnarray*}
\item[$(4)$] $\delta_r\geq 2$.
\end{itemize}  
\end{lemma}
\begin{proof}
Since $2\delta_r-2+\iota_r=\ell_r=(n-1)$, then $\iota_r\equiv \ell_r=n-1\!\pmod{2}$.
Recall that $c_r\in\Gamma_{\iota_r}$, where $\Gamma_0=\{0\}$,
$\Gamma_1=\{-1,0,1\}$ and $\Gamma_i=\{-i,-i+1,i-1,i\}$, for $i\geq 2$.
Substituting the expression $r=(a+1)\ell_r+c_r=a(n-1)+n-1+c_r$ in
$s_0\leq r<s_0+n$, we get $a(n-1)+2\leq a(n-1)+n-1+c_r<a(n-1)+2+n$, so
$2\leq n-1+c_r<2+n$. Therefore, $3-n\leq c_r<3$.

Suppose that $c_r\geq 0$. Then, $\iota_r-1\leq c_r<3$ and $\iota_r\leq 3$. Thus,
$0\leq \iota_r\leq 3$ and $0\leq c_r\leq 2$ and the only possible pairs
$(\iota_r,c_r)$ are $(0,0)$, ${\cgreen (1,0)}$, ${\cblue (1,1)}$, ${\colive (2,1)}$, 
${\cviolet (2,2)}$ and ${\colive (3,2)}$.

Suppose that $c_r<0$. In particular, $\iota_r>0$. Then, $3-n\leq c_r\leq -\iota_r+1$ and $0<\iota_r\leq
n-2$. Since $n$ and $\iota_r$ have distinct parity, $\iota_r\leq n-3$. Thus, if
$\iota_r=1$, then $c_r=-1$ and $(\iota_r,c_r)={\cred (1,-1)}$. 
If $\iota_r\geq 2$, then ${\corange c_r=-\iota_r}$ or 
${\cteal c_r=-\iota_r+1}$, where
$2\leq \iota_r\leq n-3$ and $\iota_r\equiv n-1\!\pmod 2$.

Since $\ell_r=n-1$ and $\ell_r=2\delta_r-2+\iota_r$, then $\delta_r=(n+1-\iota_r)/2$. 
If $c_r\geq 0$, by $(2)$ and since $n\geq 6$,
$\iota_r\leq 3\leq n-3$; if $c_r<0$, by $(3)$, $\iota_r\leq n-3$.  Thus,
$\delta_r=(n+1-\iota_r)/2\geq (n+1-n+3)/2=2$.
\end{proof}

Next result stratifies the elements of $[s_0,s_0+n)$ into five subsets, when $n$ is odd, and
into six subsets, when $n$ is even. This will be crucial in the sequel.

\begin{proposition}\label{proposition-class}
Let $n\geq 6$. Let $r\in [s_0,s_0+n)$. Write $r=r_k=s_0+k-1$, where $1\leq k\leq n$. 

\noindent Suppose that $n$ is odd, $n\geq 7$.
\begin{itemize}
\item If $1\leq k\leq n-4$, $k$ odd, then ${\corange c_{r_k}=-\iota_{r_k}}$, 
$2\leq \iota_{r_k}\leq n-3$, $\iota_{r_k}\equiv n-1\hspace{-0.1cm}\pmod{2}$. In fact, $\iota_{r_k}=n-k-2$.
\item If $2\leq k\leq n-3$, $k$ even, then ${\cteal c_{r_k}=-\iota_{r_k}+1}$, $2\leq \iota_{r_k}\leq n-3$,
$\iota_{r_k}\equiv n-1\hspace{-0.1cm}\pmod{2}$. In fact, $\iota_{r_k}=n-k-1$.
\item $(\iota_{r_{n-2}},c_{r_{n-2}})=(0,0)$, $(\iota_{r_{n-1}},c_{r_{n-1}})={\colive (2,1)}$
$(\iota_{r_{n}},c_{r_{n}})={\cviolet (2,2)}$.
\end{itemize}
Suppose that $n$ is even, $n\geq 6$.
\begin{itemize}
\item If $1\leq k\leq n-5$, $k$ odd, then ${\corange c_{r_k}=-\iota_{r_k}}$, 
$3\leq \iota_{r_k}\leq n-3$, $\iota_{r_k}\equiv n-1\hspace{-0.1cm}\pmod{2}$. In fact, $\iota_{r_k}=n-k-2$.
\item If $2\leq k\leq n-4$, $k$ even, then ${\cteal c_{r_k}=-\iota_{r_k}+1}$, 
$3\leq \iota_{r_k}\leq n-3$,
$\iota_{r_k}\equiv n-1\hspace{-0.1cm}\pmod{2}$. In fact, $\iota_{r_k}=n-k-1$.
\item $(\iota_{r_{n-3}},c_{r_{n-3}})={\cred(1,-1)}$, 
$(\iota_{r_{n-2}},c_{r_{n-2}})={\cgreen(1,0)}$, 
$(\iota_{r_{n-1}},c_{r_{n-1}})={\cblue(1,1)}$, 
$(\iota_{r_{n}},c_{r_{n}})={\colive(3,2)}$.
\end{itemize}
\end{proposition}
\begin{proof}
By Lemma~\ref{lemma-pre}, $\ell_r=n-1$ and by 
Lemma~\ref{lemma-general}, $\iota_{r_k}\equiv n-1\!\pmod{2}$. 
Moreover, $r_k=s_0+k-1=a(n-1)+k+1$ and 
$r_k=a\ell_{r_k}+\rem_{r_k}=a(n-1)+\rem_{r_k}$. Thus, 
$\rem_{r_k}=k+1$.

Suppose that $n$ is odd, so $\iota_{r_k}$ is even. 
Suppose that $1\leq k\leq n-4$, with $k$ odd. Then
$\rem_{r_k}=k+1$ is even and
\begin{eqnarray*}
\phi_{r_k}=\left(\ell_{r_k}-\left\lfloor\frac{\rem_{r_k}+1}{2}\right\rfloor,\rem_{r_k}-2\left\lfloor
\frac{\rem_{r_k}}{2}\right\rfloor,\left\lfloor
\frac{\rem_{r_k}}{2}\right\rfloor\right)=\left(n-1-\frac{k+1}{2},0,\frac{k+1}{2}\right),
\end{eqnarray*}
so $\iota_{r_k}=\phi_{r_k,2}+|\phi_{r_k,1}-\phi_{r_k,3}|=n-k-2$ and 
$c_{r_{k}}=\phi_{r_k,3}-\phi_{r_k,1}=k+2-n=-\iota_{r_k}<0$. 
Since $1\leq k\leq n-4$, then $2\leq \iota_{r_k}=n-k-2\leq n-3$ and
${\corange c_{r_k}=-\iota_{r_k}}$.

Suppose that $2\leq k\leq n-3$, with $k$ even. Now $\rem_{r_k}=k+1$ is odd. Thus, 
\begin{eqnarray*}
\phi_{r_k}=\left(\ell_{r_k}-\left\lfloor\frac{\rem_{r_k}+1}{2}\right\rfloor,\rem_{r_k}-2\left\lfloor
\frac{\rem_{r_k}}{2}\right\rfloor,\left\lfloor
\frac{\rem_{r_k}}{2}\right\rfloor\right)=\left(n-1-\frac{k+2}{2},1,\frac{k}{2}\right),
\end{eqnarray*}
so $\iota_{r_k}=\phi_{r,2}+|\phi_{r,1}-\phi_{r,3}|=n-k-1$ and 
$c_{r_{k}}=\phi_{r_k,3}-\phi_{r_k,1}=k+2-n=-\iota_{r_k}+1$. 
Since $2\leq k\leq n-3$, then $2\leq \iota_{r_k}=n-k-1\leq n-3$ and
${\cteal c_{r_k}=-\iota_{r_k}+1}$.

For $k=n-2$, then $\rem_{r_{n-2}}=k+1=n-1$ is even. Thus,
\begin{eqnarray*}
\phi_{r_{n-2}}=\left(\ell_{r_{n-2}}-\left\lfloor\frac{\rem_{r_{n-2}}+1}{2}\right\rfloor,
\rem_{r_{n-2}}-2\left\lfloor\frac{\rem_{r_{n-2}}}{2}\right\rfloor,\left\lfloor
\frac{\rem_{r_{n-2}}}{2}\right\rfloor\right)=\left(\frac{n-1}{2},0,\frac{n-1}{2}\right),
\end{eqnarray*}
so $\iota_{r_{n-2}}=\phi_{r_{n-1},2}+|\phi_{r_{n-2},1}-\phi_{r_{n-2},3}|=0$ and 
$c_{r_{n-2}}=\phi_{r_{n-2},3}-\phi_{r_{n-2},1}=0$. Therefore, $(\iota_{r_{n-2}},c_{r_{n-2}})=(0,0)$.

For $k=n-1$, then $\rem_{r_{n-1}}=k+1=n$ is odd. Thus,
\begin{eqnarray*}
\phi_{r_{n-1}}=\left(\ell_{r_{n-1}}-\left\lfloor\frac{\rem_{r_{n-1}}+1}{2}\right\rfloor,
\rem_{r_{n-1}}-2\left\lfloor\frac{\rem_{r_{n-1}}}{2}\right\rfloor,\left\lfloor
\frac{\rem_{r_{n-1}}}{2}\right\rfloor\right)=\left(\frac{n-3}{2},1,\frac{n-1}{2}\right),
\end{eqnarray*}
so $\iota_{r_{n-1}}=\phi_{r_{n-1},2}+|\phi_{r_{n-1},1}-\phi_{r_{n-1},3}|=2$ and 
$c_{r_{n-1}}=\phi_{r_{n-1},3}-\phi_{r_{n-1},1}=1$. Therefore, $(\iota_{r_{n-1}},c_{r_{n-1}})={\colive (2,1)}$

For $k=n$, then $\rem_{r_{n}}=k+1=n+1$ is even. Thus,
\begin{eqnarray*}
\phi_{r_{n}}=\left(\ell_{r_{n}}-\left\lfloor\frac{\rem_{r_{n}}+1}{2}\right\rfloor,
\rem_{r_{n}}-2\left\lfloor\frac{\rem_{r_{n}}}{2}\right\rfloor,\left\lfloor
\frac{\rem_{r_{n}}}{2}\right\rfloor\right)=\left(\frac{n-3}{2},0,\frac{n+1}{2}\right),
\end{eqnarray*}
so $\iota_{r_{n}}=\phi_{r_{n},2}+|\phi_{r_{n},1}-\phi_{r_{n},3}|=2$ and 
$c_{r_{n}}=\phi_{r_{n},3}-\phi_{r_{n},1}=2$. Therefore, $(\iota_{r_{n}},c_{r_{n}})={\cviolet (2,2)}$.

Now, suppose that $n$ is even, so $\iota_{r_k}$ is odd. 
Suppose that $1\leq k\leq n-5$, with $k$ odd.
Then $\rem_{r_k}=k+1$ is even and, as before, it follows that
$\iota_{r_k}=\phi_{r_k,2}+|\phi_{r_k,1}-\phi_{r_k,3}|=n-k-2$ and 
$c_{r_{k}}=\phi_{r_k,3}-\phi_{r_k,1}=k+2-n=-\iota_{r_k}$. Since $1\leq k\leq n-5$, 
then $3\leq \iota_{r_k}=n-k-2\leq n-3$ and ${\corange c_{r_k}=-\iota_{r_k}}$.

Suppose that $2\leq k\leq n-4$, with $k$ even. Then
$\rem_{r_k}=k+1$ is odd. As before, one deduces that 
$\iota_{r_k}=\phi_{r,2}+|\phi_{r,1}-\phi_{r,3}|=n-k-1$ and 
$c_{r_{k}}=\phi_{r_k,3}-\phi_{r_k,1}=k+2-n=-\iota_{r_k}+1$. Since $2\leq k\leq n-4$, 
then $3\leq \iota_{r_k}=n-k-1\leq n-3$ and ${\cteal c_{r_k}=-\iota_{r_k}+1}$.

If $k=n-3$, then $\rem_{r_{n-3}}=n-2$ is even. Thus, 
\begin{eqnarray*}
\phi_{r_{n-3}}=\left(\ell_{r_{n-3}}-\left\lfloor\frac{\rem_{r_{n-3}}+1}{2}\right\rfloor,
\rem_{r_{n-3}}-2\left\lfloor\frac{\rem_{r_{n-3}}}{2}\right\rfloor,\left\lfloor
\frac{\rem_{r_{n-3}}}{2}\right\rfloor\right)=\left(\frac{n}{2},0,\frac{n-2}{2}\right),
\end{eqnarray*}
so $\iota_{r_{n-3}}=\phi_{r_{n-3},2}+|\phi_{r_{n-3},1}-\phi_{r_{n-3},3}|=1$ and 
$c_{r_{n-3}}=\phi_{r_{n-3},3}-\phi_{r_{n-3},1}=-1$. Therefore, 
$(\iota_{r_{n-3}},c_{r_{n-3}})={\cred(1,-1)}$. 

If $k=n-2$, then $\rem_{r_{n-2}}=n-1$ is odd. Thus, 
\begin{eqnarray*}
\phi_{r_{n-2}}=\left(\ell_{r_{n-2}}-\left\lfloor\frac{\rem_{r_{n-2}}+1}{2}\right\rfloor,
\rem_{r_{n-2}}-2\left\lfloor\frac{\rem_{r_{n-2}}}{2}\right\rfloor,\left\lfloor
\frac{\rem_{r_{n-2}}}{2}\right\rfloor\right)=\left(\frac{n-2}{2},1,\frac{n-2}{2}\right),
\end{eqnarray*}
so $\iota_{r_{n-2}}=\phi_{r_{n-2},2}+|\phi_{r_{n-2},1}-\phi_{r_{n-2},3}|=1$ and 
$c_{r_{n-2}}=\phi_{r_{n-2},3}-\phi_{r_{n-2},1}=0$. Therefore, 
$(\iota_{r_{n-2}},c_{r_{n-2}})={\cgreen(1,0)}$. 

If $k=n-1$, then $\rem_{r_{n-1}}=n$ is even. Thus, 
\begin{eqnarray*}
\phi_{r_{n-1}}=\left(\ell_{r_{n-1}}-\left\lfloor\frac{\rem_{r_{n-1}}+1}{2}\right\rfloor,
\rem_{r_{n-1}}-2\left\lfloor\frac{\rem_{r_{n-1}}}{2}\right\rfloor,\left\lfloor
\frac{\rem_{r_{n-1}}}{2}\right\rfloor\right)=\left(\frac{n-2}{2},0,\frac{n}{2}\right),
\end{eqnarray*}
so $\iota_{r_{n-1}}=\phi_{r_{n-1},2}+|\phi_{r_{n-1},1}-\phi_{r_{n-1},3}|=1$ and 
$c_{r_{n-2}}=\phi_{r_{n-2},3}-\phi_{r_{n-2},1}=1$. Therefore, 
$(\iota_{r_{n-1}},c_{r_{n-1}})={\cblue(1,1)}$.

If $k=n$, then $\rem_{r_{n}}=n+1$ is odd. Thus, 
\begin{eqnarray*}
\phi_{r_{n}}=\left(\ell_{r_{n}}-\left\lfloor\frac{\rem_{r_{n}}+1}{2}\right\rfloor,
\rem_{r_{n}}-2\left\lfloor\frac{\rem_{r_{n}}}{2}\right\rfloor,\left\lfloor
\frac{\rem_{r_{n}}}{2}\right\rfloor\right)=\left(\frac{n-4}{2},1,\frac{n}{2}\right),
\end{eqnarray*}
so $\iota_{r_{n}}=\phi_{r_{n},2}+|\phi_{r_{n},1}-\phi_{r_{n},3}|=3$ and 
$c_{r_{n}}=\phi_{r_{n},3}-\phi_{r_{n},1}=2$. Therefore, 
$(\iota_{r_{n}},c_{r_{n}})={\colive(3,2)}$.
\end{proof}

Next, we give a sufficient condition to ensure that an element of the form $s=r+kq$ satisfies $s<\la$. 
This is relevant in order to apply \cite[Theorems~3.1 and 3.4]{ggp}.  

\begin{lemma}\label{lemma-cota}
Let $n\geq 6$. Let $r\in [s_0,s_0+n)$, and let $s=r+kq$, for some $k\geq0$.
\begin{itemize}
\item[$(1)$] Then $\lfloor a/2\rfloor -n\geq 1$ and, if $n\geq 7$, then $\lfloor a/2\rfloor -n\geq 3$.
\item[$(2)$] If $k\leq(\lfloor a/2\rfloor-n)\lfloor n/2\rfloor$, then $s<\frob(\nums)<\la$. 
\item[$(3)$] If $k\leq \lfloor n/2\rfloor$, then $s<\frob(\nums)<\la$. 
\item[$(4)$] If $k=(\lfloor a/2\rfloor-n+1)\lfloor n/2\rfloor$, then $\frob(\nums)<s<\lfloor a/2\rfloor(a+2)\leq\la$.
\end{itemize}
\end{lemma}
\begin{proof}
Observe that, since $n\geq 6$, then  $n^2-5n-4\geq 0$, 
so $n(n-1)\geq 4n+4$, $n(n-1)/4\geq n+1$ and $\lfloor a/2\rfloor -n\geq 1$.
Similarly, if $n\geq 7$, then $n^2-5n-12\geq 0$, 
$n(n-1)\geq 4n+12$, $n(n-1)/4\geq n+3$ and $\lfloor a/2\rfloor -n\geq 3$.
Suppose that $k\leq(\lfloor a/2\rfloor-n)\lfloor n/2\rfloor$. 
Since $\lfloor n/2\rfloor q=a$, then $s=r+kq\leq s_0+n-1+kq\leq a(n-1)+n+1+(\lfloor a/2\rfloor-n)a
=\lfloor a/2\rfloor a-a+n-1$. Since $n\geq 6$, then $a>n$ and $s=r+kq< a\lfloor a/2\rfloor-1=\frob(\nums)$. 

Suppose that $k\leq\lfloor n/2\rfloor$. Then,
$k\leq\lfloor n/2\rfloor\leq (\lfloor a/2\rfloor-n)\lfloor n/2\rfloor$ and, by $(2)$, $s<\frob(\nums)<\la$. 

Suppose that $k=(\lfloor a/2\rfloor-n+1)\lfloor n/2\rfloor$, then $s=r+kq\geq a(n-1)+2+(\lfloor a/2\rfloor-n+1)a=a\lfloor a/2\rfloor+2>a\lfloor a/2\rfloor-1=\frob(\nums)$. On the other hand, by $(1)$,
$s=r+kq\leq a(n-1)+n+1+kq=a(n-1)+n+1+(\lfloor a/2\rfloor-n+1)\lfloor n/2\rfloor q=
a\lfloor a/2\rfloor+n+1\leq
a\lfloor a/2\rfloor+\lfloor a/2\rfloor<
\lfloor a/2\rfloor(a+2)$.
\end{proof}

Recall from Notation~\ref{notation-IH} that 
$I_r=[\phi_{r,1}-(\delta_r-1),\phi_{r,1}]=[\phi_{r,1}-\kappa_r,\phi_{r,1}]$ and $H_r=[0,\phi_{r,1}]$. 
We understand $[i,j]=\emptyset$ provided that $i>j$. If $I=\{i_1,\ldots,i_l\}$ and $J=\{j_1,\ldots,j_l\}$ have de same cardinality, we say that $I\leq J$, when $i_k\leq j_k$, for all $k$.
Recall that $\{0\}\subseteq\good_r\subseteq [0,\phi_{r,1}]$ (see Definition~\ref{goods}).
In the next nine lemmas we determine $I_r$, $\good_r$ and a basis of $\ker(\rho_r^{{\rm G}_r})$, 
according to which subset of the 
partition of $[s_0,s_0+n)$ given in Proposition~\ref{proposition-class} contains the element $r$. 
Again, in the proofs, if no confusion is possible, 
we just write $\ell$, $d$, $i$ and $c$ instead of $\ell_r$, $\delta_r$, $\iota_r$ and $c_r$. 

\vspace*{0.3cm}

\begin{lemma}\label{lemma-c-i}
Let $n\geq 6$. Let $r\in [s_0,s_0+n)$. Write $r=r_k=s_0+k-1$, with $1\leq k\leq n$. 
Suppose that ${\corange c_r=-\iota_r}$, $2\leq \iota_r\leq n-3$. Then,
\begin{eqnarray*}
\phi_r=\left(\frac{n-1+\iota_r}{2},0,\frac{n-1-\iota_r}{2}\right),\;
I_r=\left[\iota_r,\frac{n-1+\iota_r}{2}\right],\; 
\good_r=\{0\}\cup \left[\left\lfloor\frac{\iota_r+3}{2}\right\rfloor,
\left\lfloor\frac{n-2}{2}\right\rfloor\right].
\end{eqnarray*}
Moreover, $\delta_r=(n+1-\iota_r)/2$, $\card(\good_r)=\delta_r-1$, $\dim\ker(\rho_r^{{\rm G}_r})=1$, 
$\good_r\leq I_r\setminus\{\phi_{r,1}\}$ and 
\begin{itemize}
\item[$(1)$] $s_1=r+q\in\nums$, $s_1<\la$, $\ell_{s_1}=n-1$, $\rem_{s_1}=n-1-\iota_r+q$ and 
$\phi_{s_1}=\left(\lfloor\frac{\iota_r-1}{2}\rfloor,0,n-\lfloor\frac{\iota_r+3}{2}\rfloor\right)$;
\item[$(2)$] $s_2=r+\lfloor n/2\rfloor q\in\nums$, $s_2<\la$, $\ell_{s_2}=n$, $\rem_{s_2}=n-1-\iota_r$ and 
$\phi_{s_2}=\left(\frac{n+\iota_r+1}{2},0,\frac{n-\iota_r-1}{2}\right)$.
\end{itemize}
Furthermore, 
\begin{itemize}
\item[$(3)$] $g=\sum_{j=0}^{\frac{n-1-\iota_r}{2}}\lambda_j\mon^{\phi_r+j\omega}$, where 
$\lambda_j=(-1)^jb_{\{0\}\sqcup[\lfloor\frac{\iota_r+3}{2}\rfloor,
\lfloor\frac{n-2}{2}\rfloor]}^{[\iota_r,\frac{n-1+\iota_r}{2}]\setminus 
\{\frac{n-1+\iota_r-2j}{2}\}}\neq 0$, for all $j$, is a basis of $\ker(\rho_r^{{\rm G}_r})$. 
\item[$(4)$]
If $n$ is odd and $1\leq k\leq n-4$, with $k$ odd, then $g$ is equal to the polynomial $g_k$ defined in 
\eqref{equality-gknoddkodd}. In particular, 
$g_k$ is a basis of $\ker\left(\rho_{r_k}^{{\rm G}_{r_k}}\right)$. 
\item[$(5)$]
If $n$ is even and $1\leq k\leq n-5$, with $k$ odd, then $g$ is equal to the polynomial $g_k$ defined in 
\eqref{equality-gknevenkodd}. In particular, $g_k$ is a basis of 
$\ker\left(\rho_{r_k}^{{\rm G}_{r_k}}\right)$. 
\end{itemize}
\end{lemma}
\begin{proof}
By Lemma~\ref{lemma-pre}, $\ell_r=n-1=2d-2+i$ and, by Lemma~\ref{lemma-general},
$d\geq 2$ and $n\pm i$ is always odd. In particular,
$d=(n+1-i)/2$ and $\rem_r=\ell+c=n-1-i$ is even. 
Therefore, $\phi_{r,1}=\ell_r-\lfloor(\rem_r+1)/2\rfloor=n-1-\lfloor(n-i)/2\rfloor=(n-1+i)/2$,
$\phi_{r,2}=\rem_r-2\lfloor\rem_r/2\rfloor=0$ and $\phi_{r,3}=\lfloor\rem_r/2\rfloor=(n-1-i)/2$. 
Since $\ell=n-1=2d-2+i$, then $\phi_{r,1}=(n-1+i)/2=(2d-2+i+i)/2=i+d-1$ and $\phi_{r,1}-(d-1)=i$. Thus, 
$I_r=[\phi_{r,1}-(d-1),\phi_{r,1}]=[i,(n-1+i)/2]$. 

Note that since $i\geq 2$, then $\phi_{r,1}=(n-1+i)/2\geq \lfloor n/2\rfloor$.

Let us see that $s_1=r+q\in\nums$. Since $s_1=r+q=(a+1)\ell+c+q=a(n-1)+n-1-i+q$, by Proposition~\ref{theo31ggp},
it is enough to see that $0\leq (n-1-i+q)\leq 2(n-1)$. Since $q\leq n$ and $i\geq 2$, then
$n-1-i+q\leq n-1-2+n< 2(n-1)$.  Since $q\geq n-1$ and $i\leq n-3$, then $n-1-i+q\geq n-1-n+3+n-1=n+1 \geq 0$.  
Note that $s_1=r+q<\la$, because $1\leq\lfloor n/2\rfloor$ and, 
by Lemma~\ref{lemma-cota}, $s_1=r+1\cdot q<\frob(S)<\la$. 

Let us calculate $\ell_{s_1}$, $\rem_{s_1}$ and $\phi_{s_1}$.  
Since $s_1=a(n-1)+(n-1-i+q)\in\nums$, $s_1<\la$ and $0\leq (n-1+i+q)\leq 2(n-1)$, by Proposition~\ref{theo31ggp}, 
$\ell_{s_1}=n-1$ and $\rem_{s_1}=n-1-i+q$. Since $q$ is odd, then
$n-i+q$ is even. Therefore, $\phi_{s_1,1}=\ell_{s_1}-\lfloor (\rem_{s_1}+1)/2\rfloor=n-1-\lfloor(n-i+q)/2\rfloor=(n-2+i-q)/2$.
If $n$ is odd, then $i$ is even, $q=n$ and $(n-2+i-q)/2=(i-2)/2=\lfloor (i-1)/2\rfloor$. 
If $n$ is even, then $i$ is odd, 
$q=n-1$ and $(n-2+i-q)/2=(i-1)/2=\lfloor (i-1)/2\rfloor$ as well. 
Therefore, $\phi_{s_1,1}=\lfloor(i-1)/2\rfloor$. 
Moreover, $\phi_{s_1,2}=\rem_{s_1}-2\lfloor\rem_{s_1}/2\rfloor=1$ and 
$\phi_{s_1,3}=\lfloor\rem_{s_1}/2\rfloor=n-\lfloor (i+3)/2\rfloor$. This shows $(1)$.

Note that $2+\phi_{s_1,1}=2+\lfloor (i-1)/2\rfloor\leq n/2+
\lfloor (i-1)/2\rfloor\leq (n-1+i)/2=\phi_{r,1}$.

Since $s_1=r+q\in\nums$, then, by Lemma~\ref{lg}, $[1,1+\phi_{s_1,1}]\cap\good_r=
\left[1,1+\lfloor(i-1)/2\rfloor\right]\cap\good_r=\emptyset$. 

Let $s_2=r+\lfloor n/2\rfloor q$. By Remark~\ref{aboutq}, $\lfloor n/2\rfloor q=a$, so
$s_2=r+a\in\nums$. By Lemma~\ref{lemma-cota}, 
$s_2=r+\lfloor n/2\rfloor q<\frob(\nums)<\la$. Therefore, $s_2\in\nums$ and $s_2<\la$.

Let us calculate $\ell_{s_2}$, $\rem_{s_2}$ and $\phi_{s_2}$. Note that 
$s_2=r+\lfloor n/2\rfloor q=r+a=(a+1)\ell+c+a=(a+1)(n-1)-i+a=an+(n-1-i)$. 
Since $s_2\in\nums$, $s_2<\la$ and $0\leq n-1-i\leq 2n$, by Proposition~\ref{theo31ggp}, $\ell_{s_2}=n$ and $\rem_{s_2}=n-1-i$. 
Thus, $\phi_{s_2,1}=\ell_{s_2}-\lfloor (\rem_{s_2}+1)/2\rfloor=
n-\lfloor(n-i)/2\rfloor=(n+i+1)/2$.  Moreover,
$\phi_{s_2,2}=\rem_{s_1}-2\lfloor\rem_{s_1}/2\rfloor=0$ and 
$\phi_{s_2,3}=\lfloor\rem_{s_1}/2\rfloor=(n-i-1)/2$. This shows $(2)$.

Since $s_2=r+\lfloor n/2\rfloor q\in\nums$, by Lemma~\ref{lg},
$\left[\lfloor n/2\rfloor,\lfloor n/2\rfloor+\phi_{s_2,1}\right]\cap\good_r=\emptyset$.
Since $\lfloor n/2\rfloor+\phi_{s_2,1}>\phi_{s_2,1}>(n-1+i)/2=\phi_{r,1}\geq \lfloor n/2\rfloor$, 
it follows that $\left[\lfloor n/2\rfloor,\phi_{r,1}\right]\cap\good_r=\emptyset$ and
$\good_r\subseteq \left[0,\lfloor n/2\rfloor-1\right]$.

Suppose that $i=n-3$. Then, $\lfloor(i-1)/2\rfloor=\lfloor (n-4)/2\rfloor=\lfloor n/2\rfloor-2$
and $[1,\lfloor n/2\rfloor-1]\cap\good_r=
\left[1,1+\lfloor(i-1)/2\rfloor\right]\cap\good_r$ which is the empty set, as shown before. 
One concludes that $\good_r=\{0\}$. Note that, since $i=n-3$, then $[\lfloor(i+3)/2\rfloor,\lfloor (n-2)/2\rfloor]=\emptyset$
and $\good_r=\{0\}=\{0\}\cup [2+\lfloor(i-1)/2\rfloor,\lfloor n/2\rfloor-1]$.

Suppose that $i<n-3$, so $2\leq i\leq n-5$ and $n\geq 7$. Let us see that 
$\lfloor n/2\rfloor-1\in [2+\phi_{s_1,1},\phi_{r,1}]$
If $n$ is even, then $\lfloor(i-1)/2\rfloor=(i-1)/2\leq (n-6)/2=(n/2)-3=\lfloor n/2\rfloor-3$. 
If $n$ is odd, then $\lfloor(i-1)/2\rfloor=(i/2)-1\leq ((n-5)/2)-1=((n-1)/2)-3=\lfloor n/2\rfloor-3$.  
In any case, $2+\phi_{s_1,2}=2+\lfloor(i-1)/2\rfloor\leq \lfloor n/2\rfloor-1<\lfloor n/2\rfloor \leq \phi_{r,1}$.

Let us prove that $r+kq\notin\nums$, for all $k\in [2,\lfloor n/2\rfloor-1]$. 
Since $r+kq =a(n-1)+n-1-i+kq$,  by Proposition~\ref{theo31ggp}, it is
enough to see that $2(n-1)< n-1-i+kq <a$. Since $k\leq \lfloor n/2\rfloor-1$,
$\lfloor n/2\rfloor q=a$, $q\geq n-1$ and $i\geq 2$, then $n-1-i+kq\leq
n-1-2+\lfloor n/2\rfloor q-q\leq n-3+a-n+1=a-2<a$.  Since $k\geq 2$, $q\geq n-1$ and
$i<n-1$, then $n-1-i+kq>n-1-n+1+2(n-1)=2(n-1)$. Thus, $r+kq\not\in\nums$ and
$\{r+kq\mid k\in \left[2,\lfloor n/2\rfloor-1\right]\}\cap\nums=\emptyset$.

Since $s_1=r+q\in\nums$, $2+\phi_{s_1,1}\leq \phi_{r,1}$ and $\lfloor n/2\rfloor-1\in
[2+\phi_{s_1,1},\phi_{r,1}]$, then, by
Lemma~\ref{lemma-diffcase}, $\{0\}\cup
\left[2+\lfloor(i-1)/2\rfloor,\lfloor n/2\rfloor-1\right]\subseteq \good_r$.  Since
$\good_r\subseteq \left[0,\lfloor n/2\rfloor-1\right]$ and
$\left[1,1+\lfloor(i-1)/2\rfloor\right]\cap\good_r=\emptyset$, it follows that
$\good_r=\{0\}\cup\left[2+\lfloor(i-1)/2\rfloor,\lfloor n/2\rfloor-1\right]=\{0\}\cup 
[\lfloor(i+3)/2\rfloor,\lfloor (n-2)/2\rfloor]$.

Since $\lfloor (n-2)/2\rfloor\leq (n-3+i)/2=\phi_{r,1}-1$, then $\good_r\leq 
I_r\setminus\{\phi_{r,1}\}$. 

Observe that $\card(\good_r)=1+\lfloor (n-2)/2\rfloor -\lfloor (i+3)/\rfloor +1$. 
If $n$ is odd, then $i$ is even and $\card(\good_r)=2+(n-3)/2-(i+2)/2=(n-1-i)/2=d-1$. If $n$ is even, then $i$ is odd and $\card(\good_r)=2+(n-2)/2-(i+3)/2=(n-1-i)/2=d-1$. 

Since $\card(\good_r)=d-1$, then, by Lemma~\ref{lemma3-rho},~$(3)$, 
$\dim\ker(\rho_r^{{\rm G}_r})=1$ and a basis of $\ker(\rho_r^{{\rm G}_r})$
is given by the polynomial $g=\sum_{j=0}^{d-1}(-1)^j
b_{{\rm G}_{r}}^{I_{r}\setminus\{\phi_{{r},1}-j\}}\mon^{\phi_{r}+j\omega}$.
Substituting the values of $d$, $I_r$ and $\good_r$ previously found, we get $(3)$. 
Since $\good_r\leq I_r\setminus\{\phi_{r,1}\}$, by 
\cite[Corollary~2]{gv} or \cite[Corollary~2.5]{gp2}, 
$\lambda_j=(-1)^jb_{{\rm G}_{r}}^{I_{r}\setminus\{\phi_{{r},1}-j\}}\neq 0$, for all $j$.

Suppose that $n$ is odd and $1\leq k\leq n-4$, with $k$ odd. 
By Proposition~\ref{proposition-class}, 
${\corange c_{r_k}=-\iota_{r_k}}$ and $\iota_{r_k}=n-k-2$.  
In particular, $\delta_{r_k}=(k+3)/2$,
\begin{eqnarray*}
\phi_{r_k}=\left(\frac{2n-k-3}{2},0,\frac{k+1}{2}\right),\;
I_{r_k}=\left[n-k-2,\frac{2n-k-3}{2}\right],\; 
\good_{r_k}=\{0\}\cup \left[\frac{n-k}{2}, \frac{n-3}{2}\right].
\end{eqnarray*} 
Thus, $g$ is equal to the polynomial $g_k$ defined in \eqref{equality-gknoddkodd}. This shows $(4)$. 

Suppose that $n$ is even and $1\leq k\leq n-5$, with $k$ odd. 
By Proposition~\ref{proposition-class}, 
${\corange c_{r_k}=-\iota_{r_k}}$ and $\iota_{r_k}=n-k-2$.
In particular, $\delta_{r_k}=(k+3)/2$,
\begin{eqnarray*}
\phi_{r_k}=\left(\frac{2n-k-3}{2},0,\frac{k+1}{2}\right)\!,
I_{r_k}=\left[n-k-2,\frac{2n-k-3}{2}\right]\!,
\good_{r_k}=\{0\}\cup \left[\frac{n-k+1}{2}, \frac{n-2}{2}\right]\!.
\end{eqnarray*} 
Thus, $g$ is equal to the polynomial $g_k$ defined in \eqref{equality-gknevenkodd}. This shows $(5)$. 
\end{proof}

\vspace*{0.3cm}

\begin{lemma}\label{lemma-c-i+1}
Let $n\geq 6$. Let $r\in [s_0,s_0+n)$. Write $r=r_k=s_0+k-1$, with $1\leq k\leq n$. 
Suppose that ${\cteal c_r=-\iota_r+1}$, $2\leq\iota_r\leq n-3$. Then,
\begin{eqnarray*}
\phi_r=\left(\frac{n-3+\iota_r}{2},1,\frac{n-1-\iota_r}{2}\right),
I_r=\left[\iota_r-1,\frac{n-3+\iota_r}{2}\right],
\good_r=\{0\}\cup \left[\left\lfloor\frac{\iota_r+3}{2}\right\rfloor,
\left\lfloor\frac{n-2}{2}\right\rfloor\right].
\end{eqnarray*}
Moreover, $\delta_r=(n+1-\iota_r)/2$, $\card(\good_r)=\delta_r-1$, $\dim\ker(\rho_r^{{\rm G}_r})=1$,
$\good_r\leq I_r\setminus\{\phi_{r,1}\}$ and 
\begin{itemize}
\item[$(1)$] $s_1=r+q\in\nums$, $s_1<\la$, $\ell_{s_1}=n-1$, $\rem_{s_1}=n-i+q$ and
$\phi_{s_1}=\left(\lfloor\frac{\iota_r-1}{2}\rfloor,0,n-\lfloor\frac{\iota_r+1}{2}\rfloor\right)$;
\item[$(2)$] $s_2=r+\lfloor n/2\rfloor q\in\nums$, $s_2<\la$, $\ell_{s_2}=n$, $\rem_{s_2}=n-i$
and $\phi_{s_2}=\left(\frac{n-1+\iota_r}{2},1,\frac{n-1-\iota_r}{2}\right)$.
\end{itemize}
Furthermore, 
\begin{itemize}
\item[$(3)$] $g=\sum_{j=0}^{\frac{n-1-\iota_r}{2}}\lambda_j\mon^{\phi_{r}+j\omega}$, where 
$\lambda_j=(-1)^jb_{\{0\}\sqcup[\lfloor\frac{\iota_r+3}{2}\rfloor,
\lfloor\frac{n-2}{2}\rfloor]}^{[\iota_r-1,\frac{n-3+\iota_r}{2}]\setminus 
\{\frac{n-3+\iota_r-2j}{2}\}}\neq 0$, for all $j$, is a basis of $\ker(\rho_r^{{\rm G}_r})$. 
\item[$(4)$]
If $n$ is odd and $2\leq k\leq n-3$, with $k$ even, then $g$ is equal to the polynomial $g_k$ defined in 
\eqref{equality-gknoddkeven}. In particular, $g_k$ is a basis of 
$\ker\left(\rho_{r_k}^{{\rm G}_{r_k}}\right)$. 
\item[$(5)$]
If $n$ is even and $2\leq k\leq n-4$ with $k$ even, then $g$ is equal to the polynomial $g_k$ defined in 
\eqref{equality-gknevenkeven}. In particular, $g_k$ is a basis of 
$\ker\left(\rho_{r_k}^{{\rm G}_{r_k}}\right)$. 
\end{itemize}
\end{lemma}
\begin{proof}
By Lemma~\ref{lemma-pre}, $\ell_r=n-1=2d-2+i$ and, by Lemma~\ref{lemma-general}, 
$d\geq 2$ and $n\pm i$ is always odd. In particular,
$d=(n+1-i)/2$ and $\rem_r=\ell+c=n-i$ is odd. Therefore,
$\phi_{r,1}=\ell_r-\lfloor(\rem_r+1)/2\rfloor=n-1-\lfloor(n-i+1)/2\rfloor=(n-3+i)/2$,
$\phi_{r,2}=\rem_r-2\lfloor\rem_r/2\rfloor=1$ and $\phi_{r,3}=\lfloor\rem_r/2\rfloor=(n-1-i)/2$. 
Since $\ell=n-1=2d-2+i$, then $\phi_{r,1}=(n-3+i)/2=i+d-2$ and $\phi_{r,1}-(d-1)=i-1$. Thus, 
$I_r=[\phi_{r,1}-(d-1),\phi_{r,1}]=[i-1,(n-3+i)/2]$. 

Observe that, since $i\geq 2$, then $\phi_{r,1}=(n-3+i)/2\geq\lfloor n/2\rfloor$. Indeed, if $i=2$, then
$n$ i odd and $(n-3+i)/2=(n-1)/2=\lfloor n/2\rfloor$; if $i\geq 3$, 
then $(n-3+i)/2\geq n/2\geq\lfloor n/2\rfloor$ as well. 

Let us see that $s_1=r+q\in\nums$. Since $s_1=r+q=(a+1)\ell+c+q=a(n-1)+n-i+q$, by Proposition~\ref{theo31ggp},
it is enough to see that $0\leq (n-i+q)\leq 2(n-1)$. Since $q\leq n$ and $i\geq 2$, then
$n-i+q\leq n-2+n=2(n-1)$. Since $q\geq n-1$ and $i\leq n-3$, then $n-i+q\geq n-n+3+n-1=n+2\geq 0$.  

Note that $s_1<\la$, because $1\leq \lfloor n/2\rfloor$ and,
by Lemma~\ref{lemma-cota}, $s_1=r+1\cdot q<\frob(S)<\la$. 

Let us calculate $\ell_{s_1}$, $\rem_{s_1}$ and $\phi_{s_1,1}$.
Since $s_1=a(n-1)+(n-i+q)\in\nums$, $s_1<\la$ and $0\leq (n-i+q)\leq 2(n-1)$, by Proposition~\ref{theo31ggp}, 
$\ell_{s_1}=n-1$ and $\rem_{s_1}=n-i+q$. Since $q$ is odd, then
$n-i+q$ is even. Therefore, $\phi_{s_1,1}=\ell_{s_1}-\lfloor (\rem_{s_1}+1)/2\rfloor=
n-1-\lfloor(n-i+q)/2\rfloor=(n-2+i-q)/2$.
If $n$ is odd, then $i$ is even, $q=n$ and $(n-2+i-q)/2=(i-2)/2=\lfloor (i-1)/2\rfloor$. 
If $n$ is even, then $i$ is odd, 
$q=n-1$ and $(n-2+i-q)/2=(i-1)/2=\lfloor (i-1)/2\rfloor$ as well. 
Therefore, $\phi_{s_1,1}=\lfloor(i-1)/2\rfloor$. 
Moreover, 
$\phi_{s_1,2}=\rem_{s_1}-2\lfloor\rem_{s_1}/2\rfloor=0$ and 
$\phi_{s_1,3}=\lfloor\rem_{s_1}/2\rfloor=n-\lfloor (i+1)/2\rfloor$. This shows $(1)$.

Since $n\geq 6$, then $2\leq (n-2)/2$ and 
$2+\phi_{s_1,1}=2+(n-2+i-q)/2\leq (n-2)/2+(n-2+i-q)/2=(n-3+i)/2=\phi_{r,1}$.

Since $s_1=r+q\in\nums$, then, by Lemma~\ref{lg}, $[1,1+\phi_{s_1,1}]\cap\good_r=
\left[1,1+\lfloor(i-1)/2\rfloor\right]\cap\good_r=\emptyset$. 

Let $s_2=r+\lfloor n/2\rfloor q$. By Remark~\ref{aboutq}, $\lfloor n/2\rfloor q=a$, so
$s_2=r+a\in\nums$. By Lemma~\ref{lemma-cota}, 
$s_2=r+\lfloor n/2\rfloor q<\frob(\nums)<\la$. Therefore, $s_2\in\nums$ and $s_2<\la$.

Let us calculate $\ell_{s_2}$, $\rem_{s_2}$ and $\phi_{s_2,1}$. Note that 
$s_2=r+\lfloor n/2\rfloor q=r+a=(a+1)\ell+c+a=(a+1)(n-1)-i+1+a=an+(n-i)$. 
Since $s_2\in\nums$, $s_2<\la$ and $0\leq n-i\leq 2n$, by Proposition~\ref{theo31ggp}, 
$\ell_{s_2}=n$ and $\rem_{s_2}=n-i$. Thus, $\phi_{s_2,1}=\ell_{s_2}-\lfloor (\rem_{s_2}+1)/2\rfloor=
n-\lfloor(n-i+1)/2\rfloor=(n-1+i)/2$. Moreover, 
$\phi_{s_2,2}=\rem_{s_1}-2\lfloor\rem_{s_1}/2\rfloor=1$ and 
$\phi_{s_2,3}=\lfloor\rem_{s_1}/2\rfloor=(n-1-i)/2$. This shows $(2)$.

Since $s_2=r+\lfloor n/2\rfloor q\in\nums$, by Lemma~\ref{lg},
$\left[\lfloor n/2\rfloor,\lfloor n/2\rfloor+\phi_{s_2,1}\right]\cap\good_r=\emptyset$.
Since $\lfloor n/2\rfloor+\phi_{s_2,1}>\phi_{s_2,1}=(n-1+i)/2>(n-3+i)/2=\phi_{r,1}\geq \lfloor n/2\rfloor$, 
it follows that $\left[\lfloor n/2\rfloor,\phi_{r,1}\right]\cap\good_r=\emptyset$ and
$\good_r\subseteq \left[0,\lfloor n/2\rfloor-1\right]$.

Suppose that $i=n-3$. Then, $\lfloor(i-1)/2\rfloor=\lfloor (n-4)/2\rfloor=\lfloor n/2\rfloor-2$
and $[1,\lfloor n/2\rfloor-1]\cap\good_r=
\left[1,1+\lfloor(i-1)/2\rfloor\right]\cap\good_r$ which is the empty set, as shown before. 
One concludes that $\good_r=\{0\}$. Note that, since $i=n-3$, then $[\lfloor(i+3)/2\rfloor,\lfloor (n-2)/2\rfloor]=\emptyset$
and $\good_r=\{0\}=\{0\}\cup [2+\lfloor(i-1)/2\rfloor,\lfloor n/2\rfloor-1]$.

Suppose that $i<n-3$, so $2\leq i\leq n-5$ and $n\geq 7$. Let us see that 
$\lfloor n/2\rfloor-1\in [2+\phi_{s_1,1},\phi_{r,1}]$
If $n$ is even, then $\lfloor(i-1)/2\rfloor=(i-1)/2\leq (n-6)/2=(n/2)-3=\lfloor n/2\rfloor-3$. 
If $n$ is odd, then $\lfloor(i-1)/2\rfloor=(i/2)-1\leq ((n-5)/2)-1=((n-1)/2)-3=\lfloor n/2\rfloor-3$.  
In any case, $2+\phi_{s_1,2}=2+\lfloor(i-1)/2\rfloor\leq \lfloor n/2\rfloor-1<\lfloor n/2\rfloor \leq \phi_{r,1}$.

Let us prove that $r+kq\notin\nums$, for all $k\in [2,\lfloor n/2\rfloor-1]$. 
Since $r+kq=a(n-1)+n-i+ kq$,  by Proposition~\ref{theo31ggp}, it is
enough to see that $2(n-1)<n-i+kq<a$. Since $k\leq \lfloor n/2\rfloor-1$,
$\lfloor n/2\rfloor q=a$, $q\geq n-1$ and $i\geq 2$, then $n-i+kq\leq
n-2+\lfloor n/2\rfloor q-q\leq n-2+a-n+1=a-1<a$.  Since $k\geq 2$, $q\geq n-1$ and
$i<n$, then $n-i+kq>n-n+2(n-1)=2(n-1)$. Thus, $r+kq\not\in\nums$ and
$\{r+kq\mid k\in \left[2,\lfloor n/2\rfloor-1\right]\}\cap\nums=\emptyset$.

Since $s_1=r+q\in\nums$, $2+\phi_{s_1,1}\leq \phi_{r,1}$ and $\lfloor n/2\rfloor-1\in
[2+\phi_{s_1,1},\phi_{r,1}]$, then, by
Lemma~\ref{lemma-diffcase}, $\{0\}\cup
\left[2+\lfloor(i-1)/2\rfloor,\lfloor n/2\rfloor-1\right]\subseteq \good_r$.  Since
$\good_r\subseteq \left[0,\lfloor n/2\rfloor-1\right]$ and
$\left[1,1+\lfloor(i-1)/2\rfloor\right]\cap\good_r=\emptyset$, it follows that
$\good_r=\{0\}\cup\left[2+\lfloor(i-1)/2\rfloor,\lfloor n/2\rfloor-1\right]=\{0\}\cup 
[\lfloor(i+3)/2\rfloor,\lfloor (n-2)/2\rfloor]$.

Note that $\lfloor (n-2)/2\rfloor\leq (n-5+i)/2=\phi_{r,1}-1$. Indeed, if $i=2$, then $n$ is odd and 
$\lfloor (n-2)/2\rfloor=(n-3)/2=(n-5+i)/2=\phi_{r,1}-1$. If $i\geq 3$, then 
$\lfloor (n-2)/2\rfloor\leq (n-2)/2\leq (n-5+i)/2=\phi_{r,1}-1$. Therefore, 
$\good_r\leq I_r\setminus\{\phi_{r,1}\}$. 

The subset $\good_r$ coincides with the one in Lemma~\ref{lemma-c-i}. Thus, like there,
$\card(\good_r)=d-1$. By Lemma~\ref{lemma3-rho},~$(3)$, $\dim\ker(\rho_r^{{\rm G}_r})=1$ and
$g=\sum_{j=0}^{d-1}(-1)^jb_{{\rm G}_{r}}^{I_{r}
\setminus\{\phi_{{r},1}-j\}}\mon^{\phi_{r}+j\omega}$ is a basis of $\ker(\rho_r^{{\rm G}_r})$.
Substituting the values of $d$, $I_r$ and $\good_r$ previously found, we get $(3)$. 
Since $\good_r\leq I_r\setminus\{\phi_{r,1}\}$, by 
\cite[Corollary~2]{gv} or \cite[Corollary~2.5]{gp2}, 
$\lambda_j=(-1)^jb_{{\rm G}_{r}}^{I_{r}\setminus\{\phi_{{r},1}-j\}}\neq 0$.

Suppose that $n$ is odd and $2\leq k\leq n-3$, with $k$ even. By Proposition~\ref{proposition-class}, 
${\cteal c_{r_k}=-\iota_{r_k}+1}$ and $\iota_{r_k}=n-k-1$.
In particular, $\delta_{r_k}=(k+2)/2$,
\begin{eqnarray*}
\phi_{r_k}=\left(\frac{2n-k-4}{2},1,\frac{k}{2}\right),
I_{r_k}=\left[n-k-2,\frac{2n-k-4}{2}\right],
\good_{r_k}=\{0\}\cup \left[\frac{n-k+1}{2},\frac{n-3}{2}\right].
\end{eqnarray*}
Thus, $g$ is equal to the polynomial $g_k$ defined in \eqref{equality-gknoddkeven}. This shows $(4)$. 

Suppose that $n$ is even and $2\leq k\leq n-4$, with $k$ even. 
By Proposition~\ref{proposition-class}, 
${\cteal c_{r_k}=-\iota_{r_k}+1}$ and $\iota_{r_k}=n-k-1$.
In particular, $\delta_{r_k}=(k+2)/2$, 
\begin{eqnarray*}
\phi_{r_k}=\left(\frac{2n-k-4}{2},1,\frac{k}{2}\right),
I_{r_k}=\left[n-k-2,\frac{2n-k-4}{2}\right],
\good_{r_k}=\{0\}\cup \left[\frac{n-k+2}{2},\frac{n-2}{2}\right],
\end{eqnarray*}
Thus, $g$ is equal to the polynomial $g_k$ defined in \eqref{equality-gknevenkeven}. This shows $(5)$. 
\end{proof}

\vspace*{0.3cm}

\begin{lemma}\label{lemma-ic00}
Let $n\geq 6$. Let $r\in [s_0,s_0+n)$. Write $r=r_k=s_0+k-1$, with $1\leq k\leq n$. 
Suppose that $(\iota_r,c_r)=(0,0)$. Then $n$ is odd, $k=n-2$ and $r=r_{n-2}$, 
\begin{eqnarray*}
\phi_r=\left(\frac{n-1}{2},0,\frac{n-1}{2}\right),\; I_r=\left[0,\frac{n-1}{2}\right],\;
\good_r=\left[0,\frac{n-3}{2}\right].
\end{eqnarray*}
Moreover, $\delta_r=(n+1)/2$, $\card(\good_r)=\delta_r-1$, $\dim\ker(\rho_r^{{\rm G}_r})=1$,
$\good_r\leq I_r\setminus\{\phi_{r,1}\}$ and 
\begin{itemize}
\item[$(1)$] $s_1=r+(\delta_r-1)q\in\nums$, $s_1<\la$, $\ell_{s_1}=n$, $\rem_{s_1}=n-1$ and 
$\phi_{s_1}=((n+1)/2,0,(n-1)/2)$.
\end{itemize}
Furthermore,
\begin{itemize}
\item[$(2)$] $g=\sum_{j=0}^{\frac{n-1}{2}}(-1)^jb_{\frac{n-1}{2},\frac{n-1-2j}{2}}\cdot 
x^{\frac{n-1-2j}{2}}y^{2j}z^{\frac{n-1-2j}{2}}$ is a basis of $\ker(\rho_r^{{\rm G}_r})$. 
\end{itemize}
In particular, $g$ is equal to the polynomial $g_{n-2}$ defined in \eqref{equality-gk00}, 
and $g_{n-2}$ is a basis of $\ker\left(\rho_{r_{n-2}}^{{\rm G}_{r_{n-2}}}\right)$. 
\end{lemma}
\begin{proof}
By Proposition~\ref{proposition-class}, if $(\iota_{r_k},c_{r_k})=(0,0)$, then $n$ is odd, 
$k=n-2$ and $r=r_{n-2}$. 
By Lemma~\ref{lemma-pre}, $\ell_r=n-1=2d-2$ and, by Lemma~\ref{lemma-general}, 
$d\geq 2$. In particular,
$d=(n+1)/2\geq 3$ and $\rem_r=\ell+c=n-1$ is even. Therefore,
$\phi_{r,1}=\ell_r-\lfloor(\rem_r+1)/2\rfloor=n-1-\lfloor n/2\rfloor=(n-1)/2$,
$\phi_{r,2}=\rem_r-2\lfloor\rem_r/2\rfloor=1$ and $\phi_{r,3}=\lfloor\rem_r/2\rfloor=(n-1)/2$. 
Since $\ell=n-1=2d-2$, then $\phi_{r,1}=(n-1)/2=d-1$ and $\phi_{r,1}-(d-1)=0$. Thus, 
$I_r=[\phi_{r,1}-(d-1),\phi_{r,1}]=[0,(n-1)/2]$. 

Let us see that $s_1=r+(d-1)q\in\nums$ and $s_1<\la$. Since $d-1=\lfloor n/2\rfloor$, by Lemma~\ref{lemma-cota} $r+(d-1)q<\frob(\nums)<\la$. 
Since $d-1=(n-1)/2$ and $n$ is odd, then $q=n$, 
$(d-1)q=(n-1)n/2=a$ and $s_1=r+(d-1)q=(a+1)(n-1)+a=an+(n-1)$, where $0\leq (n-1)\leq 2n$. By Proposition~\ref{theo31ggp}, 
$s_1=an+(n-1)\in\nums$, $\ell_{s_1}=n$, $\rem_{s_1}=n-1$, $\phi_{s_1,1}=\ell_{s_1}-\lfloor (\rem_{s_1}+1)/2\rfloor=(n+1)/2$, $\phi_{s_1,2}=\rem_{s_1}-2\lfloor\rem_{s_1}/2\rfloor=0$ and $\phi_{s_1,3}=\lfloor\rem_{s_1}/2\rfloor=(n-1)/2$. This shows $(1)$.

Since $s_1=r+(d-1)q\in\nums$, then $d-1\not\in\good_r$, where $\phi_{r,1}=d-1$, so $\good_r\subseteq[0,\phi_{r,1}]=[0,d-2]$.

Let us prove that $r+kq\not\in\nums$, for all $k\in[1,d-2]$. 
Since $s_1=r+(d-1)q<\frob(\nums)$, then $r+kq<s_1<\frob(\nums)<\la$, for all $k\in[1,d-2]$. We have
$r+kq=(a+1)\ell+kq=a\ell+(\ell+kq)$. Since $\ell=n-1$, $k\leq d-2$, $q=n$ and $(d-1)n=a$, then
$\ell+kq\leq n-1+(d-2)n=n-1+(d-1)n-n=a-1<a$. 
Since $\ell=n-1$, $k\geq 1$, $q=n$, then $\ell+kq\geq n-1+n=2n-1>2\ell$. Thus,
$2\ell\leq (\ell+kq)<a$. By Proposition~\ref{theo31ggp}, $r+kq\not\in\nums$, for all $k\in[1,d-2]$.

Therefore, $\{r+kq\mid k\in [1,d-2]\}\bigcap\nums=\emptyset$.  By
Lemma~\ref{lemma-good}, $[0,d-2]\subseteq\good_r$. Thus,
$\good_r=[0,d-2]$. Since $d=(n+1)/2$, then $d-2=(n-3)/2$ and $\good_r=[0,(n-3)/2]$, which is equal to
$I_r\setminus\{\phi_{r,1}\}$.

Clearly, $\card(\good_r)=(n-1)/2=d-1$. By Lemma~\ref{lemma3-rho},~$(3)$, 
$\dim\ker(\rho_r^{{\rm G}_r})=1$ and a basis of $\ker(\rho_r^{{\rm G}_r})$ is given by the polynomial
$g=\sum_{j=0}^{d-1}(-1)^j
b_{{\rm G}_{r}}^{I_{r}\setminus\{\phi_{{r},1}-j\}}\mon^{\phi_{r}+j\omega}$.
By \cite[Corollary~6.3]{gp2}, 
$b^{[0,\tfrac{n-1}{2}]\setminus\{\tfrac{n-1-2j}{2}\}}_{[0,\tfrac{n-3}{2}]}=
b_{\tfrac{n-1}{2},\tfrac{n-1-2j}{2}}$.
Thus, substituting the values of $d$, $I_r$, $\good_r$ and $\phi_r$ previously found, 
we get $(2)$. 
Observe that $g$ coincides with the polynomial $g_{n-2}$ defined in \eqref{equality-gk00}.
\end{proof}

\vspace*{0.3cm}

\begin{lemma}\label{lemma-ic21}
Let $n\geq 6$. Let $r\in [s_0,s_0+n)$.  Write $r=r_k=s_0+k-1$, with $1\leq k\leq n$. 
Suppose that $(\iota_r,c_r)={\colive (2,1)}$. Then $n$ is odd, $k=n-1$ and $r=r_{n-1}$,
\begin{eqnarray*}
\phi_r=\left(\frac{n-3}{2},1,\frac{n-1}{2}\right),\; I_r=\left[0,\frac{n-3}{2}\right],\;
\good_r=\left[0,\frac{n-5}{2}\right].
\end{eqnarray*}
Moreover, $\delta_r=(n-1)/2$, $\card(\good_r)=\delta_r-1$, 
$\dim\ker(\rho_r^{{\rm G}_r})=1$,
$\good_r\leq I_r\setminus\{\phi_{r,1}\}$ and 
\begin{itemize}
\item[$(1)$] $s_1=r+(\delta_r-1)q\in\nums$, $s_1<\la$, $\ell_{s_1}=n$, $\rem_{s_1}=0$ and $\phi_{s_1}=(n,0,0)$.
\end{itemize}
Furthermore, 
\begin{itemize}
\item[$(2)$] $g=\sum_{j=0}^{\frac{n-3}{2}}(-1)^jb_{\frac{n-3}{2},\frac{n-3-2j}{2}}\cdot x^{\frac{n-3-2j}{2}}y^{1+2j}z^{\frac{n-1-2j}{2}}$ is a basis of $\ker(\rho_r^{{\rm G}_r})$. 
\end{itemize}
In particular, $g$ is equal to the polynomial $g_{n-1}$ defined in \eqref{equality-gk21}, 
and $g_{n-1}$ is a basis of $\ker\left(\rho_{r_{n-1}}^{{\rm G}_{r_{n-1}}}\right)$. 
\end{lemma}
\begin{proof}
By Proposition~\ref{proposition-class}, if $(\iota_{r_k},c_{r_k})={\colive (2,1)}$,
then $n$ is odd, $k=n-1$ and $r=r_{n-1}$. 
By Lemma~\ref{lemma-pre}, $\ell_r=n-1=2d-2+i=2d$ and, by Lemma~\ref{lemma-general}, 
$d\geq 2$. In particular, $d=(n-1)/2$ and $\rem_r=\ell+c=n$ is odd. 
Therefore, $\phi_{r,1}=\ell_r-\lfloor(\rem_r+1)/2\rfloor=n-1-\lfloor (n+1)/2\rfloor=(n-3)/2$,
$\phi_{r,2}=\rem_r-2\lfloor\rem_r/2\rfloor=1$ and $\phi_{r,3}=\lfloor\rem_r/2\rfloor=(n-1)/2$. 
Moreover, since $d=(n-1)/2$, then $\phi_{r,1}=(n-3)/2=d-1$, $\phi_{r,1}-(d-1)=0$ and 
$I_r=[\phi_{r,1}-(d-1),\phi_{r,1}]=[0,(n-3)/2]$. 

Let us see that $s_1=r+(d-1)q\in\nums$ and $s_1<\la$. Since $d-1=(n-3)/2<\lfloor n/2\rfloor$, 
by Lemma~\ref{lemma-cota}, $s_1=r+(d-1)q<\frob(\nums)<\la$. 
Since $d=(n-1)/2$ and $n$ is odd, then $q=n$, $dq=(n-1)n/2=a$ and $s_1=r+(d-1)q=(a+1)\ell+c+(d-1)q=
(a+1)(n-1)+1+a-n=an\in\nums$. Moreover, $\ell_{s_1}=n$, $\rem_{s_1}=0$, $\phi_{s_1,1}=\ell_{s_1}-\lfloor (\rem_{s_1}+1)/2\rfloor=n$, $\phi_{s_1,2}=\rem_{s_1}-2\lfloor\rem_{s_1}/2\rfloor=0$ and $\phi_{s_1,3}=\lfloor\rem_{s_1}/2\rfloor=0$. This shows $(1)$.

Since $s_1=r+(d-1)q\in\nums$, then $d-1\not\in\good_r$. Since $\good_r\subseteq[0,\phi_{r,1}]$ and
$\phi_{r,1}=d-1\not\in\good_r$, then $\good_r\subseteq [0,d-2]$.

Let us see $r+kq\not\in\nums$, for all $k\in[1,d-2]$. Indeed, since
$s_1=r+(d-1)q<\frob(\nums)<\la$, then $r+kq<\frob(\nums)$, for all $k\in[1,d-2]$.
Moreover, $r+kq=(a+1)\ell+c+kq=a\ell+(\ell+1+kq)$. Since $\ell=n-1$, $k\leq d-2$, $q=n$,  
$d=(n-1)/2$ and $dq=a$, then
$\ell+1+kq\leq n+(d-2)q=n+a-2n=a-n<a$. 
Since $\ell=n-1$, $k\geq 1$, $q=n$, then $\ell+1+kq\geq n+n>2\ell$. Thus,
$2\ell\leq (\ell+1+kq)<a$. By Proposition~\ref{theo31ggp}, $r+kq\not\in\nums$, for all $k\in[1,d-2]$. 

Thus, $\{r+kq\mid k\in [1,d-2]\}\bigcap\nums=\emptyset$.  By
Lemma~\ref{lemma-good}, $[0,d-2]\subseteq\good_r$. So,
$\good_r=[0,d-2]$. Since $d=(n-1)/2$, then $d-2=(n-5)/2$ and $\good_r=[0,(n-5)/2]$, 
which is equal to $I_r\setminus\{\phi_{r,1}\}$.

We have $\card(\good_r)=(n-3)/2=d-1$. By Lemma~\ref{lemma3-rho},~$(3)$,
$\dim\ker(\rho_r^{{\rm G}_r})=1$ and a basis of $\ker(\rho_r^{{\rm G}_r})$
is given by the polynomial
$g=\sum_{j=0}^{d-1}(-1)^j
b_{{\rm G}_{r}}^{I_{r}\setminus\{\phi_{{r},1}-j\}}\mon^{\phi_{r}+j\omega}$.
By \cite[Corollary~6.3]{gp2},
$b^{[0,\tfrac{n-3}{2}]\setminus\{\tfrac{n-3-2j}{2}\}}_{[0,\tfrac{n-5}{2}]}=
b_{\tfrac{n-3}{2},\tfrac{n-3-2j}{2}}$. Thus, 
substituting the values of $d$, $I_r$, $\good_r$ and $\phi_r$ previously found, 
we get $(2)$. Note that $g$ coincides with the polynomial $g_{n-1}$ 
defined in \eqref{equality-gk21}.
\end{proof}

\vspace*{0.3cm}

\begin{lemma}\label{lemma-ic22}
Let $n\geq 6$. Let $r\in [s_0,s_0+n)$. Write $r=r_k=s_0+k-1$, with $1\leq k\leq n$. 
Suppose that $(\iota_r,c_r)={\cviolet (2,2)}$. Then $n$ is odd, $k=n$ and $r=r_n$, 
\begin{eqnarray*}
\phi_r=\left(\frac{n-3}{2},0,\frac{n+1}{2}\right),\; I_r=\left[0,\frac{n-3}{2}\right],\;
\good_r=\left[0,\frac{n-5}{2}\right].
\end{eqnarray*}
Moreover, $\delta_r=(n-1)/2$, $\card(\good_r)=\delta_r-1$, $\dim\ker(\rho_r^{{\rm G}_r})=1$,
$\good_r\leq I_r\setminus\{\phi_{r,1}\}$ and 
\begin{itemize}
\item[$(1)$] $s_1=r+(\delta_r-1)q\in\nums$, $s_1<\la$, $\ell_{s_1}=n$, $\rem_{s_1}=1$ and 
$\phi_{s_1}=(n-1,1,0)$.
\end{itemize}
Furthermore, 
\begin{itemize}
\item[$(2)$] $g=\sum_{j=0}^{\frac{n-3}{2}}(-1)^jb_{\frac{n-3}{2},\frac{n-3-2j}{2}}\cdot
x^{\frac{n-3-2j}{2}}y^{2j}z^{\frac{n+1-2j}{2}}$ is a basis of $\ker(\rho_r^{{\rm G}_r})$. 
\end{itemize}
In particular, $g$ is equal to the polynomial $g_{n}$ defined in \eqref{equality-gk22}, 
and $g_{n}$ is a basis of $\ker\left(\rho_{r_{n}}^{{\rm G}_{r_{n}}}\right)$. 
\end{lemma}
\begin{proof} 
By Proposition~\ref{proposition-class}, if $(\iota_{r_k},c_{r_k})={\cviolet (2,2)}$, 
then $n$ is odd, $k=n$ and $r=r_{n}$. 
By Lemma~\ref{lemma-pre}, $\ell_r=n-1=2d-2+i=2d$ and, by Lemma~\ref{lemma-general}, 
$d\geq 2$. In particular, $d=(n-1)/2$ and $\rem_r=\ell+c=n+1$ is even. 
Therefore, $\phi_{r,1}=\ell_r-\lfloor(\rem_r+1)/2\rfloor=n-1-\lfloor (n+2)/2\rfloor=(n-3)/2$,
$\phi_{r,2}=\rem_r-2\lfloor\rem_r/2\rfloor=0$ and $\phi_{r,3}=\lfloor\rem_r/2\rfloor=(n+1)/2$. 
Moreover, since $d=(n-1)/2$, then $\phi_{r,1}=(n-3)/2=d-1$, $\phi_{r,1}-(d-1)=0$ and 
$I_r=[\phi_{r,1}-(d-1),\phi_{r,1}]=[0,(n-3)/2]$. 

Let us see that $s_1=r+(d-1)q\in\nums$ and $s_1<\la$. Since $d-1=(n-3)/2<\lfloor n/2\rfloor$, 
by Lemma~\ref{lemma-cota}, $s_1=r+(d-1)q<\frob(\nums)<\la$. 
Since $d=(n-1)/2$ and $n$ is odd, then $q=n$, $dq=(n-1)n/2=a$ and $s_1=r+(d-1)q=(a+1)\ell+c+(d-1)q=
(a+1)(n-1)+2+a-n=an+1$, where $0\leq 1\leq 2n$. By Proposition~\ref{theo31ggp}, $s_1=an+1\in\nums$, 
$\ell_{s_1}=n$, $\rem_{s_1}=1$, $\phi_{s_1,1}=\ell_{s_1}-\lfloor (\rem_{s_1}+1)/2\rfloor=n-1$, 
$\phi_{s_1,2}=\rem_{s_1}-2\lfloor\rem_{s_1}/2\rfloor=1$ and $\phi_{s_1,3}=\lfloor\rem_{s_1}/2\rfloor=0$.
This shows $(1)$. 

Since $s_1=r+(d-1)q\in\nums$, then $d-1\not\in\good_r$. Since $\good_r\subseteq[0,\phi_{r,1}]$ and
$\phi_{r,1}=d-1\not\in\good_r$, then $\good_r\subseteq [0,d-2]$.

Let us prove that $r+kq\not\in\nums$, for all $k\in[1,d-2]$. Indeed,
$s_1=r+(d-1)q<\frob(\nums)<\la$, so $r+kq<\frob(\nums)$, for all $k\in[1,d-2]$.
We have $r+kq=(a+1)\ell+c+kq=a\ell+(\ell+2+kq)$. Since $\ell=n-1$, $k\leq d-2$, $q=n$,  
$d=(n-1)/2$ and $dn=a$, then
$\ell+2+kq\leq n+1+(d-2)q=n+1+a-2n=a-n+1<a$. 
Since $\ell=n-1$, $k\geq 1$, $q=n$, then $\ell+2+kq\geq 2n+1>2\ell$. Thus,
$2\ell\leq (\ell+2+kq)<a$. By Proposition~\ref{theo31ggp}, $r+kq\not\in\nums$, for all $k\in[1,d-2]$. 

Thus, $\{r+kq\mid k\in [1,d-2]\}\bigcap\nums=\emptyset$.  By
Lemma~\ref{lemma-good}, $[0,d-2]\subseteq\good_r$. So,
$\good_r=[0,d-2]$. Since $d=(n-1)/2$, then $d-2=(n-5)/2$ and $\good_r=[0,(n-5)/2]$, which is equal to 
$I_r\setminus\{\phi_{r,1}\}$.

We have $\card(\good_r)=(n-3)/2=d-1$. By Lemma~\ref{lemma3-rho},~$(3)$, 
$\dim\ker(\rho_r^{{\rm G}_r})=1$ and a basis of $\ker(\rho_r^{{\rm G}_r})$
is given by the polynomial
$g=\sum_{j=0}^{d-1}(-1)^jb_{{\rm G}_{r}}^{I_{r}
\setminus\{\phi_{{r},1}-j\}}\mon^{\phi_{r}+j\omega}$.
By \cite[Corollary~6.3]{gp2}, $b^{[0,\tfrac{n-3}{2}]\setminus\{\tfrac{n-3-2j}{2}\}}_{[0,\tfrac{n-5}{2}]}=
b_{\tfrac{n-3}{2},\tfrac{n-3-2j}{2}}$.
Thus, substituting the values of $d$, $I_r$, $\good_r$ and $\phi_r$ previously found, 
we get $(2)$. Clearly, $g$ coincides with the polynomial $g_{n-1}$ defined in \eqref{equality-gk22}.
\end{proof}

\vspace*{0.3cm}

\begin{lemma}\label{lemma-ic1-1}
Let $n\geq 6$. Let $r\in [s_0,s_0+n)$. Write $r=r_k=s_0+k-1$, with $1\leq k\leq n$. 
Suppose that $(\iota_r,c_r)={\cred (1,-1)}$. Then $n$ is even, $k=n-3$ and $r=r_{n-3}$,
\begin{eqnarray*}
\phi_r=\left(\frac{n}{2},0,\frac{n-2}{2}\right),\; I_r=\left[1,\frac{n}{2}\right],\;
\good_r=\{0\}\cup \left[2,\frac{n-2}{2}\right].
\end{eqnarray*}
Moreover, $\delta_r=n/2$, $\card(\good_r)=\delta_r-1$, $\dim\ker(\rho_r^{{\rm G}_r})=1$,
$\good_r\leq I_r\setminus\{\phi_{r,1}\}$ and 
\begin{itemize}
\item[$(1)$] $s_1=r+q\in\nums$, $s_1<\la$, $\ell_{s_1}=n-1$, $\rem_{s_1}=2n-3$ and 
$\phi_{s_1}=(0,1,n-2)$.
\item[$(2)$] $s_2=r+\delta_r q\in\nums$, $s_2<\la$, $\ell_{s_2}=n$, $\rem_{s_2}=n-2$ and 
$\phi_{s_2}=((n+2)/2,0,(n-2)/2)$.
\end{itemize}
Furthermore, 
\begin{itemize}
\item[$(3)$] $g=\sum_{j=0}^{\frac{n-2}{2}}
\lambda_jx^{\frac{n-2j}{2}}y^{2j}z^{\frac{n-2-2j}{2}}$, 
where $\lambda_j=(-1)^jb^{\left[1,\frac{n}{2}\right]
\setminus\{\frac{n-2j}{2}\}}_{\{0\}\cup\left[2,\frac{n-2}{2}\right]}\neq 0$, 
is a basis of $\ker(\rho_r^{{\rm G}_r})$. 
\end{itemize}
In particular, $g$ is equal to the polynomial $g_{n-3}$ defined in \eqref{equality-gk1-1}, 
and $g_{n-3}$ is a basis of $\ker\left(\rho_{r_{n-3}}^{{\rm G}_{r_{n-3}}}\right)$. 
\end{lemma}
\begin{proof}
By Proposition~\ref{proposition-class}, if $(\iota_{r_k},c_{r_k})={\cred (1,-1)}$, 
then $n$ is even, $k=n-3$ and $r=r_{n-3}$. 
By Lemma~\ref{lemma-pre}, $\ell_r=n-1=2d-2+i=2d-1$ and, by Lemma~\ref{lemma-general}, 
$d\geq 2$. In particular, $d=n/2\geq 3$ and $\rem_r=\ell+c=n-2$ is even. 
Therefore, $\phi_{r,1}=\ell_r-\lfloor(\rem_r+1)/2\rfloor=n-1-\lfloor (n-1)/2\rfloor=n/2$,
$\phi_{r,2}=\rem_r-2\lfloor\rem_r/2\rfloor=0$ and $\phi_{r,3}=\lfloor\rem_r/2\rfloor=(n-2)/2$. 
Moreover, since $d=n/2$, then $\phi_{r,1}=n/2=d$, $\phi_{r,1}-(d-1)=1$ and 
$I_r=[\phi_{r,1}-(d-1),\phi_{r,1}]=[1,n/2]$. 

Let us see that $s_1=r+q\in\nums$ and $s_1<\la$. 
Since $n\geq 6$ and $1<\lfloor n/2\rfloor$, by Lemma~\ref{lemma-cota}, $s_1=r+q<\frob(\nums)<\la$. 
Since $n$ is even, $q=n-1$, so 
$s_1=r+q=(a+1)\ell+c+q=a(n-1)+(2n-3)$, where $0\leq 2n-3\leq 2(n-1)$. By Proposition~\ref{theo31ggp}, 
$s_1=r+q\in\nums$, $\ell_{s_1}=n-1$, $\rem_{s_1}=2n-3$, 
$\phi_{s_1,1}=\ell_{s_1}-\lfloor (\rem_{s_1}+1)/2\rfloor=0$,
$\phi_{s_1,2}=\rem_{s_1}-2\lfloor\rem_{s_1}/2\rfloor=1$ and 
$\phi_{s_1,3}=\lfloor\rem_{s_1}/2\rfloor=n-2$. This proves $(1)$.

Let us see that $s_2=r+dq\in\nums$ and $s_2<\la$. Since $n\geq 6$ and 
$d=n/2=[n/2]$, by Lemma~\ref{lemma-cota}, $s_2=r+dq<\frob(\nums)<\la$. 
Since $q=n-1$ and $d=n/2$, then $dq=n(n-1)/2=a$.  Therefore,
$r+dq=(a+1)\ell+c+dq=(a+1)(n-1)-1+a=an+(n-2)$, where $0\leq n-2\leq 2n$. 
By Proposition~\ref{theo31ggp}, 
$s_2=r+dq\in\nums$, $\ell_{s_2}=n$, $\rem_{s_1}=n-2$, $\phi_{s_2,1}=\ell_{s_2}-\lfloor (\rem_{s_2}+1)/2\rfloor=(n+2)/2$, $\phi_{s_2,2}=\rem_{s_1}-2\lfloor\rem_{s_1}/2\rfloor=0$ and 
$\phi_{s_2,3}=\lfloor\rem_{s_1}/2\rfloor=(n-2)/2$. This proves $(2)$. 

Since $s_1=r+q$ and $s_2=r+dq$ are in $\nums$, then $1,d\not\in\good_r$. 
Since $\good_r\subseteq[0,\phi_{r,1}]$, where $\phi_{r,1}=n/2=d$ and
$1,d\not\in\good_r$, then $\good_r\subseteq \{0\}\cup [2,d-1]=\{0\}\cup [2,(n-2)/2]$.

Let us prove that $r+kq\not\in\nums$, for all $k\in[2,d-1]$, where $d\geq 3$. 
Since $s_2=r+dq<\frob(\nums)$, then $r+kq<\frob(\nums)$, for all $k\in[2,d-1]$.
Note that $r+kq=(a+1)\ell+c+kq=a\ell+(\ell-1+kq)$. Since $\ell=n-1$, $k\leq d-1$, $q=n-1$ and $dq=(n-1)n/2=a$, then
$\ell-1+kq\leq n-2+(d-1)q=n-2+a-(n-1)=a-1<a$. 
Since $\ell=n-1$, $k\geq 2$, $q=n-1$, then $\ell-1+kq\geq n-2+2(n-1)\geq 2(n-1)=2\ell$. Thus,
$2\ell\leq (\ell-1+kq)<a$. By Proposition~\ref{theo31ggp}, $r+kq\not\in\nums$, for all $k\in[2,d-1]$. 

Thus, $\{r+kq\mid k\in [2,d-1]\}\bigcap\nums=\emptyset$. 
Since $s_1=r+q\in\nums$ and $\phi_{s_1,1}=0$, then $2+\phi_{s_1,1}\leq d=\phi_{r,1}$ 
and $d-1\in [2+\phi_{s_1,1},\phi_{r,1}]=[2,d]$. By Lemma~\ref{lemma-diffcase}, 
$\{0\}\cup [2+\phi_{s_1,1},d-1]=\{0\}\cup [2,d-1]\subseteq \good_r$. Therefore, 
$\good_r=\{0\}\cup [2,d-1]=\{0\}\cup [2,(n-2)/2]$, which is smaller than
$[1,n/2]\setminus\{n/2\}=I_r\setminus\{\phi_{r,1}\}$. 

Since $\card(\good_r)=d-1$, By Lemma~\ref{lemma3-rho},~$(3)$, it follows that
$\dim\ker(\rho_r^{{\rm G}_r})=1$ and a basis of $\ker(\rho_r^{{\rm G}_r})$ is given by
$g=\sum_{j=0}^{d-1}(-1)^jb_{{\rm G}_{r}}^{I_{r}
\setminus\{\phi_{{r},1}-j\}}\mon^{\phi_{r}+j\omega}$.
Substituting the values of $d$, $I_r$, $\good_r$ and $\phi_r$ previously found, 
we get $(3)$. Since $\good_r\leq I_r\setminus\{\phi_{r,1}\}$, by 
\cite[Corollary~2]{gv} or \cite[Corollary~2.5]{gp2}, 
$\lambda_j=(-1)^jb_{{\rm G}_{r}}^{I_{r}\setminus\{\phi_{{r},1}-j\}}\neq 0$.
Clearly, $g$ coincides with the polynomial $g_{n-1}$ defined in \eqref{equality-gk1-1}.
\end{proof}

\vspace*{0.3cm}

\begin{lemma}\label{lemma-ic10}
Let $n\geq 6$. Let $r\in [s_0,s_0+n)$. Write $r=r_k=s_0+k-1$, with $1\leq k\leq n$. 
Suppose that $(\iota_r,c_r)={\cgreen(1,0)}$, Then $n$ is even, $k=n-2$ and $r=r_{n-2}$,
\begin{eqnarray*}
\phi_r=\left(\frac{n-2}{2},1,\frac{n-2}{2}\right),\; I_r=\left[0,\frac{n-2}{2}\right],\;
\good_r=\{0\}\cup \left[2,\frac{n-4}{2}\right]=L_{r,1}\cap L_{r,2},\mbox{ where}
\end{eqnarray*}
$L_{r,1}=\{0\}\cup \left[2,\frac{n-2}{2}\right]$,
$L_{r,2}=\left[0,\frac{n-4}{2}\right]$.
Moreover, $\delta_r=n/2$, $\card(L_{r,i})=\delta_r-1$, $\dim\ker(\rho_r^{L_{r,i}})=1$, 
$\ker(\rho_r^{{\rm G}_r})=\ker(\rho_{r}^{L_{r,1}})\oplus \ker(\rho_r^{L_{r,2}})$ and
$\dim \ker(\rho_r^{{\rm G}_r})=2$. Besides, $L_{r,2}\leq I_r\setminus\{\phi_{r,1}\}$ and
\begin{itemize}
\item[$(1)$] $s_1=r+q\in\nums$, $s_1<\la$, $\ell_{s_1}=n-1$, $\rem_{s_1}=2n-2$ and $\phi_{s_1}=(0,0,n-1)$.
\item[$(2)$] $s_2=r+(\delta_r-1)q\in\nums$, $s_2<\la$, $\ell_{s_2}=n$, $\rem_{s_2}=0$ and $\phi_{s_2}=(n,0,0)$.
\end{itemize}
Furthermore,
\begin{itemize}
\item[$(3)$] $g=xy^{n-3}z-y^{n-1}$ is a basis of $\ker(\rho_r^{L_{r,1}})$.
\item[$(4)$] $\tilde{g}=\sum_{j=0}^{\frac{n-2}{2}}(-1)^{j}
b_{\frac{n-2}{2},\frac{n-2-2j}{2}}
\cdot x^{\frac{n-2-2j}{2}}y^{1+2j}z^{\frac{n-2-2j}{2}}$ is a basis of $\ker(\rho_r^{L_{r,2}})$. 
\end{itemize}
In particular, $g,\tilde{g}$ is a basis of $\ker(\rho_r^{{\rm G}_r})$, $g$ is equal to the polynomial
$g_{n-2}$ defined in \eqref{equality-gk10a}, $\tilde{g}$ is equal to the polynomial
$g_{n-1}$ defined in \eqref{equality-gk10b}, and 
$g_{n-2},g_{n-1}$ is a basis of $\ker\left(\rho_{r_{n-2}}^{{\rm G}_{r_{n-2}}}\right)$.
\end{lemma}
\begin{proof}
By Proposition~\ref{proposition-class}, if $(\iota_{r_k},c_{r_k})={\cgreen(1,0)}$, 
then $n$ is even, $k=n-2$ and $r=r_{n-2}$. By Lemma~\ref{lemma-pre}, 
$\ell_r=n-1=2d-2+i=2d-1$ and, by Lemma~\ref{lemma-general}, 
$d\geq 2$. In particular, $d=n/2\geq 3$ and $\rem_r=\ell+c=n-1$ is odd. 
Therefore, $\phi_{r,1}=\ell_r-\lfloor(\rem_r+1)/2\rfloor=n-1-\lfloor n/2\rfloor=(n-2)/2$,
$\phi_{r,2}=\rem_r-2\lfloor\rem_r/2\rfloor=1$ and $\phi_{r,3}=\lfloor\rem_r/2\rfloor=(n-2)/2$. 
Moreover, since $d=n/2$, then $\phi_{r,1}=(n-2)/2=d-1$, $\phi_{r,1}-(d-1)=0$ and 
$I_r=[\phi_{r,1}-(d-1),\phi_{r,1}]=[0,(n-2)/2]$. 

Let us see that $s_1=r+q\in\nums$ and $s_1<\la$. 
Since $n\geq 6$ and $1<\lfloor n/2\rfloor$, by Lemma~\ref{lemma-cota}, $s_1=r+q<\frob(\nums)<\la$. 
Since $n$ is even, $q=n-1$, so 
$s_1=r+q=(a+1)\ell+c+q=a(n-1)+(2n-2)$, where $0\leq (2n-2)=2\ell$. By Proposition~\ref{theo31ggp}, 
$s_1=r+q\in\nums$, $\ell_{s_1}=n-1$, $\rem_{s_1}=2n-2$, $\phi_{s_1,1}=\ell_{s_1}-\lfloor 
(\rem_{s_1}+1)/2\rfloor=0$, 
$\phi_{s_1,2}=\rem_{s_1}-2\lfloor\rem_{s_1}/2\rfloor=0$ and $\phi_{s_1,3}=\lfloor\rem_{s_1}/2\rfloor=n-1$. 
This shows $(1)$. 

Let us see that $s_2=r+(d-1)q\in\nums$ and $s_2<\la$. Since $n\geq 6$ and
$d-1=(n-2)/2<[n/2]$, by Lemma~\ref{lemma-cota}, $s_2=r+(d-1)q<\frob(\nums)<\la$. 
Since $q=n-1$ and $d=n/2$, then $dq=n(n-1)/2=a$.  Thus,
$r+(d-1)q=(a+1)\ell+c+(d-1)q=(a+1)(n-1)+a-(n-1)=an$. 
By Proposition~\ref{theo31ggp}, 
$s_2=r+(d-1)q\in\nums$, $\ell_{s_2}=n$, $\rem_{s_1}=0$, 
$\phi_{s_2,1}=\ell_{s_2}-\lfloor (\rem_{s_2}+1)/2\rfloor=n$,
$\phi_{s_2,2}=\rem_{s_2}-2\lfloor\rem_{s_2}/2\rfloor=0$ and 
$\phi_{s_2,3}=\lfloor\rem_{s_2}/2\rfloor=0$. This shows $(2)$.

Since $s_1=r+q$ and $s_2=r+(d-1)q$ are in $\nums$, then $1,d-1\not\in\good_r$. 
Since $\good_r\subseteq[0,\phi_{r,1}]$, where $\phi_{r,1}=(n-2)/2=d-1$ and
$1,d-1\not\in\good_r$, then $\good_r\subseteq \{0\}\cup [2,d-2]=\{0\}\cup [2,(n-4)/2]$.

If $n=6$, then $d=3$ and $\good_r=\{0\}=\{0\}\cup [2,d-2]$, understanding that 
$[i,j]=\emptyset$ whenever $i>j$.

Suppose that $n\geq 8$, so $d=n/2\geq 4$. Let us prove that $r+kq\not\in\nums$, for
all $k\in [2,d-2]$. Since $s_2=r+(d-1)q<\frob(\nums)$, then $r+kq<\frob(\nums)$, for all $k\in[2,d-2]$.
Note that $r+kq=(a+1)\ell+c+kq=a\ell+(\ell+kq)$. Since $\ell=n-1$, $k\leq d-2$, $q=n-1$ and $dq=n(n-1)/2=a$, 
then $\ell+kq\leq n-1+(d-2)q=n-1+a-2(n-1)=a-(n-1)<a$. 
Since $\ell=n-1$, $k\geq 2$, $q=n-1$, then $\ell+kq\geq \ell+2\ell>2\ell$. Thus, 
$2\ell\leq (\ell-1+kq)<a$. By Proposition~\ref{theo31ggp}, $r+kq\not\in\nums$, for all $k\in[2,d-2]$. 

Thus, $\{r+kq \mid k\in [2,d-2]\}\cap\nums=\emptyset$. Since $d\geq 4$, $s_1=r+q\in\nums$ and
$\phi_{s_1,1}=0$, then $2+\phi_{s_1,1}=2\leq d-1=\phi_{r,1}$ and $d-2\in
[2+\phi_{s_1,1},\phi_{r,1}]=[2,d-1]$. By Lemma~\ref{lemma-diffcase},
$\{0\}\cup [2,d-2]\subseteq\good_r$. Therefore, $\good_r=\{0\}\cup [2,d-2]=\{0\}\cup [2,(n-4)/2]$.
In particular, $\card(\good_r)=(n-4)/2=d-2$. By Lemma~\ref{lemma3-rho},~$(2)$,
$\dim\ker(\rho_r^{{\rm G}_r})=2$.

One can write $\good_r=L_{r,1}\cap L_{r,2}$, where $L_{r,1}=\{0\}\cup [2,(n-2)/2]$ and
$L_{r,2}=[0,(n-4)/2]$. Note that
$L_{r,2}=[0,(n-4)/2]=[0,(n-2)/2]\setminus\{ (n-2)/2\}=I_r\setminus\{\phi_{r,1}\}$.

Clearly, $\card(L_{r,1})=(n-2)/2=d-1$.
By Lemma~\ref{lemma3-rho},~$(3)$, $\dim \ker(\rho_r^{L_{r,1}})=1$ and
a basis of $\ker(\rho_r^{L_{r,1}})$ is given by the polynomial
$\sum_{j=0}^{d-1}(-1)^j
b_{L_{r,1}}^{I_{r}\setminus\{\phi_{{r},1}-j\}}\mon^{\phi_{r}+j\omega}$.
If $0\leq j\leq \frac{n-6}{2}$, then
$\{0\}\cup\left[2,\frac{n-2}{2}\right]\not\leq \left[0,\frac{n-2}{2}\right]
\setminus\{\frac{n-2-2j}{2}\}$. By \cite[Corollary~2]{gv} or \cite[Corollary~2.5]{gp2}, 
$b^{\left[0,\frac{n-2}{2}\right]
\setminus\{\frac{n-2-2j}{2}\}}_{\{0\}\cup\left[2,\frac{n-2}{2}\right]}=0$. 
By \cite[Lemma~2.2]{gp2}, $b^{\{0\}\cup\left[2,\frac{n-2}{2}\right]}
_{\{0\}\cup\left[2,\frac{n-2}{2}\right]}=1$. By \cite[Lemma~4.3 and Lemma~2-2]{gp2},
$b^{\left[1,\frac{n-2}{2}\right]}
_{\{0\}\cup\left[2,\frac{n-2}{2}\right]}=b^{\left[1,\frac{n-4}{2}\right]}
_{\left[1,\frac{n-4}{2}\right]}=1$. 
Thus, a basis of $\ker(\rho_r^{L_{r,1}})$ is given by the polynomial
\begin{eqnarray*}
g=b^{\{0\}\cup\left[2,\frac{n-2}{2}\right]}
_{\{0\}\cup\left[2,\frac{n-2}{2}\right]}\cdot xy^{n-3}z-
b^{\left[1,\frac{n-2}{2}\right]}
_{\{0\}\cup\left[2,\frac{n-2}{2}\right]}\cdot y^{n-1}=xy^{n-3}z-y^{n-1},
\end{eqnarray*}
which proves $(3)$. 
Similarly, $\card(L_{r,2})=(n-2)/2=d-1$. By Lemma~\ref{lemma3-rho},~$(3)$,
$\dim \ker(\rho_r^{L_{r,2}})=1$ and
$\tilde{g}=\sum_{j=0}^{\delta_{r}-1}(-1)^jb_{L_{r,2}}^{I_{r}
\setminus\{\phi_{{r},1}-j\}}\mon^{\phi_{r}+j\omega}$ is a basis of $\ker(\rho_r^{L_{r,2}})$.
Substituting the values of $d$, $I_r$, $\good_r$ and $\phi_r$ previously found, 
we get $(4)$. 

Since $\good_r=L_{r,1}\cap L_{r,2}\subseteq L_{r,i}$, 
then $\ker(\rho_r^{L_{r,1}})+\ker(\rho_r^{L_{r,2}})\subseteq\ker(\rho_r^{{\rm G}_r})$.
On the other hand, $L_{r,1}\cup L_{r,2}=[0,(n-2)/2]$, which has cardinality $n/2=d$. 
By Lemma~\ref{lemma3-rho},~$(2)$, 
$\dim \ker(\rho_r^{L_{r,1}\cup L_{r,2}})=0$.
By Lemma~\ref{lemma3-rho},~$(7)$,
$\ker(\rho_r^{L_{r,1}})\cap\ker(\rho_r^{L_{r,2}})=\ker(\rho_r^{L_{r,1}\cup L_{r,2}})=0$.
Therefore, $\ker(\rho_r^{{\rm G}_r})=\ker(\rho_r^{L_{r,1}})\oplus\ker(\rho_r^{L_{r,2}})$.
In particular, $g,\tilde{g}$ is a basis of $\ker(\rho_r^{{\rm G}_r})$. Clearly, 
$g$ coincides with the polynomial $g_{n-2}$ given in \eqref{equality-gk10a} and 
$\tilde{g}$ coincides with the polynomial $g_{n-1}$ given in \eqref{equality-gk10b}. 
Therefore, $g_{n-2},g_{n-1}$ is a basis of 
$\ker\left(\rho_{r_{n-2}}^{{\rm G}_{r_{n-2}}}\right)$.
\end{proof}

\vspace*{0.3cm}

\begin{lemma}\label{lemma-ic11}
Let $n\geq 6$. Let $r\in [s_0,s_0+n)$.  Write $r=r_k=s_0+k-1$, with $1\leq k\leq n$. 
Suppose that $(\iota_r,c_r)={\cblue (1,1)}$. Then $n$ is even, $k=n-1$ and $r=r_{n-1}$,
\begin{eqnarray*}
\phi_r=\left(\frac{n-2}{2},0,\frac{n}{2}\right),\; I_r=\left[0,\frac{n-2}{2}\right],\;
\good_r=\left[0,\frac{n-4}{2}\right].
\end{eqnarray*}
Moreover, $\delta_r=n/2$, $\card(\good_r)=\delta_r-1$, $\dim\ker(\rho_r^{{\rm G}_r})=1$,
$\good_r\leq I_r\setminus\{\phi_{r,1}\}$ and 
\begin{itemize}
\item[$(1)$] $s_1=r+(\delta_r-1)q\in\nums$, $s_1<\la$, $\ell_{s_1}=n$, $\rem_{s_1}=1$ and 
$\phi_{s_1}=(n-1,1,0)$.
\end{itemize}
Furthermore, 
\begin{itemize}
\item[$(2)$] $g=\sum_{j=0}^{\frac{n-2}{2}}(-1)^jb_{\frac{n-2}{2},\frac{n-2-2j}{2}}\cdot 
x^{\frac{n-2-2j}{2}}y^{2j}z^{\frac{n-2j}{2}}$ is a basis of $\ker(\rho_r^{{\rm G}_r})$. 
\end{itemize}
In particular, $g$ is equal to the polynomial $g_{n}$ defined in \eqref{equality-gk11}, 
and $g_{n}$ is a basis of $\ker\left(\rho_{r_{n-1}}^{{\rm G}_{r_{n-1}}}\right)$. 
\end{lemma}
\begin{proof}
By Proposition~\ref{proposition-class}, if $(\iota_{r_k},c_{r_k})={\cblue (1,1)}$, 
then $n$ is even, $k=n-1$ and $r=r_{n-1}$. 
By Lemma~\ref{lemma-pre}, $\ell_r=n-1=2d-2+i=2d-1$ and, by Lemma~\ref{lemma-general}, 
$d\geq 2$. In particular, $d=n/2\geq 3$ and $\rem_r=\ell+c=n$ is even. 
Therefore,
$\phi_{r,1}=\ell_r-\lfloor(\rem_r+1)/2\rfloor=n-1-\lfloor (n+1)/2\rfloor=(n-2)/2$,
$\phi_{r,2}=\rem_r-2\lfloor\rem_r/2\rfloor=0$ and $\phi_{r,3}=\lfloor\rem_r/2\rfloor=n/2$. 
Moreover, since $d=n/2$, then $\phi_{r,1}=(n-2)/2=d-1$, $\phi_{r,1}-(d-1)=0$ and 
$I_r=[\phi_{r,1}-(d-1),\phi_{r,1}]=[0,(n-2)/2]$. 

Let us see that $s_1=r+(d-1)q\in\nums$ and $s_1<\la$. 
Since $n\geq 6$, since $d-1=(n-2)/2<\lfloor n/2\rfloor$, 
by Lemma~\ref{lemma-cota}, $s_1=r+(d-1)q<\frob(\nums)<\la$. 
Since $d=n/2$ and $n$ is even, then $q=n-1$, $dq=n(n-1)/2=a$ and $s_1=r+(d-1)q=(a+1)\ell+c+(d-1)q=
(a+1)(n-1)+1+a-(n-1)=an+1$, where $0\leq 1\leq 2n$. By Proposition~\ref{theo31ggp}, 
$s_1=an+1\in\nums$, $\ell_{s_1}=n$, $\rem_{s_1}=1$, 
$\phi_{s_1,1}=\ell_{s_1}-\lfloor (\rem_{s_1}+1)/2\rfloor=n-1$,
$\phi_{s_1,2}=\rem_{s_1}-2\lfloor\rem_{s_1}/2\rfloor=1$ and $\phi_{s_1,3}=\lfloor\rem_{s_1}/2\rfloor=0$.
This shows $(1)$.

Since $s_1=r+(d-1)q\in\nums$, then $d-1\not\in\good_r$. Since $\good_r\subseteq[0,\phi_{r,1}]$ and
$\phi_{r,1}=d-1\not\in\good_r$, then $\good_r\subseteq [0,d-2]$.

Let us prove that $r+kq\not\in\nums$, for all $k\in[1,d-2]$, where $d\geq 3$. 
Since $s_1<\frob(\nums)$, then $r+kq<\frob(\nums)$, for all $k\in[1,d-2]$.
Note that $r+kq=(a+1)\ell+c+kq=a\ell+(\ell+1+kq)$. Since $\ell=n-1$, $k\leq d-2$, $q=n-1$ and $d(n-1)=a$, then
$\ell+1+kq\leq n+(d-2)(n-1)=n+a-2(n-1)=a-n+2<a$. 
Since $\ell=n-1$, $k\geq 1$, $q=n-1$, then $\ell+1+kq\geq n+n-1=2n-1>2\ell$. Thus,
$2\ell\leq (\ell+1+kq)<a$. By Proposition~\ref{theo31ggp}, $r+kq\not\in\nums$, for all $k\in[1,d-2]$. 

Thus, $\{r+kq\mid k\in [1,d-2]\}\bigcap\nums=\emptyset$.  By
Lemma~\ref{lemma-good}, $[0,d-2]\subseteq\good_r$. So,
$\good_r=[0,d-2]$. Since $d=n/2$, then $d-2=(n-4)/2$ and $\good_r=[0,(n-4)/2]$, which is equal to 
$I_r\setminus\{\phi_{r,1}\}$.

Since $\card(\good_r)=d-1$, By Lemma~\ref{lemma3-rho},~$(3)$, it follows that
$\dim\ker(\rho_r^{{\rm G}_r})=1$ and a basis of $\ker(\rho_r^{{\rm G}_r})$ is given by
the polynomial
$g=\sum_{j=0}^{d-1}(-1)^jb_{{\rm G}_{r}}^{I_{r}
\setminus\{\phi_{{r},1}-j\}}\mon^{\phi_{r}+j\omega}$.
By \cite[Corollary~6.3]{gp2}, 
$b^{[0,\tfrac{n-2}{2}]\setminus\{\tfrac{n-2-2j}{2}\}}_{[0,\tfrac{n-4}{2}]}=
b_{\tfrac{n-2}{2},\tfrac{n-2-2j}{2}}$.
Thus, substituting the values of $d$, $I_r$, $\good_r$ and $\phi_r$ previously found, 
we get $(2)$. Clearly, $g$ coincides with the polynomial $g_{n}$ defined in \eqref{equality-gk11}.
\end{proof}

\vspace*{0.3cm}

\begin{lemma}\label{lemma-ic32}
Let $n\geq 6$. Let $r\in [s_0,s_0+n)$.  Write $r=r_k=s_0+k-1$, with $1\leq k\leq n$. 
Suppose that $(\iota_r,c_r)={\colive (3,2)}$. Then $n$ is even, $k=n$, $r=r_{n}$ and
$I_r=\good_r=[0,(n-4)/2]$. Moreover, $V_r=\ker(\rho_r^{{\rm G}_r})=0$.
\end{lemma}
\begin{proof}
By Proposition~\ref{proposition-class}, if $(\iota_{r_k},c_{r_k})={\colive (3,2)}$, 
then $n$ is even, $k=n$ and $r=r_{n}$. 
By Lemma~\ref{lemma-pre}, $\ell_r=n-1=2d-2+i=2d+1$. 
In particular, $d=(n-2)/2$ and $\rem_r=\ell+c=n+1$ is odd. 
Therefore, $\phi_{r,1}=\ell_r-\lfloor(\rem_r+1)/2\rfloor=n-1-\lfloor (n+2)/2\rfloor=(n-4)/2$,
$\phi_{r,2}=\rem_r-2\lfloor\rem_r/2\rfloor=0$ and $\phi_{r,3}=\lfloor\rem_r/2\rfloor=n/2$. 
Moreover, since $d=(n-2)/2$, then $\phi_{r,1}=(n-4)/2=d-1$, $\phi_{r,1}-(d-1)=0$ and 
$I_r=[\phi_{r,1}-(d-1),\phi_{r,1}]=[0,(n-4)/2]$. Moreover, $\good_r\subseteq [0,\phi_{r,1}]=[0,(n-4)/2]$.

Let us prove that $r+kq\not\in\nums$, for all $k\in[1,d-1]$. Note that 
$d-1=(n-4)/2<\lfloor n/2\rfloor$. Since $n\geq 6$ and $k\leq d-1\leq \lfloor n/2\rfloor$, 
by Lemma~\ref{lemma-cota}, $r+kq<\frob(\nums)$, for all $k\in[0,d-1]$.
We have $r+kq=(a+1)\ell+c+kq=a\ell+(\ell+2+kq)$. Since $\ell=n-1$, $k\leq d-1$, $q=n-1$,  
$d=(n/2)-1$ and $dq=a-(n-1)$, then
$\ell+2+kq\leq n+1+(d-1)q=n+1+a-2(n-1)=a-n+3<a$. 
Since $\ell=n-1$, $k\geq 1$, $q=n-1$, then $\ell+2+kq\geq 2n>2\ell$. Thus,
$2\ell\leq (\ell+2+kq)<a$. By Proposition~\ref{theo31ggp}, $r+kq\not\in\nums$. 

Thus, $\{r+kq\mid k\in [1,d-1]\}\bigcap\nums=\emptyset$.  By
Lemma~\ref{lemma-good}, $[0,d-1]\subseteq\good_r$. So, $\good_r=[0,d-1]$. 
Since $d=(n-2)/2$, then $d-1=(n-4)/2$ and $\good_r=[0,(n-4)/2]$. Moreover, 
$\card(\good_r)=d$. By Corollary~\ref{corollary-vr0}, $V_r=\ker(\rho_r^{{\rm G}_r})=0$. 
\end{proof}

As an immediate consequence of Lemmas~\ref{lemma-c-i}, \ref{lemma-c-i+1}, \ref{lemma-ic00}, 
\ref{lemma-ic21}, \ref{lemma-ic22}, \ref{lemma-ic1-1}, \ref{lemma-ic10} and
\ref{lemma-ic11}, and keeping the notations as in Definition~\ref{def-gk},
we deduce the main result of the section, which is the determination of bases for the good kernels. 

\begin{theorem}\label{ThirdStep}
Let $n\geq 6$. Let $r_k=s_0+k-1\in [s_0,s_0+n)$, where $1\leq k\leq n$. 

\noindent Suppose that $n$ is odd, $n\geq 7$. 
\begin{itemize}
\item Then, for each $1\leq k\leq n$, $g_k$ is a basis of the one-dimensional vector space 
$\ker\left(\rho_{r_k}^{{\rm G}_{r_k}}\right)$. 
\end{itemize}
Suppose that $n$ is even, $n\geq 6$.
\begin{itemize}
\item  For each $1\leq k\leq n-3$, $g_k$ is a basis of the one-dimensional vector space 
$\ker\left(\rho_{r_k}^{{\rm G}_{r_k}}\right)$. 
\item $g_{n-2},g_{n-1}$ is a basis of the two-dimensional vector space 
$\ker\left(\rho_{r_{n-2}}^{{\rm G}_{r_{n-2}}}\right)$.
\item $g_{n}$ is a basis of the one-dimensional vector space 
$\ker\left(\rho_{r_{n-1}}^{{\rm G}_{r_{n-1}}}\right)$. 
\end{itemize}
\end{theorem}

\begin{remark}
Note that, when $n$ is even, the enumeration of the last two polynomials,
might be a little bit misleading because 
$g_{n-1}$ belongs to $\ker\left(\rho_{r_{n-2}}^{{\rm G}_{r_{n-2}}}\right)$ and
$g_n$ belongs to $\ker\left(\rho_{r_{n-1}}^{{\rm G}_{r_{n-1}}}\right)$. This is not the case when
$n$ is odd, where every $g_k$ belongs to $\ker\left(\rho_{r_{k}}^{{\rm G}_{r_{k}}}\right)$ 
\end{remark}

\section{Finding a part of a minimal generating set of the kernel}\label{sec-Vr}

In this section we show that $V_r=\ker(\rho_r^{{\rm G}_r})$, for all $r\in [s_0,s_0+n)$, $n\geq 6$. 
As a consequence, and using the \emph{Extending to basis Method}, we deduce a 
part of a minimal generating set of $\ker(\rho)$. 

The proof consists in finding elements $f\in\ker(\rho)$ such that $f^{\sigma}$ is (part of) a basis of 
$\ker(\rho_r^{{\rm G}_r})$. 
These elements can be obtained explicitly on solving at most four $n\times n$ linear systems. 
We divide the reasoning into several lemmas according to which subset of the partition of $[s_0,s_0+n)$ given in Proposition~\ref{proposition-class} contains the element $r$. 
As always, for the sake of readability, we write in the proofs $\ell$, $d$, $i$ and $c$ instead of $\ell_r$, $\delta_r$, $\iota_r$ and $c_r$, if no confusion is possible. 

Next observation follows directly from Lemma~\ref{lemma1-rho}, $(1)$, Lemma~\ref{lemma2-rho}, $(1)$, and Notation~\ref{notation-binomials}. We state it here for easy quotation and since it will be used intensively all along the section.

\begin{remark}\label{rem-comp}
Let $s\in\nums$, $s<\la$. Recall that $T_s$ denotes the $\mbk$-vector space spanned 
by the linearly independent monomials 
$\mcc_s=\{t^{s+kq}\mid k\in [0,\phi_{s,1}]\}$; $\rho_s:W_s\to T_s$ denotes 
the restriction of $\rho$, and $I_s=[\phi_{s,1}-\kappa_s,\phi_{s,1}]$ and $H_s=[0,\phi_{s,1}]$. Let $h\in W_s$ and $\Lambda$ the components of $h$ 
in the basis $\mcb_s$ of $W_s$.
The components of $\rho_s(h)$ in the basis $\mcc_s$ can be calculated as
$P_{I_s}^{H_s}\cdot\Lambda$, where 
$P_{I_s}^{H_s}=\left(B_{H_s}^{I_s}\right)^{\top}\cdot\Theta_{{\rm \card}(I_s)}$.

Let $L\subseteq H_s$, $\mcc_s^L=\{t^{s+kq}\mid k\in L\}\subseteq
\mcc_s$, $T_s^L$ the $\mbk$-vector space spanned 
by the linearly independent monomials 
$\mcc_s^L$ and $\rho_s^L:W_s\to T_s^L$ the composition of $\rho_s$ with the natural projection
$T_s\to T_s^L$. 
Similarly, the components of $\rho_s^L(h)$ in the basis $\mcc_s^L$ can be calculated
as $P_{I_s}^{L}\cdot\Lambda$, where
$P_{I_s}^{L}=\left(B_{L}^{I_s}\right)^{\top}\cdot\Theta_{{\rm \card}(I_s)}$. Notice that if $h\in W_s$, then $\rho(h)=\rho_s^{H_s}(h)$. 
For the sake of easy further reading, we provide a scheme summary of the remark:
\begin{framed}
\begin{eqnarray*}
\text{If }h=
\sum_{j=\text{ first in }W_s}^{\text{last in }W_s}
\!\lambda_j\mon^{\phi_s+jq},\text{ then }
\rho^L_s(h)=
\sum_{k=\text{ first in }L}^{\text{last in }L}
\left(
\sum_{j=\text{ first in }W_s}^{\text{last in }W_s}
b_{\phi_{s,1}-j,k}\lambda_j\right)
t^{s+kq}.
\end{eqnarray*}
\end{framed}
\end{remark}

We maintain the ordering of the cases as in 
Section~\ref{sec-kernels}. However, we would suggest the reader to 
start with Lemma~\ref{lemma-h-00} so as to get an idea of how the proofs of all these lemmas 
work. 

\vspace*{0.3cm}

\begin{lemma}\label{lemma-h-i}
Let $n\geq 6$. Let $r\in [s_0,s_0+n)$. Write $r=r_k=s_0+k-1$, with $1\leq k\leq n$. 
Suppose that  ${\corange c_r=-\iota_r}$, $2\leq \iota_r\leq n-3$. Let
\begin{eqnarray*}
s_1=r+q,\quad s_2=r+\left\lfloor \frac{n}{2}\right\rfloor q,\quad s_3=s_2+q\quad\mbox{and}\quad s_4=s_3+\left\lfloor \frac{n}{2}\right\rfloor q.
\end{eqnarray*}
Then $s_1,s_2,s_3,s_4\in\nums$ and $r<s_1<s_2<s_3<s_4<\la$. Moreover,
$\phi_r=\left(\frac{n-1+\iota_r}{2},0,\frac{n-1-\iota_r}{2}\right)$,
\begin{eqnarray*}
&&\phi_{s_1}=\left(\left\lfloor\frac{\iota_r-1}{2}\right\rfloor,1,
n-\left\lfloor\frac{\iota_r+3}{2}\right\rfloor\right),
\quad\phi_{s_2}=\left(\frac{n+1+\iota_r}{2}, 0,\frac{n-1-\iota_r}{2}\right),\\ 
&&\phi_{s_3}=\left(\left\lfloor\frac{\iota_r+1}{2}\right\rfloor,1,
n-\left\lfloor\frac{\iota_r+3}{2}\right\rfloor\right),
\quad\phi_{s_4}=\left(\left\lfloor\frac{\iota_r+3}{2}\right\rfloor,1,
n-\left\lfloor\frac{\iota_r+3}{2}\right\rfloor\right).
\end{eqnarray*}
Let $g=\sum_{j=0}^{\frac{n-1-\iota_r}{2}}\lambda_j\mon^{\phi_r+j\omega}$, where 
$\lambda_j=(-1)^jb_{\{0\}\sqcup[\lfloor\frac{\iota_r+3}{2}\rfloor,
\lfloor\frac{n-2}{2}\rfloor]}^{[\iota_r,\frac{n-1+\iota_r}{2}]\setminus 
\{\frac{n-1+\iota_r-2j}{2}\}}$. Set $\Lambda^{\top}=\left(\lambda_{0},\ldots,
\lambda_{\frac{n-1-\iota_r}{2}}\right)$.
\\ 
Let $h_1=
\sum_{j=0}^{\lfloor\frac{(\iota_r-1)}{2}\rfloor}\lambda_{1,j}\mon^{\phi_{s_1}+j\omega}$, 
where 
$\Lambda_1^\top=\left(\lambda_{1,0},\ldots,
\lambda_{1,\lfloor\frac{\iota_r-1}{2}\rfloor}\right)$ satisfies
\begin{eqnarray*}
&&\left(B^{[0,\lfloor\frac{\iota_r-1}{2}\rfloor]}
_{[0,\lfloor\frac{\iota_r-1}{2}\rfloor]}\right)^\top\cdot\Theta_{\lfloor\frac{\iota_r+1}{2}\rfloor}\cdot\Lambda_1=
-\left(B^{[\iota_r,\frac{n-1+\iota_r}{2}]}_{[1,\lfloor\frac{\iota_r+1}{2}\rfloor]}\right)^\top
\cdot\Theta_{\frac{n+1-\iota_r}{2}}\cdot\Lambda.
\end{eqnarray*}
Let
$h_2=\sum_{j=1}^{\frac{n-1-\iota_r}{2}}\lambda_{2,j}\mon^{\phi_{s_2}+j\omega}$, 
where $\Lambda_2^\top=\left(\lambda_{2,1},\ldots,
\lambda_{2,\frac{n-1-\iota_r}{2}}\right)$ satisfies
\begin{eqnarray*}
&&\left(B^{[\iota_r+1,\frac{n-1+\iota_r}{2}]}_{\{0\}\sqcup 
[\lfloor\frac{\iota_r+5}{2}\rfloor,\lfloor\frac{n}{2}\rfloor]}\right)^\top\cdot
\Theta_{\frac{n-1-\iota_r}{2}}\cdot\Lambda_2=-
\left(B^{[\iota_r,\frac{n-1+\iota_r}{2}]}
_{\left\{\left\lfloor\frac{n}{2}\right\rfloor\right\}}
\middle| \; 0\; \right)^{\top}\cdot \Theta_{\frac{n+1-\iota_r}{2}}\cdot 
\Lambda.
\end{eqnarray*}
Let
$h_3=\sum_{j=0}^{\lfloor\frac{\iota_r+1}{2}\rfloor}\lambda_{3,j}\mon^{\phi_{s_3}+j\omega}$, where $\Lambda_{3}^\top=\left(\lambda_{3,0},\ldots,\lambda_{3,\lfloor\frac{\iota_r+1}{2}\rfloor}\right)$ satisfies
\begin{multline*}
\left(B^{[0,\lfloor\frac{\iota_r+1}{2}\rfloor]}
_{[0,\lfloor\frac{\iota_r+1}{2}\rfloor]}\right)^\top\cdot
\Theta_{\lfloor\frac{\iota_r+3}{2}\rfloor}\cdot\Lambda_3=\\
-\left(B^{\left[\iota_r,\frac{n-1+\iota_r}{2}\right]}
_{\left[\lfloor\frac{n+2}{2}\rfloor,
\lfloor\frac{n+2}{2}\rfloor+\lfloor\frac{\iota_r+1}{2}\rfloor\right]}\right)^{\top}
\cdot\Theta_{\frac{n+1-\iota_r}{2}}\cdot\Lambda
-\left(\begin{array}{c}
B^{\left[\iota_r+1,\frac{n-1+\iota_r}{2}\right]}_
{\left[1,\lfloor\frac{\iota_r+3}{2}\rfloor\right]}
\end{array}\right)^{\top}
\cdot\Theta_{\frac{n-1-\iota_r}{2}}\cdot\Lambda_2.
\end{multline*}
Let
$h_4=\sum_{j=\lfloor\frac{\iota_r+3}{2}\rfloor-\lfloor\frac{\iota_r-2}{2}\rfloor}
^{\lfloor\frac{\iota_r+3}{2}\rfloor}\lambda_{4,j}\mon^{\phi_{s_4}+j\omega}$, 
where $\Lambda_4^\top=\left(\lambda_{4,\lfloor\frac{\iota_r+3}{2}\rfloor-\lfloor\frac{\iota_r-2}{2}\rfloor},\ldots,\lambda_{4,\lfloor\frac{\iota_r+3}{2}\rfloor}\right)$ satisfies
\begin{eqnarray*}
&&\left(B_{[0,\lfloor\frac{\iota_r-2}{2}\rfloor]}
^{[0,\lfloor\frac{\iota_r-2}{2}\rfloor]}\right)^\top\cdot
\Theta_{\lfloor\frac{\iota_r}{2}\rfloor}\cdot\Lambda_4=
-\left(B_{[\lfloor\frac{n+2}{2}\rfloor,\frac{n-1+\iota_r}{2}]}^{[\iota_r+1,\frac{n-1+\iota_r}{2}]}\right)^{\top}\cdot
\Theta_{\frac{n-1-\iota_r}{2}}\cdot\Lambda_2.
\end{eqnarray*}
Set $h=h_1+h_2+h_3+h_4$ and  $f=g+h$. 
Then $g$ is a basis of $\ker(\rho_r^{{\rm G}_r})$,
$h_i\in W_{s_i}$, $i=1,2,3,4$,
\begin{eqnarray*}
f^\sigma=g,\phantom{+} f^\tau=h,\phantom{+}
f\in\ker(\rho)\phantom{+} \text{and}\phantom{+} g\in V_r.
\end{eqnarray*}
Hence, $V_r=\ker(\rho_r^{{\rm G}_r})$. 

Suppose that $n$ is odd and that $1\leq k\leq n-4$, with $k$ odd. Then $g$ is equal 
to the polynomial $g_k$ defined in \eqref{equality-gknoddkodd} and $f=g+h$ is equal to
the polynomial $f_k$ defined as in \eqref{poly-fknoddkodd}.
In particular, $f_k\in\ker(\rho)$ and $f_k^{\sigma}=g_k$ is a basis of $V_{r_k}$.

Suppose that $n$ is even and that $1\leq k\leq n-5$, with $k$ odd. 
Then $g$ is equal to the polynomial $g_k$ defined in \eqref{equality-gknevenkodd} 
and $f=g+h$ is equal to the polynomial $f_k$ defined as in \eqref{poly-fknevenkodd}.
In particular, $f_k\in\ker(\rho)$ and $f_k^{\sigma}=g_k$ is a basis of $V_{r_k}$.
\end{lemma}
\begin{proof}
Observe that $r=(a+1)\ell_r+c_r=(a+1)(n-1)-i$. Clearly, $r<s_1<s_2<s_3<s_4$. 
Let us prove that $s_4<\la$. 
If $n=6$, then $a=15$, $q=n-1=5$, $i$ is odd and, since $2\leq i\leq n-3$, it follows that $i=3$ and $r=77$. Thus, $s_4=r+(2\lfloor n/2\rfloor+1)q=112$ and $\la=(\lfloor a/2\rfloor+2)a=135$. Suppose that $n\geq 7$. 
By Lemma~\ref{lemma-cota}, $(1)$, $3\leq \lfloor a/2\rfloor-n$. 
Thus, $2\lfloor n/2\rfloor+1\leq 3\lfloor n/2\rfloor\leq (\lfloor a/2\rfloor-n)\lfloor n/2\rfloor$. By Lemma~\ref{lemma-cota}, $(2)$, $s_4<\frob(\nums)<\la$. In particular, $s_1,s_2,s_3<\la$.

By Lemma~\ref{lemma-c-i}, $s_1,s_2\in\nums$,
$\phi_{s_1}=(\left\lfloor\frac{i-1}{2}\right\rfloor,1,n-\left\lfloor\frac{i+3}{2}\right\rfloor)$ 
and $\phi_{s_2}=(\frac{n+1+i}{2}, 0,\frac{n-1-i}{2})$. 
Since $i\leq n-3$, then $\kappa_{s_1}=
\min(\phi_{s_1,1},\phi_{s_1,3})=\lfloor (i-1)/2\rfloor$, 
$I_{s_1}=[\phi_{s_1,1}-\kappa_{s_1},\phi_{s_1,1}]=[0,\lfloor(i-1)/2\rfloor]$,
$H_{s_1}=[0,\phi_{s_1,1}]=[0,\lfloor(i-1)/2\rfloor]$,
$W_{s_1}=\langle\mon^{\phi_{s_1}+j\omega}\mid 0\leq j\leq\lfloor(i-1)/2\rfloor\rangle$,
$h_1$ is well-defined and
$h_1\in W_{s_1}$. 
Similarly, $\kappa_{s_2}=(n-1-i)/2$, 
$I_{s_2}=[i+1,(n+1+i)/2]$, $H_{s_2}=[0,(n+1+i)/2]$, 
$W_{s_2}=\langle\mon^{\phi_{s_2}+j\omega}\mid 0\leq j\leq (n-1-i)/2\rangle$,
$h_2$ is well-defined and $h_2\in W_{s_2}$. 

Let us prove that $s_3\in\nums$. 
By Remark~\ref{aboutq}, $\lfloor n/2\rfloor q=a$. Thus, 
$s_3=s_2+q=r+(\lfloor n/2\rfloor+1)q=an+(n-1-i+q)$. 
Since $i\geq 2$ and $q\leq n$, then $n-1-i+q\leq 2n-3<2n$. 
Since $i\leq n-3$ and $q\geq n-1$, then $n-1-i+q\geq n+1>0$. Therefore, 
$0\leq n-1-i+q\leq 2n$. By Proposition~\ref{theo31ggp}, $s_3\in\nums$,  
and since $s_3<\la$, then we deduce $\ell_{s_3}=n$, 
$\rem_{s_3}=n-1-i+q$ and
$\phi_{s_3}=(\lfloor (i+1)/2\rfloor,1,n-\lfloor (i+3)/2\rfloor)$. 
As before, since $i\leq n-3$, then 
$\kappa_{s_3}=\min(\phi_{s_3,1},\phi_{s_3,3})=
\lfloor (i+1)/2\rfloor$, $I_{s_3}=[0,\lfloor(i+1)/2\rfloor]$,
$H_{s_3}=[0,\lfloor(i+1)/2\rfloor]$, 
$W_{s_3}=\langle\mon^{\phi_{s_3}+j\omega}\mid 0\leq j\leq
\lfloor(i+1)/2\rfloor\rangle$,
$h_3$ is well-defined and $h_3\in W_{s_3}$.

Let us prove that $s_4\in\nums$. Since $s_3=an+(n-1-i+q)$ and $\lfloor n/2\rfloor q=a$, 
it follows that $s_4=s_3+\lfloor\frac{n}{2}\rfloor q=a(n+1)+(n-1-i+q)$. 
Since $q\leq n$ and $i\geq 2$, then $n-1-i+q\leq 2n-3<2(n+1)$. 
Since $q\geq n-1$ and $i\leq n-3$, then $n-1-i+q\geq n+1>0$. Therefore, 
$0\leq n-1-i+q\leq 2(n+1)$. By Proposition~\ref{theo31ggp}, $s_4\in\nums$, and  since $s_4<\la$, 
we get that $\ell_{s_4}=n+1$, $\rem_{s_4}=n-1-i+q$ and 
$\phi_{s_4}=(\lfloor(i+3)/2\rfloor,1,n-\lfloor(i+3)/2\rfloor)$.  
Again, since $i\leq n-3$, then $\kappa_{s_4}=\min(\phi_{s_4,1},\phi_{s_4,3})=
\lfloor (i+3)/2\rfloor$,
$I_{s_4}=[0,\lfloor(i+3)/2\rfloor]$,
$H_{s_4}=[0,\lfloor(i+3)/2\rfloor]$,
$W_{s_4}=\langle \mon^{\phi_{s_4}+j\omega}\mid 0\leq j\leq
\lfloor(i+3)/2\rfloor\rangle$, $h_4$ is well-defined and $h_4\in W_{s_4}$.

By Lemma~\ref{lemma-c-i}, $\phi_r=\left(\frac{n-1+i}{2},0,\frac{n-1-i}{2}\right)$ 
and $g$ is a basis of $\ker(\rho_r^{{\rm G}_r})$.
Since $r<s_1<s_2<s_3<s_4$, then $f^\sigma=g$ and $f^\tau=h$. 
If we prove that $\rho(f)=0$, then we will deduce that $g\in V_r$ and 
$V_r=\ker(\rho_r^{{\rm G}_r})$.

By Lemma~\ref{lemma-c-i}, $I_r=[i,(n-1+i)/2]$, 
$H_r=[0,\frac{n-1+i}{2}]$ and
$\good_r=\{0\}
\cup[\lfloor\frac{i+3}{2}\rfloor,\lfloor\frac{n-2}{2}\rfloor]$. 

Set $L_r^1=[1,1+\lfloor\frac{i-1}{2}\rfloor]$, $L_r^2=
[\lfloor\frac{i+3}{2}\rfloor,\lfloor\frac{n-2}{2}\rfloor]$,
$L_r^3=\{\lfloor\frac{n}{2}\rfloor\}$ and $L_r^4=
[\lfloor\frac{n+2}{2}\rfloor,\frac{n-1+i}{2}]$, so that
$\good_r=\{0\}\cup L_r^2$ and $H_r=\{0\}\sqcup L_r^1\sqcup L_r^2\sqcup
L_r^3\sqcup L_r^4$. 
If $i=n-3$, then $L_r^1=[1,\lfloor\frac{n-2}{2}\rfloor]$, 
$L_r^3=\{\lfloor\frac{n}{2}\rfloor\}$ and 
we take $L_{r}^2=\emptyset$. 
Since $g\in\ker(\rho_r^{{\rm G}_r})$, then, by Lemma~\ref{lemma3-rho},~$(8)$, 
\begin{eqnarray*}
  \rho(g)=\rho_r^{{\rm G}_r}(g) +\rho_r^{H_r\setminus {{\rm G}_r}}(g)=
  \rho_r^{H_r\setminus {{\rm G}_r}}(g)=
  \rho_r^{L_r^1}(g)+\rho_r^{L_r^3}(g)+\rho_r^{L_r^4}(g).
\end{eqnarray*}
Using Remark~\ref{rem-comp} and the equalities $s_1=r+q$, $s_2=r+\lfloor n/2\rfloor q$, 
$s_3=s_2+q=r+\lfloor (n+2)/2\rfloor q$,
and $(n-1+i)/2-\lfloor (n+2)/2\rfloor=\lfloor (i-2)/2\rfloor$, 
we obtain:
\begin{eqnarray*}
&&\rho_r^{L_r^1}(g)=\sum_{k=1}^{1+\lfloor\frac{i-1}{2}\rfloor}
\left(\sum_{j=0}^{\frac{n-1-i}{2}}b_{\frac{n-1+i-2j}{2},k}\lambda_j\right)t^{r+kq}=
 \sum_{k=0}^{\lfloor\frac{i-1}{2}\rfloor}
\left(\sum_{j=0}^{\frac{n-1-i}{2}}b_{\frac{n-1+i-2j}{2},k+1}\lambda_j\right)
t^{s_1+kq},\\
&&\rho_r^{L_r^3}(g)=
\left(\sum_{j=0}^{\frac{n-1-i}{2}}b_{\frac{n-1+i-2j}{2},\lfloor\frac{n}{2}\rfloor}\lambda_j\right)
t^{r+\lfloor\frac{n}{2}\rfloor q}=
\left(\sum_{j=0}^{\frac{n-1-i}{2}}
b_{\frac{n-1+i-2j}{2},\lfloor\frac{n}{2}\rfloor}\lambda_j\right)t^{s_2}
\mbox{ and }
\\&&\rho_r^{L_r^4}(g)=\sum_{k=\lfloor\frac{n+2}{2}\rfloor}^{\frac{n-1+i}{2}}
\left(\sum_{j=0}^{\frac{n-1-i}{2}}b_{\frac{n-1+i-2j}{2},k}\lambda_j\right)t^{r+kq}
=\sum_{k=0}^{\lfloor\frac{i-2}{2}\rfloor}
\left(\sum_{j=0}^{\frac{n-1-i}{2}}b_{\frac{n-1+i-2j}{2},\lfloor\frac{n+2}{2}\rfloor +k}\lambda_j\right)t^{s_3+kq}.
\end{eqnarray*}
In the expression of $\rho_r^{L_r^4}(g)$, we can extend the summation in the $k$'s with 
$\lfloor (i-2)/2\rfloor+1=\lfloor i/2\rfloor$ 
and $\lfloor (i+1)/2\rfloor$. Indeed, for every $j\geq 0$, 
\begin{eqnarray*}
\left\lfloor\frac{n+2}{2}\right\rfloor+\left\lfloor\frac{i+1}{2}\right\rfloor
\geq\left\lfloor\frac{n+2}{2}\right\rfloor+\left\lfloor\frac{i}{2}\right\rfloor\geq 
\frac{n+1}{2}+\frac{i-1}{2}=\frac{n+i}{2}>\frac{n-1+i-2j}{2}
\end{eqnarray*}
Therefore, $b_{\frac{n-1+i-2j}{2},\lfloor\frac{n+2}{2}\rfloor+\lfloor\frac{i}{2}\rfloor}=0$ and
$b_{\frac{n-1+i-2j}{2},\lfloor\frac{n+2}{2}\rfloor+\lfloor\frac{i+1}{2}\rfloor}=0$, 
for every $j\geq 0$. Thus, 
\begin{eqnarray}\label{rhor4}
&&\rho_r^{L_r^4}(g)=
\sum_{k=0}^{\lfloor\frac{i+1}{2}\rfloor}
\left(\sum_{j=0}^{\frac{n-1-i}{2}}b_{\frac{n-1+i-2j}{2},\lfloor\frac{n+2}{2}\rfloor +k}\lambda_j\right)t^{s_3+kq}.
\end{eqnarray}
Let us find $h_1=\sum_{j=0}^{\lfloor\frac{i-1}{2}\rfloor}
\lambda_{1,j}\mon^{\phi_{s_1}+j\omega}\in W_{s_1}$
such that $\rho_r^{L_r^1}(g)+\rho(h_1)=0$, if possible.  
By Remark~\ref{rem-comp},
\begin{eqnarray*}
\rho(h_1)=\sum_{k=0}^{\lfloor\frac{i-1}{2}\rfloor}\left(
\sum_{j=0}^{\lfloor\frac{i-1}{2}\rfloor}b_{\lfloor\frac{i-1}{2}\rfloor-j,k}
\lambda_{1,j}\right)t^{s_1+kq}.
\end{eqnarray*}
Asking for $\rho(h_1)=-\rho_r^{L_r^1}(g)$, we get the system of linear
equations in the unknowns $\lambda_{1,j}$:
\begin{eqnarray*}
\sum_{j=0}^{\lfloor\frac{i-1}{2}\rfloor}b_{\lfloor\frac{i-1}{2}\rfloor-j,k}\lambda_{1,j}
=-\sum_{j=0}^{\frac{n-1-i}{2}}b_{\frac{n-1+i-2j}{2},k+1}\lambda_j,
\phantom{+}k=0,\ldots,\left\lfloor\frac{i-1}{2}\right\rfloor.
\end{eqnarray*}    
In matrix form:
\begin{eqnarray*}
\left(\begin{array}{ccc}
b_{\lfloor\frac{i-1}{2}\rfloor,0}&\ldots&b_{0,0}\\\vdots&&\vdots\\
b_{\lfloor\frac{i-1}{2}\rfloor,\lfloor\frac{i-1}{2}\rfloor}&\ldots&b_{0,\lfloor\frac{i-1}{2}\rfloor}
\end{array}\right)\!
\left(\begin{array}{c}\lambda_{1,0}\\\vdots
\\\lambda_{1,\lfloor\frac{i-1}{2}\rfloor}\end{array}\right)\!=-\left(\begin{array}{ccc}
b_{\frac{n-1+i}{2},1}&\ldots&b_{i,1}\\\vdots&&\vdots\\b_{\frac{n-1+i}{2},\lfloor\frac{i+1}{2}\rfloor}
&\ldots&b_{i,\lfloor\frac{i+1}{2}\rfloor}
\end{array}\right)\!
\left(\begin{array}{c}
\lambda_0\\\vdots\\\lambda_{\frac{n-1-i}{2}}
\end{array}\right).
\end{eqnarray*}
That is, $P^{[0,\lfloor\frac{i-1}{2}\rfloor]}_{[0,\lfloor\frac{i-1}{2}\rfloor]}\Lambda_1=
-P^{[1,\lfloor\frac{i+1}{2}\rfloor]}_{[i,\frac{n-1+i}{2}]}\Lambda$, 
where $\Lambda_1^\top=(\lambda_{1,0},\ldots,\lambda_{1,\lfloor\frac{i-1}{2}\rfloor})$
and $\Lambda^\top=(\lambda_0,\ldots,\lambda_{\frac{n-1-i}{2}})$, with 
$\lambda_j=(-1)^j
b_{\{0\}\sqcup
[\lfloor\frac{i+3}{2}\rfloor,\lfloor\frac{n-2}{2}\rfloor]}
^{[i,\frac{n-1+i}{2}]\setminus \{\frac{n-1+i-2j}{2}\}}$. 
By Notation~\ref{notation-binomials}, $P^J_I=(B^I_J)^{\top}\Theta$. So, we get the system:
\begin{eqnarray*}
\left(B^{[0,\lfloor\frac{i-1}{2}\rfloor]}_{[0,\lfloor\frac{i-1}{2}\rfloor]}\right)^\top
\cdot\Theta_{\lfloor\frac{i+1}{2}\rfloor}\cdot\Lambda_1=
-\left(B^{[i,\frac{n-1+i}{2}]}_{[1,\lfloor\frac{i+1}{2}\rfloor]}\right)^\top
\cdot\Theta_{\frac{n+1-i}{2}}\cdot\Lambda.
\end{eqnarray*}
Note that $B^{[0,\lfloor\frac{i-1}{2}\rfloor]}_{[0,\lfloor\frac{i-1}{2}\rfloor]}$ 
is a square $\left\lfloor\frac{i+1}{2}\right\rfloor\times 
\left\lfloor\frac{i+1}{2}\right\rfloor$
lower triangular matrix with 1's in the diagonal, 
so it is invertible. Moreover, $\Theta_{\lfloor\frac{i+1}{2}\rfloor}^{-1}=
\Theta_{\lfloor\frac{i+1}{2}\rfloor}$. Hence, this system has 
a unique solution $\Lambda_1$. Take $h_1$ in $W_{s_1}$ to be the 
polynomial with components $\Lambda_1$ in $\mathcal{B}_{s_1}$. Therefore,
\begin{eqnarray*}
\rho(g+h_1)=\rho_r^{L_r^1}(g)+\rho_r^{L_r^3}(g)+
\rho_r^{L_r^4}(g)+\rho(h_1)=\rho_r^{L_r^3}(g)+\rho_r^{L_r^4}(g).
\end{eqnarray*}

Set $L_{s_2}^1=[1,\lfloor(i+3)/2\rfloor]$, $L_{s_2}^2=[\lfloor(i+5)/2\rfloor,\lfloor n/2\rfloor]$ and
$L_{s_2}^3=[\lfloor(n+2)/2\rfloor,(n+1+i)/2]$. Observe that, if $i$ is odd (and
$n$ even), then $L_{s_2}^1=[1,(i+3)/2]$, $L_{s_2}^2=[(i+5)/2,n/2]$ and
$L_{s_2}^3=[(n+2)/2,(n+1+i)/2]$, whereas if $i$ is even (and $n$ odd),
then $L_{s_2}^1=[1,(i+2)/2]$,
$L_{s_2}^2=[(i+4)/2,(n-1)/2]$ and $L_{s_2}^3=[(n+1)/2,(n+1+i)/2]$. If
$i=n-3$, then $\lfloor(i+5)/2\rfloor>\lfloor n/2\rfloor$, 
$L_{s_2}^1=[1,\lfloor n/2\rfloor]$,
$L_{s_2}^2=\emptyset$ and $L_{s_2}^3=[\lfloor(n+2)/2\rfloor,(n+1+i)/2]$. If
$i<n-3$, then $\lfloor(i+5)/2\rfloor\leq \lfloor n/2\rfloor$ and $L_{s_2}^3\neq\emptyset$.
Thus, in any case, $H_{s_2}=\{0\}\sqcup L_{s_2}^1\sqcup
L_{s_2}^2\sqcup L_{s_2}^3$. 
Note that $\supp(\rho_r^{L_r^4}(g))\subseteq \{s_3+kq\mid k\in
[0,\lfloor (i-2)/2\rfloor]\}=\{s_2+kq\mid k\in [1,\lfloor i/2\rfloor]\}\subset
\{s_2+kq\mid k\in L_{s_2}^1\}$.

In general, we cannot expect to find $h_2\in W_{s_2}$ such that
$\rho_r^{L_r^3}(g)+\rho_r^{L_r^4}(g)+\rho(h_2)=0$. For instance, it may
happen that $\dim W_{s_2}<\card(L_r^3\cup L_r^4)=1+\lfloor i/2\rfloor$, 
which would lead to an overdetermined system. So let us search a $h_2\in W_{s_2}$ such that
$\rho_r^{L_r^3}(g)+\rho_{s_2}^{\{0\}}(h_2)=0$ and such that
$\rho_{s_2}^{L_{s_2}^2}(h_2)=0$. 

These two conditions determine
$1+\card(L_{s_2}^2)=(n-1-i)/2=\dim W_{s_2}-1=\kappa_{s_2}$ linear equations. In
order to have as many unknowns as linear conditions, we take
$h_2=\sum_{j=1}^{\frac{n-1-i}{2}}\lambda_{2,j}\mon^{\phi_{s_2}+j\omega}$ in
$W_{s_2}$, with zero component in the first monomial
$\mon^{\phi_{s_2}}$ of the basis $\mcb_{s_2}$. 
So, we ask for:
\begin{eqnarray*}
\rho_{s_2}^{\{0\}\sqcup L_{s_2}^2}(h_2)&=&-\rho_r^{L_r^3}(g).
\end{eqnarray*}
Using Remark~\ref{rem-comp}, we obtain the more explicit expression
\begin{eqnarray*}
\sum_{k\in\{0\}\sqcup L_{s_2}^{2}}\left(\sum_{j=1}^{\frac{n-1-i}{2}}b_{\frac{n+1+i-2j}{2},k}\cdot \lambda_{2,j}\right)t^{s_2+kq}&=&-\left(\sum_{j=0}^{\frac{n-1-i}{2}}b_{\frac{n-1+i-2j}{2},\lfloor\frac{n}{2}\rfloor}\lambda_j\right)t^{s_2}.
\end{eqnarray*}
In matrix form:
\begin{eqnarray*}
\left(\begin{array}{ccc}
b_{\frac{n-1+i}{2},0}&\cdots&b_{i+1,0}\\
b_{\frac{n-1+i}{2},\lfloor\frac{i+5}{2}\rfloor}&\cdots&
b_{i+1,\lfloor\frac{i+5}{2}\rfloor}\\
\vdots&&\vdots\\
b_{\frac{n-1+i}{2},\lfloor\frac{n}{2}\rfloor}&\cdots&
b_{i+1,\lfloor\frac{n}{2}\rfloor}
\end{array}\right)
\!\!\cdot\!\!
\left(\begin{array}{c}\lambda_{2,1}\\\vdots\\\lambda_{2,\frac{n-1-i}{2}}
\end{array}\right)
\!=\!
-\left(\begin{array}{ccc}
b_{\frac{n-1+i}{2},\lfloor\frac{n}{2}\rfloor}&\cdots&
b_{i,\lfloor\frac{n}{2}\rfloor}\vspace*{0,1cm}\\
0&\cdots&0\\\vdots&&\vdots\\0&\cdots&0
\end{array}\right)
\!\!\cdot\!\!
\left(\begin{array}{c}\lambda_{0}\\\vdots\\\lambda_{\frac{n-1-i}{2}}
\end{array}\right).
\end{eqnarray*}
Equivalently, 
\begin{eqnarray*}
P^{\{0\}\sqcup L_{s_2}^2}_{[i+1,\frac{n-1+i}{2}]}\cdot \Lambda_2=
-\left(\begin{array}{c}
P^{\{\lfloor\frac{n}{2}\rfloor\}}_{[i,\frac{n-1+i}{2}]}\\\hline 0
\end{array}\right)
\cdot \Lambda,
\end{eqnarray*}
where 
$\Lambda_2^\top=\left(\lambda_{2,1},\ldots,\lambda_{2,\frac{n-i-1}{2}}\right)$
and 
$\Lambda^{\top}=\left(\lambda_0,\ldots,\lambda_{\frac{n-1-i}{2}}\right)$.
By Notation~\ref{notation-binomials}, this says:
\begin{eqnarray*}
\left(B^{[i+1,\frac{n-1+i}{2}]}_{\{0\}\sqcup L_{s_2}^2}\right)^\top\cdot\Theta_{\frac{n-1-i}{2}}\cdot\Lambda_2=-
\left(B^{[i,\frac{n-1+i}{2}]}
_{\left\{\left\lfloor\frac{n}{2}\right\rfloor\right\}}
\middle| \; 0\; \right)^{\top}\cdot \Theta_{\frac{n+1-i}{2}}\cdot 
\Lambda.
\end{eqnarray*}
Observe that $B^{[i+1,\frac{n-1+i}{2}]}_{\{0\}\sqcup L_{s_2}^2}$ 
is a square $\frac{n-1-i}{2}\times\frac{n-1-i}{2}$ matrix. 
Since $\lfloor n/2\rfloor\leq (n-1+i)/2$, then
${\{0\}\sqcup L_{s_2}^2}\leq [i+1,\frac{n-1+i}{2}]$. 
By \cite[Corollary~2]{gv} or \cite[Corollary~2.5]{gp2},
$B^{[i+1,\frac{n-1+i}{2}]}_{\{0\}\sqcup L_{s_2}^2}$ is an invertible matrix. 
Thus, there is a unique solution $\Lambda_2$ of the above linear system. 
Take $h_2$ in $W_{s_2}$, with components $\Lambda_2$ in the linearly
independent set $\{\mon^{\phi_{s_2}+j\omega}\mid 1\leq j\leq(n-1-i)/2\}\subset \mathcal{B}_{s_2}$.

Since $\rho(g+h_1)=\rho_r^{L_r^3}(g)+\rho_r^{L_r^4}(g)$, 
$\rho(h_2)=\rho_{s_2}^{\{0\}}(h_2)+\rho_{s_2}^{L_{s_2}^1}(h_2)+
\rho_{s_2}^{L_{s_2}^2}(h_2)+\rho_{s_2}^{L_{s_2}^3}(h_2)$ and 
$\rho_r^{L_r^3}(g)+
\rho_{s_2}^{\{0\}}(h_2)+\rho_{s_2}^{L_{s_2}^2}(h_2)=0$, 
it follows that
\begin{eqnarray*}
\rho(g+h_1+h_2)=
\rho_r^{L_r^4}(g)+\rho_{s_2}^{L_{s_2}^1}(h_2)
+\rho_{s_2}^{L_{s_2}^3}(h_2). 
\end{eqnarray*}
By Remark~\ref{rem-comp}, and using $s_3=s_2+q$,  
\begin{eqnarray}\label{rhos2L1}
\rho_{s_2}^{L_{s_2}^1}(h_2)=
\sum_{k=1}^{\lfloor\frac{i+3}{2}\rfloor}
\left(\sum_{j=1}^{\frac{n-1-i}{2}}b_{\frac{n+1+i-2j}{2},k}\lambda_{2,j}\right)
t^{s_2+kq}=
\sum_{k=0}^{\lfloor\frac{i+1}{2}\rfloor}
\left(\sum_{j=1}^{\frac{n-1-i}{2}}b_{\frac{n+1+i-2j}{2},k+1}\lambda_{2,j}\right)
t^{s_3+kq}.
\end{eqnarray}
Similarly and since $b_{\frac{n+1+i-2j}{2},\frac{n+1+i}{2}}=0$, for all $j\geq 1$, then
\begin{eqnarray}\label{rhos2L3}
\rho_{s_2}^{L_{s_2}^3}(h_2)=
\sum_{k=\lfloor\frac{n+2}{2}\rfloor}^{\frac{n+1+i}{2}}
\left(\sum_{j=1}^{\frac{n-1-i}{2}}b_{\frac{n+1+i-2j}{2},k}\lambda_{2,j}\right)t^{s_2+kq}=
\sum_{k=\lfloor\frac{n+2}{2}\rfloor}^{\frac{n-1+i}{2}}
\left(\sum_{j=1}^{\frac{n-1-i}{2}}b_{\frac{n+1+i-2j}{2},k}\lambda_{2,j}\right)t^{s_2+kq}.  
\end{eqnarray}
Using the equalities \eqref{rhor4} and \eqref{rhos2L1}, we get 
\begin{eqnarray*}
\rho_r^{L_r^4}(g)+\rho_{s_2}^{L_{s_2}^1}(h_2)=
\sum_{k=0}^{\lfloor\frac{i+1}{2}\rfloor}
\left(\sum_{j=0}^{\frac{n-1-i}{2}}b_{\frac{n-1+i-2j}{2},\lfloor\frac{n+2}{2}\rfloor +k}\lambda_j+\sum_{j=1}^{\frac{n-1-i}{2}}b_{\frac{n+1+i-2j}{2},k+1}\lambda_{2,j}\right)t^{s_3+kq}.
\end{eqnarray*}
The components of 
$\rho_r^{L_r^4}(g)+\rho_{s_2}^{L_{s_2}^1}(h_2)$ in the set of monomials
$t^{s_3},\ldots,t^{s_3+\lfloor\frac{i+1}{2}\rfloor q}$ are
\begin{multline*}
\left(\begin{array}{ccc}
b_{\frac{n-1+i}{2},\lfloor\frac{n+2}{2}\rfloor}&\cdots&b_{i,\lfloor\frac{n+2}{2}\rfloor}
\\\vdots&&\vdots\\
b_{\frac{n-1+i}{2},
\lfloor\frac{n+2}{2}\rfloor+\lfloor\frac{i+1}{2}\rfloor}&\cdots
&b_{i,\lfloor\frac{n+2}{2}\rfloor+\lfloor\frac{i+1}{2}\rfloor}
\end{array}\right)\cdot
\left(\begin{array}{c}
\lambda_0\\\vdots\\\lambda_{\frac{n-1-i}{2}}
\end{array}\right)
+\\\left(\begin{array}{ccc}
b_{\frac{n-1+i}{2},1}&\cdots&b_{i+1,1}\\\vdots&&\vdots\\
b_{\frac{n-1+i}{2},\lfloor\frac{i+3}{2}\rfloor}&\cdots&b_{i+1,\lfloor\frac{i+3}{2}\rfloor}
\end{array}\right)\cdot
\left(\begin{array}{c}
\lambda_{2,1}\\\vdots\\\lambda_{2,\frac{n-1-i}{2}}
\end{array}\right).
\end{multline*}
That is, 
\begin{eqnarray*}
P^{\left[\lfloor\frac{n+2}{2}\rfloor,
\lfloor\frac{n+2}{2}\rfloor+ \lfloor\frac{i+1}{2}\rfloor\right]}_
{\left[i,\frac{n-1+i}{2}\right]}\cdot\Lambda+
P^{\left[1,\lfloor\frac{i+3}{2}\rfloor\right]
}_{\left[i+1,\frac{n-1+i}{2}\right]}\cdot\Lambda_2.
\end{eqnarray*}
By Notation~\ref{notation-binomials}, this is equivalent to:
\begin{eqnarray*}
\left(B^{\left[i,\frac{n-1+i}{2}\right]}
_{\left[\lfloor\frac{n+2}{2}\rfloor,
\lfloor\frac{n+2}{2}\rfloor+\lfloor\frac{i+1}{2}\rfloor\right]}\right)^{\top}
\cdot\Theta_{\frac{n+1-i}{2}}\cdot\Lambda+
\left(\begin{array}{c}
B^{\left[i+1,\frac{n-1+i}{2}\right]}_
{\left[1,\lfloor\frac{i+3}{2}\rfloor\right]}
\end{array}\right)^{\top}
\cdot\Theta_{\frac{n-1-i}{2}}\cdot\Lambda_2.
\end{eqnarray*}
Let us search $h_3=\sum_{j=0}^{\lfloor\frac{i+1}{2}\rfloor}\lambda_{3,j}
\mon^{\phi_{s_3}+j\omega}\in W_{s_3}$ such that
\begin{eqnarray*}
\rho_r^{L_r^4}(g)+\rho_{s_2}^{L_{s_2}^1}(h_2)+\rho(h_3)=0. 
\end{eqnarray*}
Since $I_{s_3}=H_{s_3}=[0,\frac{i+1}{2}]$ and by Remark~\ref{rem-comp},
this induces the linear system with unknowns 
$\Lambda_{3}^\top=\left(\lambda_{3,0},\ldots,
\lambda_{3,\lfloor\frac{i+1}{2}\rfloor}\right)$: 
\begin{multline*}
\left(B^{[0,\lfloor\frac{i+1}{2}\rfloor]}_{[0,\lfloor\frac{i+1}{2}\rfloor]}\right)^\top
\cdot\Theta_{\lfloor\frac{i+3}{2}\rfloor}\cdot\Lambda_3=\\-
\left(B^{\left[i,\frac{n-1+i}{2}\right]}
_{\left[\lfloor\frac{n+2}{2}\rfloor,
\lfloor\frac{n+2}{2}\rfloor+\lfloor\frac{i+1}{2}\rfloor\right]}\right)^{\top}
\cdot\Theta_{\frac{n+1-i}{2}}\cdot\Lambda-
\left(\begin{array}{c}
B^{\left[i+1,\frac{n-1+i}{2}\right]}_
{\left[1,\lfloor\frac{i+3}{2}\rfloor\right]}
\end{array}\right)^{\top}
\cdot\Theta_{\frac{n-1-i}{2}}\cdot\Lambda_2.
\end{multline*}
As before, $B^{[0,\lfloor\frac{i+1}{2}\rfloor]}_{[0,\lfloor\frac{i+1}{2}\rfloor]}$
is a square 
$\left\lfloor\frac{i+3}{2}\right\rfloor\times
\left\lfloor\frac{i+3}{2}\right\rfloor$
lower triangular matrix with 1's in the diagonal, 
so invertible. Let $\Lambda_3$ be the unique solution of this linear system and 
$h_3$ the polynomial in $W_{s_3}$ with components $\Lambda_3$ in $\mathcal{B}_{s_3}$. Then,
\begin{eqnarray*}
\rho(g+h_1+h_2+h_3)=
\rho_r^{L_r^4}(g)+\rho_{s_2}^{L^1_{s_2}}(h_2)+
\rho_{s_2}^{L^3_{s_2}}(h_2)+\rho(h_3)=
\rho_{s_2}^{L_{s_2}^3}(h_2).
\end{eqnarray*} 
Since $\frac{n-1+i}{2}-\lfloor\frac{n+2}{2}\rfloor=
\lfloor\frac{i-2}{2}\rfloor$ and $s_2+\lfloor\frac{n+2}{2}\rfloor q=s_4$ and using \eqref{rhos2L3}, 
one deduces that
\begin{eqnarray*}
\rho_{s_2}^{L_{s_2}^3}(h_2)=\sum_{k=\lfloor\frac{n+2}{2}\rfloor}^{\frac{n-1+i}{2}}
  \!\!\left(\sum_{j=1}^{\frac{n-1-i}{2}}b_{\frac{n+1+i-2j}{2},k}
  \lambda_{2,j}\right)\!t^{s_2+kq}=
\sum_{k=0}^{\lfloor\frac{i-2}{2}\rfloor}
\left(\sum_{j=1}^{\frac{n-1-i}{2}}
b_{\frac{n+1+i-2j}{2},\lfloor\frac{n+2}{2}\rfloor+k}\lambda_{2,j}\right)t^{s_4+kq}.
\end{eqnarray*}
Let us search a $h_4\in W_{s_4}$ such that $\rho_{s_2}^{L_{s_2}^3}(h_2)+\rho(h_4)=0$. 
Choose $h_4$ as a linear combination of the last $\lfloor i/2\rfloor$ monomials
of the basis $\mcb_{s_4}$: $h_{4}=\sum_{j=\lfloor\frac{i+3}{2}\rfloor-\lfloor\frac{i-2}{2}\rfloor}^{\lfloor\frac{i+3}{2}\rfloor}
\lambda_{4,j}\mon^{\phi_{s_4}+j\omega}$.  Let
$L_{s_4}^1=[0,\lfloor(i-2)/2\rfloor]$ and $L_{s_4}^2=[\lfloor i/2\rfloor,\lfloor(i+3)/2\rfloor]$, so that
$H_{s_4}=L_{s_4}^1\sqcup L_{s_4}^2$. Using Remark~\ref{rem-comp}, 
\begin{eqnarray*}
&&\rho_{s_4}^{L^1_{s_4}}(h_4)=
\sum_{k=0}^{\lfloor\frac{i-2}{2}\rfloor}\left(
\sum_{j=\lfloor\frac{i+3}{2}\rfloor-
\lfloor\frac{i-2}{2}\rfloor}^{\lfloor\frac{i+3}{2}\rfloor}
b_{\lfloor\frac{i+3}{2}\rfloor-j,k}\lambda_{4,j}\right)t^{s_4+kq}\mbox{ and }\\&&
\rho_{s_4}^{L^2_{s_4}}(h_4)=\sum_{k=\lfloor \frac{i}{2}\rfloor}^{\lfloor\frac{i+3}{2}\rfloor}\left(
\sum_{j=\lfloor\frac{i+3}{2}\rfloor-
\lfloor\frac{i-2}{2}\rfloor}^{\lfloor\frac{i+3}{2}\rfloor}
b_{\lfloor\frac{i+3}{2}\rfloor-j,k}\lambda_{4,j}\right)t^{s_4+kq}.
\end{eqnarray*}
Since $b_{\lfloor (i+3)/2\rfloor-j,k}=0$, for all 
$j\geq \left\lfloor (i+3)/2\right\rfloor-\left\lfloor(i-2)/2\right\rfloor$
and all $k\geq\left\lfloor i/2\right\rfloor$, it follows that 
$\rho_{s_4}^{L_{s_4}^2}(h_4)=0$, so $\rho(h_4)=\rho_{s_4}^{L_{s_4}^1}(h_4)$. Thus,
$\rho_{s_2}^{L_{s_2}^3}(h_2)+\rho(h_4)=0$ is equivalent to
$\rho_{s_4}^{L_{s_4}^1}(h_4)=-\rho_{s_2}^{L_{s_2}^3}(h_2)$. That is,
\begin{eqnarray*}
\sum_{k=0}^{\lfloor\frac{i-2}{2}\rfloor}\left(
\sum_{j=\lfloor\frac{i+3}{2}\rfloor-
\lfloor\frac{i-2}{2}\rfloor}^{\lfloor\frac{i+3}{2}\rfloor}
b_{\lfloor\frac{i+3}{2}\rfloor-j,k}\lambda_{4,j}\right)t^{s_4+kq}=-
\sum_{k=0}^{\lfloor\frac{i-2}{2}\rfloor}
\left(\sum_{j=1}^{\frac{n-1-i}{2}}
b_{\frac{n+1+i-2j}{2},\lfloor\frac{n+2}{2}\rfloor+k}\lambda_{2,j}\right)t^{s_4+kq}.
\end{eqnarray*}
Since $\lfloor\frac{n+2}{2}\rfloor+\lfloor\frac{i-2}{2}\rfloor=\frac{n-1+i}{2}$,
and letting $\Lambda_4^\top=\left(\lambda_{4,\lfloor\frac{i+3}{2}\rfloor-\lfloor\frac{i-2}{2}\rfloor},
\ldots,\lambda_{4,\lfloor\frac{i+3}{2}\rfloor}\right)$, we get 
\begin{eqnarray*}
P^{[0,\lfloor\frac{i-2}{2}\rfloor]}_{[0,\lfloor\frac{i-2}{2}\rfloor]}\cdot\Lambda_4=-
P^{[\lfloor\frac{n+2}{2}\rfloor,\frac{n-1+i}{2}]}_{[i+1,\frac{n-1+i}{2}]}\cdot\Lambda_2.
\end{eqnarray*}
By Notation~\ref{notation-binomials}, 
\begin{eqnarray*}
\left(B^{\left[0,\lfloor\frac{i-2}{2}\rfloor\right]}_{\left[0,\lfloor\frac{i-2}{2}\rfloor\right]}\right)^\top\cdot\Theta_{\lfloor\frac{i}{2}\rfloor}\cdot
\Lambda_4=
-\left(B_{[\lfloor\frac{n+2}{2}\rfloor,\frac{n-1+i}{2}]}^{[i+1,\frac{n-1+i}{2}]}\right)^{\top}
\Theta_{\frac{n-1-i}{2}}\cdot\Lambda_2.
\end{eqnarray*}
Since $B^{\left[0,\lfloor\frac{i-2}{2}\rfloor\right]}
_{\left[0,\lfloor\frac{i-2}{2}\rfloor\right]}$ is a square
$\left\lfloor\frac{i}{2}\right\rfloor\times 
\left\lfloor\frac{i}{2}\right\rfloor$
lower triangular matrix with 1's in the diagonal, 
this system has a unique solution $\Lambda_4$. Let
$h_4\in W_{s_4}$ be the polynomial with components $\Lambda_4$ in the linearly
independent set $\{\mon^{\phi_{s_4}+j\omega}\mid 
\lfloor(i+3)/2\rfloor-\lfloor(i-2)/2\rfloor\leq j\leq \lfloor(i+3)/2\rfloor\}\subset\mcb_{s_4}$. 
Then,
\begin{eqnarray*}
\rho(f)=\rho(g+h_1+h_2+h_3+h_4)=\rho_{s_2}^{L_{s_2}^3}(g)+\rho(h_4)=
\rho_{s_2}^{L_{s_2}^3}(g)+\rho_{s_4}^{L_{s_4}^1}(h_4)=0,
\end{eqnarray*} as desired. 

Suppose that $n$ is odd and that $1\leq k\leq n-4$, with $k$ odd. 
By Lemma~\ref{lemma-c-i}, $g$ is equal to the polynomial $g_k$ defined in 
\eqref{equality-gknoddkodd}.
By Proposition~\ref{proposition-class}, $i=n-k-2$. Thus,
\begin{eqnarray*}
&&\phi_{s_1}=\left(\frac{n-k-4}{2},1,\frac{n+k}{2}\right),
\quad\phi_{s_2}=\left(\frac{2n-k-1}{2}, 0,\frac{k+1}{2}\right),\\ 
&&\phi_{s_3}=\left(\frac{n-k-2}{2},1,\frac{n+k}{2}\right),
\quad\phi_{s_4}=\left(\frac{n-k}{2},1,\frac{n+k}{2}\right).
\end{eqnarray*}
It follows that
\begin{eqnarray*}
&&h_1=\sum_{j=0}^{\frac{n-k-4}{2}}\lambda_{1,j}x^{\frac{n-k-4-2j}{2}}y^{1+2j}z^{\frac{n+k-2j}{2}}
\phantom{+},\phantom{+}
h_2=\sum_{j=1}^{\frac{k+1}{2}}\lambda_{2,j}x^{\frac{2n-k-1-2j}{2}}y^{2j}z^{\frac{k+1-2j}{2}},\\
&&h_3=\sum_{j=0}^{\frac{n-k-2}{2}}\lambda_{3,j}x^{\frac{n-k-2-2j}{2}}y^{1+2j}z^{\frac{n+k-2j}{2}}
\phantom{+}\mbox{ and }\phantom{+}
h_4=\sum_{j=2}^{\frac{n-k}{2}}
\lambda_{4,j}x^{\frac{n-k-2j}{2}}y^{1+2j}z^{\frac{n+k+2j}{2}},
\end{eqnarray*}
which leads to the expression of $f_k$ in \eqref{poly-fknoddkodd}.

Suppose that $n$ is even and that $1\leq k\leq n-5$, with $k$ odd. 
By Lemma~\ref{lemma-c-i}, $g$ is equal to the polynomial $g_k$ defined in 
\eqref{equality-gknevenkodd}.
By Proposition~\ref{proposition-class}, $i=n-k-2$. Thus,
\begin{eqnarray*}
&&\phi_{s_1}=\left(\frac{n-k-3}{2},1,\frac{n+k-1}{2}\right),
\quad\phi_{s_2}=\left(\frac{2n-k-1}{2}, 0,\frac{k+1}{2}\right),\\ 
&&\phi_{s_3}=\left(\frac{n-k-1}{2},1,\frac{n+k-1}{2}\right),
\quad\phi_{s_4}=\left(\frac{n-k+1}{2},1,\frac{n+k-1}{2}\right).
\end{eqnarray*}
It follows that
\begin{eqnarray*}
&&h_1=\sum_{j=0}^{\frac{n-k-3}{2}}
\lambda_{1,j}x^{\frac{n-k-3-2j}{2}}y^{1+2j}z^{\frac{n+k-1-2j}{2}}
\phantom{+},\phantom{+}
h_2=\sum_{j=1}^{\frac{k+1}{2}}\lambda_{2,j}x^{\frac{2n-k-1-2j}{2}}y^{2j}z^{\frac{k+1-2j}{2}},\\
&&h_3=\sum_{j=0}^{\frac{n-k-1}{2}}\lambda_{3,j}x^{\frac{n-k-1-2j}{2}}y^{1+2j}z^{\frac{n+k-1-2j}{2}}
\phantom{+}\mbox{ and }\phantom{+}
h_4=\sum_{j=3}^{\frac{n-k+1}{2}}\lambda_{4,j}x^{\frac{n-k+1-2j}{2}}y^{1+2j}z^{\frac{n+k-1-2j}{2}},
\end{eqnarray*}
which leads to the expression of $f_k$ in \eqref{poly-fknevenkodd}.
\end{proof}

\vspace*{0.3cm}

\begin{lemma}\label{lemma-h-i+1}
Let $n\geq 6$. Let $r\in [s_0,s_0+n)$. Write $r=r_k=s_0+k-1$, with $1\leq k\leq n$. 
Suppose that $c_r={\cteal -\iota_r+1}$, $2\leq \iota_r\leq n-3$. Let 
\begin{eqnarray*}
s_1=r+q,\quad s_2=r+\left\lfloor \frac{n}{2}\right\rfloor q,\quad s_3=s_2+q\quad\mbox{and}\quad s_4=s_3+\left\lfloor \frac{n}{2}\right\rfloor q.
\end{eqnarray*}
Then $s_1,s_2,s_3,s_4\in\nums$ and $r<s_1<s_2<s_3<s_4<\la$. 
Moreover, $\phi_{r}=\left(\frac{n-3+\iota_r}{2},1,\frac{n-1-\iota_r}{2}\right)$,
\begin{eqnarray*}
&&\phi_{s_1}=\left(\left\lfloor\frac{\iota_r-1}{2}\right\rfloor,0,
n-\left\lfloor\frac{\iota_r+1}{2}\right\rfloor\right),
\quad\phi_{s_2}=\left(\frac{n-1+\iota_r}{2}, 1,\frac{n-1-\iota_r}{2}\right),\\ 
&&\phi_{s_3}=\left(\left\lfloor\frac{\iota_r+1}{2}\right\rfloor,0,
n-\left\lfloor\frac{\iota_r+1}{2}\right\rfloor\right),
\quad\phi_{s_4}=\left(\left\lfloor\frac{\iota_r+3}{2}\right\rfloor,0,
n-\left\lfloor\frac{\iota_r+1}{2}\right\rfloor\right).
\end{eqnarray*}
Let $g=\sum_{j=0}^{\frac{n-1-\iota_r}{2}}\lambda_j\cdot\mon^{\phi_{r}+j\omega}$, with
$\lambda_j=(-1)^jb_{\{0\}\sqcup[\lfloor\frac{\iota_r+3}{2}\rfloor,
\lfloor\frac{n-2}{2}\rfloor]}^{[\iota_r-1,\frac{n-3+\iota_r}{2}]\setminus 
\{\frac{n-3+\iota_r-2j}{2}\}}$. Set $\Lambda^\top=(\lambda_0,\ldots,\lambda_{\frac{n-1-\iota_r}{2}})$.\\ 
Let $h_1=
\sum_{j=0}^{\lfloor\frac{\iota_r-1}{2}\rfloor}\lambda_{1,j}\mon^{\phi_{s_1}+j\omega}$, 
where $\Lambda_1^\top=\left(\lambda_{1,0},\ldots,
\lambda_{1,\lfloor\frac{\iota_r-1}{2}\rfloor}\right)$ satisfies
\begin{eqnarray*}
&&\left(B^{[0,\lfloor\frac{\iota_r-1}{2}\rfloor]}
_{[0,\lfloor\frac{\iota_r-1}{2}\rfloor]}\right)^\top\cdot\Theta_{\lfloor\frac{\iota_r+1}{2}\rfloor}\cdot\Lambda_1=
-\left(B^{[\iota_r-1,\frac{n-3+\iota_r}{2}]}_{[1,\lfloor\frac{\iota_r+1}{2}\rfloor]}\right)^\top
\cdot\Theta_{\frac{n+1-\iota_r}{2}}\cdot\Lambda.
\end{eqnarray*}
Let
$h_2=\sum_{j=1}^{\frac{n-1-\iota_r}{2}}\lambda_{2,j}\mon^{\phi_{s_2}+j\omega}$,
where $\Lambda_2^\top=\left(\lambda_{2,1},\ldots,
\lambda_{2,\frac{n-1-\iota_r}{2}}\right)$ satisfies
\begin{eqnarray*}
&&\left(B^{[\iota_r,\frac{n-3+\iota_r}{2}]}_{\{0\}\sqcup 
[\lfloor\frac{\iota_r+5}{2}\rfloor,\lfloor\frac{n}{2}\rfloor]}\right)^\top\cdot
\Theta_{\frac{n-1-\iota_r}{2}}\cdot\Lambda_2=-
\left(B^{[\iota_r-1,\frac{n-3+\iota_r}{2}]}
_{\left\{\left\lfloor\frac{n}{2}\right\rfloor\right\}}
\middle| \; 0\; \right)^{\top}\cdot \Theta_{\frac{n+1-\iota_r}{2}}\cdot 
\Lambda.
\end{eqnarray*}
Let
$h_3=\sum_{j=0}^{\lfloor\frac{\iota_r+1}{2}\rfloor}\lambda_{3,j}\mon^{\phi_{s_3}+j\omega}$,
where $\Lambda_{3}^\top=\left(\lambda_{3,0},\ldots,
\lambda_{3,\lfloor\frac{\iota_r+1}{2}\rfloor}\right)$ satisfies
\begin{multline*}
\left(B^{[0,\lfloor\frac{\iota_r+1}{2}\rfloor]}
_{[0,\lfloor\frac{\iota_r+1}{2}\rfloor]}\right)^\top\cdot
\Theta_{\lfloor\frac{\iota_r+3}{2}\rfloor}\cdot\Lambda_3=\\
-\left(B^{[\iota_r-1,\frac{n-3+\iota_r}{2}]}_{[\lfloor\frac{n+2}{2}\rfloor,\lfloor\frac{n+2}{2}\rfloor+\lfloor\frac{\iota_r+1}{2}\rfloor]}\right)^\top\cdot\Theta_{\frac{n+1-\iota_r}{2}}\cdot\Lambda -\left(B^{[\iota_r,\frac{n-3+\iota_r}{2}]}_{[1,\lfloor\frac{\iota_r+3}{2}\rfloor]}\right)^\top\cdot\Theta_{\frac{n-1-\iota_r}{2}}\cdot\Lambda_2 .
\end{multline*} 
If $2\leq\iota_r\leq 3$, let $h_4=0$. If $\iota_r\geq 4$, let 
$h_4$ be defined as $h_4=\sum_{j=\lfloor\frac{\iota_r+3}{2}\rfloor-\lfloor\frac{\iota_r-4}{2}\rfloor}
^{\lfloor\frac{\iota_r+3}{2}\rfloor}\lambda_{4,j}\mon^{\phi_{s_4}+j\omega}$, 
where $\Lambda_4^\top=
\left(\lambda_{4,\lfloor\frac{\iota_r+3}{2}\rfloor-
\lfloor\frac{\iota_r-4}{2}\rfloor},\ldots,
\lambda_{4,\lfloor\frac{\iota_r+3}{2}\rfloor}\right)$ satisfies
\begin{eqnarray*}
&&\left(B_{[0,\lfloor\frac{\iota_r-4}{2}\rfloor]}
^{[0,\lfloor\frac{\iota_r-4}{2}\rfloor]}\right)^\top\cdot
\Theta_{\lfloor\frac{\iota_r-2}{2}\rfloor}\cdot\Lambda_4=-\left(B^{[\iota_r,\frac{n-3+\iota_r}{2}]}_{[\lfloor\frac{n+2}{2}\rfloor,\frac{n-3+\iota_r}{2}\rfloor]}\right)^\top\cdot\Theta_{\frac{n-1-\iota_r}{2}}\cdot\Lambda_2.
\end{eqnarray*}
Set $h=h_1+h_2+h_3+h_4$ and $f=g+h$. 
Then $g$ is a basis of $\ker(\rho_r^{{\rm G}_r})$,
$h_i\in W_{s_i}$, $i=1,2,3,4$,
\begin{eqnarray*}
f^\sigma=g,\phantom{+} f^\tau=h,\phantom{+}
f\in\ker(\rho)\phantom{+} \text{and}\phantom{+} g\in V_r.
\end{eqnarray*}
Hence, $V_r=\ker(\rho_r^{{\rm G}_r})$.

Suppose that $n$ is odd and $2\leq k\leq n-3$, with $k$ even. Then $g$ is equal 
to the polynomial $g_k$ defined in \eqref{equality-gknoddkeven} 
and $f=g+h$ is equal to the polynomial $f_k$ defined as in \eqref{poly-fknoddkeven}.

Suppose that $n$ is even and $2\leq k\leq n-4$, with $k$ even. 
Then $g$ is equal to the polynomial $g_k$ defined in \eqref{equality-gknevenkeven} 
and $f=g+h$ is equal to the polynomial $f_k$ defined as in \eqref{poly-fknevenkeven}.
In particular, $f_k\in\ker(\rho)$ and $f_k^{\sigma}=g_k$ is a basis of $V_{r_k}$.
\end{lemma}
\begin{proof} 
Observe that $r=(a+1)\ell_r+c_r=(a+1)(n-1)-i+1$. Clearly, $r<s_1<s_2<s_3<s_4$. 
Let us prove that $s_4<\la$. 
If $n=6$, then $a=15$, $q=n-1=5$, $i$ is odd and, since $2\leq i\leq n-3$, it follows that $i=3$ and $r=78$. Thus, $s_4=r+(2\lfloor n/2\rfloor+1)q=113$ and $\la=(\lfloor a/2\rfloor+2)a=135$. Suppose that $n\geq 7$. 
By Lemma~\ref{lemma-cota}, $(1)$, $3\leq \lfloor a/2\rfloor-n$. Thus,
$2\lfloor n/2\rfloor+1\leq 3\lfloor n/2\rfloor\leq (\lfloor a/2\rfloor-n)\lfloor n/2\rfloor$. By Lemma~\ref{lemma-cota}, $(2)$, $s_4<\frob(\nums)<\la$. In particular, 
$s_1,s_2,s_3<\la$.

By Lemma~\ref{lemma-c-i+1}, $s_1,s_2\in\nums$, $\phi_{s_1}=(\lfloor\frac{i-1}{2}\rfloor,0,n-\lfloor\frac{i+1}{2}\rfloor)$ and $\phi_{s_2}=(\frac{n-1+i}{2},1,\frac{n-1-i}{2})$. In particular, since $i\leq n-3$, then $\kappa_{s_1}=
\min(\phi_{s_1,1},\phi_{s_1,3})=\lfloor (i-1)/2\rfloor$, 
$I_{s_1}=[\phi_{s_1,1}-\kappa_{s_1},\phi_{s_1,1}]=[0,\lfloor(i-1)/2\rfloor]$, $H_{s_1}=[0,\phi_{s_1,1}]=[0,\lfloor(i-1)/2\rfloor]$, 
$W_{s_1}=\langle \mon^{\phi_{s_1+j\omega}}\mid 0\leq j\leq\lfloor(i-1)/2\rfloor\rangle$,
$h_1$ is well-defined and
$h_1\in W_{s_1}$. Similarly, $\kappa_{s_2}=(n-1-i)/2$, $I_{s_2}=[i,(n-1+i)/2]$, 
$H_{s_2}=[0,(n-1+i)/2]$,
$W_{s_2}=\langle\mon^{\phi_{s_2}+j\omega}\mid0\leq j\leq (n-1-i)/2\rangle$,
$h_2$ is well-defined and $h_2\in W_{s_2}$.

Let us prove that $s_3\in\nums$. 
By Remark~\ref{aboutq}, $\lfloor n/2\rfloor q=a$. Thus, 
$s_3=s_2+q=r+(\lfloor n/2\rfloor+1)q=an+(n-i+q)$. 
Since $i\geq 2$ and $q\leq n$, then $n-i+q\leq 2n-2<2n$. 
Since $i\leq n-3$ and $q\geq n-1$, then $n-i+q\geq n+2>0$. Therefore, 
$0\leq n-i+\leq 2n$. By Proposition~\ref{theo31ggp}, $s_3\in\nums$,
and since $s_3<\la$, then we deduce $\ell_{s_3}=n$, 
$\rem_{s_3}=n-i+q$ and
$\phi_{s_3}=(\lfloor (i+1)/2\rfloor,0,n-\lfloor (i+1)/2\rfloor)$. 
As before, since $i\leq n-3$, then 
$\kappa_{s_3}=\min(\phi_{s_3,1},\phi_{s_3,3})=
\lfloor (i+1)/2\rfloor$, $I_{s_3}=[0,\lfloor(i+1)/2\rfloor]$, $H_{s_3}=[0,\lfloor(i+1)/2\rfloor]$,
$W_{s_3}=\langle\mon^{\phi_{s_3}+j\omega}\mid 0\leq j\leq
\lfloor(i+1)/2\rfloor\rangle$, $h_3$ is well-defined and $h_3\in W_{s_3}$.

Let us prove that $s_4\in\nums$. Since $s_3=an+(n-i+q)$ and $\lfloor n/2\rfloor q=a$, 
it follows that $s_4=s_3+\lfloor\frac{n}{2}\rfloor q=a(n+1)+(n-i+q)$. 
Since $q\leq n$ and $i\geq 2$, then $n-i+q\leq 2n-2<2(n+1)$. 
Since $q\geq n-1$ and $i\leq n-3$, then $n-i+q\geq n+2>0$. Therefore, 
$0\leq n-i+q\leq 2(n+1)$. By Proposition~\ref{theo31ggp}, $s_4\in\nums$, 
and  since $s_4<\la$, 
we get that $\ell_{s_4}=n+1$, $\rem_{s_4}=n-i+q$ and 
$\phi_{s_4}=(\lfloor(i+3)/2\rfloor,0,n-\lfloor(i+1)/2\rfloor)$. 
Again, since $i\leq n-3$, then $\kappa_{s_4}=\min(\phi_{s_4,1},\phi_{s_4,3})=
\lfloor (i+3)/2\rfloor$, $I_{s_4}=[0,\lfloor(i+3)/2\rfloor]$, 
$H_{s_4}=[0,\lfloor(i+3)/2\rfloor]$, $W_{s_4}=\langle\mon^{\phi_{s_4}+j\omega}\mid0\leq j\leq\lfloor(i+3)/2\rfloor\rangle$,
$h_4$ is well-defined and $h_4\in W_{s_4}$. 

By Lemma~\ref{lemma-c-i+1}, $\phi_{r}=\left(\frac{n-3+i}{2},1,\frac{n-1-i}{2}\right)$
and $g$ is a basis of $\ker(\rho_r^{{\rm G}_r})$. 
Since $r<s_1<s_2<s_3<s_4$, then $f^\sigma=g$ and $f^\tau=h$. 
If we prove that $\rho(f)=0$, then we will deduce that $g\in V_r$ and 
$V_r=\ker(\rho_r^{{\rm G}_r})$.

By Lemma~\ref{lemma-c-i+1}, $I_r=[i-1,(n-3+i)/2]$, 
$H_r=[0,\frac{n-3+i}{2}]$ and
$\good_r=\{0\}
\cup[\lfloor\frac{i+3}{2}\rfloor,\lfloor\frac{n-2}{2}\rfloor]$.

Set $L_r^1=[1,1+\lfloor\frac{i-1}{2}\rfloor]$, $L_r^2=
[\lfloor\frac{i+3}{2}\rfloor,\lfloor\frac{n-2}{2}\rfloor]$,
$L_r^3=\{\lfloor\frac{n}{2}\rfloor\}$ and $L_r^4=
[\lfloor\frac{n+2}{2}\rfloor,\frac{n-3+i}{2}]$, so that
$\good_r=\{0\}\cup L_r^2$ and $H_r=\{0\}\sqcup L_r^1\sqcup L_r^2\sqcup
L_r^3\sqcup L_r^4$. 
If $2\leq i\leq 3$, we take $L_r^4=\emptyset$. 
If $i=n-3$,  then $L_r^1=[1,\lfloor\frac{n-2}{2}\rfloor]$, 
$L_r^3=\{\lfloor\frac{n}{2}\rfloor\}$ and we take $L_{r}^2=\emptyset$. 
(If $i=3$ and $n=6$, we take $L_r^2=\emptyset$ and $L_r^4=\emptyset$.)
Since $g\in\ker(\rho_r^{{\rm G}_r})$, then, by Lemma~\ref{lemma3-rho},~$(8)$,
\begin{eqnarray*}
  \rho(g)=\rho_r^{{\rm G}_r}(g) +\rho_r^{H_r\setminus {{\rm G}_r}}(g)=
  \rho_r^{H_r\setminus {{\rm G}_r}}(g)=
  \rho_r^{L_r^1}(g)+\rho_r^{L_r^3}(g)+\rho_r^{L_r^4}(g).
\end{eqnarray*}
Using Remark~\ref{rem-comp}, and the equalities $s_1=r+q$, $s_2=r+\lfloor n/2\rfloor q$, $s_3=s_2+q=r+\lfloor(n+2)/2\rfloor$, and $(n-3+i)/2-\lfloor (n+2)/2\rfloor=\lfloor (i-4)/2\rfloor$ (if $i\geq 4$), we obtain:
\begin{eqnarray*}
&&\rho_r^{L_r^1}(g)=\sum_{k=1}^{1+\lfloor\frac{i-1}{2}\rfloor}
\left(\sum_{j=0}^{\frac{n-1-i}{2}}b_{\frac{n-3+i-2j}{2},k}\lambda_j\right)t^{r+kq}=
 \sum_{k=0}^{\lfloor\frac{i-1}{2}\rfloor}
\left(\sum_{j=0}^{\frac{n-1-i}{2}}b_{\frac{n-3+i-2j}{2},k+1}\lambda_j\right)
t^{s_1+kq},\\
&&\rho_r^{L_r^3}(g)=
\left(\sum_{j=0}^{\frac{n-1-i}{2}}b_{\frac{n-3+i-2j}{2},\lfloor\frac{n}{2}\rfloor}\lambda_j\right)
t^{r+\lfloor\frac{n}{2}\rfloor q}=
\left(\sum_{j=0}^{\frac{n-1-i}{2}}
b_{\frac{n-3+i-2j}{2},\lfloor\frac{n}{2}\rfloor}\lambda_j\right)t^{s_2}
\mbox{ and }
\\&&\rho_r^{L_r^4}(g)=\sum_{k=\lfloor\frac{n+2}{2}\rfloor}^{\frac{n-3+i}{2}}
\left(\sum_{j=0}^{\frac{n-1-i}{2}}b_{\frac{n-3+i-2j}{2},k}\lambda_j\right)t^{r+kq}
=\sum_{k=0}^{\lfloor\frac{i-4}{2}\rfloor}
\left(\sum_{j=0}^{\frac{n-1-i}{2}}b_{\frac{n-3+i-2j}{2},\lfloor\frac{n+2}{2}\rfloor +k}\lambda_j\right)t^{s_3+kq}.
\end{eqnarray*}
In the expression of $\rho_r^{L_r^4}(g)$, we can extend the summation in the $k's$ with $\lfloor(i-4)/2\rfloor+1=\lfloor (i-2)/2\rfloor$, $\lfloor i/2\rfloor$ and $\lfloor (i+1)/2\rfloor$. 
Indeed, for every $j\geq 0$, 
\begin{multline*}
 \left\lfloor\frac{n+2}{2}\right\rfloor+\left\lfloor\frac{i+1}{2}\right\rfloor
\geq\left\lfloor\frac{n+2}{2}\right\rfloor+\left\lfloor\frac{i}{2}\right\rfloor>\left\lfloor\frac{n+2}{2}\right\rfloor+\left\lfloor\frac{i-2}{2}\right\rfloor
\geq\\
\frac{n+1}{2}+\frac{i-3}{2}=\frac{n-2+i}{2}>\frac{n-3+i-2j}{2}.
\end{multline*}
Thus, for all $j\geq 0$, 
$b_{\frac{n-3+i-2j}{2},\lfloor\frac{n+2}{2}\rfloor+\lfloor\frac{i-2}{2}\rfloor}=0$, 
$b_{\frac{n-3+i-2j}{2},\lfloor\frac{n+2}{2}\rfloor+\lfloor\frac{i}{2}\rfloor}=0$ and
$b_{\frac{n-3+i-2j}{2},\lfloor\frac{n+2}{2}\rfloor+\lfloor\frac{i+1}{2}\rfloor}=0$. Hence,
\begin{eqnarray}\label{rhor4-i}
\rho_r^{L_r^4}(g)=\sum_{k=0}^{\lfloor\frac{i+1}{2}\rfloor}
\left(\sum_{j=0}^{\frac{n-1-i}{2}}b_{\frac{n-3+i-2j}{2},\lfloor\frac{n+2}{2}\rfloor +k}\lambda_j\right)t^{s_3+kq}.
\end{eqnarray}
Let us find $h_1=
\sum_{j=0}^{\lfloor\frac{i-1}{2}\rfloor}\lambda_{1,j}\mon^{\phi_{s_1}+j\omega}\in W_{s_1}$ 
such that $\rho_{r}^{L_{r}^1}(g)+\rho(h_1)=0$, if possible.   
By Remark~\ref{rem-comp},
\begin{eqnarray*}
\rho(h_1)=\sum_{k=0}^{\lfloor\frac{i-1}{2}\rfloor}\left(
\sum_{j=0}^{\lfloor\frac{i-1}{2}\rfloor}b_{\lfloor\frac{i-1}{2}\rfloor-j,k}
\lambda_{1,j}\right)t^{s_1+kq}.
\end{eqnarray*}
Asking for $\rho(h_1)=-\rho_r^{L_r^1}(g)$, we get the system of linear
of equations in the unknowns $\lambda_{1,j}$:
\begin{eqnarray*}
\sum_{j=0}^{\lfloor\frac{i-1}{2}\rfloor}b_{\lfloor\frac{i-1}{2}\rfloor-j,k}\lambda_{1,j}
=-\sum_{j=0}^{\frac{n-1-i}{2}}b_{\frac{n-3+i-2j}{2},k+1}\lambda_j,
\phantom{+}k=0,\ldots,\left\lfloor\frac{i-1}{2}\right\rfloor.
\end{eqnarray*}
In matrix form:
\begin{eqnarray*}
\left(\begin{array}{ccc}
b_{\lfloor\frac{i-1}{2}\rfloor,0}&\ldots&b_{0,0}\\\vdots&&\vdots\\
b_{\lfloor\frac{i-1}{2}\rfloor,\lfloor\frac{i-1}{2}\rfloor}&\ldots&b_{0,\lfloor\frac{i-1}{2}\rfloor}
\end{array}\right)\!\!
\left(\begin{array}{c}\lambda_{1,0}\\\vdots
\\\lambda_{1,\lfloor\frac{i-1}{2}\rfloor}\end{array}\right)\!\!=-\!\!\left(\begin{array}{ccc}
b_{\frac{n-3+i}{2},1}&\ldots&b_{i-1,1}\\\vdots&&\vdots\\b_{\frac{n-3+i}{2},\lfloor\frac{i+1}{2}\rfloor}
&\ldots&b_{i-1,\lfloor\frac{i+1}{2}\rfloor}
\end{array}\right)\!\!
\left(\begin{array}{c}
\lambda_0\\\vdots\\\lambda_{\frac{n-1-i}{2}}
\end{array}\right)\!.
\end{eqnarray*}
That is, $P^{[0,\lfloor\frac{i-1}{2}\rfloor]}_{[0,\lfloor\frac{i-1}{2}\rfloor]}\Lambda_1=
-P^{[1,\lfloor\frac{i+1}{2}\rfloor]}_{[i-1,\frac{n-3+i}{2}]}\Lambda$, 
where $\Lambda_1^\top=(\lambda_{1,0},\ldots,\lambda_{1,\lfloor\frac{i-1}{2}\rfloor})$
and $\Lambda^\top=(\lambda_0,\ldots,\lambda_{\frac{n-1-i}{2}})$, with $\lambda_j=(-1)^jb_{\{0\}\sqcup[\lfloor\frac{i+3}{2}\rfloor,
\lfloor\frac{n-2}{2}\rfloor]}^{[i-1,\frac{n-3+i}{2}]\setminus 
\{\frac{n-3+i-2j}{2}\}}$.
By Notation~\ref{notation-binomials}, $P^J_I=(B^I_J)^{\top}\Theta$. So, we get the system:
\begin{eqnarray*}
\left(B^{[0,\lfloor\frac{i-1}{2}\rfloor]}_{[0,\lfloor\frac{i-1}{2}\rfloor]}\right)^\top
\cdot\Theta_{\lfloor\frac{i+1}{2}\rfloor}\cdot\Lambda_1=
-\left(B^{[i-1,\frac{n-3+i}{2}]}_{[1,\lfloor\frac{i+1}{2}\rfloor]}\right)^\top
\cdot\Theta_{\frac{n+1-i}{2}}\cdot\Lambda.
\end{eqnarray*}
Note that $B^{[0,\lfloor\frac{i-1}{2}\rfloor]}_{[0,\lfloor\frac{i-1}{2}\rfloor]}$ 
is a square $\left\lfloor\frac{i+1}{2}\right\rfloor\times\left\lfloor\frac{i+1}{2}\right\rfloor$ lower triangular matrix with 1's in the diagonal, 
so it is invertible. Moreover, $\Theta_{\lfloor\frac{i+1}{2}\rfloor}^{-1}=
\Theta_{\lfloor\frac{i+1}{2}\rfloor}$. Hence, this system has 
a unique solution $\Lambda_1$. Let $h_1$ be the polynomial 
in $W_1$ whose components in $\mathcal{B}_{s_1}$ are
$\Lambda_1$. Therefore,
\begin{eqnarray*}
\rho(g+h_1)=\rho_r^{L_r^1}(g)+\rho_r^{L_r^3}(g)+
\rho_r^{L_r^4}(g)+\rho(h_1)=\rho_r^{L_r^3}(g)+\rho_r^{L_r^4}(g).
\end{eqnarray*}
 Set $L_{s_2}^1=[1,\lfloor (i+3)/2\rfloor]$, $L_{s_2}^2=[\lfloor(i+5)/2\rfloor,\lfloor n/2\rfloor]$ and
$L_{s_2}^3=[\lfloor(n+2)/2\rfloor,(n-1+i)/2]$. 
Exactly as in the proof of Lemma~\ref{lemma-h-i}, 
$H_{s_2}=\{0\}\sqcup L_{s_2}^1\sqcup
L_{s_2}^2\sqcup L_{s_2}^3$. Furthermore, we have
$\supp(\rho_r^{L_r^4}(g))\subseteq \{s_2+kq\mid k\in
[1,\lfloor i/2\rfloor]\}=\{s_2+kq\mid k\in L_{s_2}^1\}$. 

Let us search a $h_2\in W_{s_2}$ such that
$\rho_r^{L_r^3}(g)+\rho_{s_2}^{\{0\}}(h_2)=0$ and such that
$\rho_{s_2}^{L_{s_2}^2}(h_2)=0$. These two conditions determine
$1+\card(L_{s_2}^2)=(n-1-i)/2=\dim W_{s_2}-1$ linear equations. In
order to have as many unknowns as linear conditions, we take
$h_2=\sum_{j=1}^{\frac{n-1-i}{2}}\lambda_{2,j}\mon^{\phi_{s_2}+j\omega}$ in
$W_{s_2}$, with zero component in the first monomial
$\mon^{\phi_{s_2}}$ of the basis $\mcb_{s_2}$. 
So, we ask for:
\begin{eqnarray*}
\rho_{s_2}^{\{0\}\sqcup L_{s_2}^2}(h_2)&=&-\rho_r^{L_r^3}(g).
\end{eqnarray*}
Using Remark~\ref{rem-comp}, we obtain the more explicit expression:
\begin{eqnarray*}
\sum_{k\in\{0\}\sqcup L_{s_2}^{2}}\left(\sum_{j=1}^{\frac{n-1-i}{2}}b_{\frac{n-1+i-2j}{2},k}\cdot \lambda_{2,j}\right)t^{s_2+kq}&=&-\left(\sum_{j=0}^{\frac{n-1-i}{2}}b_{\frac{n-3+i-2j}{2},\lfloor\frac{n}{2}\rfloor}\lambda_j\right)t^{s_2}.
\end{eqnarray*}
In matrix form:
\begin{eqnarray*}
\left(\begin{array}{ccc}
b_{\frac{n-3+i}{2},0}&\cdots&b_{i,0}\\
b_{\frac{n-3+i}{2},\lfloor\frac{i+5}{2}\rfloor}&\cdots&
b_{i,\lfloor\frac{i+5}{2}\rfloor}\\
\vdots&&\vdots\\
b_{\frac{n-3+i}{2},\lfloor\frac{n}{2}\rfloor}&\cdots&
b_{i,\lfloor\frac{n}{2}\rfloor}
\end{array}\right)
\!\!\cdot\!\!
\left(\begin{array}{c}\lambda_{2,1}\\\vdots\\\lambda_{2,\frac{n-1-i}{2}}
\end{array}\right)
\!=\!
-\left(\begin{array}{ccc}
b_{\frac{n-3+i}{2},\lfloor\frac{n}{2}\rfloor}&\cdots&
b_{i-1,\lfloor\frac{n}{2}\rfloor}\vspace*{0,1cm}\\
0&\cdots&0\\\vdots&&\vdots\\0&\cdots&0
\end{array}\right)
\!\!\cdot\!\!
\left(\begin{array}{c}\lambda_{0}\\\vdots\\\lambda_{\frac{n-1-i}{2}}
\end{array}\right).
\end{eqnarray*}
Equivalently, 
\begin{eqnarray*}
P^{\{0\}\sqcup L_{s_2}^2}_{[i,\frac{n-3+i}{2}]}\cdot \Lambda_2=
-\left(\begin{array}{c}
P^{\{\lfloor\frac{n}{2}\rfloor\}}_{[i-1,\frac{n-3+i}{2}]}\\\hline 0
\end{array}\right)
\cdot \Lambda,
\end{eqnarray*}
where 
$\Lambda_2^\top=\left(\lambda_{2,1},\ldots,\lambda_{2,\frac{n-i-1}{2}}\right)$
and 
$\Lambda^{\top}=\left(\lambda_0,\ldots,\lambda_{\frac{n-1-i}{2}}\right)$.
By Notation~\ref{notation-binomials}, this says:
\begin{eqnarray*}
\left(B^{[i,\frac{n-3+i}{2}]}_{\{0\}\sqcup L_{s_2}^2}\right)^\top\cdot\Theta_{\frac{n-1-i}{2}}\cdot\Lambda_2=-
\left(B^{[i-1,\frac{n-3+i}{2}]}
_{\left\{\left\lfloor\frac{n}{2}\right\rfloor\right\}}
\middle| \; 0\; \right)^{\top}\cdot \Theta_{\frac{n+1-i}{2}}\cdot 
\Lambda.
\end{eqnarray*}
Observe that $B^{[i,\frac{n-3+i}{2}]}_{\{0\}\sqcup L_{s_2}^3}$ 
is a square $\frac{n-1-i}{2}\times\frac{n-1-i}{2}$ matrix. 
Since $\lfloor n/2\rfloor\leq (n-3+i)/2$, then
${\{0\}\sqcup L_{s_2}^2}\leq [i,\frac{n-3+i}{2}]$. 
By \cite[Corollary~2]{gv} or \cite[Corollary~2.5]{gp2},
$B^{[i,\frac{n-3+i}{2}]}_{\{0\}\sqcup L_{s_2}^2}$ is an invertible matrix. 
Thus, there is a unique solution $\Lambda_2$ of the above linear system. 
Take $h_2$ in $W_{s_2}$, with components $\Lambda_2$ in the linearly
independent set $\{\mon^{\phi_{s_2}+j\omega}\mid 1\leq j\leq(n-1-i)/2\}\subset \mathcal{B}_{s_2}$. 

Since $\rho(g+h_1)=\rho_r^{L_r^3}(g)+\rho_r^{L_r^4}(g)$, $\rho(h_2)=\rho_{s_2}^{\{0\}}(h_2)+\rho_{s_2}^{L_{s_2}^1}(h_2)+\rho_{s_2}^{L_{s_2}^2}(h_2)+\rho_{s_2}^{L_{s_2}^3}(h_2)$ and $\rho_{r}^{L_r^3}(g)+\rho_{s_2}^{\{0\}}(h_2)+\rho_{s_2}^{L_{s_2}^2}(h_2)=0$, it follows that 
\begin{eqnarray*}
\rho(g+h_1+h_2)=\rho_r^{L_r^4}(g)+\rho_{s_2}^{L_{s_2}^1}(h_2)+\rho_{s_2}^{L_{s_2}^3}(h_2).
\end{eqnarray*}
By Remark~\ref{rem-comp} and using $s_3=s_2+q$, 
\begin{eqnarray}\label{rhos2L1-i}
\rho_{s_2}^{L_{s_2}^1}(h_2)=
\sum_{k=1}^{\lfloor\frac{i+3}{2}\rfloor}
\left(\sum_{j=1}^{\frac{n-1-i}{2}}b_{\frac{n-1+i-2j}{2},k}\lambda_{2,j}\right)
t^{s_2+kq}=
\sum_{k=0}^{\lfloor\frac{i+1}{2}\rfloor}
\left(\sum_{j=1}^{\frac{n-1-i}{2}}b_{\frac{n-1+i-2j}{2},k+1}\lambda_{2,j}\right)
t^{s_3+kq}.
\end{eqnarray}
Similarly, and since $b_{\frac{n-1+i-2j}{2},\frac{n-1+i}{2}}=0$, for all $j\geq 1$, then 
\begin{eqnarray}\label{rhos2L3-i}
\rho_{s_2}^{L_{s_2}^3}(h_2)=
\sum_{k=\lfloor\frac{n+2}{2}\rfloor}^{\frac{n-1+i}{2}}
\left(\sum_{j=1}^{\frac{n-1-i}{2}}b_{\frac{n-1+i-2j}{2},k}\lambda_{2,j}\right)t^{s_2+kq}=
\sum_{k=\lfloor\frac{n+2}{2}\rfloor}^{\frac{n-3+i}{2}}
\left(\sum_{j=1}^{\frac{n-1-i}{2}}b_{\frac{n-1+i-2j}{2},k}\lambda_{2,j}\right)t^{s_2+kq}.  
\end{eqnarray}
Note that if $2\leq i\leq 3$, then $\lfloor (n+2)/2\rfloor>(n-3+i)/2$
and so $\rho_{s_2}^{L_{s_2}^3}(h_2)=0$. 

Thus, $\rho(g+h_1+h_2)=\rho_r^{L_r^4}(g)+\rho_{s_2}^{L_{s_2}^1}(h_2)$. 
Let us search $h_3=\sum_{j=0}^{\lfloor\frac{i+1}{2}\rfloor}\lambda_{3,j}
\mon^{\phi_{s_3}+j\omega}\in W_{s_3}$ such that
$\rho_r^{L_r^4}(g)+\rho_{s_2}^{L_{s_2}^1}(h_2)+\rho(h_3)=0$,
or, equivalently, 
\begin{eqnarray*}
\rho(h_3)=-\rho_r^{L_r^4}(g)-\rho_{s_2}^{L_{s_2}^1}(h_2). 
\end{eqnarray*}
By Remark~\ref{rem-comp}, 
\begin{eqnarray*}
\rho(h_3)=\rho_{s_{3}}^{H_{s_3}}(h_3)=
\sum_{k=0}^{\lfloor\frac{i+1}{2}\rfloor}
\left(\sum_{j=0}^{\lfloor\frac{i+1}{2}\rfloor}
b_{\lfloor\frac{i+1-2j}{2}\rfloor,k}\lambda_{3,j}\right)
t^{s_3+kq}.
\end{eqnarray*}
Using the equalities \eqref{rhor4-i} and \eqref{rhos2L1-i}, we get
\begin{eqnarray*}
\rho_r^{L_r^4}(g)+\rho_{s_2}^{L_{s_2}^1}(h_2)=
\sum_{k=0}^{\lfloor\frac{i+1}{2}\rfloor}
\left(\sum_{j=0}^{\frac{n-1-i}{2}}b_{\frac{n-3+i-2j}{2},\lfloor\frac{n+2}{2}\rfloor +k}\lambda_j+\sum_{j=1}^{\frac{n-1-i}{2}}b_{\frac{n-1+i-2j}{2},k+1}\lambda_{2,j}\right)t^{s_3+kq}.
\end{eqnarray*}
Therefore, $\rho(h_3)=-\rho_r^{L_r^4}(g)-\rho_{s_2}^{L_{s_2}^1}(h_2)$ is equivalent to:
\begin{multline*}
\left(\begin{array}{ccc}
b_{\lfloor\frac{i+1}{2}\rfloor,0}&\cdots&b_{0,0}\\\vdots&&\vdots\\
b_{\lfloor\frac{i+1}{2}\rfloor,\lfloor\frac{i+1}{2}\rfloor}&\cdots&b_{0,\lfloor\frac{i+1}{2}\rfloor}
\end{array}\right)\cdot
\left(\begin{array}{c}
\lambda_{3,0}\\\vdots\\\lambda_{3,\lfloor\frac{i+1}{2}\rfloor}
\end{array}\right)=\\
-\left(\begin{array}{ccc}
b_{\frac{n-3+i}{2},\lfloor\frac{n+2}{2}\rfloor}&\cdots&b_{i-1,\lfloor\frac{n+2}{2}\rfloor}
\\\vdots&&\vdots\\
b_{\frac{n-3+i}{2},
\lfloor\frac{n+2}{2}\rfloor+\lfloor\frac{i+1}{2}\rfloor}&\cdots
&b_{i-1,\lfloor\frac{n+2}{2}\rfloor+\lfloor\frac{i+1}{2}\rfloor}
\end{array}\right)\cdot
\left(\begin{array}{c}
\lambda_0\\\vdots\\\lambda_{\frac{n-1-i}{2}}
\end{array}\right)-\\
\left(\begin{array}{ccc}
b_{\frac{n-3+i}{2},1}&\cdots&b_{i,1}\\\vdots&&\vdots\\
b_{\frac{n-3+i}{2},\lfloor\frac{i+3}{2}\rfloor}&\cdots&b_{i,\lfloor\frac{i+3}{2}\rfloor}
\end{array}\right)\cdot
\left(\begin{array}{c}
\lambda_{2,1}\\\vdots\\\lambda_{2,\frac{n-1-i}{2}}
\end{array}\right).
\end{multline*}
That is, $
P^{\left[0,\lfloor\frac{i+1}{2}\rfloor\right]
}_{\left[0,\lfloor\frac{i+1}{2}\rfloor\right]}\cdot\Lambda_3=
-P^{\left[\lfloor\frac{n+2}{2}\rfloor,
\lfloor\frac{n+2}{2}\rfloor+\lfloor\frac{i+1}{2}\rfloor\right]}_
{\left[i-1,\frac{n-3+i}{2}\right]}\cdot\Lambda-
P^{\left[1,\lfloor\frac{i+3}{2}\rfloor\right]
}_{\left[i,\frac{n-3+i}{2}\right]}\cdot\Lambda_2$, 
where $\Lambda_{3}^\top=\left(\lambda_{3,0},\ldots,
\lambda_{3,\lfloor\frac{i+1}{2}\rfloor}\right)$.
By Notation~\ref{notation-binomials}, this is equivalent to the system:
\begin{multline*}
\left(B^{[0,\lfloor\frac{i+1}{2}\rfloor]}_{[0,\lfloor\frac{i+1}{2}\rfloor]}\right)^\top
\cdot\Theta_{\lfloor\frac{i+3}{2}\rfloor}\cdot\Lambda_3=\\
-\left(B^{[i-1,\frac{n-3+i}{2}]}_{[\lfloor\frac{n+2}{2}\rfloor,\lfloor\frac{n+2}{2}\rfloor+\lfloor\frac{i+1}{2}\rfloor]}\right)^\top\cdot\Theta_{\frac{n+1-i}{2}}\cdot\Lambda -\left(B^{[i,\frac{n-3+i}{2}]}_{[1,\lfloor\frac{i+3}{2}\rfloor]}\right)^\top\cdot\Theta_{\frac{n-1-i}{2}}\cdot\Lambda_2 .
\end{multline*}
As before, $B^{[0,\lfloor\frac{i+1}{2}\rfloor]}_{[0,\lfloor\frac{i+1}{2}\rfloor]}$ is a square $\left\lfloor\frac{i+3}{2}\right\rfloor\times\left\lfloor\frac{i+3}{2}\right\rfloor$ lower triangular matrix with 1's in the diagonal, 
so invertible. Let $\Lambda_3$ be the unique solution of this system and let
$h_3$ be the polynomial in $W_{s_3}$ with components $\Lambda_3$ in 
$\mathcal{B}_{s_3}$. Then,
\begin{eqnarray*}
\rho(g+h_1+h_2+h_3)=\rho_r^{L_r^4}(g)+\rho_{s_2}^{L_{s_2}^1}(h_2)+\rho_{s_2}^{L_{s_2}^3}(h_2)+\rho(h_3)=
\rho_{s_2}^{L_{s_2}^3}(h_2).
\end{eqnarray*}
If $2\leq i\leq 3$, we have seen before that $\rho_{s_2}^{L_{s_2}^{3}}(h_2)=0$. In such a case, we just take $h_4=0$ and we deduce that 
$\rho(g+h_1+h_2+h_3+h_4)=0$, as desired. So, suppose that $i\geq 4$.
Since $\frac{n-3+i}{2}-\lfloor\frac{n+2}{2}\rfloor=
\lfloor\frac{i-4}{2}\rfloor$ and 
$s_2+\left\lfloor\frac{n+2}{2}\right\rfloor q=s_4$, 
and using \eqref{rhos2L3-i}, then
\begin{eqnarray*}
\rho_{s_2}^{L_{s_2}^3}(h_2)=\sum_{k=\lfloor\frac{n+2}{2}\rfloor}^{\frac{n-3+i}{2}} \!\!\left(\sum_{j=1}^{\frac{n-1-i}{2}}b_{\frac{n-1+i-2j}{2},k}
\lambda_{2,j}\right)\!t^{s_2+kq}=
\sum_{k=0}^{\lfloor\frac{i-4}{2}\rfloor}
\left(\sum_{j=1}^{\frac{n-i-1}{2}}
b_{\frac{n-1+i-2j}{2},\lfloor\frac{n+2}{2}\rfloor+k}\lambda_{2,j}\right)t^{s_4+kq}.
\end{eqnarray*}
Let us search a $h_4\in W_{s_4}$ such that $\rho_{s_2}^{L_{s_2}^3}(h_2)+\rho(h_4)=0$. 
Choose $h_4$ as a linear combination of the last $\lfloor (i-2)/2\rfloor$ monomials
of the basis $\mcb_{s_4}$: $h_{4}=\sum_{j=\lfloor\frac{i+3}{2}\rfloor-\lfloor\frac{i-4}{2}\rfloor}^{\lfloor\frac{i+3}{2}\rfloor}
\lambda_{4,j}\mon^{\phi_{s_4}+j\omega}$.  Let
$L_{s_4}^1=[0,\lfloor(i-4)/2\rfloor]$ and $L_{s_4}^2=[\lfloor (i-2)/2\rfloor,\lfloor(i+3)/2\rfloor]$, so that
$H_{s_4}=L_{s_4}^1\sqcup L_{s_4}^2$. Using Remark~\ref{rem-comp},
\begin{eqnarray*}
&&\rho_{s_4}^{L_{s_4}^1}(h_4)=\sum_{k=0}^{\lfloor\frac{i-4}{2}\rfloor}\left(
\sum_{j=\lfloor\frac{i+3}{2}\rfloor-
\lfloor\frac{i-4}{2}\rfloor}^{\lfloor\frac{i+3}{2}\rfloor}
b_{\lfloor\frac{i+3}{2}\rfloor-j,k}\lambda_{4,j}\right)t^{s_4+kq}\text{ and }\\
&&\rho_{s_4}^{L_{s_4}^2}(h_4)=\sum_{k=\lfloor \frac{i-2}{2}\rfloor}^{\lfloor\frac{i+3}{2}\rfloor}\left(
\sum_{j=\lfloor\frac{i+3}{2}\rfloor-
\lfloor\frac{i-4}{2}\rfloor}^{\lfloor\frac{i+3}{2}\rfloor}
b_{\lfloor\frac{i+3}{2}\rfloor-j,k}\lambda_{4,j}\right)t^{s_4+kq}. 
\end{eqnarray*}
Since $b_{\lfloor (i+3)/2\rfloor-j,k}=0$, for all 
$j\geq \left\lfloor (i+3)/2\right\rfloor-\left\lfloor(i-4)/2\right\rfloor$
and all $k\geq\left\lfloor (i-2)/2\right\rfloor$, it follows that $\rho_{s_4}^{L_{s_4}^2}(h_4)=0$, so $\rho(h_4)=\rho_{s_4}^{L_{s_4}^1}(h_4)$. Thus,
$\rho_{s_2}^{L_{s_2}^3}(h_2)+\rho(h_4)=0$ is equivalent to
$\rho_{s_4}^{L_{s_4}^1}(h_4)=-\rho_{s_2}^{L_{s_2}^3}(h_2)$. 
That is,
\begin{eqnarray*}
\sum_{k=0}^{\lfloor\frac{i-4}{2}\rfloor}\left(
\sum_{j=\lfloor\frac{i+3}{2}\rfloor-
\lfloor\frac{i-4}{2}\rfloor}^{\lfloor\frac{i+3}{2}\rfloor}
b_{\lfloor\frac{i+3}{2}\rfloor-j,k}\lambda_{4,j}\right)t^{s_4+kq}=-
\sum_{k=0}^{\lfloor\frac{i-4}{2}\rfloor}
\left(\sum_{j=1}^{\frac{n-1-i}{2}}
b_{\frac{n-1+i-2j}{2},\lfloor\frac{n+2}{2}\rfloor+k}\lambda_{2,j}\right)t^{s_4+kq}.
\end{eqnarray*}
Since $\lfloor\frac{n+2}{2}\rfloor+\lfloor\frac{i-4}{2}\rfloor=\frac{n-3+i}{2}$,
and letting $\Lambda_4^\top=\left(\lambda_{4,\lfloor\frac{i+3}{2}\rfloor-\lfloor\frac{i-4}{2}\rfloor},
\ldots,\lambda_{4,\lfloor\frac{i+3}{2}\rfloor}\right)$, we get
\begin{eqnarray*}
P^{[0,\lfloor\frac{i-4}{2}\rfloor]}_{[0,\lfloor\frac{i-4}{2}\rfloor]}\cdot\Lambda_4=-
P^{[\lfloor\frac{n+2}{2}\rfloor,\frac{n-3+i}{2}]}_{[i,\frac{n-3+i}{2}]}\cdot\Lambda_2.
\end{eqnarray*}
By Notation~\ref{notation-binomials},
\begin{eqnarray*}
\left(B^{\left[0,\lfloor\frac{i-4}{2}\rfloor\right]}_{\left[0,\lfloor\frac{i-4}{2}\rfloor\right]}\right)^\top\cdot\Theta_{\lfloor\frac{i-2}{2}\rfloor}\cdot
\Lambda_4=-\left(B^{[i,\frac{n-3+i}{2}]}_{[\lfloor\frac{n+2}{2}\rfloor,\frac{n-3+i}{2}\rfloor]}\right)^\top\cdot\Theta_{\frac{n-1-i}{2}}\cdot\Lambda_2.
\end{eqnarray*}
Since $B^{\left[0,\lfloor\frac{i-4}{2}\rfloor\right]}
_{\left[0,\lfloor\frac{i-4}{2}\rfloor\right]}$ is a square
$\left\lfloor\frac{i-2}{2}\right\rfloor\times\left\lfloor\frac{i-2}{2}\right\rfloor$ lower triangular matrix with 1's in the diagonal, 
this system has a unique solution $\Lambda_4$. Let
$h_4$ be the polynomial in $W_{s_4}$ 
with components $\Lambda_4$ in the linearly
independent set $\{\mon^{\phi_{s_4}+j\omega}\mid 
\lfloor(i+3)/2\rfloor-\lfloor(i-4)/2\rfloor\leq j\leq \lfloor(i+3)/2\rfloor\}$. Then,
\begin{eqnarray*}
\rho(f)=\rho(g+h_1+h_2+h_3+h_4)=\rho_{s_2}^{L_{s_2}^4}(g)+\rho(h_4)=0,
\end{eqnarray*}
as desired.

Suppose that $n$ is odd and that $2\leq k\leq n-3$, with $k$ even. 
By Lemma~\ref{lemma-c-i+1}, $g$ is equal to the polynomial $g_k$ defined in 
\eqref{equality-gknoddkeven}.
By Proposition~\ref{proposition-class}, $i=n-k-1$. Thus,
\begin{eqnarray*}
&&\phi_{s_1}=\left(\frac{n-k-3}{2},0,\frac{n+k+1}{2}\right),
\quad\phi_{s_2}=\left(\frac{2n-k-2}{2},1,\frac{k}{2}\right),\\ 
&&\phi_{s_3}=\left(\frac{n-k-1}{2}, 0,\frac{n+k+1}{2}\right),
\quad\phi_{s_4}=\left(\frac{n-k+1}{2},0,\frac{n+k+1}{2}\right).
\end{eqnarray*}
It follows that
\begin{eqnarray*}
&&h_1=\sum_{j=0}^{\frac{n-k-3}{2}}\lambda_{1,j}x^{\frac{n-k-3-2j}{2}}y^{2j}z^{\frac{n+k+1-2j}{2}}
\phantom{+},\phantom{+}
h_2=\sum_{j=1}^{\frac{k}{2}}
\lambda_{2,j}x^{\frac{2n-k-2-2j}{2}}y^{1+2j}z^{\frac{k-2j}{2}},\\
&&h_3=\sum_{j=0}^{\frac{n-k-1}{2}}\lambda_{3,j}x^{\frac{n-k-1-2j}{2}}y^{2j}z^{\frac{n+k+1-2j}{2}}
\phantom{+}\mbox{ and }\phantom{+}
h_4=\sum_{j=3}^{\frac{n-k+1}{2}}
\lambda_{4,j}x^{\frac{n-k+1-2j}{2}}y^{2j}z^{\frac{n+k+1+2j}{2}},
\end{eqnarray*}
where $h_4=0$ if $i=2$ or, equivalently, if $k=n-3$. 
This leads to the expression of $f_k$ in \eqref{poly-fknoddkeven}.

Suppose that $n$ is even and that $2\leq k\leq n-4$, with $k$ even. 
By Lemma~\ref{lemma-c-i+1}, $g$ is equal to the polynomial $g_k$ defined in 
\eqref{equality-gknevenkeven}
By Proposition~\ref{proposition-class}, $i=n-k-1$. Thus,
\begin{eqnarray*}
&&\phi_{s_1}=\left(\frac{n-k-2}{2},0,\frac{n+k}{2}\right),
\quad\phi_{s_2}=\left(\frac{2n-k-2}{2},1,\frac{k}{2}\right),\\ 
&&\phi_{s_3}=\left(\frac{n-k}{2},0,\frac{n+k}{2}\right),
\quad\phi_{s_4}=\left(\frac{n-k+2}{2},0,\frac{n+k}{2}\right).
\end{eqnarray*}
It follows that
\begin{eqnarray*}
&&h_1=\sum_{j=0}^{\frac{n-k-2}{2}}
\lambda_{1,j}x^{\frac{n-k-2-2j}{2}}y^{2j}z^{\frac{n+k-2j}{2}}
\phantom{+},\phantom{+}
h_2=\sum_{j=1}^{\frac{k}{2}}\lambda_{2,j}x^{\frac{2n-k-2-2j}{2}}y^{1+2j}z^{\frac{k-2j}{2}},\\
&&h_3=\sum_{j=0}^{\frac{n-k}{2}}\lambda_{3,j}x^{\frac{n-k-2j}{2}}y^{2j}z^{\frac{n+k-2j}{2}}
\phantom{+}\mbox{ and }\phantom{+}
h_4=\sum_{j=4}^{\frac{n-k+2}{2}}\lambda_{4,j}x^{\frac{n-k+2-2j}{2}}y^{2j}z^{\frac{n+k-2j}{2}},
\end{eqnarray*}
where $h_4=0$ if $i=3$ or, equivalently, if $k=n-4$.
This leads to the expression of $f_k$ in \eqref{poly-fknevenkeven}.
\end{proof}

\vspace*{0.3cm}

\begin{lemma}\label{lemma-h-00}
Let $n\geq 6$. Let $r\in [s_0,s_0+n)$. Write $r=r_k=s_0+k-1$, with $1\leq k\leq n$.
Suppose that $(\iota_r,c_r)=(0,0)$, so $n$ is odd, $k=n-2$ and $r=r_{n-2}$.
Let
\begin{eqnarray*}
s_1=r+\left(\frac{n-1}{2}\right)q\phantom{+}\mbox{ and }
\phantom{+}s_2=s_1+q=r+\left(\frac{n+1}{2}\right)q. 
\end{eqnarray*}
Then $s_1,s_2\in\nums$ and $r<s_1<s_2<\la$. Moreover,
\begin{eqnarray*}
\phi_{s_1}=\left(\frac{n+1}{2},0,\frac{n-1}{2}\right)\phantom{+}\mbox{ and }\phantom{+}
\phi_{s_2}=(0,1,n-1).
\end{eqnarray*}
Let $g=\sum_{j=0}^{\frac{n-1}{2}}(-1)^jb_{\frac{n-1}{2},\frac{n-1-2j}{2}}\cdot 
x^{\frac{n-1-2j}{2}}y^{2j}z^{\frac{n-1-2j}{2}}$,
$h_1=-xy^{n-1}$, $h_2=yz^{n-1}$, $h=h_1+h_2$ and $f=g+h$. Then
$g$ is a basis of $\ker(\rho_r^{{\rm G}_r})$, $h_1\in W_{s_1}$, $h_2\in W_{s_2}$, 
\begin{eqnarray*}
f^\sigma=g,\phantom{+} f^\tau=h,\phantom{+}
f\in\ker(\rho)\phantom{+}\text{and}\phantom{+}g\in V_r.
\end{eqnarray*}
Hence, $V_r=\ker(\rho_r^{{\rm G}_r})$. Moreover, 
$g$ is equal to the polynomial $g_{n-2}$ defined in \eqref{equality-gk00} 
and $f=g+h$ is equal to the polynomial $f_{n-2}$ defined as in \eqref{poly-fk00}.
In particular, $f_{n-2}\in\ker(\rho)$ and $f_{n-2}^{\sigma}$ is a basis of $V_{r_{n-2}}$. 
\end{lemma}
\begin{proof}
By Proposition~\ref{proposition-class}, if $(\iota_{r_k},c_{r_k})=(0,0)$, then $n$ is odd, 
$k=n-2$ and $r=r_{n-2}$. 
Observe that $r=(a+1)\ell_r+c_r=(a+1)(n-1)$. 
Clearly, $r<s_1<s_2$. By Lemma~\ref{lemma-ic00}, $\delta_r=(n+1)/2$, so $s_1=r+((n-1)/2)q=r+(\delta_r-1)q\in\nums$, 
$s_1<\la$ and $\phi_{s_1}=((n+1)/2,0,(n-1)/2)$.

Since $n\geq 7$, by Lemma~\ref{lemma-cota}, $\lfloor a/2\rfloor-n\geq 3$, so
$(\lfloor a/2\rfloor-n)\lfloor n/2\rfloor\geq 3\lfloor n/2\rfloor>
\lfloor n/2\rfloor +1=(n+1)/2=d$. Thus, $s_2=r+dq<\frob(\nums)<\la$.  
Since $n$ is odd, then $q=n$ and 
$qd=q(n+1)/2=(n(n-1)/2)+n=a+n$. Thus, $s_2=r+dq=
an+(2n-1)$, where $0\leq 2n-1\leq 2n$. By Proposition~\ref{theo31ggp}, 
$s_2\in\nums$, $\ell_{s_2}=n$ and $\rem_{s_2}=2n-1$. Then, 
$\phi_{s_2,1}=\ell_{s_2}-\lfloor (\rem_{s_2}+1)/2\rfloor=0$, 
$\phi_{s_2,2}=\rem_{s_2}-2\lfloor \rem_{s_2}/2\rfloor=1$ and
$\phi_{s_2,3}=2\lfloor \rem_{s_2}/2\rfloor=n-1$.

By Lemma~\ref{lemma-ic00},
$g$ is a basis of $\ker(\rho_r^{{\rm G}_r})$. Moreover, 
$h_1=-xy^{n-1}=-\mon^{\phi_{s_1}+\frac{n-1}{2}\omega}\in W_{s_1}$ and 
$h_2=yz^{n-1}=\mon^{\phi_{s_2}}\in W_{s_2}$. Since $r<s_1<s_2$, it follows that $f^\sigma=g$ and $f^\tau=h$. It is enough to prove that $\rho(f)=0$. Because then, $g\in V_r$ and 
$V_r=\ker(\rho_r^{{\rm G}_r})$. One could check, using the definition, that 
$f\in\ker(\rho)$. Alternatively, we explain why we choose 
$h_1$ and $h_2$ while, at the same time, we prove that $\rho(f)=0$. Clearly, $H_r=\good_r\sqcup (H_r\setminus\good_r)$. Since
$g\in\ker(\rho_r^{{\rm G}_r})$, by Lemma~\ref{lemma3-rho}, 
$\rho(g)=\rho_r^{{\rm G}_r}(g)+\rho_r^{H_r\setminus {\rm G}_r}(g)=\rho_r^{H_r\setminus {\rm G}_r}(g)$. 
By Lemma~\ref{lemma-ic00}, $\phi_r=((n-1)/2,0,(n-1)/2)$, 
$I_r=[0,\frac{n-1}{2}]$, $H_r=[0,\frac{n-1}{2}]$ and $\good_r=[0,\frac{n-3}{2}]$.
Therefore, $H_r\setminus {\rm G}_r=\{\frac{n-1}{2}\}$. 
By Remark~\ref{rem-comp}, and using that $b_{\frac{n-1-2j}{2},\frac{n-1}{2}}=0$, for all $j\geq 1$,
\begin{eqnarray*}
\rho(g)=\rho_r^{\{\frac{n-1}{2}\}}(g)=
\left(\sum_{j=0}^{\frac{n-1}{2}}b_{\frac{n-1-2j}{2},\frac{n-1}{2}}
(-1)^jb_{\frac{n-1}{2},\frac{n-1-2j}{2}}\right)t^{r+\frac{n-1}{2}q}=t^{r+\frac{n-1}{2}q}=t^{s_1}.
\end{eqnarray*}
Let us find $h_1\in W_{s_1}$ such that, if possible, $\rho(g)+\rho(h_1)=0$ and, 
if not, such that at least 
$\rho(h_1)$ cancels the term $t^{s_1}$. Since $s_1\in\nums$ and $s_1<\la$, 
by Summary~\ref{summary-semigroup}, we deduce that a basis of $W_{s_1}$ is 
$\mcb_{s_1}=\{\mon^{\phi_{s_1}+j\omega}\mid 0\leq j\leq\min(\phi_{s_1,1},\phi_{s_1,3})\}=\{
x^{\frac{n+1-2j}{2}}y^{2j}z^{\frac{n-1-2j}{2}}\mid 0\leq j\leq \frac{n-1}{2}\}$.
Due to the definition of $\rho$, where $\rho(x)=t^a(1+t^q)$, $\rho(y)=t^{a+1}$ and $\rho(z)=t^{a+2}$, 
one realizes that, in terms of simplicity, the best monomials to choose are those whose degree in $x$ is as small as possible. Thus, let $h_1=\lambda_{1,1}xy^{n-1}\in W_{s_1}$. 
By Lemma~\ref{lemma1-rho}, $\rho(h_1)=\lambda_{1,1}t^{s_1}+\lambda_{1,1}t^{s_1+q}$. Taking $\lambda_{1,1}=-1$, one gets $\rho(g+h_1)=t^{s_1}-t^{s_1}-t^{s_1+q}=-t^{s_2}$.
Hence, we choose $h_1=-xy^{n-1}$.

As before, let us find $h_2\in W_{s_2}$ such that, if possible, $\rho(g+h_1)+\rho(h_2)=0$ and, if not, such that
$\rho(h_2)$ cancels the term $-t^{s_2}$.
Since $s_2\in\nums$, $s_2<\la$ and $\phi_{s_2}=(0,1,n-1)$, by Summary~\ref{summary-semigroup}, 
it follows that $W_{s_2}=\langle yz^{n-1}\rangle$. So, take $h_2=\lambda_{2,1}yz^{n-1}$. Note that 
$\rho(h_2)=\lambda_{2,1}t^{a+1+(n-1)(a+2)}=\lambda_{2,1}t^{s_2}$. Asking for the equality $0=\rho(g+h_1)+\rho(h_2)=-t^{s_2}+\lambda_{2,1}t^{s_{2}}$, we obtain $\lambda_{2,1}=1$ and $h_2=yz^{n-1}$. 

Finally, by Lemma~\ref{lemma-ic00}, $g$ is equal to the polynomial $g_{n-2}$ defined in 
\eqref{equality-gk00}. The expression of $f=g+h$ as $f_{n-2}$ follows from the values of $h_1$ and $h_2$. 
\end{proof}

\vspace*{0.3cm}

\begin{lemma}\label{lemma-h-21}
Let $r\in [s_0,s_0+n)$. 
Write $r=r_k=s_0+k-1$, with $1\leq k\leq n$.
Suppose that  $(\iota_r,c_r)={\colive (2,1)}$, so $n$ is odd, 
$k=n-1$ and $r=r_{n-1}$.
Let
\begin{eqnarray*}
s_1=r+\left(\frac{n-3}{2}\right)q,\quad s_2=s_1+q,\quad s_3=s_1+\left(\frac{n+1}{2}\right)q\quad\mbox{and}\quad s_4=s_3+q.
\end{eqnarray*}
Then $s_1,s_2,s_3,s_4\in\nums$ and $r<s_1<s_2<s_3<s_4<\la$. Moreover,
\begin{eqnarray*}
&&\phi_{s_1}=\left(n,0,0\right),
\quad\phi_{s_2}=\left(\frac{n-1}{2},1,\frac{n-1}{2}\right),\quad\phi_{s_3}=\left(\frac{n+1}{2},1,\frac{n-1}{2}\right),
\quad\phi_{s_4}=\left(1,0,n\right).
\end{eqnarray*}
Let 
$g=\sum_{j=0}^{\frac{n-3}{2}}(-1)^jb_{\frac{n-3}{2},\frac{n-3-2j}{2}}\cdot x^{\frac{n-3-2j}{2}}y^{1+2j}z^{\frac{n-1-2j}{2}}$ and $h_1=-x^n$. Let
$h_2=\sum_{j=1}^{\frac{n-1}{2}}\lambda_{2,j}\mon^{\phi_{s_2}+j\omega}$, where $\Lambda_2^\top=\left(\lambda_{2,1},\ldots,
\lambda_{2,\frac{n-1}{2}}\right)$ satisfies
\begin{eqnarray*}
&&\left(B^{[0,\frac{n-3}{2}]}_{[0,\frac{n-3}{2}]}\right)^\top\cdot
\Theta_{\frac{n-1}{2}}\cdot\Lambda_2=\left(B^{\{n\}}_{[1,\frac{n-1}{2}]}\right)^\top.
\end{eqnarray*}
Let
$h_3=\sum_{j=1}^{\frac{n-1}{2}}\lambda_{3,j}\mon^{\phi_{s_3}+j\omega}$, where $\Lambda_{3}^\top=\left(\lambda_{3,1},\ldots,\lambda_{3,\frac{n-1}{2}}\right)$ satisfies
\begin{eqnarray*}
\left(B^{[1,\frac{n-1}{2}]}
_{\{0\}\sqcup[2,\frac{n-1}{2}]}\right)^\top\cdot
\Theta_{\frac{n-1}{2}}\cdot\Lambda_3=
\left(B^{\{n\}}_
{\{\frac{n+1}{2}\}\sqcup [\frac{n+5}{2},n]}\right)^\top.
\end{eqnarray*} 
Let
$h_4=\lambda_{4,1}y^2z^{n-1}$, where 
$\lambda_{4,1}=b_{n,\frac{n+3}{2}}-
\sum_{j=1}^{\frac{n-1}{2}}b_{\frac{n+1-2j}{2},1}\cdot\lambda_{3,j}$.

Set $h=h_1+h_2+h_3+h_4$ and  $f=g+h$. Then $g$ is a basis of $\ker(\rho_r^{{\rm G}_r})$,
$h_i\in W_{s_i}$, $i=1,2,3,4$,
\begin{eqnarray*}
f^\sigma=g,\phantom{+} f^\tau=h,\phantom{+}
f\in\ker(\rho)\phantom{+} \text{and}\phantom{+} g\in V_r.
\end{eqnarray*}
Hence, $V_r=\ker(\rho_r^{{\rm G}_r})$. Moreover, 
$g$ is equal to the polynomial $g_{n-2}$ defined in \eqref{equality-gk21} 
and $f=g+h$ is equal to the polynomial $f_{n-1}$ defined as in \eqref{poly-fk21}.
In particular, $f_{n-1}\in\ker(\rho)$ and $f_{n-1}^{\sigma}$ is a basis of $V_{r_{n-1}}$. 
\end{lemma}
\begin{proof}
By Proposition~\ref{proposition-class}, if $(\iota_{r_k},c_{r_k})={\colive (2,1)}$, 
then $n$ is odd, $k=n-1$ and $r=r_{n-1}$. 
Observe that $r=(a+1)\ell_r+c_r=(a+1)(n-1)+1=a(n-1)+n$. Clearly, $r<s_1<s_2<s_3<s_4$. Let us prove that $s_4<\la$. Note that 
$s_4=s_3+q=r+(n-1)q+q=r+nq$. Since $n$ is odd, $n\geq 7$. By Lemma~\ref{lemma-cota},~$(1)$,
$3\leq \lfloor a/2\rfloor-n$. Clearly, $2n<3n-1$, so
$n<3(n-1)/2\leq(\lfloor a/2\rfloor-n)(n-1)/2=(\lfloor a/2\rfloor-n)\lfloor n/2\rfloor $. By Lemma~\ref{lemma-cota},~$(2)$,  
$s_4<\frob(\nums)<\la$. In particular $s_1,s_2,s_3<\la$.

By Lemma~\ref{lemma-ic21}, $\delta_r=(n-1)/2$, so
$s_1=r+((n-3)/2)q=r+(\delta_r-1)q\in\nums$, $s_1<\la$ and $\phi_{s_1}=(n,0,0)$. In particular, $\kappa_{s_1}=
\min(\phi_{s_1,1},\phi_{s_1,3})=0$, $I_{s_1}=[\phi_{s_1,1}-\kappa_{s_1},\phi_{s_1,1}]=\{n\}$, $H_{s_1}=[0,\phi_{s_1,1}]=[0,n]$, $W_{s_1}=\langle x^{n}\rangle$ and $h_1\in W_{s_1}$.

Let us prove that $s_2\in\nums$. Since $n$ is odd, $q=n$ and $((n-1)/2)q=a$. 
Note that $s_2=s_1+q=r+((n-3)/2)q+q=a(n-1)+n+((n-1)/2)q=an+n$. So, $s_2=an+n$, with 
$0\leq n\leq 2n$. By Proposition~\ref{theo31ggp}, $s_2\in\nums$ and, 
since $s_2<\la$, then $\ell_{s_2}=n$, $\rem_{s_2}=n$ and
$\phi_{s_2}=((n-1)/2,1,(n-1)/2)$. 
In particular,
$\kappa_{s_2}=\min(\phi_{s_2,1},\phi_{s_2,3})=(n-1)/2$, 
$I_{s_2}=[0,(n-1)/2]$, $H_{s_2}=[0,(n-1)/2]$, 
$W_{s_2}=\langle\mon^{\phi_{s_2}+j\omega}\mid 0\leq j\leq (n-1)/2\rangle$, 
$h_2$ is well-defined and $h_2\in W_{s_2}$. 

Let us prove that $s_3\in\nums$. Again, since $((n-1)/2)q=a$, $s_1=r+((n-3)/2)q$ and $r=a(n-1)+n$, it follows that $s_3=s_1+\frac{n+1}{2}q=r+\frac{n-3}{2}q+\frac{n+1}{2}q=r+(\frac{n-1}{2}+\frac{n-1}{2})q=a(n-1)+n+2a=a(n+1)+n$. So, $s_3=a(n+1)+n$, where $0\leq n\leq 2(n+1)$.  
By Proposition~\ref{theo31ggp}, $s_3\in\nums$ and, since $s_3<\la$, 
then $\ell_{s_3}=n+1$, 
$\rem_{s_3}=n$ and
$\phi_{s_3}=((n+1)/2,1,(n-1)/2)$. In particular,
$\kappa_{s_3}=\min(\phi_{s_3,1},\phi_{s_3,3})=(n-1)/2$, 
$I_{s_3}=[1,(n+1)/2]$, $H_{s_3}=[0,(n+1)/2]$, 
$W_{s_3}=\langle\mon^{\phi_{s_3}+j\omega}\mid 0\leq j\leq (n-1)/2\rangle$, 
$h_3$ is well-defined and $h_3\in W_{s_3}$. 

Let us prove that $s_4\in\nums$. 
Since $s_3=a(n+1)+n$ and $q=n$, it follows that $s_4=s_3+q=a(n+1)+2n$, where
$0\leq 2n\leq 2(n+1)$. By Proposition~\ref{theo31ggp}, $s_4\in\nums$ and,
since $s_4<\la$, then $\ell_{s_4}=n+1$, 
$\rem_{s_4}=2n$ and $\phi_{s_4}=(1,0,n)$. In particular,
$\kappa_{s_4}=\min(\phi_{s_4,1},\phi_{s_4,3})=
1$, $I_{s_4}=[0,1]$, $H_{s_4}=[0,1]$, $W_{s_4}=\langle xz^n,y^2z^{n-1}\rangle$ and
$h_4\in W_{s_4}$. 

By Lemma~\ref{lemma-ic21}, 
$g$ is a basis of $\ker(\rho_r^{{\rm G}_r})$.
Since $r<s_1<s_2<s_3<s_4$, then $f^\sigma=g$ and $f^\tau=h$. 
Let us prove that $\rho(f)=0$, so $g\in V_r$ and $V_r=\ker(\rho_r^{{\rm G}_r})$.

By Lemma~\ref{lemma-ic21}, $\phi_{r}=((n-3)/2,1,(n-1)/2)$, 
$I_r=[0,(n-3)/2]$, $H_r=[0,\phi_{r,1}]=[0,(n-3)/2]$ and
$\good_r=[0,(n-5)/2]$. 

Set $L_r^1=\{(n-3)/2\}$, so that
$H_r=\good_r\sqcup L_r^1$. Since $g\in\ker(\rho_r^{{\rm G}_r})$, then, by Lemma~\ref{lemma3-rho},~$(8)$,
\begin{eqnarray*}
\rho(g)=\rho_r^{{\rm G}_r}(g) +\rho_r^{H_r\setminus {{\rm G}_r}}(g)=
\rho_r^{H_r\setminus {{\rm G}_r}}(g)=\rho_r^{L_r^1}(g).
\end{eqnarray*}
By Remark~\ref{rem-comp}, and using that 
$s_1=r+\frac{n-3}{2}q$ and that $b_{\frac{n-3-2j}{2},\frac{n-3}{2}}=0$, for all $j\geq 1$, 
then 
\begin{eqnarray*}
\rho_r^{L_r^1}(g)=\left(\sum_{j=0}^{\frac{n-3}{2}}
b_{\frac{n-3-2j}{2},\frac{n-3}{2}}\cdot(-1)^j\cdot b_{\frac{n-3}{2},\frac{n-3-2j}{2}}\right)
t^{r+\frac{n-3}{2}q}=t^{s_1}.
\end{eqnarray*}
Set $L_{s_1}^1=[1,(n-1)/2]$, $L_{s_1}^2=\{(n+1)/2\}$, 
$L_{s_1}^3=\{(n+3)/2\}$ and $L_{s_1}^4=[(n+5)/2,n]$. Note that 
$H_{s_1}=\{0\}\sqcup L_{s_1}^1\sqcup L_{s_1}^2\sqcup L_{s_1}^3\sqcup L_{s_1}^4$. 

Let us find $h_1=\lambda_{1,0}x^n\in W_{s_1}$ such that 
$\rho_r^{L_r^1}(g)+\rho_{s_1}^{\{0\}}(h_1)=0$. 
By Remark~\ref{rem-comp}, $\rho_{s_1}^{\{0\}}(h_1)=\lambda_{1,0}t^{s_1}$. Thus, 
$\lambda_{1,0}=-1$. Therefore,
\begin{eqnarray*}
\rho(g+h_1)&=&\rho_r^{L_r^1}(g)+\rho_{s_1}^{\{0\}}(h_1)+\rho_{s_1}^{L_{s_1}^1}(h_1)+\rho_{s_1}^{L_{s_1}^2}(h_1)+\rho_{s_1}^{L_{s_1}^3}(h_1)+\rho_{s_1}^{L_{s_1}^4}(h_1)\\
&=&\rho_{s_1}^{L_{s_1}^1}(h_1)+\rho_{s_1}^{L_{s_1}^2}(h_1)+\rho_{s_1}^{L_{s_1}^3}(h_1)+\rho_{s_1}^{L_{s_1}^4}(h_1).
\end{eqnarray*}
Using Remark~\ref{rem-comp} and the equalities $s_2=s_1+q$, $s_3=s_1+\frac{n+1}{2}q$ 
and $s_4=s_1+\frac{n+3}{2}q$, we get
\begin{eqnarray*}
&&\rho_{s_1}^{L_{s_1}^1}(h_1)=-\sum_{k=1}^{\frac{n-1}{2}}b_{n,k} t^{s_1+kq}=
-\sum_{k=0}^{\frac{n-3}{2}}b_{n,k+1} t^{s_2+kq},\\
&&\rho_{s_1}^{L_{s_1}^2}(h_1)=-b_{n,\frac{n+1}{2}} t^{s_1+\frac{n+1}{2}q}
=-b_{n,\frac{n+1}{2}} t^{s_3},\\
&&\rho_{s_1}^{L_{s_1}^3}(h_1)=-b_{n,\frac{n+3}{2}} t^{s_1+\frac{n+3}{2}q}=
-b_{n,\frac{n+3}{2}} t^{s_3+q}=-b_{n,\frac{n+3}{2}} t^{s_4},\\
&&\rho_{s_1}^{L_{s_1}^4}(h_1)=-\sum_{k=\frac{n+5}{2}}^{n}b_{n,k} t^{s_1+kq}=
-\sum_{k=2}^{\frac{n-1}{2}}b_{n,\frac{n+1}{2}+k} t^{s_3+kq}.
\end{eqnarray*}
Set $L_{s_2}^{1}=[0,(n-3)/2]$ and $L_{s_2}^2=\{(n-1)/2\}$. 
Note that $H_{s_2}=L_{s_2}^1\sqcup L_{s_2}^2$. 

Let us find $h_2\in W_{s_2}$ such that $\rho_{s_1}^{L_{s_1}^1}(h_1)+\rho(h_2)=0$. Choose $h_2=\sum_{j=1}^{\frac{n-1}{2}}\lambda_{2,j}\mon^{\phi_{s_2+j\omega}}$ in $W_{s_2}$ with 
zero component in the first element of the basis $\mcb_{s_2}$. 

By Remark~\ref{rem-comp}, and using that 
$b_{\frac{n-1-2j}{2},\frac{n-1}{2}}=0$, for all $j\geq 1$, we get
\begin{eqnarray*}
&&\rho_{s_2}^{L_{s_2}^1}(h_2)=\sum_{k=0}^{\frac{n-3}{2}}
\left(\sum_{j=1}^{\frac{n-1}{2}}
b_{\frac{n-1-2j}{2},k}\lambda_{2,j}\right)t^{s_2+kq}, \\
&&\rho_{s_2}^{L_{s_2}^2}(h_2)=
\left(\sum_{j=1}^{\frac{n-1}{2}}
b_{\frac{n-1-2j}{2},\frac{n-1}{2}}\lambda_{2,j}\right)t^{s_2+\frac{n-1}{2}q}=0.
\end{eqnarray*}
Therefore, $\rho(h_2)=\rho_{s_2}^{L_{s_2}^1}(h_2)+\rho_{s_2}^{L_{s_2}^2}(h_2)=
\rho_{s_2}^{L_{s_2}^1}(h_2)$. Thus, $\rho_{s_1}^{L_{s_1}^1}(h_1)+\rho(h_2)=0$ is equivalent to asking for $\rho_{s_2}^{L_{s_2}^1}(h_2)=-\rho_{s_1}^{L_{s_1}^1}(h_1)$, which leads to the linear system:
\begin{eqnarray*}
\sum_{k=0}^{\frac{n-3}{2}}
\left(\sum_{j=1}^{\frac{n-1}{2}}
b_{\frac{n-1-2j}{2},k}\lambda_{2,j}\right)t^{s_2+kq}=
\sum_{k=0}^{\frac{n-3}{2}}b_{n,k+1} t^{s_2+kq}.
\end{eqnarray*}
In matrix form, $P_{[0,\frac{n-3}{2}]}^{[0,\frac{n-3}{2}]}\cdot \Lambda _2=
P_{\{n\}}^{[1,\frac{n-1}{2}]}$, where 
$\Lambda_2^\top=\left(\lambda_{2,1},\ldots,\lambda_{2,\frac{n-1}{2}}\right)$. 
By Notation~\ref{notation-binomials}, this is equivalent to: 
\begin{eqnarray*}
\left(B^{[0,\frac{n-3}{2}]}_{[0,\frac{n-3}{2}]}\right)^\top
\cdot\Theta_{\frac{n-1}{2}}\cdot\Lambda_2=\left(B^{\{n\}}_{[1,\frac{n-1}{2}]}\right)^\top\cdot \Theta_{1}=\left(B^{\{n\}}_{[1,\frac{n-1}{2}]}\right)^\top.
\end{eqnarray*}
Note that $B_{[0,\frac{n-3}{2}]}^{[0,\frac{n-3}{2}]}$ is a square $\frac{n-1}{2}\times\frac{n-1}{2}$ lower triangular matrix with $1$'s in the diagonal, so invertible. Let $\Lambda_2$ be the unique solution of the system and let $h_2$ be the polynomial in $W_{s_2}$, 
with components $\Lambda_2$ in the linearly 
independent set $\{\mon^{\phi_{s_2}+j\omega}\mid 1\leq j\leq (n-1)/2\}\subset \mathcal{B}_{s_2}$. 
Then, 
\begin{eqnarray*}
\rho(g+h_1+h_2)=\rho_{s_1}^{L_{s_1}^2}(h_1)+\rho_{s_1}^{L_{s_1}^3}(h_1)+\rho_{s_1}^{L_{s_1}^4}(h_1).
\end{eqnarray*}
Set $L_{s_3}^1=\{1\}$, $L_{s_3}^2=[2,\frac{n-1}{2}]$ and 
$L_{s_3}^3=\{\frac{n+1}{2}\}$. Note that $H_{s_3}=\{0\}\sqcup L_{s_3}^1
\sqcup L_{s_3}^2\sqcup L_{s_3}^3$. 

The best thing that could happen now is that there would be an element $h_3$ in $W_{s_3}$ 
such that $\rho(g+h_1+h_2+h_3)=0$. Instead, we look for $h_3=
\sum_{j=1}^{\frac{n-1}{2}}\lambda_{3,j}\mon^{\phi_{s_3}+j\omega}$ in $W_{s_3}$, with 
zero component in the first element of the basis $\mcb_{s_3}$, and such that 
$\rho_{s_1}^{L_{s_1}^2}(h_1)+\rho_{s_1}^{L_{s_1}^4}(h_1)+
\rho_{s_3}^{\{0\}}(h_3)+\rho_{s_3}^{L_{s_3}^2}(h_3)=0$. 

By Remark~\ref{rem-comp}, and using that $s_4=s_3+q$ and
$b_{\frac{n+1-2j}{2},\frac{n+1}{2}}=0$, for all $j\geq 1$, we get:
\begin{eqnarray*}
&&\rho_{s_3}^{\{0\}}(h_3)=\left(\sum_{j=1}^{\frac{n-1}{2}}
b_{\frac{n+1-2j}{2},0}\lambda_{3,j}\right)t^{s_3},
\\&&\rho_{s_3}^{L_{s_3}^1}(h_3)=\left(\sum_{j=1}^{\frac{n-1}{2}}
b_{\frac{n+1-2j}{2},1}\lambda_{3,j}\right)t^{s_3+q}=\left(\sum_{j=1}^{\frac{n-1}{2}}
b_{\frac{n+1-2j}{2},1}\lambda_{3,j}\right)t^{s_4},
\\&&\rho_{s_3}^{L_{s_3}^2}(h_3)=\sum_{k=2}^{\frac{n-1}{2}}
\left(\sum_{j=1}^{\frac{n-1}{2}}
b_{\frac{n+1-2j}{2},k}\lambda_{3,j}\right)t^{s_3+kq},
\\&&\rho_{s_3}^{L_{s_3}^3}(h_3)=\left(\sum_{j=1}^{\frac{n-1}{2}}
b_{\frac{n+1-2j}{2},\frac{n+1}{2}}\lambda_{3,j}\right)t^{s_3+\frac{n+1}{2}q}=0.
\end{eqnarray*}
Asking for $\rho_{s_3}^{\{0\}}(h_3)+\rho_{s_3}^{L_{s_3}^2}(h_3)=
-\rho_{s_1}^{L_{s_1}^2}(h_1)-\rho_{s_1}^{L_{s_1}^4}(h_1)$ 
induces the linear system:
\begin{eqnarray*}
\sum_{k\in \{0\}\sqcup [2,\frac{n-1}{2}]}
\left(\sum_{j=1}^{\frac{n-1}{2}}b_{\frac{n+1-2j}{2},k}\lambda_{3,j}\right)t^{s_3+kq}
=\sum_{k\in \{0\}\sqcup [2,\frac{n-1}{2}]}
b_{n,\frac{n+1}{2}+k} t^{s_3+kq}.
\end{eqnarray*}
In matrix form, 
$P_{[1,\frac{n-1}{2}]}^{\{0\}\sqcup[2,\frac{n-1}{2}]}\cdot \Lambda_3=
P_{\{n\}}^{\{\frac{n+1}{2}\}\sqcup[\frac{n+5}{2},n]}$, where 
$\Lambda_3^\top=\left(\lambda_{3,1},\ldots,\lambda_{3,\frac{n-1}{2}}\right)$. By Notation~\ref{notation-binomials}, 
this is equivalent to:
\begin{eqnarray*}
\left(B^{[1,\frac{n-1}{2}]}
_{\{0\}\sqcup[2,\frac{n-1}{2}]}\right)^\top\cdot
\Theta_{\frac{n-1}{2}}\cdot\Lambda_3=
\left(B^{\{n\}}_{[\frac{n+1}{2},n]\setminus\{\frac{n+3}{2}\}}\right)^\top
\cdot\Theta_1=
\left(B^{\{n\}}_{[\frac{n+1}{2},n]\setminus\{\frac{n+3}{2}\}}\right)^\top.
\end{eqnarray*}
Observe that $B^{[1,\frac{n-1}{2}]}_{\{0\}\sqcup[2,\frac{n-1}{2}]}$ 
is a square $\frac{n-1}{2}\times\frac{n-1}{2}$ matrix and 
$\{0\}\sqcup[2,(n-1)/2]\leq[1,(n-1)/2]$. 
Therefore, by \cite[Corollary~2]{gv} or \cite[Corollary~2.5]{gp2}, 
$B^{[1,\frac{n-1}{2}]}_{\{0\}\sqcup[2,\frac{n-1}{2}]}$ is an invertible matrix. 
Let $\Lambda_3$ be the unique solution of the system and let $h_3$ be the polynomial in $W_{s_3}$, 
with components $\Lambda_3$ in the linearly independent set 
$\{\mon^{\phi_{s_3}+j\omega}\mid 1\leq j\leq (n-1)/2\}\subset\mcb_{s_3}$. Then,
\begin{eqnarray*}
\rho(g+h_1+h_2+h_3)=\rho_{s_1}^{L_{s_1}^3}(h_1)+\rho_{s_3}^{L_{s_3}^1}(h_3)=
\left(-b_{n,\frac{n+3}{2}}+\sum_{j=1}^{\frac{n-1}{2}}b_{\frac{n+1-2j}{2},1}\cdot\lambda_{3,j}\right)t^{s_4}.
\end{eqnarray*}
Take $h_4=\lambda_{4,1}y^2z^{n-1}$ in $W_{s_4}=\langle xz^n,y^2z^{n-1}\rangle$, 
with $\lambda_{4,1}=b_{n,\frac{n+3}{2}}-\sum_{j=1}^{\frac{n-1}{2}}b_{\frac{n+1-2j}{2},1}\cdot\lambda_{3,j}$. By Lemma~\ref{lemma1-rho},~$(1)$,
$\rho(h_4)=\lambda_{4,1}t^{s_4}$. Then, 
\begin{eqnarray*}
\rho(f)=\rho(g+h_1+h_2+h_3+h_4)=\rho_{s_1}^{L_{s_1}^3}(h_1)+\rho_{s_3}^{L_{s_3}^1}(h_3)+\rho(h_4)=0,
\end{eqnarray*}
as desired.

Finally, by Lemma~\ref{lemma-ic21}, $g$ is equal to the polynomial $g_{n-1}$ defined in 
\eqref{equality-gk21}. The expression of $f=g+h$ as $f_{n-1}$ follows from substituting
the values of $\phi_{s_2}$ and $\phi_{s_3}$ in $h_2$ and $h_3$. 
\end{proof}

\vspace*{0.3cm}

\begin{lemma}\label{lemma-h-22}
Let $r\in [s_0,s_0+n)$. 
Write $r=r_k=s_0+k-1$, with $1\leq k\leq n$.
Suppose that  $(\iota_r,c_r)={\cviolet (2,2)}$, so $n$ is odd, $k=n$ and $r=r_n$. 
Let
\begin{eqnarray*}
s_1=r+\left(\frac{n-3}{2}\right)q,\quad s_2=s_1+q,\quad\text{and}\quad 
s_3=s_1+\left(\frac{n+1}{2}\right)q.
\end{eqnarray*}
Then $s_1,s_2,s_3\in\nums$ and $r<s_1<s_2<s_3<\la$. Moreover,
\begin{eqnarray*}
&&\phi_{s_1}=\left(n-1,1,0\right),
\quad\phi_{s_2}=\left(\frac{n-1}{2},0,\frac{n+1}{2}\right),\quad\phi_{s_3}=\left(\frac{n+1}{2},0,\frac{n+1}{2}\right).
\end{eqnarray*}
Let $g=\sum_{j=0}^{\frac{n-3}{2}}(-1)^jb_{\frac{n-3}{2},\frac{n-3-2j}{2}}\cdot
x^{\frac{n-3-2j}{2}}y^{2j}z^{\frac{n+1-2j}{2}}$ and $h_1=-x^{n-1}y$. Let
$h_2=\sum_{j=1}^{\frac{n-1}{2}}\lambda_{2,j}\mon^{\phi_{s_2}+j\omega}$ in
$W_{s_2}$, where $\Lambda_2^\top=\left(\lambda_{2,1},\ldots,
\lambda_{2,\frac{n-1}{2}}\right)$ satisfies
\begin{eqnarray*}
&&\left(B^{[0,\frac{n-3}{2}]}_{[0,\frac{n-3}{2}]}\right)^\top\cdot
\Theta_{\frac{n-1}{2}}\cdot\Lambda_2= \left(B^{\{n-1\}}_{[1,\frac{n-1}{2}]}\right)^\top.
\end{eqnarray*}
Let
$h_3=\sum_{j=2}^{\frac{n+1}{2}}\lambda_{3,j}\mon^{\phi_{s_3}+j\omega}\in
W_{s_3}$, where $\Lambda_{3}^\top=\left(\lambda_{3,2},\ldots,\lambda_{3,\frac{n+1}{2}}\right)$ 
satisfies
\begin{eqnarray*}
\left(B^{[0,\frac{n-3}{2}]}
_{[0,\frac{n-3}{2}]}\right)^\top\cdot
\Theta_{\frac{n-1}{2}}\cdot\Lambda_3= \left(B^{\{n-1\}}_{[\frac{n+1}{2},n-1]}\right)^\top.
\end{eqnarray*} 
Set $h=h_1+h_2+h_3$ and  $f=g+h$. 
Then $g$ is a basis of $\ker(\rho_r^{{\rm G}_r})$, $h_i\in W_{s_i}$, $i=1,2,3$,
\begin{eqnarray*}
f^\sigma=g,\phantom{+} f^\tau=h,\phantom{+}
f\in\ker(\rho)\phantom{+} \text{and}\phantom{+} g\in V_r.
\end{eqnarray*}
Hence, $V_r=\ker(\rho_r^{{\rm G}_r})$. Moreover, 
$g$ is equal to the polynomial $g_{n}$ defined in \eqref{equality-gk22}
and $f=g+h$ is equal to the polynomial $f_{n}$ defined as in \eqref{poly-fk22}.
In particular, $f_{n}\in\ker(\rho)$ and $f_{n}^{\sigma}$ is a basis of $V_{r_{n}}$. 
\end{lemma}
\begin{proof}  
By Proposition~\ref{proposition-class}, if $(\iota_{r_k},c_{r_k})={\cviolet (2,2)}$, 
then $n$ is odd, $k=n$ and $r=r_{n}$. 
Observe that $r=(a+1)\ell_r+c_r=(a+1)(n-1)+2=a(n-1)+n+1$. Clearly, $r<s_1<s_2<s_3$. Let us prove that $s_3<\la$. Since $s_1=r+((n-3)/2)q$, it follows that $s_3=s_1+((n+1)/2)q=r+(n-1)q$. Since $n$ is odd, $n\geq 7$. By Lemma~\ref{lemma-cota},~$(1)$,
$\lfloor a/2\rfloor-n\geq 3$. Hence, 
$n-1=2(n-1)/2<3(n-1)/2\leq(\lfloor a/2\rfloor-n)(n-1)/2=(\lfloor a/2\rfloor-n)\lfloor n/2\rfloor$. 
By Lemma~\ref{lemma-cota},~$(2)$, it follows that $s_3<\frob(\nums)<\la$. In particular, 
$s_1,s_2<\la$.  

By Lemma~\ref{lemma-ic22}, $\delta_r=(n-1)/2$, so
$s_1=r+((n-3)/2)q=r+(\delta_r-1)q\in\nums$, $s_1<\la$ and $\phi_{s_1}=(n-1,1,0)$. In particular, $\kappa_{s_1}=
\min(\phi_{s_1,1},\phi_{s_1,3})=0$, $I_{s_1}=[\phi_{s_1,1}-\kappa_{s_1},\phi_{s_1,1}]=\{n-1\}$,  $H_{s_1}=[0,\phi_{s_1,1}]=[0,n-1]$, $W_{s_1}=\langle x^{n-1}y\rangle$ and $h_1\in W_{s_1}$.

Let us prove that $s_2\in\nums$. Since $n$ is odd, $q=n$ and $((n-1)/2)q=a$. Note
that $s_2=s_1+q=r+((n-1)/2)q=a(n-1)+n+1+a=an+n+1$. So, $s_2=an+(n+1)$, with 
$0\leq n+1\leq 2n$. By Proposition~\ref{theo31ggp}, $s_2\in\nums$, 
and since $s_2<\la$, then $\ell_{s_2}=n$, $\rem_{s_2}=n+1$ and
$\phi_{s_2}=((n-1)/2,0,(n+1)/2)$. In particular, 
$\kappa_{s_2}=\min(\phi_{s_2,1},\phi_{s_2,3})=
(n-1)/2$, $I_{s_2}=[0,(n-1)/2]$, $H_{s_2}=[0,(n-1)/2]$,
$W_{s_2}=\langle\mon^{\phi_{s_2}+j\omega}\mid 0\leq j\leq(n-1)/2\rangle$,
$h_2$ is well-defined and $h_2\in W_{s_2}$. 

Let us prove that $s_3\in\nums$. Again, since $((n-1)/2)q=a$, 
$s_3=r+(n-1)q=a(n-1)+n+1+2a=a(n+1)+n+1$, where 
$0\leq n+1\leq 2(n+1)$. By Proposition~\ref{theo31ggp}, $s_3\in\nums$, 
and since $s_3<\la$, then $\ell_{s_3}=n+1$, $\rem_{s_3}=n+1$ and
$\phi_{s_3}=((n+1)/2,0,(n+1)/2)$. In particular, 
$\kappa_{s_3}=\min(\phi_{s_3,1},\phi_{s_3,3})=
(n+1)/2$,  $I_{s_3}=[0,(n+1)/2]$, $H_{s_3}=[0,(n+1)/2]$, 
$W_{s_3}=\langle\mon^{\phi_{s_3}+j\omega}\mid 0\leq j\leq(n+1)/2\rangle$,
$h_3$ is well-defined and $h_3\in W_{s_3}$. 

By Lemma~\ref{lemma-ic22}, g is a basis of $\ker(\rho_r^{\rm G})$.
Since $r<s_1<s_2<s_3$, then $f^\sigma=g$ and $f^\tau=h$. 
Let us prove that $\rho(f)=0$, so $g\in V_r$ and $V_r=\ker(\rho_r^{{\rm G}_r})$. 

By Lemma~\ref{lemma-ic22}, $I_r=[0,(n-3)/2]$, $H_r=[0,(n-3)/2]$ and
$\good_r=[0,(n-5)/2]$. 

Set $L_r^1=\{(n-3)/2\}$, so that
$H_r=\good_r\sqcup L_r^1$. Since $g\in\ker(\rho_r^{{\rm G}_r})$, then, by Lemma~\ref{lemma3-rho},~$(8)$
\begin{eqnarray*}
  \rho(g)=\rho_r^{{\rm G}_r}(g) +\rho_r^{H_r\setminus {{\rm G}_r}}(g)=
  \rho_r^{H_r\setminus {{\rm G}_r}}(g)=
  \rho_r^{L_r^1}(g).
\end{eqnarray*}
Using Remark~\ref{rem-comp}, $s_1=r+\tfrac{n-3}{2}q$ and 
$b_{\frac{n-3-2j}{2},\frac{n-3}{2}}=0$, for all $j\geq 1$, then
\begin{eqnarray*}
\rho_r^{L_r^1}(g)=\hspace*{-1.5mm}\left(\sum_{j=0}^{\frac{n-3}{2}}b_{\frac{n-3-2j}{2},\frac{n-3}{2}}\cdot(-1)^j\cdot b_{\frac{n-3}{2},\frac{n-3-2j}{2}}\hspace*{-1.5mm}\right)t^{r+\frac{n-3}{2}q}=b_{\frac{n-3}{2},\frac{n-3}{2}}\cdot(-1)^0\cdot b_{\frac{n-3}{2},\frac{n-3}{2}}\,
t^{s_1}=t^{s_1}
\end{eqnarray*}
Set $L_{s_1}^1=[1,(n-1)/2]$ and $L_{s_1}^2=[(n+1)/2,n-1]$. Note that $H_{s_1}=\{0\}\sqcup L_{s_1}^1\sqcup L_{s_1}^2$. 

Let us find $h_1=\lambda_{1,0}x^{n-1}y\in W_{s_1}$ such that $\rho_r^{L_r^1}(g)+\rho_{s_1}^{\{0\}
}(h)=0$. By Remark~\ref{rem-comp}, $\rho_{s_1}^{\{0\}}(h_1)=\lambda_{1,0}t^{s_1}$. Thus, $\lambda_{1,0}=-1$. Therefore, 
\begin{eqnarray*}
\rho(g+h_1)=\rho_r^{L_r^1}(g)+\rho_{s_1}^{\{0\}}(h_1)+\rho_{s_1}^{L_{s_1}^1}(h_1)+\rho_{s_1}^{L_{s_1}^2}(h_1)
=\rho_{s_1}^{L_{s_1}^1}(h_1)+\rho_{s_1}^{L_{s_1}^2}(h_1).
\end{eqnarray*}
Using Remark~\ref{rem-comp} and the equalities $s_2=s_1+q$ and $s_3=s_1+((n+1)/2)q$, it follows that 
\begin{eqnarray*}
&&\rho_{s_1}^{L_{s_1}^1}(h_1)=\sum_{k=1}^{\frac{n-1}{2}}-b_{n-1,k} t^{s_1+kq}=
\sum_{k=0}^{\frac{n-3}{2}}-b_{n-1,k+1} t^{s_2+kq},\\
&&\rho_{s_1}^{L_{s_1}^2}(h_1)=\sum_{k=\frac{n+1}{2}}^{n-1}-b_{n-1,k} t^{s_1+kq}=
\sum_{k=0}^{\frac{n-3}{2}}-b_{n-1,k+\frac{n+1}{2}} t^{s_3+kq}.
\end{eqnarray*}
Set $L_{s_2}^1=[0,(n-3)/2]$ and $L_{s_2}^2=\{(n-1)/2\}$, so that 
$H_{s_2}=L_{s_2}^1\sqcup L_{s_2}^2$.  

Let us search $h_2\in W_{s_2}$ such that $\rho_{s_1}^{L_{s_1}^1}(h_1)+\rho(h_2)=0$. 
Let $h_2$ be a linear combination of the last $(n-1)/2$ monomials of the basis $\mcb_{s_2}$: $h_2=\sum_{j=1}^{\frac{n-1}{2}}\lambda_{2,j}\mon^{\phi_{s_2+j\omega}}$. 
By Remark~\ref{rem-comp}, and using that 
$b_{\frac{n-1-2j}{2},\frac{n-1}{2}}=0$, for all $j\geq 1$, we get
\begin{eqnarray*}
&&\rho_{s_2}^{L_{s_2}^1}(h_2)=\sum_{k=0}^{\frac{n-3}{2}}
\left(\sum_{j=1}^{\frac{n-1}{2}}
b_{\frac{n-1-2j}{2},k}\lambda_{2,j}\right)t^{s_2+kq}, \\
&&\rho_{s_2}^{L_{s_2}^2}(h_2)=
\left(\sum_{j=1}^{\frac{n-1}{2}}
b_{\frac{n-1-2j}{2},\frac{n-1}{2}}\lambda_{2,j}\right)t^{s_2+\frac{n-1}{2}q}=0.
\end{eqnarray*}
Therefore, $\rho_{s_1}^{L_{s_1}^1}(h_1)+\rho(h_2)=0$ 
is equivalent to $\rho_{s_2}^{L_{s_2}^1}(h_2)=-\rho_{s_1}^{L_{s_1}^1}(h_1)$, 
which induces the linear system:
\begin{eqnarray*}
\sum_{k=0}^{\frac{n-3}{2}}
\left(\sum_{j=1}^{\frac{n-1}{2}}
b_{\frac{n-1-2j}{2},k}\lambda_{2,j}\right)t^{s_2+kq}=
\sum_{k=0}^{\frac{n-3}{2}}b_{n-1,k+1} t^{s_2+kq}.
\end{eqnarray*}
In matrix form, $P_{[0,\frac{n-3}{2}]}^{[0,\frac{n-3}{2}]}\cdot\Lambda_2=P_{\{n-1\}}^{[1,\frac{n-1}{2}]}$, where $\Lambda_2^\top=\left(\lambda_{2,1},\ldots,\lambda_{2,\frac{n-1}{2}}\right)$. By Notation~\ref{notation-binomials}, this is equivalent to:
\begin{eqnarray*}
\left(B^{[0,\frac{n-3}{2}]}_{[0,\frac{n-3}{2}]}\right)^\top
\cdot\Theta_{\frac{n-1}{2}}\cdot\Lambda_2=
\left(B^{\{n-1\}}_{[1,\frac{n-1}{2}]}\right)^\top\cdot\Theta_1=
\left(B^{\{n-1\}}_{[1,\frac{n-1}{2}]}\right)^\top.
\end{eqnarray*}
Note that $B_{[0,\frac{n-3}{2}]}^{[0,\frac{n-3}{2}]}$ is a square $\frac{n-1}{2}\times\frac{n-1}{2}$ lower triangular matrix with $1$'s in the diagonal, so invertible. Let $\Lambda_2$ be the 
unique solution of the system above and let $h_2$ be the polynomial in $W_{s_2}$, with components $\Lambda_2$ in the linearly independent set 
$\{\mon^{\phi_{s_2}+j\omega}\mid 1\leq j\leq (n-1)/2\}\subset\mathcal{B}_{s_2}$. Then, 
\begin{eqnarray*}
\rho(g+h_1+h_2)=\rho_{s_1}^{L_{s_1}^2}(h_1).
\end{eqnarray*}
Set $L_{s_3}^1=[0,(n-3)/2]$ and $L_{s_3}^2=\{(n-1)/2,(n+1)/2\}$. Thus, 
$H_{s_3}=L_{s_3}^1\sqcup L_{s_3}^2$. 

Let us search $h_3\in W_{s_3}$ such that $\rho_{s_1}^{L_{s_1}^2}(h_1)+\rho(h_3)=0$. 
We choose $h_3$ a linear combination of the last $(n-1)/2$ monomials of the basis $\mcb_{s_3}$: $h_3=\sum_{j=2}^{\frac{n+1}{2}}\lambda_{3,j}\mon^{\phi_{s_3}+j\omega}$. 
By Remark~\ref{rem-comp}, and using that 
$b_{\frac{n+1-2j}{2},\frac{n-1}{2}}=0$ and $b_{\frac{n+1-2j}{2},\frac{n+1}{2}}=0$,
for all $j\geq 2$, we get:
\begin{eqnarray*}
&&\rho_{s_3}^{L_{s_3}^1}(h_3)=\sum_{k=0}^{\frac{n-3}{2}}
\left(\sum_{j=2}^{\frac{n+1}{2}}
b_{\frac{n+1-2j}{2},k}\lambda_{3,j}\right)t^{s_3+kq}, \\
&&\rho_{s_3}^{L_{s_3}^2}(h_3)=\sum_{k=\frac{n-1}{2}}^{\frac{n+1}{2}}
\left(\sum_{j=2}^{\frac{n+1}{2}}
b_{\frac{n+1-2j}{2},k}\lambda_{3,j}\right)t^{s_3+kq}=0.
\end{eqnarray*}
Therefore, $\rho(h_3)=\rho_{s_3}^{L_{s_3}^1}(h_3)+\rho_{s_3}^{L_{s_3}^2}(h_3)
=\rho_{s_3}^{L_{s_3}^1}(h_3)$. Thus, $\rho_{s_1}^{L_{s_1}^2}(h_1)+\rho(h_3)=0$
is equivalent to $\rho_{s_3}^{L_{s_3}^1}(h_3)=-\rho_{s_1}^{L_{s_1}^2}(h_1)$. 
This induces the linear system:
\begin{eqnarray*}
\sum_{k=0}^{\frac{n-3}{2}}
\left(\sum_{j=2}^{\frac{n+1}{2}}
b_{\frac{n+1-2j}{2},k}\lambda_{3,j}\right)t^{s_3+kq}=
\sum_{k=0}^{\frac{n-3}{2}}b_{n-1,k+\frac{n+1}{2}} t^{s_3+kq}.
\end{eqnarray*}
In matrix form, $P_{[0,\frac{n-3}{2}]}^{[0,\frac{n-3}{2}]}\cdot\Lambda_3=P_{\{n-1\}}^{[\frac{n+1}{2},n-1]}$, where $\Lambda_3^\top=\left(\lambda_{3,2},\ldots,\lambda_{3,\frac{n+1}{2}}\right)$. By Notation~\ref{notation-binomials}, this is equivalent to:
\begin{eqnarray*}
\left(B^{[0,\frac{n-3}{2}]}_{[0,\frac{n-3}{2}]}\right)^\top
\cdot\Theta_{\frac{n-1}{2}}\cdot\Lambda_3=
\left(B^{\{n-1\}}_{[\frac{n+1}{2},n-1]}\right)^\top\cdot\Theta_1=
\left(B^{\{n-1\}}_{[\frac{n+1}{2},n-1]}\right)^\top.
\end{eqnarray*}
Note that $B_{[0,\frac{n-3}{2}]}^{[0,\frac{n-3}{2}]}$ is a square $\frac{n-1}{2}\times\frac{n-1}{2}$ lower triangular matrix with $1$'s in the diagonal, so it is invertible. 
Let $\Lambda_3$ be the unique solution of the former system and let $h_3$ be the polynomial
in $W_{s_3}$, with components $\Lambda_3$ in the linearly independent subset 
$\{\mon^{\phi_{s_3}+j\omega}\mid 2\leq j\leq (n+1)/2\}$ of $\mathcal{B}_{s_3}$. Then, $\rho(g+h_1+h_2+h_3)=0$, as desired. 

Finally, by Lemma~\ref{lemma-ic22}, $g$ is equal to the polynomial $g_{n}$ defined in 
\eqref{equality-gk22}. The expression of $f=g+h$ as $f_{n}$ follows from substituting
the values of $\phi_{s_2}$ and $\phi_{s_3}$ in $h_2$ and $h_3$. 
\end{proof}

\vspace*{0.3cm}

\begin{lemma}\label{lemma-h-1-1}
Let $r\in [s_0,s_0+n)$. 
Write $r=r_k=s_0+k-1$, with $1\leq k\leq n$.
Suppose that  $(\iota_r,c_r)={\cred (1,-1)}$, so n is even, $k=n-3$ and $r=r_{n-3}$. 
Let
\begin{eqnarray*}
s_1=r+q,\quad s_2=r+\left(\frac{n}{2}\right)q,\quad\text{and}\quad  s_3=s_2+q.
\end{eqnarray*}
Then $s_1,s_2,s_3\in\nums$ and $r<s_1<s_2<s_3<\la$. Moreover,
\begin{eqnarray*}
&&\phi_{s_1}=\left(0,1,n-2\right),
\quad\phi_{s_2}=\left(\frac{n+2}{2},0,\frac{n-2}{2}\right),\quad\phi_{s_3}=\left(1,1,n-2\right).
\end{eqnarray*}
Let $g=\sum_{j=0}^{\frac{n-2}{2}}(-1)^jb^{\left[1,\frac{n}{2}\right]
\setminus\{\frac{n-2j}{2}\}}
_{\{0\}\cup\left[2,\frac{n-2}{2}\right]}\cdot 
x^{\frac{n-2j}{2}}y^{2j}z^{\frac{n-2-2j}{2}}$. Let $h_1=\lambda_{1,0}yz^{n-2}$ where 
\begin{eqnarray*}
\lambda_{1,0}=-\sum_{j=0}^{\frac{n-2}{2}}(-1)^jb_{\{0\}\sqcup[2,\frac{n-2}{2}]}^{[1,\frac{n}{2}]\setminus\{\frac{n-2j}{2}\}}b_{\frac{n-2j}{2},1}.    
\end{eqnarray*}
Let $h_2=-x^2y^{n-2}$ and $h_3=xyz^{n-2}+y^3z^{n-3}$. Set $h=h_1+h_2+h_3$ and  $f=g+h$. 
Then $g$ is a basis of $\ker(\rho_r^{{\rm G}_r})$,
$h_i\in W_{s_i}$, $i=1,2,3$,
\begin{eqnarray*}
f^\sigma=g,\phantom{+} f^\tau=h,\phantom{+}
f\in\ker(\rho)\phantom{+} \text{and}\phantom{+} g\in V_r.
\end{eqnarray*}
Hence, $V_r=\ker(\rho_r^{{\rm G}_r})$. Moreover, 
$g$ is equal to the polynomial $g_{n-3}$ defined in \eqref{equality-gk1-1}
and $f=g+h$ is equal to the polynomial $f_{n-2}$ defined as in \eqref{poly-fk1-1}.
In particular, $f_{n-3}\in\ker(\rho)$ and $f_{n-3}^{\sigma}$ is a basis of $V_{r_{n-3}}$.
\end{lemma}
\begin{proof}
By Proposition~\ref{proposition-class}, if $(\iota_{r_k},c_{r_k})={\cred (1,-1)}$, 
then $n$ is odd, $k=n-3$ and $r=r_{n-3}$. 
Observe that $r=(a+1)\ell_r+c_r=(a+1)(n-1)-1=a(n-1)+n-2$. Clearly, $r<s_1<s_2<s_3$. Let us prove that $s_3<\la$. Since $n$ is even, then $q=n-1$ and $\tfrac{n}{2}q=a$. 
Then, $s_3=s_2+q=r+\tfrac{n}{2}q+(n-1)=
a(n-1)+n-2+a+n-1=an+2n-3$. If $n=6$, then $a=15$ and 
$s_3=an+2n-3=99$ and $\la=(\lfloor a/2\rfloor+2)a=105$.
If $n\geq 8$, by Lemma~\ref{lemma-cota},~$(1)$, $\lfloor a/2\rfloor-n\geq 3$. Hence 
$\tfrac{n+2}{2}<3\tfrac{n}{2}\leq(\lfloor\tfrac{a}{2}\rfloor-n)\tfrac{n}{2}=
(\lfloor \tfrac{a}{2}\rfloor-n)\lfloor \tfrac{n}{2}\rfloor$. By Lemma~\ref{lemma-cota},~$(2)$, it follows that $s_3=s_2+q=r+\tfrac{n+2}{2}q<\frob(\nums)<\la$. In particular, $s_1,s_2<\la$.

By Lemma~\ref{lemma-ic1-1}, $\delta_r=n/2$, so $s_2=r+\delta_rq$, and
$s_1,s_2\in\nums$, $\phi_{s_1}=(0,1,n-2)$ and $\phi_{s_2}=((n+2)/2,0,(n-2)/2)$. 
In particular, $\kappa_{s_1}=
\min(\phi_{s_1,1},\phi_{s_1,3})=0$, 
$I_{s_1}=[\phi_{s_1,1}-\kappa_{s_1},\phi_{s_1,1}]=\{0\}$, 
$H_{s_1}=[0,\phi_{s_1,1}]=\{0\}$, $W_{s_1}=\langle yz^{n-2}\rangle$ and $h_1\in W_{s_1}$. 
Similarly, 
$\kappa_{s_2}=(n-2)/2$, $I_{s_2}=[1,(n+2)/2]$, $H_{s_2}=[0,(n+2)/2]$, $W_{s_2}=\langle\mon^{\phi_{s_2}+j\omega}\mid 0\leq j\leq(n-2)/2\rangle$. Taking
$j=(n-2)/2$, then $h_2=-x^2y^{n-2}\in W_{s_2}$.

Let us prove that $s_3\in\nums$. We have seen $s_3=an+(2n-3)$, where
$0\leq 2n-3\leq 2n$. By Proposition~\ref{theo31ggp}, $s_3\in\nums$, and 
since $s_3<\la$, then $\ell_{s_3}=n$, $\rem_{s_3}=2n-3$ and
$\phi_{s_3}=(1,1,n-2)$. In particular, 
$\kappa_{s_3}=\min(\phi_{s_3,1},\phi_{s_3,3})=1$, 
$I_{s_3}=[0,1]$, $H_{s_3}=[0,1]$, 
$W_{s_3}=\langle xyz^{n-2},y^3z^{n-3}\rangle$ and $h_3=xyz^{n-2}+y^3z^{n-3}\in W_{s_3}$.

By Lemma~\ref{lemma-ic1-1}, $g$ is a basis of $\ker(\rho_r^{\rm G})$.
Since $r<s_1<s_2<s_3$, then $f^\sigma=g$ and $f^\tau=h$. 
Let us prove that $\rho(f)=0$, so $g\in V_r$ and $V_r=\ker(\rho_r^{{\rm G}_r})$.

By Lemma~\ref{lemma-ic1-1}, $I_r=[1,n/2]$, 
$H_r=[0,\phi_{r,1}]=[0,n/2]$ and
$\good_r=\{0\}\cup\left[2,\frac{n-2}{2}\right]$. 

Set $L_r^1=\{1\}$ and $L_r^2=\{n/2\}$, so that
$H_r=\good_r\sqcup L_r^1\sqcup L_r^2$. Since $g\in\ker(\rho_r^{{\rm G}_r})$, then, by Lemma~\ref{lemma3-rho},~$(8)$,
\begin{eqnarray*}
  \rho(g)=\rho_r^{{\rm G}_r}(g) +\rho_r^{H_r\setminus {{\rm G}_r}}(g)=
  \rho_r^{H_r\setminus {{\rm G}_r}}(g)=
  \rho_r^{L_r^1}(g)+\rho_r^{L_r^2}(g). 
\end{eqnarray*}
By Remark~\ref{rem-comp}, and using that $s_1=r+q$, $s_2=r+\tfrac{n}{2}q$, $b_{\frac{n-2j}{2},\frac{n}{2}}=0$, for all $j\geq 1$, and $b_{\{0\}\sqcup[2,\frac{n-2}{2}]}^{[1,\frac{n-2}{2}]}=1$ ($B_{\{0\}\sqcup[2,\frac{n-2}{2}]}^{[1,\frac{n-2}{2}]}=1$ is a lower triangular matrix with 1's in the diagonal), we get:
\begin{eqnarray*}
&&\rho_r^{L_r^1}(g)=
\left(\sum_{j=0}^{\frac{n-2}{2}}
b_{\frac{n-2j}{2},1}\cdot
(-1)^j\cdot b_{\{0\}\sqcup[2,\frac{n-2}{2}]}^
{[1,\frac{n}{2}]\setminus\{\frac{n-2j}{2}\}}
\right)t^{r+q}=
\left(\sum_{j=0}^{\frac{n-2}{2}}
b_{\frac{n-2j}{2},1}\cdot (-1)^j\cdot
b_{\{0\}\sqcup[2,\frac{n-2}{2}]}^
{[1,\frac{n}{2}]\setminus\{\frac{n-2j}{2}\}}\right)t^{s_1},\\
&&\rho_r^{L_r^2}(g)=
\left(\sum_{j=0}^{\frac{n-2}{2}}
b_{\frac{n-2j}{2},\frac{n}{2}}\cdot (-1)^j\cdot
b_{\{0\}\sqcup[2,\frac{n-2}{2}]}^{[1,\frac{n}{2}]\setminus\{\frac{n-2j}{2}\}}
\right)t^{r+\frac{n}{2}q}=
b_{\frac{n}{2},\frac{n}{2}}\cdot (-1)^0\cdot
b^{[1,\frac{n-2}{2}]}_{\{0\}\sqcup[2,\frac{n-2}{2}]}
\cdot t^{s_2}=t^{s_2}.
\end{eqnarray*}
We want to find an $h_1\in W_{s_1}$ such that 
$\rho_{r}^{L_r^1}(g)+\rho(h_1)=0$. Let $h_1=\lambda_{1,0}yz^{n-2}$ such that 
$$\lambda_{1,0}=-\sum_{j=0}^{\frac{n-2}{2}}(-1)^jb_{\{0\}\sqcup[2,\frac{n-2}{2}]}^{[1,\frac{n}{2}]\setminus\{\frac{n-2j}{2}\}}b_{\frac{n-2j}{2},1}.$$
Then, $\rho(g+h_1)=\rho_r^{L_r^1}(g)+\rho_r^{L_r^2}(g)+\rho(h_1)=\rho_r^{L_r^2}(g)$.
Set $L_{s_2}^1=\{1,2\}$ and $L_{s_2}^2=[3,(n+2)/2]$, so that $H_{s_2}=\{0\}\sqcup L_{s_2}^1\sqcup L_{s_2}^2$.  
Let us find $h_2=\lambda_{2,0}x^2y^{n-2}\in W_{s_2}$ such that $\rho_r^{L_r^2}(g)+\rho_{s_2}^{\{0\}}(h_2)=0$. 

By Remark~\ref{rem-comp}, $\rho_{s_2}^{\{0\}}(h_2)=
b_{\tfrac{n+2}{2}-\tfrac{n-2}{2},0}\lambda_{2,0}t^{s_2}=b_{2,0}\lambda_{2,0}t^{s_2}$. Take 
$\lambda_{2,0}=-1$. Then, we get 
\begin{eqnarray*}
\rho(g+h_1+h_2)=\rho_r^{L^2_r}(g)+\rho_{s_2}^{\{0\}}(h_2)+
\rho_{s_2}^{L_{s_2}^1}(h_2)+\rho_{s_2}^{L_{s_2}^2}(h_2)=
\rho_{s_2}^{L_{s_2}^1}(h_2)+\rho_{s_2}^{L_{s_2}^2}(h_2).
\end{eqnarray*}
By Remark~\ref{rem-comp}, and using that $\lambda_{2,0}=-1$, $s_3=s_2+q$ and
$b_{2,k}=0$, for all $k\geq 3$, it follows that
\begin{eqnarray*}
&&\rho_{s_2}^{L_{s_2}^1}(g)=\sum_{k=1}^{2}b_{2,k}\lambda_{2,0}t^{s_2+kq}
=-2 t^{s_2+q}-t^{s_2+2q}=-2t^{s_3}-t^{s_3+q},\\
&&\rho_{s_2}^{L_{s_2}^2}(h_2)=\sum_{k=3}^{\frac{n+2}{2}}b_{2,k}\lambda_{2,0}t^{s_2+kq}=0.
\end{eqnarray*}
Therefore, 
\begin{eqnarray*}
\rho(g+h_1+h_2)=\rho_{s_2}^{L_{s_2}^1}(g)=-2t^{s_3}-t^{s_3+q}.
\end{eqnarray*}
 Let us find $h_3=\lambda_{3,0}xyz^{n-2}+\lambda_{3,1}y^3z^{n-3}\in W_{s_3}$ such that 
$\rho_{s_2}^{L_{s_2}^1}(h_2)+\rho(h_3)=0$. By Remark~\ref{rem-comp}, that $\rho(h_3)=\sum_{k=0}^{1}\left(b_{1,k}\lambda_{3,0}+b_{0,k}\lambda_{3,1}\right)t^{s_3+kq}=\left(\lambda_{3,0}+\lambda_{3,1}\right)t^{s_3}+\lambda_{3,0}t^{s_3+q}$. This induces the linear system: 
\begin{eqnarray*}
\left(\lambda_{3,0}+\lambda_{3,1}\right)t^{s_3}+\lambda_{3,0}t^{s_3+q}=2t^{s_3}+t^{s_3+q},
\end{eqnarray*}
whose solution is $\lambda_{3,0}=1$ and $\lambda_{3,1}=1$. So, if $h_3=xyz^{n-2}+y^3z^{n-3}$,
then $\rho(g+h_1+h_2+h_3)=0$, as desired.

Finally, by Lemma~\ref{lemma-ic1-1}, $g$ is equal to the polynomial $g_{n-3}$ defined in 
\eqref{equality-gk1-1}. The expression of $f=g+h$ as $f_{n-3}$ follows from the values of $h_1$, $h_2$ and $h_3$. 
\end{proof}

\vspace*{0.3cm}

\begin{lemma}\label{lemma-h-10}
Let $r\in [s_0,s_0+n)$. Write $r=r_k=s_0+k-1$, with $1\leq k\leq n$.
Suppose that  $(\iota_r,c_r)={\cgreen (1,0)}$, so $n$ is even, $k=n-2$ and $r=r_{n-2}$. 
Let
\begin{eqnarray*}
s_1=r+q,\quad s_2=r+\left(\frac{n-2}{2}\right)q,\quad s_3=s_2+q,\quad s_4=s_2+\left(\frac{n+2}{2}\right)q\quad\mbox{and}\quad s_5=s_4+q.
\end{eqnarray*}
Then $s_1,s_2,s_3,s_4,s_5\in\nums$ and $r<s_1<s_2<s_3<s_4<s_5<\la$. Moreover, $\phi_{s_1}=\left(0,0,n-1\right)$, 
\begin{eqnarray*}
\quad\phi_{s_2}=\left(n,0,0\right),\;
\phi_{s_3}=\left(\frac{n}{2},1,\frac{n-2}{2}\right),\;
\phi_{s_4}=\left(\frac{n+2}{2},1,\frac{n-2}{2}\right)\;\mbox{ and }\,
\phi_{s_5}=\left(2,0,n-1\right).
\end{eqnarray*}
Let $g=xy^{n-3}z-y^{n-1}$ and 
$\tilde{g}=\sum_{j=0}^{\frac{n-2}{2}}(-1)^{j}b_{\frac{n-2}{2},\frac{n-2-2j}{2}}
x^{\frac{n-2-2j}{2}}y^{1+2j}z^{\frac{n-2-2j}{2}}$. Let $h_1=-z^{n-1}$ and $h_2=-x^n$. Let 
$h_3=\sum_{j=0}^{\frac{n-2}{2}}\lambda_{3,j}\mon^{\phi_{s_3}+j\omega}$, where $\Lambda_3^\top=\left(\lambda_{3,0},\ldots,
\lambda_{3,\frac{n-2}{2}}\right)$ satisfies
\begin{eqnarray*}
\left(B^{[1,\frac{n}{2}]}_{[0,\frac{n-2}{2}]}\right)^\top\cdot
\Theta_{\frac{n}{2}}\cdot\Lambda_3=\left(B^{\{n\}}_{[1,\frac{n}{2}]}\right)^\top.
\end{eqnarray*}
If $n=6$, let $h_4=(b_{6,4}-\lambda_{3,0})x^2y^5$. Suppose that $n\geq 8$. Let
$h_4=\sum_{j=2}^{\frac{n-2}{2}}\lambda_{4,j}\mon^{\phi_{s_4}+j\omega}$, where $\Lambda_{4}^\top=\left(\lambda_{4,2},\ldots,\lambda_{4,\frac{n-2}{2}}\right)$ satisfies
\begin{eqnarray*}
\left(B^{[2,\frac{n-2}{2}]}
_{\{0\}\sqcup[3,\frac{n-2}{2}]}\right)^\top\cdot
\Theta_{\frac{n-4}{2}}\cdot\Lambda_4=\left(\begin{array}{c}
b_{n,\frac{n+2}{2}}-\lambda_{3,0} \vspace*{0.1cm} \\
\hline \vspace*{-0.3cm}\\
\left(B^{\{n\}}_{[\frac{n+8}{2},n]}\right)^\top
\end{array}\right).
\end{eqnarray*} 
Let $h_5=\lambda_{5,1}xy^2z^{n-2}+\lambda_{5,2}y^4z^{n-3}$, where 
$\Lambda_5^\top=(\lambda_{5,1},\lambda_{5,2})$ satisfies
\begin{eqnarray*}
\left(B^{[0,1]}_{[0,1]}\right)^\top\cdot\Theta_{2}\cdot\Lambda_5=\left(B^{\{n\}}_{[\frac{n+4}{2},\frac{n+6}{2}]}\right)^\top-\left(B^{[2,\frac{n-2}{2}]}_{[1,2]}\right)^\top\cdot\Theta_{\frac{n-4}{2}}\cdot\Lambda_4
\end{eqnarray*}
Set $\tilde{h}=h_2+h_3+h_4+h_5$, $f=g+h_1$ and $\tilde{f}=\tilde{g}+\tilde{h}$. 
Then $g,\tilde{g}$ is a basis of $\ker(\rho_r^{{\rm G}_r})$,
$h_i\in W_{s_i}$, $i=1,2,3,4,5$,
\begin{eqnarray*}
f^\sigma=g,\phantom{+} f^\tau=h_1,\phantom{+}
f\in\ker(\rho)\phantom{+} \text{and}\phantom{+} g\in V_r;\phantom{+}
\tilde{f}^\sigma=\tilde{g},\phantom{+} \tilde{f}^\tau=\tilde{h},\phantom{+}
\tilde{f}\in\ker(\rho)\phantom{+} \text{and}\phantom{+} \tilde{g}\in V_r.
\end{eqnarray*}
Hence, $V_r=\ker(\rho_r^{{\rm G}_r})$. 
Moreover, $g$ is equal to the polynomial $g_{n-2}$ defined in \eqref{equality-gk10a} 
and $f=g+h_1$ is equal to the polynomial $f_{n-2}$ defined as in \eqref{poly-fk10a};
$\tilde{g}$ is equal to the polynomial $g_{n-1}$ defined in \eqref{equality-gk10b} 
and $\tilde{f}=\tilde{g}+\tilde{h}$ is equal to the polynomial $f_{n-1}$ defined as in 
\eqref{poly-fk10b}.
In particular, $f_{n-2},f_{n-1}\in\ker(\rho)$ and $f_{n-2}^{\sigma},f_{n-1}^{\sigma}$ 
is a basis of $V_{r_{n-2}}$. 
\end{lemma}
\begin{proof}
By Proposition~\ref{proposition-class}, if $(\iota_{r_k},c_{r_k})={\cgreen (1,0)}$, 
then $n$ is odd, $k=n-2$ and $r=r_{n-2}$. 
Observe that $r=(a+1)\ell_r+c_r=(a+1)(n-1)=a(n-1)+n-1$. Clearly, $r<s_1<s_2<s_3<s_4<s_5$. Let us prove that $s_5<\la$. Since $s_4=s_2+((n+2)/2)q$ and $s_2=r+((n-2)/2)q$, it follows that $s_5=s_4+q=r+(n+1)q$. Since $n$ is even, then $q=n-1$. Suppose that $n=6$. Then $a=15$ and $s_5=r+(n+1)q=(a+1)(n-1)+(n+1)(n-1)=115<135=(\lfloor a/2\rfloor+2)a=\la$. If $n\geq 8$, by Lemma~\ref{lemma-cota},~$(1)$, then $\lfloor a/2\rfloor-n\geq 3$. Hence, $n+1<3(n/2)\leq(\lfloor a/2\rfloor-n)(n/2)=(\lfloor a/2\rfloor-n)\lfloor n/2\rfloor$. By Lemma~\ref{lemma-cota},~$(2)$, $s_5=r+(n+1)q<\frob(\nums)<\la$. In particular, $s_1<s_2<s_3<s_4<\la$. 

By Lemma~\ref{lemma-ic10}, $\delta_r=n/2$, so $s_1=r+q$ and $s_2=r+(\delta_r-1)q\in\nums$. Moreover, $\phi_{s_1}=(0,0,n-1)$ and $\phi_{s_2}=(n,0,0)$. In particular, $\kappa_{s_1}=
\min(\phi_{s_1,1},\phi_{s_1,3})=0$, $I_{s_1}=[\phi_{s_1,1}-\kappa_{s_1},\phi_{s_1,1}]=\{0\}$, $H_{s_1}=[0,\phi_{s_1,1}]=\{0\}$, $W_{s_1}=\langle z^{n-1}\rangle$ and $h_1=-z^{n-1}\in W_{s_1}$. Similarly, $\kappa_{s_2}=0$, $I_{s_2}=\{n\}$, $H_{s_2}=[0,n]$, $W_{s_2}=\langle x^n\rangle$ and $h_2=-x^{n}\in W_{s_2}$.

Let us prove that $s_3\in\nums$. Since $n$ is even, $q=n-1$ and $(n/2)q=a$. 
Then, $s_3=s_2+q=r+(n/2)q=a(n-1)+n-1+a=an+(n-1)$, where $0\leq n-1\leq 2n$. By
Proposition~\ref{theo31ggp}, $s_3\in\nums$ and, since $s_3<\la$, then $\ell_{s_3}=n$, 
$\rem_{s_2}=n-1$ and $\phi_{s_2}=(n/2,1,(n-2)/2)$. 
In particular, $\kappa_{s_3}=\min(\phi_{s_3,1},\phi_{s_3,3})=(n-2)/2$, 
$I_{s_3}=[1,n/2]$, $H_{s_3}=[0,n/2]$, 
$W_{s_3}=\langle\mon^{\phi_{s_3}+j\omega}\mid 0\leq j\leq (n-1)/2\rangle$
and $h_3$ is well-defined and $h_3\in W_{s_3}$.

Let us prove that $s_4\in\nums$. Again, $(n/2)q=a$. Then, 
$s_4=s_2+\tfrac{n+2}{2}q=r+\frac{n-2}{2}q+a+q=a(n-1)+n-1+a-q+a+q=a(n+1)+(n-1)$. Thus, $s_4=a(n+1)+(n-1)$, where $0\leq n-1\leq 2(n+1)$. By Proposition~\ref{theo31ggp}, $s_4\in\nums$ and, since $s_4<\la$, then $\ell_{s_4}=n+1$, $\rem_{s_4}=n-1$ and $\phi_{s_4}=((n+2)/2,1,(n-2)/2)$.  In particular,
$\kappa_{s_4}=(n-2)/2$, $I_{s_4}=[2,(n+2)/2]$, $H_{s_4}=[0,(n+2)/2]$, $W_{s_4}=\langle\mon^{\phi_{s_4}+j\omega}\mid 0\leq j\leq (n-2)/2\rangle$ and
$h_4$ is well-defined and $h_4\in W_{s_4}$. 

Let us see that $s_5\in\nums$. We have seen before that $s_5=r+(n+1)q$. 
Since $q=n-1$, then $s_5=a(n-1)+(n-1)+nq+q=a(n-1)+(n-1)+2a+(n-1)=a(n+1)+(2n-2)$, where
$0\leq 2n-2\leq 2(n+1)$. By Proposition~\ref{theo31ggp}, $s_5\in\nums$ and, 
since $s_5<\la$, then $\ell_{s_5}=n+1$, $\rem_{s_5}=2n-2$ and
$\phi_{s_5}=(2,0,n-1)$. In particular,
$\kappa_{s_5}=\min(\phi_{s_5,1},\phi_{s_5,3})=2$, $I_{s_5}=[0,2]$, $H_{s_5}=[0,2]$, 
$W_{s_5}=\langle x^2z^{n-1},xy^2z^{n-2},y^4z^{n-3}\rangle$ and
$h_5$ is well-defined and $h_5\in W_{s_5}$. 

By Lemma~\ref{lemma-ic10}, $I_r=[0,(n-2)/2]$, $H_r=[0,(n-2)/2]$ and
$G_r=\{0\}\cup [2,(n-4)/2]=L_{r,1}\cap L_{r,2}$, where
$L_{r,1}=\{0\}\sqcup[2,(n-2)/2]$ and $L_{r,2}=[0,(n-4)/2]$. 

By the same lemma, $g$ is a basis of 
$\ker\left(\rho_{r}^{L_{r,1}}\right)$ and $\tilde{g}$ is a basis of $\ker\left(\rho_{r}^{L_{r,2}}\right)$. Moreover,  
$\ker(\rho_r^{{\rm G}_r})=\ker(\rho_{r}^{L_{r,1}})\oplus \ker(\rho_r^{L_{r,2}})$. Therefore, 
$g,\tilde{g}$ is a basis of $\ker(\rho_r^{{\rm G}_r})$. 

We have $r<s_1$, $f=g+h_1$, with $g\in W_{r}$ and $h_1\in W_{s_1}$. Thus, $f^\sigma=g$ and $f^\tau=h_1$. Similarly, $r<s_2<s_3<s_4<s_5$, $\tilde{f}=\tilde{g}+h_2+h_3+h_4+h_5$, with $\tilde{g}\in W_{r}$ and
$h_i\in W_{s_i}$. Therefore, $\tilde{f}^\sigma=\tilde{g}$ and $\tilde{f}^\tau=\tilde{h}$. 
If we prove $\rho(f)=0$ and $\rho(\tilde{f})=0$, then one deduces that $g,\tilde{g}\in V_r$ and $V_r=\ker(\rho_r^{{\rm G}_r})$. So let us prove $\rho(f)=0$ and $\rho(\tilde{f})=0$. 

Since $g\in \ker\left(\rho_r^{L_{r,1}}\right)$, by
Lemma~\ref{lemma3-rho},~$(8)$, and by Lemma~\ref{lemma1-rho},~$(1)$ (for instance), we have:
\begin{eqnarray*}
\rho(g)=\rho_r^{L_{r,1}}(g)+\rho_r^{H_r\setminus L_{r,1}}(g)=\rho_r^{\{1\}}(g)=
\rho(xy^{n-3}z)-\rho(y^{n-1})=t^r(1+t^q)-t^r=t^{r+q}=t^{s_1}.
\end{eqnarray*}
Let us find $h_1=\lambda_{1,0}z^{n-1}\in W_{s_1}=\langle z^{n-1}\rangle$ such that $\rho(h_1)+\rho(g)=0$.
By Lemma~\ref{lemma1-rho},~$(1)$ (for instance), $\rho(h_1)=\lambda_{1,0}t^{s_1}$. Taking
$\lambda_{1,0}=-1$, we get $\rho(f)=\rho(g+h_1)=0$, as desired. 

By Lemma~\ref{lemma3-rho},~$(8)$, and since $\tilde{g}\in\ker\left(\rho_r^{L_{r,2}}\right)$, 
$\rho(\tilde{g})=\rho_r^{L_{r,2}}(\tilde{g})+\rho_r^{H_r\setminus L_{r,2}}(\tilde{g})=\rho_r^{\{\frac{n-2}{2}\}}(\tilde{g})$.
By Remark~\ref{rem-comp}, and using that $s_2=r+\tfrac{n-2}{2}q$ and 
$b_{\frac{n-2-2j}{2},\frac{n-2}{2}}=0$, for all $j\geq 1$, we have:
\begin{multline*}
\rho(\tilde{g})=\rho_r^{\{\frac{n-2}{2}\}}(\tilde{g})=
\left(\sum_{j=0}^{\frac{n-2}{2}}b_{\frac{n-2-2j}{2},\frac{n-2}{2}}\cdot(-1)^j
\cdot b_{\frac{n-2}{2},\frac{n-2-2j}{2}}\right) t^{r+\frac{n-2}{2}q}=\\
b_{\frac{n-2}{2},\frac{n-2}{2}}
\cdot (-1)^0\cdot b_{\frac{n-2}{2},\frac{n-2}{2}}\cdot t^{s_2}=t^{s_2}.
\end{multline*}

Set $L_{s_2}^1=[1,n/2]$, $L_{s_2}^2=\{(n+2)/2\}$, $L_{s_2}^3=\{(n+4)/2,(n+6)/2\}$ and $L_{s_2}^4=[(n+8)/2,n]$. If $n=6$, we understand $L_{s_2}^4=\emptyset$.
Note that $H_{s_2}=\{0\}\sqcup L_{s_2}^1\sqcup L_{s_2}^2\sqcup L_{s_2}^3\sqcup L_{s_2}^4$. 

Let us find $h_2=\lambda_{2,0}x^n\in W_{s_2}=\langle x^n\rangle$ such that $\rho(\tilde{g})+\rho_{s_2}^{\{0\}}(h_2)=0$. 
By Remark~\ref{rem-comp}, $\rho_{s_2}^{\{0\}}(h_2)=
b_{n,0}\lambda_{2,0}t^{s_2}=\lambda_{2,0}t^{s_2}$. On taking
$\lambda_{2,0}=-1$, we get: 
\begin{eqnarray*}
\rho(\tilde{g}+h_2)=\rho_{s_2}^{L_{s_2}^1}(h_2)+\rho_{s_2}^{L_{s_2}^2}(h_2)+\rho_{s_2}^{L_{s_2}^3}(h_2)+\rho_{s_2}^{L_{s_2}^4}(h_2).
\end{eqnarray*}
By Remark~\ref{rem-comp}, and using that $h_2=-x^n$, 
$s_3=s_2+q$, $s_4=s_2+\tfrac{n+2}{2}q$ and $s_5=s_2+\tfrac{n+4}{2}q$, 
then:
\begin{eqnarray*}
&&\rho_{s_2}^{L_{s_2}^1}(h_2)=\sum_{k=1}^{\frac{n}{2}}-b_{n,k}t^{s_2+kq}
=\sum_{k=0}^{\frac{n-2}{2}}-b_{n,k+1}t^{s_3+kq};\\
&&\rho_{s_2}^{L_{s_2}^2}(h_2)=-b_{n,\frac{n+2}{2}}t^{s_2+\frac{n+2}{2}q}=
-b_{n,\frac{n+2}{2}}t^{s_4};\\
&&\rho_{s_2}^{L_{s_2}^3}(h_2)=
-b_{n,\frac{n+4}{2}}t^{s_2+\frac{n+4}{2}q}-b_{n,\frac{n+6}{2}}t^{s_2+\frac{n+6}{2}q}=
-b_{n,\frac{n+4}{2}}t^{s_5}-b_{n,\frac{n+6}{2}}t^{s_5+q};\\
&&\rho_{s_2}^{L_{_2}^4}(h_2)=\sum_{k=\frac{n+8}{2}}^n-b_{n,k}t^{s_2+kq}=\sum_{k=3}^{\frac{n-2}{2}}-b_{n,\frac{n+2+2k}{2}}t^{s_4+kq}.
\end{eqnarray*}
Let $L_{s_3}^1=[0,(n-2)/2]$ and $L_{s_3}^2=\{n/2\}$. Note that $H_{s_3}=L_{s_3}^1\sqcup L_{s_3}^2$. 

Let us search $h_3=\sum_{j=0}^{\frac{n-2}{2}}\lambda_{3,j}\mon^{\phi_{s_3}+j\omega}\in W_{s_3}$ such that $\rho_{s_2}^{L_{s_2}^1}(h_2)+\rho_{s_3}^{L_{s_3}^1}(h_3)=0$. 
Using Remark~\ref{rem-comp}, it follows that $\rho_{s_3}^{L_{s_3}^1}(h_3)=\sum_{k=0}^{\frac{n-2}{2}}\left(\sum_{j=0}^{\frac{n-2}{2}}b_{\frac{n-2j}{2},k}\lambda_{3,j}\right)t^{s_3+kq}$. This induces the linear system:
\begin{eqnarray*}
\sum_{k=0}^{\frac{n-2}{2}}\left(\sum_{j=0}^{\frac{n-2}{2}}b_{\frac{n-2j}{2},k}\lambda_{3,j}\right)t^{s_3+kq}=\sum_{k=0}^{\frac{n-2}{2}}b_{n,k+1}t^{s_3+kq}.
\end{eqnarray*}
In matrix form, $P_{[1,\frac{n}{2}]}^{[0,\frac{n-2}{2}]}\cdot\Lambda_{3}=P_{\{n\}}^{[1,\frac{n}{2}]}$, where $\Lambda_3^\top=\left(\lambda_{3,0},\ldots,\lambda_{3,\frac{n-2}{2}}\right)$. By Notation~\ref{notation-binomials}, this is equivalent to:
\begin{eqnarray*}
\left(B^{[1,\frac{n}{2}]}_{[0,\frac{n-2}{2}]}\right)^\top\cdot
\Theta_{\frac{n}{2}}\cdot\Lambda_3=
\left(B^{\{n\}}_{[1,\frac{n}{2}]}\right)^\top\cdot\Theta_{1}=
\left(B^{\{n\}}_{[1,\frac{n}{2}]}\right)^\top.
\end{eqnarray*}
Observe that $B^{[1,\frac{n}{2}]}
_{[0,\frac{n-2}{2}]}$ is a square $\frac{n}{2}\times\frac{n}{2}$ matrix and 
$[0,(n-2)/2]\leq[1,n/2]$. By \cite[Corollary~2]{gv} or \cite[Corollary~2.5]{gp2}, 
we deduce that $B^{[1,\frac{n}{2}]}
_{[0,\frac{n-2}{2}]}$ is an invertible matrix. Let $\Lambda_{3}$ be the unique solution
of the above linear system and let $h_3$ be the polynomial in $W_{s_3}$, 
with components $\Lambda_{s_3}$ in the basis $\mcb_{s_3}$. Then, 
\begin{eqnarray*}
\rho(\tilde{g}+h_2+h_3)=\rho_{s_2}^{L_{s_2}^2}(h_2)+\rho_{s_2}^{L_{s_2}^3}(h_2)+\rho_{s_2}^{L_{s_2}^4}(h_2)+\rho_{s_3}^{L_{s_3}^2}(h_3). 
\end{eqnarray*}
By Remark~\ref{rem-comp}, and using that  $b_{\tfrac{n-2j}{2},\tfrac{n}{2}}=0$, for all $j\geq 1$,
and $s_4=s_3+\tfrac{n}{2}q$, then,
\begin{eqnarray*}
\rho_{s_3}^{L_{s_3}^2}(h_3)=\sum_{j=0}^{\frac{n-2}{2}}b_{\frac{n-2j}{2},\frac{n}{2}}\cdot \lambda_{3,j}\cdot t^{s_3+\frac{n}{2}q}=b_{\frac{n}{2},\frac{n}{2}}\cdot\lambda_{3,0}\cdot t^{s_4}=\lambda_{3,0}t^{s_4}.
\end{eqnarray*}
Set $L_{s_4}^1=\{1,2\}$, $L_{s_4}^2=[3,\tfrac{n-2}{2}]$ and 
$L_{s_4}^3=\{\tfrac{n}{2},\tfrac{n+2}{2}\}$. Observe that 
$H_{s_4}=\{0\}\sqcup L_{s_4}^1\sqcup L_{s_4}^2\sqcup L_{s_4}^3$. 

Let us find $h_4$ in $W_{s_4}$ such that $\rho_{s_2}^{L_{s_2}^2}(h_2)+\rho_{s_3}^{L_{s_3}^2}(h_3)+\rho_{s_4}^{\{0\}}(h_4)=0$ and, if $n\geq 8$, such that
$\rho_{s_2}^{L_{s_2}^4}(h_2)+\rho_{s_4}^{L_{s_4}^2}(h_4)=0$. 
We choose $h_4$ to be a linear combination of the last $(n-4)/2$ monomials of the basis $\mcb_{s_4}$, namely,  $h_4=\sum_{j=2}^{\frac{n-2}{2}}\lambda_{4,j}\mon^{\phi_{s_4}+j\omega}$. 
By Remark~\ref{rem-comp}:  
\begin{eqnarray*}
\rho_{s_4}^{\{0\}}(h_4)=
\left(\sum_{j=2}^{\frac{n-2}{2}}b_{\frac{n+2-2j}{2},0}\lambda_{4,j}\right)t^{s_4}\;\mbox{ and }\;
\rho_{s_4}^{L_{s_4}^2}(h_4)=
\sum_{k=3}^{\frac{n-2}{2}}
\left(\sum_{j=2}^{\frac{n-2}{2}}b_{\frac{n+2-2j}{2},k}\lambda_{4,j}\right)t^{s_4+kq}. 
\end{eqnarray*}
When asking for $\rho_{s_4}^{\{0\}}(h_4)=-\rho_{s_2}^{L_{s_2}^2}(h_2)-\rho_{s_3}^{L_{s_3}^2}(h_3)$ and, if $n\geq 8$, $\rho_{s_4}^{L_{s_4}^2}(h_4)=-\rho_{s_2}^{L_{s_2}^4}(h_2)$, we get the linear system:
\begin{eqnarray*}
&&\left(\sum_{j=2}^{\frac{n-2}{2}}b_{\frac{n+2-2j}{2},0}\lambda_{4,j}\right)t^{s_4}=b_{n,\frac{n+2}{2}}t^{s_4}-\lambda_{3,0}t^{s_4},\\
&&\sum_{k=3}^{\frac{n-2}{2}}\left(\sum_{j=2}^{\frac{n-2}{2}}b_{\frac{n+2-2j}{2},k}\lambda_{4,j}\right)t^{s_4+kq}=\sum_{k=3}^{\frac{n-2}{2}}b_{n,\frac{n+2+2k}{2}}t^{s_4+kq}.
\end{eqnarray*}
If $n=6$, take $\lambda_{4,2}=b_{6,4}-\lambda_{3,0}$, the unique solution of the 
first equation. Suppose that $n\geq 8$. In matrix form:  
\begin{eqnarray*}
P_{[2,\frac{n-2}{2}]}^{\{0\}\sqcup[3,\frac{n-2}{2}]}\cdot\Lambda_4=
\left(\begin{array}{c}
b_{n,\frac{n+2}{2}}-\lambda_{3,0} \vspace*{0.1cm} \\ 
\hline \vspace*{-0.3cm}\\
P_{\{n\}}^{[\frac{n+8}{2},n]}
\end{array}\right),
\end{eqnarray*}
where $\Lambda_{4}^\top=(\lambda_{4,2},\ldots,\lambda_{4,\frac{n-2}{2}})$. 
By Notations~\ref{notation-binomials}, this is equivalent to:
\begin{eqnarray*}
\left(B^{[2,\frac{n-2}{2}]}
_{\{0\}\sqcup[3,\frac{n-2}{2}]}\right)^\top\cdot
\Theta_{\frac{n-4}{2}}\cdot\Lambda_4=\left(\begin{array}{c}
b_{n,\frac{n+2}{2}}-\lambda_{3,0} \vspace*{0.1cm} \\
\hline \vspace*{-0.3cm}\\
\left(B^{\{n\}}_{[\frac{n+8}{2},n]}\right)^\top
\end{array}\right).
\end{eqnarray*} 
Observe that $B^{[2,\frac{n-2}{2}]}
_{\{0\}\sqcup[3,\frac{n-2}{2}]}$ is a square $\frac{n-4}{2}\times\frac{n-4}{2}$ matrix and $\{0\}\sqcup[3,(n-2)/2]\leq[2,(n-2)/2]$. 
By \cite[Corollary~2]{gv}, or \cite[Corollary~2.5]{gp2}, 
$B^{[2,\frac{n-2}{2}]}_{\{0\}\sqcup[3,\frac{n-2}{2}]}$ is an invertible matrix. 
Let $\Lambda_{4}$ be the unique solution of linear system above. Let $h_4$ be
the polynomial in $W_{s_4}$, with components $\Lambda_{s_4}$ in the linearly 
independent subset $\{\mon^{\phi_{s_4}+j\omega}\mid 2\leq j\leq (n-2)/2\}$
of $\mcb_{s_4}$. Then, 
\begin{eqnarray*}
\rho(\tilde{g}+h_2+h_3+h_4)=\rho_{s_2}^{L_{s_2}^3}(h_2)+
\rho_{s_4}^{L_{s_4}^1}(h_4)+\rho_{s_4}^{L_{s_4}^3}(h_4).
\end{eqnarray*}

By Remark~\ref{rem-comp}, and using that $b_{\frac{n+2-2j}{2},\frac{n}{2}}=0$ and $b_{\frac{n+2-2j}{2},\frac{n+2}{2}}=0$, for all $j\geq 2$, we obtain:
\begin{eqnarray*}
\rho_{s_4}^{L_{s_4}^3}(h_4)=\sum_{k=\frac{n}{2}}^{\frac{n+2}{2}}\left(\sum_{j=2}^{\frac{n-2}{2}}b_{\frac{n+2-2j}{2},k}\lambda_{4,j}\right)t^{s_4+kq}=0.
\end{eqnarray*}
Therefore, 
\begin{eqnarray*}
\rho(\tilde{g}+h_2+h_3+h_4)=\rho_{s_2}^{L_{s_2}^3}(h_2)+\rho_{s_4}^{L_{s_4}^1}(h_4).
\end{eqnarray*}
By Remark~\ref{rem-comp}, and using that $s_5=s_4+q$, we get that: 
\begin{eqnarray*}
\rho_{s_4}^{L_{s_4}^1}(h_4)=\sum_{k=1}^{2}\left(\sum_{j=2}^{\frac{n-2}{2}}b_{\frac{n+2-2j}{2},k}\lambda_{4,j}\right)t^{s_4+kq}=\sum_{k=0}^{1}\left(\sum_{j=2}^{\frac{n-2}{2}}b_{\frac{n+2-2j}{2},k+1}\lambda_{4,j}\right)t^{s_5+kq}.
\end{eqnarray*}
Set $L_{s_5}^1=\{0,1\}$ and $L_{s_5}^2=\{2\}$. Note that $H_{s_5}=L_{s_5}^1\sqcup L_{s_5}^2$. 

Let us find $h_5=\lambda_{5,1}xy^2z^{n-2}+\lambda_{5,2}y^4z^{n-3}\in W_{s_5}=
\langle x^2z^{n-1}, xy^2z^{n-2},y^4z^{n-3}\rangle$ such that 
$\rho(h_5)+\rho_{s_2}^{L_{s_2}^3}(h_2)+\rho_{s_4}^{L_{s_4}^1}(h_4)=0$. 
By Remark~\ref{rem-comp}, and since $b_{1,2}=0$, $b_{0,2}=0$, it follows that:
\begin{eqnarray*}
\rho_{s_5}^{L_{s_5}^2}(h_5)=\left(b_{1,2}\lambda_{5,1}+b_{0,2}\lambda_{5,2}\right)t^{s_5+2q}=0.
\end{eqnarray*}
Therefore, $\rho(h_5)=
\rho_{s_5}^{L_{s_5}^1}(h_5)+\rho_{s_5}^{L_{s_5}^2}(h_5)=\rho_{s_5}^{L_{s_5}^1}(h_5)$. 
Thus, $\rho(h_5)+\rho_{s_2}^{L_{s_2}^3}(h_2)+\rho_{s_4}^{L_{s_4}^2}(h_4)=0$ is equivalent to 
$\rho_{s_5}^{L_{s_5}^1}(h_5)=-\rho_{s_2}^{L_{s_2}^3}(h_2)-\rho_{s_4}^{L_{s_4}^1}(h_4)$. 
This induces the linear system:
\begin{eqnarray*}
\sum_{k=0}^1\left(b_{1,k}\lambda_{5,1}+b_{0,k}\lambda_{5,2}\right)t^{s_5+kq}=\sum_{k=0}^1b_{n,\frac{n+4}{2}+k}t^{s_5+kq}-\sum_{k=0}^{1}\left(\sum_{j=2}^{\frac{n-2}{2}}b_{\frac{n+2-2j}{2},k+1}\lambda_{4,j}\right)t^{s_5+kq}.
\end{eqnarray*}
In matrix form, $P_{[0,1]}^{[0,1]}\cdot\Lambda_5=P_{\{n\}}^{[\frac{n+4}{2},\frac{n+6}{2}]}-P_{[2,\frac{n-2}{2}]}^{[1,2]}\cdot\Lambda_4$, 
where $\Lambda_5^\top=(\lambda_{5,1},\lambda_{5,2})$. By Notation~\ref{notation-binomials}, this is equivalent to:
\begin{eqnarray*}
\left(B^{[0,1]}_{[0,1]}\right)^\top\cdot\Theta_{2}\cdot\Lambda_5=\left(B^{\{n\}}_{[\frac{n+4}{2},\frac{n+6}{2}]}\right)^\top-\left(B^{[2,\frac{n-2}{2}]}_{[1,2]}\right)^\top\cdot\Theta_{\frac{n-4}{2}}\cdot\Lambda_4.
\end{eqnarray*}
Observe that $B^{[0,1]}_{[0,1]}$ is a square $2\times 2$ lower triangular matrix with $1$'s in the diagonal, so invertible. Let $\Lambda_5$ be the unique solution of the linear system
and let $h_5$ be the polynomial in $W_{s_5}$,
with components $\Lambda_5$ in the linearly independent subset 
$\{xy^2z^{n-2},y^4z^{n-3}\}$ of $\mcb_{s_5}$. Then, 
\begin{eqnarray*}
\rho(\tilde{f})=\rho(\tilde{g}+h_2+h_3+h_4+h_5)=\rho_{s_2}^{L_{s_2}^3}(h_2)+\rho_{s_4}^{L_{s_4}^2}(h_4)+\rho_{s_5}^{L_{s_5}^1}(h_5)=0, 
\end{eqnarray*}
as desired.  

Finally, by Lemma~\ref{lemma-ic10}, $g$ is equal to the polynomial $g_{n-2}$ defined in 
\eqref{equality-gk10a} and $\tilde{g}$ is equal to the polynomial $g_{n-1}$ defined in 
\eqref{equality-gk10b}. The expression of $f=g+h_1$ as $f_{n-2}$ follows from 
substituting the values of $g_{n-2}$ and $h_1$. The expression of 
$\tilde{f}=\tilde{g}+\tilde{h}$ follows from substituting the values of 
$\phi_{s_3}$ and $\phi_{s_4}$ in $h_3$ and $h_4$.
\end{proof}

\vspace*{0.3cm}

\begin{lemma}\label{lemma-h-11}
Let $r\in [s_0,s_0+n)$. Write $r=r_k=s_0+k-1$, with $1\leq k\leq n$.
Suppose that  $(\iota_r,c_r)={\cblue (1,1)}$, so $n$ is even, $k=n-1$ and $r=r_{n-1}$. 
Let
\begin{eqnarray*}
s_1=r+\left(\frac{n-2}{2}\right)q,\quad s_2=s_1+q,\quad s_3=s_1+\left(\frac{n+2}{2}\right)q\quad\mbox{and}\quad s_4=s_3+q.
\end{eqnarray*}
Then $s_1,s_2,s_3,s_4\in\nums$ and $r<s_1<s_2<s_3<s_4<\la$. Moreover,
\begin{eqnarray*}
&&\phi_{s_1}=\left(n-1,1,0\right),
\quad\phi_{s_2}=\left(\frac{n}{2},0,\frac{n}{2}\right),\quad\phi_{s_3}=\left(\frac{n+2}{2},0,\frac{n}{2}\right),
\quad\phi_{s_4}=\left(1,1,n-1\right).
\end{eqnarray*}
Let $g=\sum_{j=0}^{\frac{n-2}{2}}(-1)^jb_{\frac{n-2}{2},\frac{n-2-2j}{2}}\cdot 
x^{\frac{n-2-2j}{2}}y^{2j}z^{\frac{n-2j}{2}}$ and $h_1=-x^{n-1}y$. Let
$h_2=\sum_{j=1}^{\frac{n}{2}}\lambda_{2,j}\mon^{\phi_{s_2}+j\omega}$, where $\Lambda_2^\top=\left(\lambda_{2,1},\ldots,
\lambda_{2,\frac{n-2}{2}}\right)$ satisfies
\begin{eqnarray*}
&&\left(B^{[0,\frac{n-2}{2}]}_{[0,\frac{n-2}{2}]}\right)^\top\cdot
\Theta_{\frac{n}{2}}\cdot\Lambda_2=\left(B^{\{n-1\}}_{[1,\frac{n}{2}]}\right)^\top.
\end{eqnarray*}
Let
$h_3=\sum_{j=3}^{\frac{n}{2}}\lambda_{3,j}\mon^{\phi_{s_3}+j\omega}$, where $\Lambda_{3}^\top=\left(\lambda_{3,3},\ldots,\lambda_{3,\frac{n}{2}}\right)$ satisfies
\begin{eqnarray*}
\left(B^{[1,\frac{n-4}{2}]}
_{\{0\}\sqcup[2,\frac{n-4}{2}]}\right)^\top\cdot
\Theta_{\frac{n-4}{2}}\cdot\Lambda_3=\left(B^{\{n-1\}}_{\{\frac{n+2}{2}\}\sqcup[\frac{n+6}{2},n-1]}\right)^\top.
\end{eqnarray*} 
Let $h_4=\lambda_{4,1}y^3z^{n-2}$, where 
\begin{eqnarray*}
\lambda_{4,1}=b_{n-1,\frac{n+4}{2}}-\sum_{j=3}^{\frac{n}{2}}b_{\frac{n+2-2j}{2},1}\cdot\lambda_{3,j}.
\end{eqnarray*}
Set $h=h_1+h_2+h_3+h_4$ and  $f=g+h$. Then, $g$ is a basis of $\ker(\rho_r^{{\rm G}_r})$,
$h_i\in W_{s_i}$, $i=1,2,3,4$,
\begin{eqnarray*}
f^\sigma=g,\phantom{+} f^\tau=h,\phantom{+}
f\in\ker(\rho)\phantom{+} \text{and}\phantom{+} g\in V_r.
\end{eqnarray*}
Hence, $V_r=\ker(\rho_r^{{\rm G}_r})$. 
Moreover, $g$ is equal to the polynomial $g_{n}$ defined in \eqref{equality-gk11}
and $f=g+h$ is equal to the polynomial $f_{n}$ defined as in \eqref{poly-fk11}.
In particular, $f_{n}\in\ker(\rho)$ and $f_{n}^{\sigma}$ is a basis of $V_{r_{n-1}}$. 
\end{lemma}
\begin{proof}
By Proposition~\ref{proposition-class}, if $(\iota_{r_k},c_{r_k})={\cblue (1,1)}$, 
then $n$ is odd, $k=n-1$ and $r=r_{n-1}$. 
Observe that $r=(a+1)\ell_r+c_r=(a+1)(n-1)+1=a(n-1)+n$. Clearly, $r<s_1<s_2<s_3<s_4$. Let us prove that $s_4<\la$. Since $n$ is even, then $q=n-1$ and $\tfrac{n}{2}q=a$. In particular, 
$s_4=s_3+q=s_1+\tfrac{n+4}{2}q=r+(n+1)q=a(n-1)+n+nq+(n-1)=a(n+1)+2n-1$. 
If $n=6$, then $a=15$ and $s_4=a(n+1)+2n-1=116<135=\la$. If $n\geq 8$, by Lemma~\ref{lemma-cota},~$(1)$, $\lfloor a/2\rfloor-n\geq 3$. Hence $n+1=2(n/2)+1<3(n/2)\leq(\lfloor a/2\rfloor-n)(n/2)=(\lfloor a/2\rfloor-n)\lfloor n/2\rfloor $. By Lemma~\ref{lemma-cota},~$(2)$, it follows that  
$s_4=r+(n+1)q<\frob(\nums)<\la$. In particular, $s_1<s_2<s_3<\la$.

By Lemma~\ref{lemma-ic11}, $\delta_r=n/2$, so $s_1=r+(\delta_r-1)\in\nums$, $s_1<\la$ and $\phi_{s_1}=(n-1,1,0)$. In particular, $\kappa_{s_1}=\min(\phi_{s_1,1},\phi_{s_1,3})=0$, $I_{s_1}=[\phi_{s_1,1}-\kappa_{s_1},\phi_{s_1,1}]=\{n-1\}$, $H_{s_1}=[0,\phi_{s_1,1}]=[0,n-1]$, $W_{s_1}=\langle x^{n-1}y\rangle$ and $h_1\in W_{s_1}$.

Let us prove that $s_2\in\nums$. As before, $q=n-1$ and $(n/2)q=a$. Thus, 
$s_2=s_1+q=r+(n/2)q=a(n-1)+n+a=an+n$, where $0\leq n\leq 2n$. By Proposition~\ref{theo31ggp}, $s_2\in\nums$ and, and since $s_2<\la$, then $\ell_{s_2}=n$, $\rem_{s_2}=n$ and
$\phi_{s_2}=(n/2,0,n/2)$. 
In particular,
$\kappa_{s_2}=\min(\phi_{s_2,1},\phi_{s_2,3})=n/2$, 
$I_{s_2}=[0,n/2]$, $H_{s_2}=[0,n/2]$, $W_{s_2}=\langle\mon^{\phi_{s_2}+j\omega}\mid 0\leq j\leq n/2\rangle$, $h_2$ is well-defined and $h_2\in W_{s_2}$. 

Let us prove that $s_3\in\nums$. Again, since $(n/2)q=a$, then 
$s_3=s_1+((n+2)/2)q=r+((n-2)/2)q+((n+2)/2)q=a(n-1)+n+(\frac{n}{2}+\frac{n}{2})q
=a(n+1)+n$. Thus, $s_3=a(n+1)+n$, where $0\leq n\leq 2(n+1)$. By Proposition~\ref{theo31ggp}, $s_3\in\nums$ and, since $s_3<\la$, then $\ell_{s_3}=n+1$, 
$\rem_{s_3}=n$ and $\phi_{s_3}=((n+2)/2,0,n/2)$. In particular,
$\kappa_{s_3}=\min(\phi_{s_3,1},\phi_{s_3,3})=n/2$, $I_{s_3}=[1,(n+2)/2]$, $H_{s_3}=[0,(n+2)/2]$, 
$W_{s_3}=\langle\mon^{\phi_{s_3}+j\omega}\mid 0\leq j\leq n/2\rangle$,
$h_3$ is well-defined and $h_3\in W_{s_3}$

Let us prove that $s_4\in\nums$. Since $q=n-1$ and $s_3=a(n+1)+n$, 
then $s_4=s_3+q=a(n+1)+2n-1$, where $0\leq 2n-1\leq 2(n+1)$. By
Proposition~\ref{theo31ggp}, $s_4\in\nums$ and, 
since $s_4<\la$, then $\ell_{s_4}=n+1$, $\rem_{s_4}=2n-1$ and
$\phi_{s_4}=(1,1,n-1)$. In particular,
$\kappa_{s_4}=\min(\phi_{s_4,1},\phi_{s_4,3})=1$, $I_{s_4}=[0,1]$, 
$H_{s_4}=[0,1]$, $W_{s_4}=\langle xyz^{n-1},y^3z^{n-2}\rangle$
and $h_4\in W_{s_4}$.

By Lemma~\ref{lemma-ic11}, $g$ is a basis of $\ker(\rho_r^{\rm G})$.
Since $r<s_1<s_2<s_3<s_4$, then $f^\sigma=g$ and $f^\tau=h$. 
Let us prove that $\rho(f)=0$, so $g\in V_r$ and $V_r=\ker(\rho_r^{{\rm G}_r})$.

By Lemma~\ref{lemma-ic11}, $I_r=[0,(n-2)/2]$, $H_r=[0,(n-2)/2]$ and
$\good_r=[0,(n-4)/2]$. 

Set $L_r^1=\{(n-2)/2\}$, so that
$H_r=\good_r\sqcup L_r^1$. Since $g\in\ker(\rho_r^{{\rm G}_r})$, then, by Lemma~\ref{lemma3-rho},~$(8)$,
\begin{eqnarray*}
  \rho(g)=\rho_r^{{\rm G}_r}(g) +\rho_r^{H_r\setminus {{\rm G}_r}}(g)=
  \rho_r^{H_r\setminus {{\rm G}_r}}(g)=
  \rho_r^{L_r^1}(g).
\end{eqnarray*}
Using Remark~\ref{rem-comp}, $s_1=r+((n-2)/2)q$ and $b_{\frac{n-2-2j}{2},\frac{n-2}{2}}=0$, 
for all $j\geq 1$, then
\begin{eqnarray*}
\rho_r^{L_r^1}(g)=\left(\sum_{j=0}^{\frac{n-2}{2}}b_{\frac{n-2-2j}{2},\frac{n-2}{2}}\cdot(-1)^j\cdot b_{\frac{n-2}{2},\frac{n-2-2j}{2}}\right)t^{r+\frac{n-2}{2}q}=b_{\frac{n-2}{2},\frac{n-2}{2}}\cdot(-1)^0\cdot b_{\frac{n-2}{2},\frac{n-2}{2}}\cdot t^{s_1}=t^{s_1}
\end{eqnarray*}
Set $L_{s_1}^1=[1,n/2]$, $L_{s_1}^2=\{(n+2)/2\}$, $L_{s_1}^3=\{(n+4)/2\}$ and $L_{s_1}^4=[(n+6)/2,n-1]$. Note that $H_{s_1}=\{0\}\sqcup L_{s_1}^1\sqcup L_{s_1}^2\sqcup L_{s_1}^3\sqcup L_{s_1}^4$. 

Let us find $h_1=\lambda_{1,0}x^{n-1}y\in W_{s_1}$ such that $\rho_r^{L_r^1}(g)+\rho_{s_1}^{\{0\}}(h_1)=0$. By Remark~\ref{rem-comp}, $\rho_{s_1}^{\{0\}}(h_1)=b_{n-1,0}\lambda_{1,0}t^{s_1}$. Thus, 
we take $\lambda_{1,0}=-1$. Then,
\begin{eqnarray*}
\rho(g+h_1)&=&\rho_r^{L_r^1}(g)+\rho_{s_1}^{\{0\}}(h_1)+\rho_{s_1}^{L_{s_1}^1}(h_1)+\rho_{s_1}^{L_{s_1}^2}(h_1)+\rho_{s_1}^{L_{s_1}^3}(h_1)+\rho_{s_1}^{L_{s_1}^4}(h_1)\\
&=&\rho_{s_1}^{L_{s_1}^1}(h_1)+\rho_{s_1}^{L_{s_1}^2}(h_1)+\rho_{s_1}^{L_{s_1}^3}(h_1)+\rho_{s_1}^{L_{s_1}^4}(h_1).
\end{eqnarray*}
Using Remark~\ref{rem-comp} and the equalities $s_2=s_1+q$, $s_3=s_1+((n+2)/2)q$ and $s_4=s_3+q=s_1+((n+4)/2)q$, it follows that: 
\begin{eqnarray*}
&&\rho_{s_1}^{L_{s_1}^1}(h_1)=\sum_{k=1}^{\frac{n}{2}}-b_{n-1,k}
t^{s_1+kq}=\sum_{k=0}^{\frac{n-2}{2}}-b_{n-1,k+1} t^{s_2+kq},\\
&&\rho_{s_1}^{L_{s_1}^2}(h_1)=-b_{n-1,\frac{n+2}{2}} t^{s_1+\frac{n+2}{2}q}
=-b_{n-1,\frac{n+2}{2}} t^{s_3},\\
&&\rho_{s_1}^{L_{s_1}^3}(h_1)=-b_{n-1,\frac{n+4}{2}}t^{s_1+\frac{n+4}{2}q}=
-b_{n-1,\frac{n+4}{2}}t^{s_4},\\
&&\rho_{s_1}^{L_{s_1}^4}(h_1)=
\sum_{k=\frac{n+6}{2}}^{n-1}-b_{n-1,k}t^{s_1+kq}=
\sum_{k=2}^{\frac{n-4}{2}}-b_{n-1,\frac{n+2}{2}+k}t^{s_3+kq}.
\end{eqnarray*}
Set $L_{s_2}^1=[0,(n-2)/2]$ and $L_{s_2}^{2}=\{n/2\}$. Note that $H_{s_2}=L_{s_2}^1\sqcup L_{s_2}^2$. 

Let us find $h_2\in W_{s_2}$ such that $\rho_{s_1}^{L_{s_1}^1}(h_1)+\rho(h_2)=0$. Choose $h_2$ as a linear combination of the last $n/2$ monomials of the basis $\mcb_{s_2}$: $h_2=\sum_{j=1}^{\frac{n}{2}}\lambda_{2,j}\mon^{\phi_{s_2+j\omega}}$. By Remark~\ref{rem-comp}, and using that $b_{\frac{n-2j}{2},\frac{n}{2}}=0$ for all $1\leq j\leq n/2$, we get that: 
\begin{eqnarray*}
\rho_{s_2}^{L_{s_2}^1}(h_2)=\sum_{k=0}^{\frac{n-2}{2}}
\left(\sum_{j=1}^{\frac{n}{2}}
b_{\frac{n-2j}{2},k}\lambda_{2,j}\right)t^{s_2+kq}\;\mbox{ and }\;
\rho_{s_2}^{L_{s_2}^2}(h_2)=
\left(\sum_{j=1}^{\frac{n}{2}}
b_{\frac{n-2j}{2},\frac{n}{2}}\lambda_{2,j}\right)t^{s_2+\frac{n}{2}q}=0.
\end{eqnarray*}
Therefore, $\rho_{s_1}^{L_{s_1}^1}(h_1)+\rho(h_2)=0$ is equivalent to $\rho_{s_2}^{L_{s_2}^1}(h_2)=-\rho_{s_1}^{L_{s_1}^1}(h_1)$. This induces the linear system:
\begin{eqnarray*}
\sum_{k=0}^{\frac{n-2}{2}}
\left(\sum_{j=1}^{\frac{n}{2}}
b_{\frac{n-2j}{2},k}\lambda_{2,j}\right)t^{s_2+kq}=
\sum_{k=0}^{\frac{n-2}{2}}b_{n-1,k+1} t^{s_2+kq}.
\end{eqnarray*}
In matrix form, $P_{[0,\frac{n-2}{2}]}^{[0,\frac{n-2}{2}]}\cdot \Lambda_2=P_{\{n-1\}}^{[1,\frac{n}{2}]}$,
where $\Lambda_2^\top=\left(\lambda_{2,1},\ldots,\lambda_{2,\frac{n}{2}}\right)$. By Notation~\ref{notation-binomials}, this is equivalent to: 
\begin{eqnarray*}
\left(B^{[0,\frac{n-2}{2}]}_{[0,\frac{n-2}{2}]}\right)^\top
\cdot\Theta_{\frac{n}{2}}\cdot\Lambda_2=\left(B^{\{n-1\}}_{[1,\frac{n}{2}]}\right)^\top\cdot\Theta_1=
\left(B^{\{n-1\}}_{[1,\frac{n}{2}]}\right)^\top.
\end{eqnarray*}
Note that $B_{[0,\frac{n-2}{2}]}^{[0,\frac{n-2}{2}]}$ is a square $\frac{n}{2}\times\frac{n}{2}$ lower triangular matrix with $1$'s in the diagonal, so invertible. Let $\Lambda_2$ be the unique solution of the system above and let $h_2$ be the polynomial in $W_{s_2}$, with components $\Lambda_2$ in the linearly independent set $\{\mon^{\phi_{s_2}+j\omega}\mid 1\leq j\leq n/2\}\subset \mathcal{B}_{s_2}$. Then, 
\begin{eqnarray*}
\rho(g+h_1+h_2)=\rho_{s_1}^{L_{s_1}^2}(h_1)+\rho_{s_1}^{L_{s_1}^3}(h_1)+\rho_{s_1}^{L_{s_1}^4}(h_1).
\end{eqnarray*}
Set $L_{s_3}^1=\{1\}$, $L_{s_3}^2=[2,(n-4)/2]$ and $L_{s_3}^3=[(n-2)/2,(n+2)/2]$. Note that $H_{s_3}=\{0\}\sqcup L_{s_3}^1\sqcup L_{s_3}^2\sqcup L_{s_3}^3$. 

Let us find an $h_3\in W_{s_3}$ such that $\rho_{s_1}^{L_{s_1}^2}(h_1)+\rho_{s_1}^{L_{s_1}^4}(h_1)+\rho_{s_3}^{\{0\}}(h_3)+\rho_{s_3}^{L_{s_3}^2}(h_3)=0$. We choose $h_3$ as a linear combination of the last $(n-6)/2$ monomials of the basis $\mcb_{s_3}$: $h_3=\sum_{j=3}^{\frac{n}{2}}\lambda_{3,j}\mon^{\phi_{s_3}+j\omega}$. By Remark~\ref{rem-comp}, it follows that
\begin{eqnarray*}
\rho_{s_3}^{\{0\}}(h_3)=\left(\sum_{j=3}^{\frac{n}{2}}
b_{\frac{n+2-2j}{2},0}\lambda_{3,j}\right)t^{s_3}
\;\mbox{ and }\; 
\rho_{s_3}^{L_{s_3}^2}(h_3)=\sum_{k=2}^{\frac{n-4}{2}}\left(\sum_{j=3}^{\frac{n}{2}}b_{\frac{n+2-2j}{2},k}\lambda_{3,j}\right)t^{s_3+kq}.
\end{eqnarray*}
Asking for $\rho_{s_3}^{\{0\}}(h_3)+\rho_{s_3}^{L_{s_3}^2}(h_3)=-\rho_{s_1}^{L_{s_1}^2}(h_1)-\rho_{s_1}^{L_{s_1}^4}(h_1)$ leads to the linear system:
\begin{eqnarray*}
\left(\sum_{j=3}^{\frac{n}{2}}b_{\frac{n+2-2j}{2},0}\lambda_{3,j}\right)t^{s_3}+
\sum_{k=2}^{\frac{n-4}{2}}\left(\sum_{j=3}^{\frac{n}{2}}b_{\frac{n+2-2j}{2},k}\lambda_{3,j}\right)t^{s_3+kq}
=b_{n-1,\frac{n+2}{2}}t^{s_3}+
\sum_{k=2}^{\frac{n-4}{2}}b_{n-1,\frac{n+2}{2}+k} t^{s_3+kq}.
\end{eqnarray*}
In matrix form, $P_{[1,\frac{n-4}{2}]}
^{\{0\}\sqcup[2,\frac{n-4}{2}]}\cdot\Lambda_3=
P_{\{n-1\}}^{\{\frac{n+2}{2}\}\sqcup[\frac{n+6}{2},n-1]}$, where $\Lambda_3^\top=(\lambda_{3,3},\ldots ,\lambda_{3,\frac{n}{2}})$. By Notation~\ref{notation-binomials}, this is equivalent to:
\begin{eqnarray*}
\left(B^{[1,\frac{n-4}{2}]}
_{\{0\}\sqcup[2,\frac{n-4}{2}]}\right)^\top\cdot
\Theta_{\frac{n-4}{2}}\cdot\Lambda_3
=\left(B^{\{n-1\}}_{\{\frac{n+2}{2}\}\sqcup[\frac{n+6}{2},n-1]}\right)^\top\cdot\Theta_1=
\left(B^{\{n-1\}}_{\{\frac{n+2}{2}\}\sqcup[\frac{n+6}{2},n-1]}\right)^\top.
\end{eqnarray*}
Observe that $B^{[1,\frac{n-4}{2}]}
_{\{0\}\sqcup[2,\frac{n-4}{2}]}$ is a square $\frac{n-4}{2}\times\frac{n-4}{2}$ matrix and $\{0\}\sqcup[2,(n-4)/2]\leq[1,(n-4)/2]$. By \cite[Corollary~2]{gv}, or \cite[Corollary~2.5]{gp2}, $B^{[1,\frac{n-4}{2}]}_{\{0\}\sqcup[2,\frac{n-4}{2}]}$ is an invertible matrix. 
Let $\Lambda_{3}$ be the unique solution of this linear system and 
let $h_3$ be the polynomial in $W_{s_3}$, with components $\Lambda_{3}$ 
in the linearly independent subset 
$\{\mon^{\phi_{s_3}+j\omega}\mid 3\leq j\leq n/2\}$ of $\mcb_{s_3}$. Then,
\begin{eqnarray*}
\rho(g+h_1+h_2+h_3)=\rho_{s_1}^{L_{s_1}^3}(h_1)+\rho_{s_3}^{L_{s_3}^1}(h_3)+
\rho_{s_3}^{L_{s_3}^3}(h_3).
\end{eqnarray*}
By Remark~\ref{rem-comp}, using the equality $s_4=s_3+q$ and that $b_{\frac{n+2-2j}{2},\frac{n-2}{2}}=0$, $b_{\frac{n+2-2j}{2},\frac{n}{2}}=0$ and $b_{\frac{n+2-2j}{2},\frac{n+2}{2}}=0$, for all $j\geq 3$, we deduce:
\begin{eqnarray*}
&&\rho_{s_3}^{L_{s_3}^1}(h_3)=\left(\sum_{j=3}^{\frac{n}{2}}
b_{\frac{n+2-2j}{2},1}\lambda_{3,j}\right) t^{s_3+q}=
\left(\sum_{j=3}^{\frac{n}{2}}b_{\frac{n+2-2j}{2},1}\lambda_{3,j}\right)
t^{s_4},\\
&&\rho_{s_3}^{L_{s_3}^3}=\sum_{k=\frac{n-2}{2}}^{\frac{n+2}{2}}\left(\sum_{j=3}^{\frac{n}{2}}b_{\frac{n+2-2j}{2},k}\lambda_{3,j}\right)t^{s_3+kq}=0.
\end{eqnarray*}
 Therefore, 
\begin{eqnarray*}
\rho(g+h_1+h_2+h_3)=\rho_{s_1}^{L_{s_1}^3}(h_1)+\rho_{s_3}^{L_{s_3}^1}(h_3)=\left(-b_{n-1,\frac{n+4}{2}}+\sum_{j=3}^{\frac{n}{2}}b_{\frac{n+2-2j}{2},1}\lambda_{3,j}\right)t^{s_4}.
\end{eqnarray*}
Let us find $h_4=\lambda_{4,1}y^3z^{n-2}\in W_{s_4}$ such that 
$\rho_{s_1}^{L_{s_1}^3}(h_1)+\rho_{s_3}^{L_{s_3}^1}(h_3)+\rho(h_4)=0$. Note that 
$\rho(h_4)=\lambda_{4,1}t^{s_4}$. On taking 
$\lambda_{4,1}=b_{n-1,\frac{n+4}{2}}-\sum_{j=3}^{\frac{n}{2}}b_{\frac{n+2-2j}{2},1}\lambda_{3,j}$, we obtain
\begin{eqnarray*}
\rho(f)=\rho(g+h_1+h_2+h_3+h_4)=\rho_{s_1}^{L_{s_1}^3}(h_1)+\rho_{s_3}^{L_{s_3}^2}(h_3)+\rho(h_4)=0, 
\end{eqnarray*}
as desired.

Finally, by Lemma~\ref{lemma-ic11}, $g$ is equal to the polynomial $g_{n}$ defined in 
\eqref{equality-gk11}. The expression of $f=g+h$ as $f_{n}$ follows from substituting 
the values of $\phi_{s_2}$, $\phi_{s_3}$ and $\phi_{s_4}$ in $h_2$, $h_3$ and $h_4$.
\end{proof}

\begin{summary}\label{resum}
Let $n\geq 6$. Let $r=r_k=s_0+k-1\in [s_0,s_0+n)$, where $1\leq k\leq n$. 

\vspace*{0.2cm}

\noindent \underline{Suppose that $n$ is odd and $n\geq 7$.}

\vspace*{0.2cm}

\noindent $\bullet$ Let $1\leq k\leq n-4$, $k$ odd. 
By Lemma~\ref{lemma-h-i}, $g_k$, as in \eqref{equality-gknoddkodd},
is a basis of $V_{r_k}=V_{s_0+k-1}$ and $f_k$, 
as in \eqref{poly-fknoddkodd}, satisfies $\rho(f_k)=0$ and $f_k^{\sigma}=g_k$. 

\vspace*{0.2cm}

\noindent $\bullet$ Let $2\leq k\leq n-3$, $k$ even. 
By Lemma~\ref{lemma-h-i+1}, $g_k$, as in \eqref{equality-gknoddkeven},
is a basis of $V_{r_k}=V_{s_0+k-1}$ and $f_k$, 
as in \eqref{poly-fknoddkeven},
satisfies $\rho(f_k)=0$ and $f_k^{\sigma}=g_k$. 

\vspace*{0.2cm}

\noindent $\bullet$ Let $k=n-2$.
By Lemma~\ref{lemma-h-00}, $g_{n-2}$, as in \eqref{equality-gk00}, 
is a basis of $V_{r_{n-2}}=V_{s_0+n-3}$ and $f_{n-2}$, 
as in \eqref{poly-fk00},
satisfies $\rho(f_{n-2})=0$ and $f_k^{\sigma}=g_{n-2}$. 

\vspace*{0.2cm}

\noindent $\bullet$ Let $k=n-1$.
By Lemma~\ref{lemma-h-21}, $g_{n-1}$, as in \eqref{equality-gk21}, 
is a basis of $V_{r_{n-1}}=V_{s_0+n-2}$ and $f_{n-1}$, 
as in \eqref{poly-fk21},
satisfies $\rho(f_{n-1})=0$ and $f_{n-1}^{\sigma}=g_{n-1}$.  

\vspace*{0.2cm}

\noindent $\bullet$ Let $k=n$.
By Lemma~\ref{lemma-h-22}, $g_n$, as in \eqref{equality-gk22}, 
is a basis of $V_{r_n}=V_{s_0+n-1}$ and $f_n$, as in \eqref{poly-fk22},
satisfies $\rho(f_n)=0$ and $f_n^{\sigma}=g_n$.  

\vspace*{0.2cm}

In particular, and by Theorem~\ref{SecondStep}, 
$\min\{\sord(f)\mid f\in \ker(\rho), f\neq 0\}=s_0=a(n-1)+2$. Take
\begin{eqnarray*}
&&\msd_{0}=\{f_1\},\ldots,\msd_{n-1}=\{f_n\},\msd_{n}=\emptyset,\ldots,\msd_{a-1}=\emptyset,\mbox{ where}\\&&
\msd_{0}^{\sigma}=\{g_1\}\mbox{ is a basis of }V_{s_0},\ldots, 
\msd_{n-1}^{\sigma}=\{g_n\}\mbox{ is a basis of }V_{s_0+n-1}.
\end{eqnarray*}
By the \emph{Extending to basis Method} (page \pageref{moh43}),
it follows that $\cup_{k=0}^{a-1}\msd_k=\{f_1,\ldots,f_n\}$ 
can be extended to a minimal generating set of $\ker(\rho)$.

\vspace*{0.2cm}

\noindent \underline{Suppose that $n$ is even and $n\geq 6$.}

\vspace*{0.2cm}

\noindent $\bullet$ Let $1\leq k\leq n-5$, $k$ odd. 
By Lemma~\ref{lemma-h-i}, $g_k$, as in \eqref{equality-gknevenkodd},
is a basis of $V_{r_k}=V_{s_0+k-1}$ and $f_k$, 
as in \eqref{poly-fknevenkodd},
satisfies $\rho(f_k)=0$ and $f_k^{\sigma}=g_k$. 

\vspace*{0.2cm}

\noindent $\bullet$ Let $2\leq k\leq n-4$, $k$ even. 
By Lemma~\ref{lemma-h-i+1}, $g_k$, as in \eqref{equality-gknevenkeven},
is a basis of $V_{r_k}=V_{s_0+k-1}$ and $f_k$,
as in \eqref{poly-fknevenkeven},
satisfies $\rho(f_k)=0$ and $f_k^{\sigma}=g_k$. 

\vspace*{0.2cm}

\noindent $\bullet$ Let $k=n-3$.
By Lemma~\ref{lemma-h-1-1}, $g_{n-3}$, as in \eqref{equality-gk1-1}, 
is a basis of $V_{r_{n-3}}=V_{s_0+n-4}$ and $f_{n-3}$, 
as in \eqref{poly-fk1-1},
satisfies $\rho(f_{n-3})=0$ and $f_{n-3}^{\sigma}=g_{n-3}$. 

\vspace*{0.2cm}

\noindent $\bullet$ Let $k=n-2$.
By Lemma~\ref{lemma-h-10}, $g_{n-2}$, as in \eqref{equality-gk10a}, 
is a non zero element of $V_{r_{n-2}}=V_{s_0+n-3}$ and $f_{n-2}$, 
as in \eqref{poly-fk10a},
satisfies $\rho(f_{n-2})=0$ and $f_{n-2}^{\sigma}=g_{n-2}$.  

\vspace*{0.2cm}

\noindent $\bullet$ Let $k=n-1$.
By Lemma~\ref{lemma-h-10}, $g_{n-1}$, as in \eqref{equality-gk10b}, 
is a non zero element of $V_{r_{n-2}}=V_{s_0+n-3}$ and $f_{n-1}$, 
as in \eqref{poly-fk10b},
satisfies $\rho(f_{n-1})=0$ and $f_{n-1}^{\sigma}=g_{n-1}$. 
Moreover, $g_{n-2},g_{n-1}$ are a basis of
$V_{r_{n-2}}=V_{s_0+n-3}$.

\noindent $\bullet$ Let $k=n$.
By Lemmas~\ref{lemma-ic11} and \ref{lemma-h-11} $g_{n}$, as in \eqref{equality-gk11}, 
is a basis of $V_{r_{n-1}}=V_{s_0+n-2}$ and $f_{n}$, as in \eqref{poly-fk11} and Lemma~\ref{lemma-h-11},
satisfies $\rho(f_{n})=0$ and $f_{n}^{\sigma}=g_{n}$. 

\vspace*{0.2cm}

In particular, and by Theorem~\ref{SecondStep}, 
$\min\{\sord(f)\mid f\in \ker(\rho), f\neq 0\}=s_0=a(n-1)+2$. Take
\begin{eqnarray*}
&&\msd_{0}=\{f_1\},\ldots,\msd_{n-4}=\{f_{n-3}\},\msd_{n-3}=\{f_{n-2},f_{n-1}\},\msd_{n-2}=\{f_{n}\},\\
&&\msd_{n-1}=\emptyset,\ldots,\msd_{a-1}=\emptyset,\mbox{ where}\\&&
\msd_{0}^{\sigma}=\{g_1\}\mbox{ is a basis of }V_{s_0},\ldots, 
\msd_{n-4}^{\sigma}=\{g_{n-3}\}\mbox{ is a basis of }V_{s_0+n-4},\\&&
\msd_{n-3}^{\sigma}=\{g_{n-2},g_{n-1}\}\mbox{ is a basis of }V_{s_0+n-3},
\msd_{n-1}^{\sigma}=\{g_n\}\mbox{ is a basis of }V_{s_0+n-2}.
\end{eqnarray*}
By the \emph{Extending to basis Method} (page \pageref{moh43}),
it follows that $\cup_{k=0}^{a-1}\msd_k=\{f_1,\ldots,f_n\}$ 
can be extended to a minimal generating set of $\ker(\rho)$.
\end{summary}

Now, we can state and prove the main result of the section. 

\vspace*{0.3cm}

\begin{theorem}\label{FourthStep}
Let $n\geq 3$. 
The set of polynomials $\{f_1,\ldots,f_n\}$ is part of a minimal generating set of $\ker(\rho)$.
\end{theorem}
\begin{proof}
The previous Summary~\ref{resum} shows the cases $n\geq 6$. It remains to prove the cases
$n=3,4,5$.

\vspace*{0.2cm}

Suppose that $n=3$. Then $a=(n-1)n/2=3$, $s_0=a(n-1)+2=8$ and
$\nums=\langle 3,4,5\rangle$.
By Theorem~\ref{SecondStep}, $V_r=0$, for all $r\in\nums\cap (0,8)$. 
Let us calculate $V_r$, for $r\in [s_0,s_0+n)=\{8,9,10\}$. Since
$\fac(8,\nums)=\{(1,0,1),(0,2,0)\}$, 
$\fac(9,\nums)=\{(3,0,0),(0,1,1)\}$ and $\fac(10,\nums)=\{(2,1,0),(0,0,2)\}$, 
then $W_8=\langle xz, y^2\rangle$, $W_9=\langle x^3,yz\rangle$ and 
$W_{10}=\langle x^2y,z^2\rangle$. 
By Lemma~\ref{lemma1-rho}, $V_8=0$ or 
$V_8=\langle xz-y^2\rangle$, $V_9=0$ or
$V_9=\langle x^3-yz\rangle$ and $V_{10}=0$ or 
$V_{10}=\langle x^2y-z^2\rangle$.
For $k=1,2,3$, let $g_k$ and $f_k$ be as in
\eqref{equality-gfn=3}. 
Note that $g_1=xz-y^2\in W_8$ and $f_1=g_1-xy^2+yz^2$
satisfies $\rho(f_1)=0$ and $f_1^{\sigma}=g_1$. 
Similarly, $g_2=x^3-yz\in W_9$ and $f_2=g_2-3xyz+z^3-x^2yz+xz^3$
satisfies $\rho(f_2)=0$ and $f_2^{\sigma}=g_2$. Finally, 
$g_3=x^2y-z^2\in W_{10}$ and $f_3=g_3-y^2z-xz^2$
satisfies $\rho(f_3)=0$ and $f_3^{\sigma}=g_3$. Therefore, 
$=g_1\in V_8$, $f_2^{\sigma}=g_2\in V_{9}$, $f_3^{\sigma}=g_3\in V_{10}$ and 
\begin{eqnarray*}
\min\{\sord(f)\mid f\in \ker(\rho), f\neq 0\}=8=s_0.
\end{eqnarray*}
Take $\msd_0=\{f_1\}$, $\msd_1=\{f_2\}$ and $\msd_2=\{f_3\}$, where $a=3$. 
By the \emph{Extending to basis Method} (see page \pageref{moh43}), we deduce that
$\msd_0\cup\msd_1\cup\msd_2=\{f_1,f_2,f_3\}$ is part of a minimal generating set of $\ker(\rho)$. 

\vspace*{0.2cm}

Suppose that $n=4$. Then $a=(n-1)n/2=6$, $s_0=a(n-1)+2=20$ and
$\nums=\langle 6,7,8\rangle$. By Theorem~\ref{SecondStep},
$V_r=0$, for all $r\in\nums\cap (0,20)$.
Let us calculate $V_r$, for $r\in [s_0,s_0+n)=\{20,21,22,23\}$ and also for $r=24$. 
One can check that $W_{20}=\langle xy^2,x^2z\rangle$, 
$W_{21}=\langle xyz,y^3\rangle$, $W_{22}=\langle xz^2,y^2z\rangle$, $W_{23}=\langle yz^2\rangle$ and
$W_{24}=\langle x^4,z^3\rangle$.
By Lemma~\ref{lemma1-rho}, $V_{20}=0$ or 
$V_{20}=\langle xy^2-x^2z\rangle$, $V_{21}=0$ or 
$V_{21}=\langle xyz-y^3\rangle$, $V_{22}=0$ or
$V_{22}=\langle xz^2-y^2z\rangle$, $V_{23}=0$, $V_{24}=0$ or 
$V_{24}=\langle x^4-z^3\rangle$.
For $k=1,2,3,4$, let $g_k$ and $f_k$ be as in
\eqref{equality-gfn=4}. Then $g_1=xy^2-x^2z\in W_{20}$ and $f_1=g_1+yz^2+x^2y^2-2xyz^2+z^4$
satisfies $\rho(f_1)=0$ and $f_1^{\sigma}=g_1$. Similarly, 
$g_2=xyz-y^3\in W_{21}$ and $f_2=g_2-z^3$ satisfies $\rho(f_2)=0$ and $f_2^{\sigma}=g_2$; 
$g_3=xz^2-y^2z\in W_{22}$ and $f_3=g_3-x^3y+x^2z^2+xy^2z+y^4$ satisfies $\rho(f_3)=0$ and 
$f_3^{\sigma}=g_3$ and $g_4=x^4-z^3\in W_{24}$ and $f_4=g_4-4x^2yz+2y^2z^2-xy^2z^2+yz^4$
satisfies $\rho(f_4)=0$ and $f_4^{\sigma}=g_4$. 
Therefore, 
$f_1^{\sigma}=g_1\in V_{20}$, $f_2^{\sigma}=g_2\in V_{21}$, $f_3^{\sigma}=g_3\in V_{22}$,
$f_4^{\sigma}=g_4\in V_{24}$ and 
$\min\{\sord(f)\mid f\in \ker(\rho), f\neq 0\}=20=s_0$.
Take $\msd_0=\{f_1\}$, $\msd_1=\{f_2\}$, $\msd_2=\{f_3\}$, $\msd_3=\emptyset$, 
$\msd_{4}=\{f_4\}$ and $\msd_{5}=\emptyset$, where $a=6$.
By the \emph{Extending to basis Method} (see page \pageref{moh43}), we deduce that
$\cup_{i=0}^{5}\msd_i=\{f_1,f_2,f_3,f_4\}$ is part of a minimal generating set of $\ker(\rho)$. 

\vspace*{0.2cm}

Suppose that $n=5$. Then $a=(n-1)n/2=10$, $s_0=a(n-1)+2=42$ and
$\nums=\langle 10,11,12\rangle$. By Theorem~\ref{SecondStep},
$V_r=0$, for all $r\in\nums\cap (0,42)$.
Let us calculate $V_r$, for $r\in [s_0,s_0+n)=\{42,43,44,45,46\}$.
It is readily seen that 
$W_{42}=\langle x^2y^2,x^3z,\rangle$, 
$W_{43}=\langle x^2yz,xy^3\rangle$,
$W_{44}=\langle x^2z^2,xy^2z,y^4\rangle$,
$W_{45}=\langle xyz^2,y^3z\rangle$ and 
$W_{46}=\langle xz^3,y^2z^2\rangle$. 
For $k=1,\ldots,5$, lef $g_k$ and $f_k$ be as in 
\eqref{equality-gfn=5}. Then $g_1=x^3z-x^2y^2\in W_{42}$ and 
$f_1=g_1-yz^3-2x^3y^2+6xyz^3-y^3z^2+2y^5z$ satisfies $\rho(f_1)=0$ and $f_1^{\sigma}=g_1$. 
Similarly, $g_2=x^2yz-xy^3\in W_{43}$ and $f_2=g_2-z^4-x^2y^3+xz^4+y^2z^3$ satisfies
$\rho(f_2)=0$ and $f_2^{\sigma}=g_2$; $g_3=x^2z^2-2xy^2z+y^4\in W_{44}$ and $f_3=g_3-xy^4+yz^4$ satisfies
$\rho(f_3)=0$ and $f_3^{\sigma}=g_3$; $g_4=xyz^2-y^3z\in W_{45}$
$f_4=g_4-x^5+10xy^3z-5y^5+x^2y^3z+9xy^5-6y^2z^4$ satisfies $\rho(f_4)=0$ and $f_4^{\sigma}=g_4$
and $g_5=xz^3-y^2z^2\in W_{46}$ and $f_5=g_5-x^4y+6xy^2z^2-2y^4z+xy^4z+3y^6$
satisfies $\rho(f_5)=0$ and $f_5^{\sigma}=g_5$. 
Therefore, $f_1^{\sigma}=g_1\in V_{42}$, $f_2^{\sigma}=g_2\in V_{43}$, $f_3^{\sigma}=g_3\in V_{44}$,
$f_4^{\sigma}=g_4\in V_{45}$, $f_5^{\sigma}=g_5\in V_{46}$
and $\min\{\sord(f)\mid f\in \ker(\rho), f\neq 0\}=42=s_0$.
Take $\msd_0=\{f_{1}\}$, $\msd_1=\{f_2\}$, $\msd_2=\{f_3\}$,  
$\msd_{3}=\{f_4\}$, $\msd_{4}=\{f_5\}$ and $\msd_{5}=\ldots=\msd_{9}=\emptyset$, where $a=10$.
By the \emph{Extending to basis Method} (see page \pageref{moh43}), we deduce that
$\cup_{i=0}^9\msd_i=\{f_1,f_2,f_3,f_4,f_5\}$ is part of a minimal generating set of $\ker(\rho)$. 
\end{proof}

\section{Proof of the main result: a minimal generating set of the kernel}\label{sec-min}

Let $f_1,\ldots,f_n$ be as in Definition~\ref{def-gk}, where, for $n\geq 6$, we take the coefficients
given by Lemmas~\ref{lemma-h-i}, \ref{lemma-h-i+1}, \ref{lemma-h-00}, \ref{lemma-h-21}, 
\ref{lemma-h-22}, \ref{lemma-h-1-1}, \ref{lemma-h-10}, \ref{lemma-h-11}. 
The purpose of this section is to prove that 
$\{f_1,\ldots,f_n\}$ defines a minimal generating set of $\ker(\rho)$. 
We start with some previous results.

\begin{lemma}\label{lemma-In6}
Let $n\geq 6$. Let $R=\mbk\lser x,y,z\rser$, $\mfm=(x,y,z)$ and $I=(f_1,\ldots,f_n)$.
\begin{itemize}
\item[$(1)$] If $n$ is odd and $1\leq k\leq n-3$, then $(f_k,z)=(x^{n-k-2}y^{k+1},z)$.
\item[$(2)$] If $n$ is even and $2\leq k\leq n-3$, then $(f_k,z)=(x^{n-k-2}y^{k+1},z)$.
\item[$(3)$] $(f_{n-2},z)=(y^{n-1},z)$. 
\item[$(4)$] If $n$ is even, then $(f_1,f_{n-2},z)=(x^{n-3}y^2,y^{n-1},z)$. 
\item[$(5)$] $(f_{n-2},f_{n-1},z)=(x^{n},y^{n-1},z)$.
\item[$(6)$] $(f_{n-2},f_n,z)=(x^{n-1}y,y^{n-1},z)$.
\item[$(7)$] $I+zR=(x^n,x^{n-1}y,x^{n-3}y^2,\ldots,xy^{n-2},y^{n-1},z)$.
\end{itemize}
\end{lemma}
\begin{proof}[Proof of $(1)$] Suppose that $n$ is odd, $1\leq k\leq n-4$, $k$ odd.
By the equalities \eqref{equality-gknoddkodd} and \eqref{poly-fknoddkodd}, 
\begin{eqnarray*}
f_k&=&\sum_{j=0}^{\frac{k+1}{2}}\lambda_j
x^{\frac{2n-k-3-2j}{2}}y^{2j}z^{\frac{k+1-2j}{2}}+
\sum_{j=0}^{\frac{n-k-4}{2}}\lambda_{1,j}x^{\frac{n-k-4-2j}{2}}y^{1+2j}z^{\frac{n+k-2j}{2}}+\\
&&\sum_{j=1}^{\frac{k+1}{2}}\lambda_{2,j}x^{\frac{2n-k-1-2j}{2}}y^{2j}z^{\frac{k+1-2j}{2}}+
\sum_{j=0}^{\frac{n-k-2}{2}}\lambda_{3,j}x^{\frac{n-k-2-2j}{2}}y^{1+2j}z^{\frac{n+k-2j}{2}}+\\&&
\sum_{j=2}^{\frac{n-k}{2}}\lambda_{4,j}x^{\frac{n-k-2j}{2}}y^{1+2j}z^{\frac{n+k+2j}{2}}.
\end{eqnarray*}
Let us see that $f_k=ax^{n-k-2}y^{k+1}+bz$, for some $a,b\in\mbk\lser x,y,z\rser$, where 
$a$ is invertible. Indeed,
\begin{eqnarray*}
f_k&=&\left(\sum_{j=0}^{\frac{k-1}{2}}\lambda_j
x^{\frac{2n-k-3-2j}{2}}y^{2j}z^{\frac{k-1-2j}{2}}+
\sum_{j=0}^{\frac{n-k-4}{2}}\lambda_{1,j}x^{\frac{n-k-4-2j}{2}}y^{1+2j}z^{\frac{n+k-2-2j}{2}}\right.+\\
&&\phantom{+}\sum_{j=1}^{\frac{k-1}{2}}\lambda_{2,j}x^{\frac{2n-k-1-2j}{2}}y^{2j}z^{\frac{k-1-2j}{2}}+
\sum_{j=0}^{\frac{n-k-2}{2}}\lambda_{3,j}x^{\frac{n-k-2-2j}{2}}y^{1+2j}z^{\frac{n+k-2-2j}{2}}+\\
&&\phantom{+}\left.\sum_{j=2}^{\frac{n-k}{2}}
\lambda_{4,j}x^{\frac{n-k-2j}{2}}y^{1+2j}z^{\frac{n+k-2+2j}{2}}\right)z+
\left(\lambda_{\frac{k+1}{2}}+\lambda_{2,\frac{k+1}{2}}x\right)x^{n-k-2}y^{k+1}.
\end{eqnarray*}
If $k=1$, we understand vacuous the third summation.
By Lemma~\ref{lemma-c-i}, the coefficient $\lambda_{\frac{k+1}{2}}\neq 0$.
Thus, $f_k=ax^{n-k-2}y^{k+1}+bz$, where
$a=\left(\lambda_{\frac{k+1}{2}}+\lambda_{2,\frac{k+1}{2}}x\right)$ is invertible.

\vspace*{0.3cm}

Suppose that $n$ is odd and $2\leq k\leq n-3$, with $k$ even. 
By the equalities \eqref{equality-gknoddkeven} and \eqref{poly-fknoddkeven}, 
\begin{eqnarray*}
f_k&=&\sum_{j=0}^{\frac{k}{2}}\lambda_j
x^{\frac{2n-k-4-2j}{2}}y^{1+2j}z^{\frac{k-2j}{2}}+
\sum_{j=0}^{\frac{n-k-3}{2}}\lambda_{1,j}x^{\frac{n-k-3-2j}{2}}y^{2j}z^{\frac{n+k+1-2j}{2}}+\\
&&\sum_{j=1}^{\frac{k}{2}}\lambda_{2,j}x^{\frac{2n-k-2-2j}{2}}y^{1+2j}z^{\frac{k-2j}{2}}+
\sum_{j=0}^{\frac{n-k-1}{2}}\lambda_{3,j}x^{\frac{n-k-1-2j}{2}}y^{2j}z^{\frac{n+k+1-2j}{2}}+\\&&
\sum_{j=3}^{\frac{n-k+1}{2}}\lambda_{4,j}x^{\frac{n-k+1-2j}{2}}y^{2j}z^{\frac{n+k+1+2j}{2}}.
\end{eqnarray*} 
If $k=n-3$, we understand vacuous the last summation. Let us
see that $f_k=ax^{n-k-2}y^{k+1}+bz$, for some $a,b\in\mbk\lser x,y,z\rser$, 
where $a$ is invertible. But,
\begin{eqnarray*}
f_k&=&\left(\sum_{j=0}^{\frac{k-2}{2}}\lambda_j
x^{\frac{2n-k-4-2j}{2}}y^{1+2j}z^{\frac{k-2-2j}{2}}+
\sum_{j=0}^{\frac{n-k-3}{2}}\lambda_{1,j}x^{\frac{n-k-3-2j}{2}}y^{2j}z^{\frac{n+k-1-2j}{2}}\right.+\\
&&\phantom{+}\sum_{j=1}^{\frac{k-2}{2}}\lambda_{2,j}x^{\frac{2n-k-2-2j}{2}}y^{1+2j}z^{\frac{k-2-2j}{2}}+
\sum_{j=0}^{\frac{n-k-1}{2}}\lambda_{3,j}x^{\frac{n-k-1-2j}{2}}y^{2j}z^{\frac{n+k-1-2j}{2}}+\\
&&\phantom{+}\left.
\sum_{j=3}^{\frac{n-k+1}{2}}\lambda_{4,j}x^{\frac{n-k+1-2j}{2}}y^{2j}z^{\frac{n+k-1+2j}{2}}
\right)z+\left(\lambda_{\frac{k}{2}}+\lambda_{2,\frac{k}{2}}x\right)x^{n-k-2}y^{k+1}.
\end{eqnarray*}
By Lemma~\ref{lemma-c-i+1}, $\lambda_\frac{k}{2}\neq 0$. 
Thus, $f_k=ax^{n-k-2}y^{k+1}+bz$, where 
$a=\left(\lambda_{\frac{k}{2}}+\lambda_{2,\frac{k}{2}}x\right)$ is invertible.
This proves $(1)$. 

\vspace*{0.3cm}

\noindent $\bullet$ \emph{Proof of $(2)$}. 
Suppose that $n$ is even, $3\leq k\leq n-5$, $k$ odd. 
By the equalities \eqref{equality-gknevenkodd} and \eqref{poly-fknevenkodd}, 
\begin{eqnarray*}
f_k&=&\sum_{j=0}^{\frac{k+1}{2}}\lambda_j
x^{\frac{2n-k-3-2j}{2}}y^{2j}z^{\frac{k+1-2j}{2}}+\sum_{j=0}^{\frac{n-k-3}{2}}
\lambda_{1,j}x^{\frac{n-k-3-2j}{2}}y^{1+2j}z^{\frac{n+k-1-2j}{2}}+\\
&&\sum_{j=1}^{\frac{k+1}{2}}\lambda_{2,j}x^{\frac{2n-k-1-2j}{2}}y^{2j}z^{\frac{k+1-2j}{2}}+
\sum_{j=0}^{\frac{n-k-1}{2}}\lambda_{3,j}x^{\frac{n-k-1-2j}{2}}y^{1+2j}z^{\frac{n+k-1-2j}{2}}+\\&&
\sum_{j=3}^{\frac{n-k+1}{2}}\lambda_{4,j}x^{\frac{n-k+1-2j}{2}}y^{1+2j}z^{\frac{n+k-1-2j}{2}}.
\end{eqnarray*}
Note that, if $k=1$, the last summation gives 
$\sum_{j=3}^{\frac{n}{2}}\lambda_{4,j}x^{\frac{n-2j}{2}}y^{1+2j}z^{\frac{n-2j}{2}}$, 
and the last term of this summation is $\lambda_{4,\frac{n}{2}}y^{1+n}$. 
This is the reason why we treat separately 
the case $k=1$. Let us see that $f_k=ax^{n-k-2}y^{k+1}+bz$, for some $a,b\in\mbk\lser x,y,z\rser$, 
where $a$ is invertible. Indeed,
\begin{eqnarray*}
f_k&=&\left(\sum_{j=0}^{\frac{k-1}{2}}\lambda_j
x^{\frac{2n-k-3-2j}{2}}y^{2j}z^{\frac{k-1-2j}{2}}+\sum_{j=0}^{\frac{n-k-3}{2}}
\lambda_{1,j}x^{\frac{n-k-3-2j}{2}}y^{1+2j}z^{\frac{n+k-3-2j}{2}}\right.+\\
&&\phantom{+}\sum_{j=1}^{\frac{k-1}{2}}\lambda_{2,j}x^{\frac{2n-k-1-2j}{2}}y^{2j}z^{\frac{k-1-2j}{2}}+
\sum_{j=0}^{\frac{n-k-1}{2}}\lambda_{3,j}x^{\frac{n-k-1-2j}{2}}y^{1+2j}z^{\frac{n+k-3-2j}{2}}+\\
&&\phantom{+}\left.\sum_{j=3}^{\frac{n-k+1}{2}}\lambda_{4,j}
x^{\frac{n-k+1-2j}{2}}y^{1+2j}z^{\frac{n+k-3-2j}{2}}\right)z+
\left(\lambda_{\frac{k+1}{2}}+\lambda_{2,\frac{k+1}{2}}x\right)x^{n-k-2}y^{k+1}.
\end{eqnarray*}
By Lemma~\ref{lemma-c-i}, $\lambda_{\frac{k+1}{2}}\neq 0$.
Thus, $f_k=ax^{n-k-2}y^{k+1}+bz$, where
$a=\left(\lambda_{\frac{k+1}{2}}+\lambda_{2,\frac{k+1}{2}}x\right)$ is invertible. 

\vspace*{0.3cm}

Suppose that $n$ is even and that $2\leq k\leq n-4$, with $k$ even. 
By the equalities \eqref{equality-gknevenkeven} and \eqref{poly-fknevenkeven}, 
\begin{eqnarray*}
f_k&=&\sum_{j=0}^{\frac{k}{2}}\lambda_j
x^{\frac{2n-k-4-2j}{2}}y^{1+2j}z^{\frac{k-2j}{2}}+
\sum_{j=0}^{\frac{n-k-2}{2}}
\lambda_{1,j}x^{\frac{n-k-2-2j}{2}}y^{2j}z^{\frac{n+k-2j}{2}}+\\&&
\sum_{j=1}^{\frac{k}{2}}\lambda_{2,j}x^{\frac{2n-k-2-2j}{2}}y^{1+2j}z^{\frac{k-2j}{2}}+
\sum_{j=0}^{\frac{n-k}{2}}\lambda_{3,j}x^{\frac{n-k-2j}{2}}y^{2j}z^{\frac{n+k-2j}{2}}+\\&&
\sum_{j=4}^{\frac{n-k+2}{2}}\lambda_{4,j}x^{\frac{n-k+2-2j}{2}}y^{2j}z^{\frac{n+k-2j}{2}},
\end{eqnarray*}
where the last summation is taken void if $k=n-4$. 
Let us see that $f_k=ax^{n-k-2}y^{k+1}+bz$, for some $a,b\in\mbk\lser x,y,z\rser$, 
where $a$ is invertible. Indeed,
\begin{eqnarray*}
f_k&=&\left(\sum_{j=0}^{\frac{k-2}{2}}\lambda_j
x^{\frac{2n-k-4-2j}{2}}y^{1+2j}z^{\frac{k-2-2j}{2}}+
\sum_{j=0}^{\frac{n-k-2}{2}}
\lambda_{1,j}x^{\frac{n-k-2-2j}{2}}y^{2j}z^{\frac{n+k-2-2j}{2}}\right.+\\
&&\phantom{+}\sum_{j=1}^{\frac{k-2}{2}}\lambda_{2,j}x^{\frac{2n-k-2-2j}{2}}y^{1+2j}z^{\frac{k-2-2j}{2}}+
\sum_{j=0}^{\frac{n-k}{2}}\lambda_{3,j}x^{\frac{n-k-2j}{2}}y^{2j}z^{\frac{n+k-2-2j}{2}}+\\
&&\phantom{+}
\left.\sum_{j=4}^{\frac{n-k+2}{2}}\lambda_{4,j}x^{\frac{n-k+2-2j}{2}}y^{2j}z^{\frac{n+k-2-2j}{2}}
\right)z+\left(\lambda_{\frac{k}{2}}+\lambda_{2,\frac{k}{2}}x\right)x^{n-k-2}y^{k+1}.
\end{eqnarray*}
By Lemma~\ref{lemma-c-i+1}, $\lambda_{\frac{k}{2}}\neq 0$.
Thus, $f_k=ax^{n-k-2}y^{k+1}+bz$, where 
$a=\left(\lambda_{\frac{k}{2}}+\lambda_{2,\frac{k}{2}}x\right)$ is invertible.

\vspace*{0.3cm}

Suppose that $n$ is even and $k=n-3$. By the equalities
\eqref{equality-gk1-1} and \eqref{poly-fk1-1},
\begin{eqnarray*}
f_k&=&\sum_{j=0}^{\frac{n-2}{2}}\lambda_jx^{\frac{n-2j}{2}}y^{2j}z^{\frac{n-2-2j}{2}}+
\lambda_{1,0}yz^{n-2}-x^2y^{n-2}+xyz^{n-2}+y^3z^{n-3}\\
&=&\left(\sum_{j=0}^{\frac{n-4}{2}}\lambda_jx^{\frac{n-2j}{2}}y^{2j}z^{\frac{n-4-2j}{2}}+
\lambda_{1,0}yz^{n-3}+xyz^{n-3}+y^3z^{n-4}\right)z+\left(\lambda_{\frac{n-2}{2}}
-x\right)xy^{n-2}.
\end{eqnarray*}
By Lemma~\ref{lemma-ic1-1}, $\lambda_{\frac{n-2}{2}}\neq 0$. 
Thus, $f_{n-3}=axy^{n-2}+bz$, for some $a,b\in\mbk\lser x,y,z\rser$, where 
$a=\left(\lambda_{\frac{n-2}{2}}+\lambda_{2,\frac{k+1}{2}}x\right)$ is invertible. Hence, 
$(f_{n-3},z)=(xy^{n-2},z)$. This shows $(2)$. 

\vspace*{0.3cm}

\noindent $\bullet$ \emph{Proof of $(3)$}. Suppose that $n$ is odd. 
By the equalities \eqref{equality-gk00} and \eqref{poly-fk00}, 
\begin{eqnarray*}
f_{n-2}&=&\sum_{j=0}^{\frac{n-1}{2}}(-1)^jb_{\frac{n-1}{2},\frac{n-1-2j}{2}}\cdot 
x^{\frac{n-1-2j}{2}}y^{2j}z^{\frac{n-1-2j}{2}}-xy^{n-1}+yz^{n-1}\\
&=&\left(\sum_{j=0}^{\frac{n-3}{2}}(-1)^jb_{\frac{n-1}{2},\frac{n-1-2j}{2}}\cdot 
x^{\frac{n-1-2j}{2}}y^{2j}z^{\frac{n-3-2j}{2}}+yz^{n-2}\right)z+
\left((-1)^{\frac{n-1}{2}}-x\right)y^{n-1}.
\end{eqnarray*}
Since $(-1)^{\frac{n-1}{2}}-x$ is invertible in $\mbk\lser x,y,z\rser$, 
then $(f_{n-2},z)=(y^{n-1},z)$.
Suppose that $n$ is even. By the equality \eqref{poly-fk10a}, 
$f_{n-2}=xy^{n-3}z-y^{n-1}-z^{n-1}=(xy^{n-3}-z^{n-2})z-y^{n-1}$. 
Thus, $(f_{n-2},z)=(y^{n-1},z)$, as well. This shows $(3)$. 

\vspace*{0.3cm}

\noindent $\bullet$ \emph{Proof of $(4)$}. 
By the equalities~\ref{equality-gknevenkodd} and \ref{poly-fknevenkodd} (with $n$ even and $k=1$)
\begin{eqnarray*}
f_1&=&\sum_{j=0}^{1}\lambda_j
x^{n-2-j}y^{2j}z^{1-j}+
\sum_{j=0}^{\frac{n-4}{2}}
\lambda_{1,j}x^{\frac{n-4-2j}{2}}y^{1+2j}z^{\frac{n-2j}{2}}+\lambda_{2,1}x^{n-2}y^{2}\\&&
\sum_{j=0}^{\frac{n-2}{2}}\lambda_{3,j}x^{\frac{n-2-2j}{2}}y^{1+2j}z^{\frac{n-2j}{2}}+
\sum_{j=3}^{\frac{n}{2}}\lambda_{4,j}x^{\frac{n-2j}{2}}y^{1+2j}z^{\frac{n-2j}{2}}
\\&=&\left(x^{n-2}+\sum_{j=0}^{\frac{n-4}{2}}
\lambda_{1,j}x^{\frac{n-4-2j}{2}}y^{1+2j}z^{\frac{n-2-2j}{2}}+
\sum_{j=0}^{\frac{n-2}{2}}\lambda_{3,j}x^{\frac{n-2-2j}{2}}y^{1+2j}z^{\frac{n-2-2j}{2}}+\right.\\&&
\phantom{+}\left.\sum_{j=3}^{\frac{n-2}{2}}\lambda_{4,j}x^{\frac{n-2j}{2}}y^{1+2j}z^{\frac{n-2-2j}{2}}\right)z+
\left(-1+\lambda_{2,1}x\right)x^{n-3}y^2+\lambda_{4,\frac{n}{2}}y^{n+1}.
\end{eqnarray*}
(We used that $\lambda_0=1$ and $\lambda_1=-1$.)
Therefore, $f_1=ax^{n-3}y^{2}+by^{n+1}+cz$, where $a,b,c\in\mbk\lser x,y,z\rser$ and $a=(-1+\lambda_{2,1}x)$ is invertible. So, $f_1\in (x^{n-3}y^2,y^{n+1},z)\subset (x^{n-3}y^2,y^{n-1},z)$. 
By $(3)$, $(f_{n-2},z)=(y^{n-1},z)$. In particular, $(f_1,f_{n-2},z)=(f_1,y^{n-1},z)\subseteq 
(x^{n-3}y^2,y^{n-1},z)$. On the other hand, since $a$ is invertible, 
then $x^{n-3}y^2\in (f_1,y^{n+1},z)\subset (f_1,y^{n-1},z)$. We conclude that
$(f_1,f_{n-2},z)=(x^{n-3}y^2,y^{n-1},z)$. This proves $(4)$. 

\vspace*{0.3cm}

\noindent $\bullet$ \emph{Proof of $(5)$}.
Suppose that $n$ is odd. By the equalities \eqref{equality-gk21} and \eqref{poly-fk21}, 
\begin{eqnarray*}
f_{n-1}&=&\sum_{j=0}^{\frac{n-3}{2}}(-1)^jb_{\frac{n-3}{2},\frac{n-3-2j}{2}}\cdot 
x^{\frac{n-3-2j}{2}}y^{1+2j}z^{\frac{n-1-2j}{2}}-x^n+\\&&
\sum_{j=1}^{\frac{n-1}{2}}\lambda_{2,j}x^{\frac{n-1-2j}{2}}y^{1+2j}z^{\frac{n-1-2j}{2}}+
\sum_{j=1}^{\frac{n-1}{2}}\lambda_{3,j}x^{\frac{n+1-2j}{2}}y^{1+2j}z^{\frac{n-1-2j}{2}}+
\lambda_{4,1}y^2z^{n-1}\\
&=&\left(\sum_{j=0}^{\frac{n-3}{2}} (-1)^jb_{\frac{n-3}{2},\frac{n-3-2j}{2}}
x^{\frac{n-3-2j}{2}}y^{1+2j}z^{\frac{n-3-2j}{2}}\right.+
\sum_{j=1}^{\frac{n-3}{2}}\lambda_{2,j}x^{\frac{n-1-2j}{2}}y^{1+2j}z^{\frac{n-3-2j}{2}}+\\
&&\left.\sum_{j=1}^{\frac{n-3}{2}}\lambda_{3,j}x^{\frac{n+1-2j}{2}}y^{1+2j}z^{\frac{n-3-2j}{2}} 
+\lambda_{4,0}y^2z^{n-2}\right)z-x^n+\left(\lambda_{2,\frac{n-1}{2}}y+
\lambda_{3,\frac{n-1}{2}}xy\right)y^{n-1}.
\end{eqnarray*}
Therefore, $f_{n-1}=-x^n+ay^{n-1}+bz$, where $a,b\in\mbk\lser x,y,z\rser$.
So, $f_{n-1}\in (x^{n},y^{n-1},z)\subset (x^n,y^{n-1},z)$. 
By $(3)$, $(f_{n-2},z)=(y^{n-1},z)$. In particular, 
$(f_{n-2},f_{n-1},z)=(f_{n-1},y^{n-1},z)\subseteq 
(x^{n},y^{n-1},z)$. On the other hand, $x^{n}\in (f_{n-1},y^{n-1},z)$. 
We conclude that
$(f_{n-2},f_{n-1},z)=(x^{n},y^{n-1},z)$.

Now suppose that $n$ is even. 
By the equalities \eqref{equality-gk10b} and \eqref{poly-fk10b}, 
\begin{eqnarray*}
f_{n-1}&=&\sum_{j=0}^{\frac{n-2}{2}}(-1)^{j}b_{\frac{n-2}{2},\frac{n-2-2j}{2}}
x^{\frac{n-2-2j}{2}}y^{1+2j}z^{\frac{n-2-2j}{2}}-x^n+\\
&&\sum_{j=0}^{\frac{n-2}{2}}\lambda_{3,j}x^{\frac{n-2j}{2}}y^{1+2j}z^{\frac{n-2-2j}{2}}+
\sum_{j=2}^{\frac{n-2}{2}}\lambda_{4,j}x^{\frac{n+2-2j}{2}}y^{1+2j}z^{\frac{n-2-2j}{2}}+
\lambda_{5,1}xy^2z^{n-2}+\lambda_{5,2}y^4z^{n-3}\\
&=&\left(\sum_{j=0}^{\frac{n-4}{2}}(-1)^{j}b_{\frac{n-2}{2},\frac{n-2-2j}{2}}
x^{\frac{n-2-2j}{2}}y^{1+2j}z^{\frac{n-4-2j}{2}}+
\sum_{j=0}^{\frac{n-4}{2}}\lambda_{3,j}x^{\frac{n-2j}{2}}y^{1+2j}z^{\frac{n-4-2j}{2}}\right.+\\&&
\left.\sum_{j=2}^{\frac{n-4}{2}}\lambda_{4,j}x^{\frac{n+2-2j}{2}}y^{1+2j}z^{\frac{n-4-2j}{2}}+
\lambda_{5,1}xy^2z^{n-3}+\lambda_{5,2}y^{4}z^{n-4}\right)z+\\
&&\left((-1)^{\frac{n-2}{2}}+ \lambda_{3,\frac{n-2}{2}}x+
\lambda_{4,\frac{n-2}{2}}x^2\right)y^{n-1}-x^n.
\end{eqnarray*}
Therefore, $f_{n-1}=-x^n+ay^{n-1}+bz$, where $a,b\in\mbk\lser x,y,z\rser$.
So, $f_{n-1}\in (x^{n},y^{n-1},z)\subset (x^n,y^{n-1},z)$. 
By $(3)$, $(f_{n-2},z)=(y^{n-1},z)$. In particular, 
$(f_{n-2},f_{n-1},z)=(f_{n-1},y^{n-1},z)\subseteq 
(x^{n},y^{n-1},z)$. On the other hand, $x^{n}\in (f_{n-1},y^{n-1},z)$. 
Hence, $(f_{n-2},f_{n-1},z)=(x^{n},y^{n-1},z)$. This shows $(5)$. 

\vspace*{0.3cm}

\noindent $\bullet$ \emph{Proof of $(6)$}. 
Suppose that $n$ is odd. By the equalities~\ref{equality-gk22} and \ref{poly-fk22},
\begin{eqnarray*}
f_n&=&\sum_{j=0}^{\frac{n-3}{2}}\lambda_j 
x^{\frac{n-3-2j}{2}}y^{2j}z^{\frac{n+1-2j}{2}}-x^{n-1}y+\\&&
\sum_{j=1}^{\frac{n-1}{2}}\lambda_{2,j}x^{\frac{n-1-2j}{2}}y^{2j}z^{\frac{n+1-2j}{2}}+
\sum_{j=2}^{\frac{n+1}{2}}\lambda_{3,j}x^{\frac{n+1-2j}{2}}y^{2j}z^{\frac{n+1-2j}{2}}\\
&=&\left(\sum_{j=0}^{\frac{n-3}{2}}\lambda_j 
x^{\frac{n-3-2j}{2}}y^{2j}z^{\frac{n+1-2j}{2}}+
\sum_{j=1}^{\frac{n-1}{2}}\lambda_{2,j}x^{\frac{n-1-2j}{2}}y^{2j}z^{\frac{n-1-2j}{2}}\right.+\\&&
\phantom{+}\left.\sum_{j=2}^{\frac{n-1}{2}}\lambda_{3,j}x^{\frac{n+1-2j}{2}}y^{2j}z^{\frac{n+1-2j}{2}}\right)z
-x^{n-1}y+\lambda_{3,\frac{n+1}{2}}y^{n+1}.
\end{eqnarray*}
Therefore, $f_n=-x^{n-1}y+ay^{n+1}+bz$, where $a,b\in\mbk\lser x,y,z\rser$.
So, $f_n\in (x^{n-1}y,y^{n+1},z)\subset (x^{n-1}y,y^{n-1},z)$. 
By $(3)$, $(f_{n-2},z)=(y^{n-1},z)$. In particular, $(f_{n-2},f_n,z)=(f_n,y^{n-1},z)\subseteq 
(x^{n-1}y,y^{n-1},z)$. On the other hand, $x^{n-1}y\in (f_n,y^{n+1},z)\subset (f_n,y^{n-1},z)$. 
Hence, $(f_{n-2},f_{n},z)=(x^{n-1}y,y^{n-1},z)$.

Now, suppose that $n$ is even. 
By the equalities~\ref{equality-gk11} and \ref{poly-fk11},
\begin{eqnarray*}
f_n&=&\sum_{j=0}^{\frac{n-2}{2}}
\lambda_jx^{\frac{n-2-2j}{2}}y^{2j}z^{\frac{n-2j}{2}}-x^{n-1}y+
\sum_{j=1}^{\frac{n}{2}}\lambda_{2,j}x^{\frac{n-2j}{2}}y^{2j}z^{\frac{n-2j}{2}}+\\&&
\sum_{j=3}^{\frac{n}{2}}\lambda_{3,j}x^{\frac{n+2-2j}{2}}y^{2j}z^{\frac{n-2j}{2}}+
\lambda_{4,1}y^3z^{n-2}\\
&=&\left(\sum_{j=0}^{\frac{n-2}{2}}\lambda_jx^{\frac{n-2-2j}{2}}y^{2j}z^{\frac{n-2-2j}{2}}+
\sum_{j=1}^{\frac{n-2}{2}}\lambda_{2,j}x^{\frac{n-2j}{2}}y^{2j}z^{\frac{n-2-2j}{2}}\right.+\\&&
\phantom{+}\left.\sum_{j=3}^{\frac{n-2}{2}}\lambda_{3,j}x^{\frac{n+2-2j}{2}}y^{2j}z^{\frac{n-2-2j}{2}}+
\lambda_{4,1}y^3z^{n-3}\right)z-x^{n-1}y+
\left(\lambda_{2,\frac{n}{2}}+\lambda_{3,\frac{n}{2}}x\right)y^{n}.
\end{eqnarray*}
Thus, $f_n=-x^{n-1}y+ay^{n}+bz$, where $a,b\in\mbk\lser x,y,z\rser$.
So, $f_n\in (x^{n-1}y,y^{n},z)\subset (x^{n-1}y,y^{n-1},z)$. 
By $(3)$, $(f_{n-2},z)=(y^{n-1},z)$. In particular, $(f_{n-2},f_n,z)=(f_n,y^{n-1},z)\subseteq 
(x^{n-1}y,y^{n-1},z)$. On the other hand, $x^{n-1}y\in (f_n,y^{n},z)\subset (f_n,y^{n-1},z)$. 
We conclude that $(f_{n-2},f_{n},z)=(x^{n-1}y,y^{n-1},z)$.
This shows $(6)$. 

\vspace*{0.3cm}

\noindent \noindent $\bullet$ \emph{Proof of $(7)$}. Suppose that $n$ is odd. 
Then, by $(1)$, 
\begin{eqnarray*}
(f_1,f_2,\ldots,f_{n-3},z)=(x^{n-3}y^2,x^{n-4}y^3,\ldots,xy^{n-2},z).   
\end{eqnarray*}
Using $(5)$ and $(6)$, we deduce that: 
\begin{eqnarray*}
I+zR=(f_1,\ldots,f_n,z)=(x^n,x^{n-1}y,x^{n-3}y^2,\ldots,xy^{n-2},y^{n-1},z).  
\end{eqnarray*}

Suppose that $n$ is even. 
Then, by $(2)$ and $(4)$, 
\begin{eqnarray*}
(f_1,f_2,\ldots,f_{n-3},f_{n-2},z)=(x^{n-3}y^2,x^{n-4}y^3,\ldots,xy^{n-2},y^{n-1},z). 
\end{eqnarray*}
Using $(5)$ and $(6)$, then:
$I+zR=(f_1,\ldots,f_n,z)=
(x^n,x^{n-1}y,x^{n-3}y^2,\ldots,xy^{n-2},y^{n-1},z)$.
\end{proof}

\vspace*{0.3cm}

\begin{lemma}\label{lemma-In3}
Let $n\geq 3$. Let $R=\mbk\lser x,y,z\rser$, $\mfm=(x,y,z)$ and $I=(f_1,\ldots,f_n)$. Then 
$\height(I)=2$, $I+zR$ is an $\mfm$-primary ideal of $R$ and $\length_R(R/(I+zR))=a+2$.
\end{lemma}
\begin{proof}
Let $P=\ker(\rho)$. 
By Theorem~\ref{FourthStep}, $I=(f_1,\ldots,f_n)\subseteq P$. 
Clearly, $P\subset\mfm$, because if $g\not\in\mfm$, then $g$ is invertible, $1=\rho(gg^{-1})=\rho(g)\rho(g^{-1})$ and $g\not\in P$.
Since $\rho(z)=t^{a+2}$, $z\not\in P$. Thus,
$I\subseteq P\subsetneq\mfm$. Since $P$ is prime, 
then $\height(I)\leq \height(P)<\height(\mfm)=3$. Therefore, $\height(I)\leq \height(P)\leq 2$.

\vspace*{0.2cm}

Suppose that $n=3$. Let $f_1,f_2,f_3$ be as in \eqref{equality-gfn=3}.
Then $(f_1,f_2,z)=(x^3,(1+x)y^2,z)=(x^3,y^2,z)$. 
By \cite[Corollary 1.6.19.]{bh}, $\grade(f_{1},f_{2},z)=\grade(x^3,y^2,z)=3$ 
and $f_{1},f_{2},z$ is a regular sequence. Therefore,
$2\leq\grade(I)=\height(I)\leq\height(P)\leq 2$ and $\height(I)=\height(P)=2$.
Since $I+zR=(x^3,x^2y,y^2,z)$, it follows that
$I+zR$ is an $\mfm$-primary ideal of $R$, 
$\height(I+zR)=3$, $R/(I+zR)$ is an Artinian ring and 
$\length_R(R/(I+zR))$ is finite. 

\vspace*{0.2cm}

Suppose that $n=4$. Let $f_1,f_2,f_3,f_4$ be as in \eqref{equality-gfn=4}
and $I=(f_1,f_2,f_3,f_4)$.
Then $(f_2,f_4,z)=(x^4,y^3,z)$. 
By \cite[Corollary 1.6.19.]{bh}, $\grade(f_{2},f_{4},z)=\grade(x^4,y^3,z)=3$ 
and $f_{2},f_{4},z$ is a regular sequence. Therefore, $\height(I)=\height(P)=2$.
Since, $I+zR=(x^4,x^3y,xy^2,y^3,z)$,
it follows that
$I+zR$ is an $\mfm$-primary ideal of $R$, 
$\height(I+zR)=3$, $R/(I+zR)$ is an Artinian ring and 
$\length_R(R/(I+zR))$ is finite. 

\vspace*{0.2cm}

Suppose that $n=5$. Let $f_1,f_2,f_3,f_4,f_5$ be as in \eqref{equality-gfn=5}.
Then $(f_3,f_4,z)=(x^5,(1+x)y^4,z)=(x^5,y^4,z)$.
By \cite[Corollary 1.6.19.]{bh}, $\grade(f_{3},f_{4},z)=\grade(x^5,y^4,z)$
and $f_{3},f_{4},z$ is a regular sequence. Therefore, $\height(I)=\height(P)=2$.
Since $I+zR=(x^5,x^4y,x^2y^2,xy^3,y^4,z)$, it follows that
$I+zR$ is an $\mfm$-primary ideal of $R$, 
$\height(I+zR)=3$, $R/(I+zR)$ is an Artinian ring and 
$\length_R(R/(I+zR))$ is finite. 

\vspace*{0.2cm}

Suppose that $n\geq 6$. 
By Lemma~\ref{lemma-In6},~$(5)$, $(f_{n-2},f_{n-1},z)=(x^n,y^{n-1},z)$. 
By \cite[Corollary 1.6.19.]{bh}, $\grade(f_{n-2},f_{n-1},z)=\grade(x^{n},y^{n-1},z)=3$ 
and $f_{n-2},f_{n-1},z$ is a regular sequence. Therefore, $\height(I)=\height(P)=2$.
By Lemma~\ref{lemma-In6},~$(7)$,
$I+zR=(x^n,x^{n-1}y,x^{n-3}y^2,\ldots,xy^{n-2},y^{n-1},z)$. 
Therefore, $I+zR$ is an $\mfm$-primary ideal of $R$,
$\height(I+zR)=3$, $R/(I+zR)$ is an Artinian ring and 
$\length_R(R/(I+zR))$ is finite.

Let $S=R/zR$, which is a regular local ring. Let $\mfm_S=\mfm/zR$ be its maximal ideal and 
$k_S=S/\mfm_S=R/\mfm=\mbk$ its residue field. By abuse of notation, let $x,y$ still 
denote the classes of the variables in $S$, so $\mfm_S=(x,y)$. 
Let $J=I+zR$ and set $\overline{J}=J/zR$. 

If $n=3$, $\overline{J}=(x^3,x^2y,y^2)=x^2\mfm_S+y^2R\subset\mfm_S^2$. 

If $n=4$, $\overline{J}=(x^4,x^3y,xy^2,y^3)=x^3\mfm_S+\mfm_Sy^2\subset\mfm_S^3$. 

If $n=5$, $\overline{J}=(x^5,x^4y,x^2y^2,xy^3,y^4)=x^4\mfm_S+\mfm_S^2y^2\subset\mfm_S^4$. 

If $n\geq 6$, $\overline{J}=(x^n,x^{n-1}y,x^{n-3}y^2,x^{n-4}y^3,\ldots,xy^{n-2},y^{n-1})= 
x^{n-1}\mfm_S+\mfm_S^{n-3}y^2\subset \mfm_S^{n-1}$.

Since $R\to S$ is a surjective map and $R/J=S/\overline{J}$, it follows that 
\begin{eqnarray*}
\length_R(R/(I+zR))=\length_R(R/J)=\length_R(S/\overline{J})=\length_S(S/\overline{J}), 
\end{eqnarray*}
because the set of $R$-submodules of $S/\overline{J}$ coincides with the set of 
$S$-submodules of $S/\overline{J}$. Consider the short exact sequences:
\begin{eqnarray*}
0\to \mfm_S/\overline{J}\to S/\overline{J}\to S/\mfm_S\to 0,\phantom{+}
0\to \mfm_S^{i+1}/\overline{J}\to\mfm_S^{i}/\overline{J}
\to\mfm_S^{i}/\mfm_S^{i+1}\to 0,\text{ for }1\leq i\leq n-2.
\end{eqnarray*}
Using the additivity of length and that the length of a vector space coincides with its dimension:
\begin{eqnarray*}
\length_{S}(S/\overline{J})&=&\length_{S}(S/\mfm_S)+\sum_{i=1}^{n-2}\length_S(\mfm_S^i/\mfm_S^{i+1})+\length_S(\mfm_S^{n-1}/\overline{J})\\
&=&\frac{n(n-1)}{2}+\length_S(\mfm_S^{n-1}/\overline{J})=a+\length_S(\mfm_S^{n-1}/\overline{J})
\end{eqnarray*} 
Since 
$\mfm_S(\mfm_S^{n-1}/\overline{J})=0$, it follows that 
$\mfm_S^{{n-1}}/\overline{J}=\langle x^{n-1},x^{n-2}y\rangle$ is a two-dimensional 
vector space. Thus, $\length_S(S/\overline{J})=a+2$.
\end{proof}


\begin{lemma}\label{lemma-P}
Let $n\geq 3$. Let $R=\mbk\lser x,y,z\rser$, $\mfm=(x,y,z)$, $P=\ker(\rho)$ and $D=R/P$. The following hold.
\begin{itemize}
\item[$(1)$] The ideal $P$ has height $2$ and 
the length of $R/(P+zR)$ is finite as an $R$-module.
\item[$(2)$] $D=\mbk\lser t^a+t^{a+q},t^{a+1},t^{a+2}\rser$ is a one-dimensional 
Cohen-Macaulay local domain.
\item[$(3)$] $D[t]=\mbk\lser t\rser$.
\item[$(4)$] The field of fractions of $D$ is $K(D)=K(\mbk\lser t\rser)$.
\item[$(5)$] The integral closure of $D=\mbk\lser t^a+t^{a+q},t^{a+1},t^{a+2}\rser$
in its field of fractions is $\mbk\lser t\rser$.
\item[$(6)$] If $V=\mbk\lser t\rser$, then 
$\length_R(R/(P+zR))=\length_D(D/zD)=\length_D(V/zV)$.
\item[$(7)$] $\length_R(R/(P+zR))=\length_D(V/zV)=\length_V(V/zV)=a+2$.
\end{itemize}
\end{lemma}
\begin{proof}[Proof of $(1)$]
By Theorem~\ref{FourthStep}, $I=(f_1,\ldots,f_n)\subseteq P$, where $P\subsetneq\mfm$,
and $I+zR\subseteq P+zR\subset\mfm$. 
By Lemma~\ref{lemma-In3}, $\height(I)=2$ and $I+zR$ is $\mfm$-primary. Thus, 
$\height(P)=2$, $P+zR$ is $\mfm$-primary,
$\height(P+zR)=3$, $R/(P+zR)$ is Artinian and $\length_R(R/(P+zR))$ is finite.

\vspace*{0.3cm}

\noindent $\bullet$ \emph{Proof of $(2)$}.
Observe that $R$ is a regular local ring with maximal ideal $\mfm$, in particular,
Cohen-Macaulay local. Thus, $\dim(D)=\dim(R/P)=\dim(R)-\height(P)=1$.
Since $R$ is Noetherian local and $P$ is a prime ideal of $R$, 
then $D=R/P$ is a Noetherian local domain. 
In particular, any non-zero and non-invertible element of $D$ 
is a $D$-regular sequence, so  
$1\leq\grade(\mfm_D,D)=\depth(D)\leq\dim(D)=1$ and $D$ is a Cohen-Macaulay ring.
Let $\mfm_D=\mfm/P$ be its maximal ideal and $k_D=D/\mfm_D=R/\mfm=\mbk$ its residue field. By abuse of notation, let $x,y,z$ still denote the classes of the variables in $D$, so $\mfm_D=(x,y,z)$. 
Note that $D=R/P=\mbk\lser x,y,z\rser/\ker(\rho)\cong \im(\rho)=\mbk\lser t^a+t^{a+q},t^{a+1},t^{a+2}\rser$. 
Through this isomorphism, $x$, $y$ and $z$ are identified with $t^a+t^{a+q}$, $t^{a+1}$ and $t^{a+2}$, respectively. 

\vspace*{0.3cm}

\noindent $\bullet$ \emph{Proof of $(3)$}.
Clearly, $D[t]=\mbk\lser t^a+t^{a+q},t^{a+1},t^{a+2}\rser[t]$ is a subring of 
$\mbk\lser t\rser$. Let us prove the equality. To do that, consider 
the numerical semigroup $S=\langle a+1,a+2\rangle$, whose 
Frobenius number is $\frob(S)=(a+1)(a+2)-(a+1)-(a+2)=a^2+a-1$ 
(see, e.g., \cite[Theorem~2.1.1]{ramirez}). 

Let $f\in\mbk\lser t\rser$. Write $f=\varphi_1+\varphi_2$, where
$\varphi_1:=\sum_{i=0}^{a^2+a-1}\lambda_it^{i}$ and $\varphi_2:=\sum_{i\geq 0}\lambda_{a^2+a+i}t^{a^2+a+i}$. Clearly $\varphi_1\in\mbk[t]\subset D[t]$. Since $a^2+a-1$ is the last number not in $S$, it follows that $t^{a^2+a+i}\in\mbk\lser t^{a+1},t^{a+2}\rser$ for all $i\geq 0$. Thus, $\varphi_2\in D$. So, $f\in D[t]$ and $D[t]=\mbk\lser t\rser$. 
 
\vspace*{0.3cm}

\noindent $\bullet$ \emph{Proof of $(4)$}.
Since $D=\mbk\lser t^{a}+t^{a+q},t^{a+1},t^{a+2}\rser\subset
\mbk\lser t\rser$, then $K(D)\subseteq K(\mbk\lser t\rser)$.
To see the other inclusion, take $g\in K(\mbk\lser t\rser)$. 
By {\sc Step 5}, $\mbk\lser t\rser=D[t]$. Thus, 
$g\in K(\mbk\lser t\rser)=K(D[t])$ and
$g=p(t)/q(t)$, where $p(t),q(t)\in D[t]$, $q(t)\neq 0$.
Since $t^{a+1}$ and $t^{a+2}$ are in $D$, then $t=t^{a+2}/t^{a+1}\in K(D)$.
Since $K(D)$ is a field that contains $t$ and $D$, it follows that $K(D)$ 
contains too any quotient of polynomial expressions in $t$ with coefficients in $D$, thus
$g=p(t)/q(t)\in K(D)$.

\vspace*{0.3cm}

\noindent $\bullet$ \emph{Proof of $(5)$}.
The element $t$ is a root of the monic polynomial $T^{a+1}-t^{a+1}\in D[T]$.
Thus, $t$ is integral over $D$ and $D\subset D[t]$ is an integral extension.
Since $\mbk\lser t\rser$ is regular local, 
it is integrally closed in its field of fractions.
Since $D\subset D[t]=\mbk\lser t\rser$ is an integral extension, 
$\mbk\lser t\rser$ is integrally closed in $K((\mbk\lser t\rser)$
and $K(D)=K(\mbk\lser t\rser)$, then the integral closure of $D$ 
is $\mbk\lser t\rser$.

\vspace*{0.3cm}

\noindent $\bullet$ \emph{Proof of $(6)$}. 
Since $R\to D=R/P$ is a surjective map and $D/zD=R/(P+zR)$ is a $D$-module, 
it follows that $\length_R(R/(P+zR))=\length_R(D/zD)=\length_D(D/zD)$.

Since $z$ is a $D$-regular sequence in $\mfm_D$
and $D$ is a one dimensional Cohen-Macaulay local ring, then $z$ is a system of parameters of $D$ 
(see, e.g., \cite[\S14]{matsumura} and \cite[Theorem~2.1.2]{bh}). 

Let $V=\mbk\lser t\rser$, which is a discrete valuation ring, DVR for short. 
Set $\mfm_V=(t)$ its maximal ideal and $k_V=V/\mfm_V=\mbk$ its residue field. 

We have seen that $D\subset D[t]=\mbk\lser t\rser=V$, where $t$ is integral over $D$. 
Thus, $V$ is a finitely generated $D$-module (see, e.g., \cite[Proposition~5.1]{am}). 
Since $t^{a+2}$ is a $V$-regular sequence in $\mfm_D$, then
$1\leq \grade(\mfm_D,V)=\depth_D(V)\leq \dim V=1$ and $V$ is a finitely generated 
Cohen-Macaulay $D$-module 
(see, e.g., \cite[Definition~1.2.7, Proposition~1.2.12 and Definition~2.1.1]{bh}).

Note that $K(D)=(D\setminus\{0\})^{-1}D\subseteq (D\setminus\{0\})^{-1}V
\subseteq (D\setminus\{0\})^{-1}K(D)=K(D)$. Therefore, 
$V\otimes_DK(D)=(D\setminus\{0\})^{-1}V=K(D)$.
So, $\rank_D(V)=1$ (see, e.g., \cite[Definition~1.4.2]{bh}). 
Using that $D$ is a Cohen-Macaulay local ring, that $V$ is a finitely generated
Cohen-Macaulay $D$-module of $\rank_D(V)=1$ and that
$z$ is a system of parameters of $D$, we conclude that 
$\length_D(V/zV)=\length_D(D/zD)\cdot \rank_D(V)=\length_D(D/zD)$
(see, e.g., \cite[Corollary~4.6.11]{bh}).

\vspace*{0.3cm}

\noindent $\bullet$ \emph{Proof of $(7)$}.
Let $0=M_{a+2}\subset M_{a+1}\subset\ldots\subset M_1\subset M_0=V/zV$ 
be the chain defined by 
$M_i=\mfm_V^{i}/\mfm_V^{a+2}$, $0\leq i\leq a+2$, where each $M_i$ is both a 
$V$-module and a $D$-module. 
Note that $M_0=V/t^{a+2}V=V/zV$. Each consecutive quotient satisfies 
$M_i/M_{i+1}=\mfm_V^i/\mfm_V^{i+1}\cong k_V=k_D$, which is both a simple 
$V$-module and a simple $D$-module.
Thus, this chain is a series of composition of $V/zV$, 
considered as a $V$-module and considered as a $D$-module. It follows that $\length_D(V/zV)
=\length_V(V/zV)=a+2$. In particular, by $(6)$, 
$\length_R(R/(P+zR))=\length_D(V/zV)=\length_V(V/zV)=a+2$. 
\end{proof}

We can prove now the main result of the paper. 

\begin{theorem}\label{FifthStep}
Let $n\geq 3$. Then $\{f_1,\ldots,f_n\}$ defines a minimal generating set of $\ker(\rho)$. 
\end{theorem}
\begin{proof}
As in the previous lemmas, set $I=(f_1,\ldots,f_n)$ and $P=\ker(\rho)$.
Since $\rho(z)=t^{a+2}$, then $z\not\in P$, $P\cap zR=zP$ and $\text{Tor}_1(R/P,R/zR)=0$. 
Consider the short exact sequence 
\begin{eqnarray*}
0\rightarrow P/I\rightarrow R/I\rightarrow R/P\rightarrow 0. 
\end{eqnarray*}
On tensor this short exact sequence by $R/zR$ we obtain:
\begin{eqnarray*}
\cdots\rightarrow\underbrace{\text{Tor}_1(R/P,R/zR)}_{=0}\rightarrow (P/I)/z(P/I)\rightarrow R/(I+zR)\rightarrow R/(P+zR)\rightarrow 0.
\end{eqnarray*}
By Lemmas~\ref{lemma-In3} and \ref{lemma-P}, $\length_R(R/(I+zR))=\length_R(R/(P+zR))=a+2$.
By the additivity of the length, we get $\length_R((P/I)/z(P/I))=0$ and $P/I=z(P/I)$. 
By Nakayama's Lemma, $I=P$.
By Theorem~\ref{FourthStep}, $f_1,\ldots,f_n$ is part of a minimal generating set of $P$, so 
$\mu(P)\geq n$. Since $I=P$, then $f_1,\ldots,f_n$ is a generating set of $P$, so $\mu(P)\leq n$. 
Hence, $\mu(P)=n$ and $f_1,\ldots ,f_n$ is a minimal system of generators of $P$. 
\end{proof}

As a corollary, and keeping the same notations for $a$ and $q$, we get:

\begin{corollary}\label{cor-inA}
Let $\mbk$ a field of characteristic zero. Consider the following commutative diagram of 
$\mbk$-algebra homomorphisms. 
\begin{center}
\begin{tikzcd}[row sep = huge, column sep = huge]
A_1=\mbk[x,y,z]\arrow[r, "\varrho_1"]  \arrow[d, "\varphi_1", swap]
&\mbk[t]\arrow[d, "\psi_1"]
\\
A_2=\mbk[x,y,z]_{(x,y,z)} 
\arrow[r, "{\varrho_2}"] \arrow[d, "\varphi_2", swap] 
&\mbk[t]_{(t)}\arrow[d, "\psi_2"] 
\\
R=\mbk\lser x,y,z\rser \arrow[r, "\rho", swap] 
&\mbk\lser t\rser                            
\end{tikzcd}
\end{center}
where $\varrho_1(x)=\varrho_2(x)=\rho(x)=t^a(1+t^q)$,
$\varrho_1(y)=\varrho_2(y)=\rho(y)=t^{a+1}$, $\varrho_1(z)=\varrho_2(z)=\rho(z)=t^{a+2}$, 
$\varphi_1$ and $\psi_1$ are the localization homomorphisms and 
$\varphi_2$ and $\psi_2$ are the completion homomorphisms.
Then the prime ideals $\ker(\varrho_1)$ and 
$\ker(\varrho_2)$ are minimally generated by at least $n$ elements.
\end{corollary}
\begin{proof}
Let $H_2=(f_1,\ldots,f_n)A_2$ be the ideal of $A_2$ generated by the polynomials
$f_1,\ldots,f_n$. By Theorem~\ref{FifthStep}, 
$H_2R=\langle \varphi_2(H_2)\rangle=(f_1,\ldots,f_n)R$ is equal to $\ker(\rho)$, which is
a prime ideal of $R$ minimally generated by 
$f_1,\ldots,f_n$. Since the completion homomorphism $\varphi_2$
is faithfully flat, $H_2=H_2R\cap A_2$, i.e., contraction is left inverse to extension of ideals 
(see, e.g., \cite[\S 8, page 63]{matsumura}). 
In particular, $H_2$ is a prime ideal of $A_2$. 
Moreover, any generating set of $H_2$ has at least $n$ elements, otherwise, if 
$H_2=(g_1,\ldots,g_m)$, with $m<n$, then $\ker(\rho)=H_2R=(g_1,\ldots,g_m)R$ 
would be generated by less than $n$ elements, a contradiction.
Since $\rho\circ\varphi_2=\psi_2\circ\varrho_2$ and $\psi_2$ 
is injective, it follows that
$\psi_2(\varrho_2(f_i))=\rho(\varphi_2(f_i))=\rho(f_i)=0$ and 
$f_i\in\ker(\varrho_2)$. Thus, $H_2\subseteq\ker(\varrho_2)$. Since $H_2$ is a non-principal prime ideal of a Unique Factorization Domain, $\height(H_2)>1$. So $H_2\subseteq\ker(\varrho_2)$ are 
prime ideals with $1<\height(H_2)\leq \height(\ker(\varrho_2))<3$. Therefore,  
$\height(H_2)=\height(\ker(\varrho_2))=2$ and $H_2=\ker(\varrho_2)$.

Let $H_1=\varphi_1^{-1}(H_2)$.
Since $\varrho_2\circ\varphi_1=\psi_1\circ\varrho_1$ and $\psi_1$ is injective, 
it follows that $H_1=\varphi_1^{-1}(H_2)=\varphi_1^{-1}(\ker(\varrho_2))=\ker(\varrho_1)$.
Since $\varphi_1$ is a localization, p$H_2=\langle \varphi_1(\varphi_1^{-1}(H_2))\rangle=\langle\varphi_1(H_1)\rangle=H_1A_2$, i.e., extension is left inverse to contraction of ideals
(see, e.g.,the proof of \cite[Propostion~3.11]{am}). 
In particular, any generating set of $\ker(\varrho_1)$ 
has at least $n$ elements, for, if $\ker(\varrho_1)=H_1=(g_1,\ldots,g_m)$, with $m<n$, then 
$H_2=H_1A_2=(g_1,\ldots,g_m)A_2$, 
a contradiction.
\end{proof}

In connection with the conjecture of Sally in \cite[Remark p. 53]{sally}, we have:

\begin{corollary}\label{cor-Sally}
Let $(A,\mfm)$ be a $3$-dimensional regular local ring containing a field $\mbk$
of characteristic zero. Let $n\in\mbn$, $n\geq 3$. Then there exists a prime ideal $\mfp$ of $A$ minimally generated by $n$ elements. 
\end{corollary}
\begin{proof} 
Let $(R,\mfn)=(\hat{A},\hat{\mfm})$ be the $\mfm$-adic completion of $(A,\mfm)$. 
If $\mfm=(x,y,z)$, then $\mfn=(x,y,z)R$. 
Since $A$ is regular local and $A\supset\mbk$, then $R$ is a complete regular local ring containing
the field $\mbk$. Hence, $R$ is isomorphic to the power series ring $\mbk\lser x,y,z\rser$ in the variables $x,y,z$ over the field $\mbk$. Let $f_1,\ldots,f_n\in R$ be defined as in Definition~\ref{def-gk}
and $P=(f_1,\ldots,f_n)R$. 
By Theorem~\ref{FifthStep}, $P$ is a prime ideal of $R$ minimally generated by $f_1,\ldots,f_n$.
Since $f_1,\ldots,f_n$ are polynomials in $x,y,z$ with coefficients in $\mbk$, they can also be considered as elements of $A$. Let $\mfp=(f_1,\ldots,f_n)A$ be the ideal they generate in $A$.
By definition, $\mfp R=P$. Since the completion homomorphism $\varphi:A\to\hat{A}=R$ is faithfully flat, 
it follows that  $\mfp=\mfp R\cap A$. Thus $\mfp$ is prime. Moreover, $f_1,\ldots,f_n$ is a minimal
generating set of $\mfp$ for, if $\mfp=(q_1,\ldots,q_m)A$, with $m<n$, then $P=\mfp R=(q_1,\ldots,q_m)R$ would be generated by $m$ elements, a contradiction. 
\end{proof}

\section{Example and a Python code}

\allowdisplaybreaks[4]

To help understanding the whole process of the proof, we detail a little bit the case $n=6$.  

\begin{example}\label{genn6}
Let $n=6$. Then $a=(n-1)n/2=15$, $\nums=\langle 15,16,17\rangle$, 
$q=2\lfloor (n+1)/2\rfloor-1=5$,
and $\rho(x)=t^{15}(1+t^5)$, $\rho(y)=t^{16}$, $\rho(z)=t^{17}$. 
Moreover, $s_0=a(n-1)+2=77$. 

According to Summary~\ref{resum}, for $n=6$, the $\sigma$-leading forms
$g_1,\ldots,g_6$ of the minimal generating set $f_1,\ldots,f_6$ are determined by the subspaces
$V_r$, where $r\in [s_0,s_0+n-1)=\{77,78,79,80,81\}$. 
By Lemma~\ref{lemma-pre}, $\ell_r=\lfloor r/a\rfloor=n-1=5$. 
Recall that $\rem_r=r-a\ell_r=r-75$, $\phi_r=(\phi_{r,1},\phi_{r,2},\phi_{r,3})=
\left(\ell_r-\left\lfloor \frac{\rem_r+1}{2}\right\rfloor,
\rem_r-2\left\lfloor \frac{\rem_r}{2}\right\rfloor ,\left\lfloor
\frac{\rem_r}{2} \right\rfloor \right)$, $\iota_r=\phi_{r,2}+|\phi_{r,1}-\phi_{r,3}|$,
$c_{r}=\phi_{r,3}-\phi_{r,1}$, $\kappa_r=\min(\phi_{r,1},\phi_{r,3})$, 
$W_r=\langle\mon^{\phi_r+j\omega}\mid 0\leq j\leq\kappa_r \rangle$ and  $\delta_r=\dim W_r=\kappa_r+1$.
We summarize all this information in the following table, 
where the polynomials $g_1,\ldots,g_6$ are obtained by using the equalities 
\eqref{equality-gknevenkodd}, \eqref{equality-gknevenkeven}, \eqref{equality-gk1-1}, 
\eqref{equality-gk10a}, \eqref{equality-gk10b} and \eqref{equality-gk11}. 
As always, write $r_k=s_0+k-1=77+k-1$, for $1\leq k\leq n$. 

\begingroup
\renewcommand*{\arraystretch}{1.5}
\begin{eqnarray*}
\begin{array}{|c|c|c|c|c|c|c|}\hline
k & r=r_k & (\ell_r,\rem_r) & \phi_r & (\iota_r,c_r) & W_r &  V_r\\  \hline
1 & 77 & (5,2) & (4,0,1) & {\corange (3,-3)} & \langle x^4z,x^3y^2\rangle & 
g_1=x^4z-x^3y^2  \\ \hline
2 & 78 & (5,3) & (3,1,1) & {\cteal (3,-2)} & \langle x^3yz, x^2y^3\rangle & 
g_2=x^3yz-x^2y^3 \\ \hline
3 & 79 & (5,4) & (3,0,2) & {\cred (1,-1)} & \langle x^3z^2, x^2y^2z, xy^4\rangle &  
g_3=x^3z^2-3x^2y^2z+2xy^4 \\ \hline
\begin{array}{c}4\\5\end{array} & 80 & (5,5) & (2,1,2) & {\cgreen (1,0)} &  
\langle x^2yz^2, xy^3z, y^5\rangle &  
\begin{array}{c}g_4=xy^3z-y^5,\\g_5=x^2yz^2-2xy^3z+y^5\end{array} \\ \hline
6 & 81 & (5,6) & (2,0,3) & {\cblue (1,1)} & \langle x^2z^3, xy^2z^2, y^4z\rangle &  
g_6=x^2z^3-2xy^2z^2+y^4z\\ \hline
\end{array}
\end{eqnarray*}
\endgroup

\vspace*{0.2cm}

\noindent Once determined the $\sigma$-leading forms, we calculate the tail of
$f_1$, so $k=1$ and $r=r_1=77$. 
Taking Summary~\ref{resum} as our guide, we look at Lemma~\ref{lemma-h-i}. 
One has $f_1^{\tau}=h_1+h_2+h_3+h_4$. To find $h_1$ we
need to solve the first system in Lemma~\ref{lemma-h-i}, where 
$\Lambda_1^{\top}=(\lambda_{1,0},\lambda_{1,1})$ and $\Lambda^{\top}=(1,-1)$.
\begin{eqnarray*}
&&\left(B^{[0,\lfloor\frac{\iota_r-1}{2}\rfloor]}
_{[0,\lfloor\frac{\iota_r-1}{2}\rfloor]}\right)^\top\cdot\Theta_{\lfloor\frac{\iota_r+1}{2}\rfloor}\cdot\Lambda_1=
-\left(B^{[\iota_r,\frac{n-1+\iota_r}{2}]}_{[1,\lfloor\frac{\iota_r+1}{2}\rfloor]}\right)^\top
\cdot\Theta_{\frac{n+1-\iota_r}{2}}\cdot\Lambda. 
\end{eqnarray*}
Substituting the values of $n=6$, $k=1$, $\iota_r=n-k-2=3$, we get: 
\begin{eqnarray*}
\left(\begin{array}{cc} 1 & 1\\ 0 & 1\end{array}\right)
\left(\begin{array}{cc} 0 & 1\\ 1 & 0\end{array}\right)
\left(\begin{array}{c} \lambda_{1,0}\\\lambda_{1,1}\end{array}\right)=-
\left(\begin{array}{cc} 3 & 4 \\ 3 & 6\end{array}\right)
\left(\begin{array}{cc} 0 & 1\\ 1 & 0\end{array}\right)
\left(\begin{array}{r} 1\\-1\end{array}\right),
\end{eqnarray*}
whose solution is $\lambda_{1,0}=-3$ and $\lambda_{1,1}=2$. Summarizing:
\begingroup
\renewcommand*{\arraystretch}{1.5}
\begin{eqnarray*}
\begin{array}{|c|c|c|c|c|r|c|}\hline
s_1=r+q & \phi_{s_1} & W_{s_1} & h_1 & \lambda_{1,0} & \lambda_{1,1}\\\hline 
82 & (1,1,3) & \langle xyz^3,y^3z\rangle & 
\lambda_{1,0}xyz^3+\lambda_{1,1}y^3z & -3 & 2 \\\hline
\end{array}
\end{eqnarray*}
\endgroup
Thus, $h_1=-3xyz^3+2y^3z$. Now, we turn to $h_2$. Consider the second system in Lemma~\ref{lemma-h-i}.
\begin{eqnarray*}
&&\left(B^{[\iota_r+1,\frac{n-1+\iota_r}{2}]}_{\{0\}\sqcup 
[\lfloor\frac{\iota_r+5}{2}\rfloor,\lfloor\frac{n}{2}\rfloor]}\right)^\top\cdot
\Theta_{\frac{n-1-\iota_r}{2}}\cdot\Lambda_2=-
\left(B^{[\iota_r,\frac{n-1+\iota_r}{2}]}
_{\left\{\left\lfloor\frac{n}{2}\right\rfloor\right\}}
\middle| \; 0\; \right)^{\top}\cdot \Theta_{\frac{n+1-\iota_r}{2}}\cdot 
\Lambda.  
\end{eqnarray*}
Substituting the values of $n=6$, $k=1$ and $\iota_r=n-k-2=3$, we obtain:
\begin{eqnarray*}
\left(\begin{array}{c} 1 \end{array}\right)
\left(\begin{array}{c} 1 \end{array}\right)
\left(\begin{array}{c} \lambda_{2,1} \end{array}\right)=-
\left(\begin{array}{cc} 1 & 4 \end{array}\right)
\left(\begin{array}{cc} 0 & 1\\ 1 & 0\end{array}\right)
\left(\begin{array}{r} 1\\-1\end{array}\right), 
\end{eqnarray*}
whose solution is $\lambda_{2,1}=-3$. Summarizing
\begingroup
\renewcommand*{\arraystretch}{1.5}
\begin{eqnarray*}
\begin{array}{|c|c|c|c|c|}\hline
s_2=r+\lfloor\frac{n}{2}\rfloor q & \phi_{s_2} & W_{s_2} & h_2 & \lambda_{2,1}\\\hline
92 & (5,0,1) & \langle x^5z,x^4y^2\rangle &
\lambda_{2,1}x^4y^2 & -3 \\\hline
\end{array}
\end{eqnarray*}
\endgroup
Thus, $h_2=-3x^4y^2$. Let us calculate $h_3$. Consider the third system in Lemma~\ref{lemma-h-i}.
\begin{multline*}
\left(B^{[0,\lfloor\frac{\iota_r+1}{2}\rfloor]}
_{[0,\lfloor\frac{\iota_r+1}{2}\rfloor]}\right)^\top\cdot
\Theta_{\lfloor\frac{\iota_r+3}{2}\rfloor}\cdot\Lambda_3=\\
-\left(B^{\left[\iota_r,\frac{n-1+\iota_r}{2}\right]}
_{\left[\lfloor\frac{n+2}{2}\rfloor,
\lfloor\frac{n+2}{2}\rfloor+\lfloor\frac{\iota_r+1}{2}\rfloor\right]}\right)^{\top}
\cdot\Theta_{\frac{n+1-\iota_r}{2}}\cdot\Lambda
-\left(\begin{array}{c}
B^{\left[\iota_r+1,\frac{n-1+\iota_r}{2}\right]}_
{\left[1,\lfloor\frac{\iota_r+3}{2}\rfloor\right]}
\end{array}\right)^{\top}
\cdot\Theta_{\frac{n-1-\iota_r}{2}}\cdot\Lambda_2.
\end{multline*}
Substituting the values of $n=6$, $k=1$ and $\iota_r=n-k-2=3$, we obtain:
\begin{eqnarray*}
\left(\begin{array}{ccc}1&1&1\\0&1&2\\0&0&1\end{array}\right)\!
\left(\begin{array}{ccc}0&0&1\\0&1&0\\1&0&0\end{array}\right)\!
\left(\begin{array}{c}\lambda_{3,0}\\\lambda_{3,1}\\\lambda_{3,2}\end{array}\right)=-
\left(\begin{array}{cc}0&1\\0&0\\0&0\end{array}\right)
\left(\begin{array}{cc}0&1\\1&0\end{array}\right)
\left(\begin{array}{r} 1\\-1\end{array}\right)+
\left(\begin{array}{c}4\\6\\4\end{array}\right)\!
\left(\begin{array}{c}1\end{array}\right)\!
\left(\begin{array}{c}3\end{array}\right),
\end{eqnarray*}
whose solution is $\lambda_{3,0}=12$, $\lambda_{3,1}=-6$ and $\lambda_{3,2}=5$. Summarizing,
\begingroup
\renewcommand*{\arraystretch}{1.5}
{\small
\begin{eqnarray*}
\begin{array}{|c|c|c|c|c|c|c|c|}\hline
s_3=s_2+q & \phi_{s_3} & W_{s_3} & h_3 &\lambda_{3,0} &\lambda_{3,1}& \lambda_{3,2}\\\hline  
97 & (2,1,3) & \langle x^2yz^3,xy^3z^2,y^5z\rangle &
\lambda_{3,0}x^2yz^3+\lambda_{3,1}xy^3z^2+\lambda_{3,2}y^5z & 12 &-6 & 5\\\hline
\end{array}
\end{eqnarray*}
}
\endgroup
Thus, $h_3=12x^2yz^3-6xy^3z^2+5y^5z$. Let us calculate $h_4$. Consider the fourth
system in Lemma~\ref{lemma-h-i}.
\begin{eqnarray*}
&&\left(B_{[0,\lfloor\frac{\iota_r-2}{2}\rfloor]}
^{[0,\lfloor\frac{\iota_r-2}{2}\rfloor]}\right)^\top\cdot
\Theta_{\lfloor\frac{\iota_r}{2}\rfloor}\cdot\Lambda_4=
-\left(B_{[\lfloor\frac{n+2}{2}\rfloor,\frac{n-1+\iota_r}{2}]}^{[\iota_r+1,\frac{n-1+\iota_r}{2}]}\right)^{\top}\cdot
\Theta_{\frac{n-1-\iota_r}{2}}\cdot\Lambda_2.
\end{eqnarray*}
Using $n=6$, $k=1$ and $\iota_r=n-k-2=3$, we obtain:
\begin{eqnarray*}
\left(\begin{array}{c}1\end{array}\right)
\left(\begin{array}{c}1\end{array}\right)
\left(\begin{array}{c}\lambda_{4,3}\end{array}\right)=-
\left(\begin{array}{c}1\end{array}\right)
\left(\begin{array}{c}1\end{array}\right)
\left(\begin{array}{c}-3\end{array}\right),
\end{eqnarray*}
so that $\lambda_{4,3}=3$. Summarizing,
\begingroup
\renewcommand*{\arraystretch}{1.5}
\begin{eqnarray*}
\begin{array}{|c|c|c|c|c|c|}\hline
s_4=s_3+\lfloor\frac{n}{2}\rfloor q & \phi_{s_4} & W_{s_4} & h_4 & \lambda_{4,2}\\\hline 
112 & (3,1,3) & \langle x^3yz^3,x^2y^3z^2,xy^5z,y^7\rangle & \lambda_{4,3}y^5 & 3\\\hline
\end{array}
\end{eqnarray*}
\endgroup
Therefore, $f_1=x^4z-x^3y^2-3xyz^3+2y^3z^2-3x^4y^2+12x^2yz^3-6xy^3z^2+5y^5z+3y^7$. 
In a similar vein, we would obtain, 
\begin{eqnarray*}
f_2 &=& x^3yz -x^2y^3 -2xz^4 +y^2z^3 -x^3y^3 +x^2z^4 +xy^2z^3 +y^4z^2,\\
f_3 &=& x^3z^2 -3x^2y^2z +2xy^4 +yz^4 -x^2y^4 +xyz^4 +y^3z^3,\\
f_4 &=& xy^3z -y^5 -z^5,\\
f_5 &=& x^2yz^2 -2xy^3z +y^5 -x^6 +11x^3yz^2-13x^2y^3z +8xy^5 +4x^2y^5 -3xy^2z^4 +y^4z^3,\\
f_6 &=& x^2z^3 -2xy^2z^2 +y^4z -x^5y +10x^2y^2z^2 -10xy^4z +5y^6 +5xy^6 -4y^3z^4. 
\end{eqnarray*}

\end{example}

\begin{remark}\label{code}
We have developed a Python code that computes the minimal generating set 
$f_1,\ldots,f_n$ of $\ker(\rho)$. It is a pleasure to thank Oriol Almirall 
for his help to the development of this code.

We briefly explain the structure of the code.

First, we define the function\verb| build_pascal(n)|, which constructs the $n\times n$ lower-triangular Pascal matrix. This matrix will be used to compute the binomial determinants in a more efficient way.  

Next, we create the function \verb|bin_det(I,J)| that computes the binomial matrix $B^{I}_{J}$ and its determinant $b^I_J$, when the sets $I$ and $J$ (defined as Python lists) have the same length.  
By default, this function uses the Pascal matrix defined previously so to compute more efficiently. However, the same results can be obtained from the SymPy command \verb|sy.binomial| by calling \verb|bin_det(I,J,False)|.

Finally, the function \verb|get_minimal_gen(n)| computes the minimal generating set 
$f_1,\ldots,f_n$ for the ideal $\ker(\rho)$, when $n\geq 5$. Note that the equalities in Definition~\ref{def-gk}
for the specific case $n=5$ and for the general case $n$ odd, $n\geq 7$, coincide 
if we allow in the general case $n$ to be equal to $5$.

The acquisition of $f_1,\ldots,f_n$ proceeds in two main phases. 

First, we compute the $\sigma$-leading forms $g_1,\ldots,g_n$. They are obtained directly by applying Definition~\ref{def-gk}. Here, we need to calculate some binomial determinants. We do so by using the command \verb|bin_det(I,J)[2]| (calling the \verb|bin_det| function with the 
second argument obtain just the determinant). 

Second, we define the tails. Here we solve the linear 
systems described in Section~\ref{sec-Vr}.
Each system is expressed in terms of binomial matrices, which again are computed by 
using the \verb|bin_det| function (in this case, called with the first argument).
The systems are solved either using the \verb|solve| command or alternatively 
the \verb|gauss_jordan_solve| command from the SymPy library.

As said before, the main advantage of our method is that we do not have to 
calculate any Gr\"obner basis. Computing a minimal generating set 
for $n\geq 60$ using a Gr\"obner basis approach often leads to timeout,
whereas our Python algorithm can still run efficiently. 

For the implementation details and the full code, we refer the reader to the GitHub repository:
\begin{center}
\url{https://github.com/laura-gonzalez-hernandez/minimal-generators}
\end{center}
\end{remark}

\end{document}